\documentclass[11pt, a4paper,reqno]{amsart}
\usepackage[a4paper,margin=0.85in]{geometry}
\usepackage[utf8]{inputenc}
\usepackage{amsmath,amsfonts,amsthm,amssymb, amsopn,mathtools}
\usepackage[final, pdftex, pdfpagelabels, pdfstartview = {FitH}, bookmarks, colorlinks, plainpages = false, linktoc=all, linkcolor=magenta, citecolor=cyan, urlcolor=blue,filecolor=black]{hyperref}

\usepackage{xcolor}
\usepackage{tikz}
\usepackage[textwidth=20mm]{todonotes}
\usepackage{cleveref}
\usepackage{autonum}
\usepackage{pdfrender, tikz-cd, rotating}
\usepackage{fancyhdr}
\usepackage[shortlabels]{enumitem}
\usepackage{etoolbox, array, setspace}
\usepackage{lscape}
\usepackage{tabularx}
\usepackage{longtable} % para tablas largas
\usepackage{caption}
\usepackage{subcaption}
\usepackage{float}
\usepackage{genyoungtabtikz}
\usetikzlibrary{patterns,patterns.meta}
\usetikzlibrary{arrows.meta,calc}
\makeatletter
\let\amsart@tocsection\tocsection
\let\amsart@tocchapter\tocchapter
\let\tocsection\@undefined
\let\tocchapter\@undefined
\makeatother
\usepackage[nottoc]{tocbibind}
\makeatletter
\let\tocsection\amsart@tocsection
\let\tocchapter\amsart@tocchapter
\providecommand{\@openbib@code}{}
\makeatother

\usepackage{ytableau} % para insertar tableau

\usepackage{url}
\newcounter{rowcntr}[table]
\renewcommand{\therowcntr}{\arabic{rowcntr}}
\newcolumntype{N}{>{\refstepcounter{rowcntr}\therowcntr}c}
\AtBeginEnvironment{table}{\setcounter{rowcntr}{0}}

\usepackage[margin=2cm]{caption}
\usepackage{stackengine}
\usepackage{xparse}

\newcommand{\DE}[3]{D({#1}_{#2,#3}) }

\ExplSyntaxOn
\NewDocumentCommand{\D}{>{\SplitArgument{2}{,}}m}{\plaza_D_aux#1}
\NewDocumentCommand{\plaza_D_aux}{mmm}
  { \DE{\tl_trim_spaces:n{#1}}{\tl_trim_spaces:n{#2}}{\tl_trim_spaces:n{#3}} }
\ExplSyntaxOff

\newcommand{\bfN}{\mathbf{N}}
\newcommand{\bfM}{\mathbf{M}}
\newcommand{\bfH}{\underline{\Tilde{\mathbf{H}}}}

\newcommand{\sph}{\mathcal{H}^{\textbf{sph}}}

\newcommand{\OP}[1]{P_{#1}}
\newcommand{\OQ}[1]{Q_{#1}}
\newcommand{\labs}{\left|}
\newcommand{\rabs}{\right|}
\newcommand{\bfP}{\mathbf{P}}
\newcommand{\hgt}{\text{ht}}
\newcommand{\pint}[2]{\langle #1 , #2 \rangle}

\newtheorem{introtheorem}{Theorem}

\newtheorem{introconjecture}{Conjecture}

\newtheorem{theorem}{Theorem}[section]
\newtheorem{claim}[theorem]{Claim}
\newtheorem{lemma}[theorem]{Lemma}
\newtheorem{remark}[theorem]{Remark}
\newtheorem{definition}[theorem]{Definition}
\newtheorem{example}[theorem]{Example}

\newtheorem{corollary}[theorem]{Corollary}
\newtheorem{proposition}[theorem]{Proposition}

\newenvironment{proof.}{\paragraph{\textit{Proof:}}}{\hfill$\blacksquare$}

\newenvironment{myproof}[1][Proof]{%
  \par\addvspace{\medskipamount}%
  \noindent\textbf{#1.} \ignorespaces%
}{%
  \hfill$\square$\par\addvspace{\medskipamount}%
}

\usepackage{multicol}
\newcommand{\interior}[1]{%
  {\kern0pt#1}^{\mathrm{o}}%
}

\tikzset{
  pics/square/.default={1},
  pics/square/.style = {
    code = {
    \draw[pic actions] (0,0) rectangle (#1,#1);
    }
  }
}

\newcommand{\grayrect}[2][5]{%
\begin{tikzpicture}[scale=0.45]
    \foreach \r/\c in {#2}{
        \pgfmathtruncatemacro{\x}{\c-6}
        \fill[gray!35] (\x,\r-1) rectangle ++(1,1);
    }

    \draw (0,0) rectangle (3,#1);
    \foreach \x in {1,2} {
        \draw (\x,0) -- (\x,#1);
    }
    \foreach \y in {1,...,#1} {
        \draw (0,\y) -- (3,\y);
    }

    \node at (0.5,#1+0.45) {$6$};
    \node at (1.5,#1+0.45) {$7$};
    \node at (2.5,#1+0.45) {$8$};
\end{tikzpicture}%
}

\def\alcoves{\begin{tikzpicture}
    \fill[blue!30] (0,0) -- (60:10 ) -- (120:10 ) -- cycle;
    \fill[red] (0,0) -- (240:2 ) -- (300:2 ) -- cycle;
    \foreach \i in {0, 1, ..., 5} {
      \draw[->, ultra thick, blue] (0, 0) -- (30 + \i*60:3.4641);
      }
      
    \foreach \m in {0,1,2,3,4,5} {
       \draw[red] (0,0) -- (\m*60:10);
    }
    
      \draw[->, ultra thick, blue] (0, 0) -- (1*60:2);
      \draw[->, ultra thick, blue] (0, 0) -- (2*60:2);
      
      \pgfmathsetmacro\cx{2 * cos(120)}
      \pgfmathsetmacro\sx{2 * sin(120)}
      \pgfmathsetmacro\sy{2 * sin(60)}
      \pgfmathsetmacro\cy{2 * cos(60)}

      \foreach \k in {-5,...,5} {
        \foreach \j in {-5,...,5} {
            \pgfmathsetmacro\g{\k+\j};
            \ifthenelse{ 6 > \g \AND  -6 < \g  }{\filldraw (\k*\cx+\j*\cy,\k*\sx+\j*\sy) circle (2pt);}{}
        } 
      }
    \foreach \m / \n in {2/8,4/6,6/4,8/2,10/0} {
        \foreach \r in {0,1,2,3,4,5} {
            \draw[shorten >=-\n cm,shorten <=-\n cm] (\r*60 - 60:\m) -- (\r*60:\m);
        }
    }
\end{tikzpicture}}

\def\positiveroots{
\begin{tikzpicture}[baseline={(-5,-3)}]
\draw (-2,1) -- (-11.5,-8.5) -- (7.5, -8.5) -- (-2,1);
\foreach \y in {0,...,10}
    \foreach \x in {0,..., \y }{
        \pgfmathsetmacro\cx{int(\x + 1)};
        \pgfmathsetmacro\cy{int(10 - \y + \x)};
        \pgfmathsetmacro\z{int(10 - \y - 1)};
        \ifthenelse{ \z > 0 }{ 
        \draw (-\y + 2*\x-2,-\y) node {\huge $\alpha_{\cx,\cy}$};}{}
       }

\end{tikzpicture}
}

\def\positiverootsgeqcincosiete{
\begin{tikzpicture}[baseline={(-5,-3)}]
\draw (-2,1) -- (-9.5,-6.5) -- (-2.5,-6.5) -- (-1.5,-7.5) -- (6.5, -7.5) -- (-2,1);
\foreach \y in {0,...,10}
    \foreach \x in {0,..., \y }{
        \pgfmathsetmacro\cx{int(\x + 1)};
        \pgfmathsetmacro\cy{int(10 - \y + \x)};
        \pgfmathsetmacro\z{int(10 - \y - 1)};
        \ifthenelse{ \z > 1 }{ 
        \draw (-\y + 2*\x-2,-\y) node {\huge $\alpha_{\cx,\cy}$};}{}
       }

\end{tikzpicture}
}

\def\actionofsseisincincosiete{
\begin{tikzpicture}[baseline={(-5,-3)}]
\draw[fill, green!20!] (-8,-5) -- (-6,-3) -- (-2,-7) -- (-4,-9) -- (-8,-5);
\draw[fill, green!20!] (3,-4) -- (5,-6) -- (2,-9) -- (0,-7)-- (3,-4);
\draw[blue] (2,-8) node {\huge $\alpha_{7,8}$};
\draw[blue] (-4,-8) node {\huge $\alpha_{4,5}$};
\draw[<->] (-6.3,-4.3) -- (-6.7,-4.7); 
\draw[<->] (-5.3,-5.3) -- (-5.7,-5.7);
\draw[<->] (-4.3,-6.3) -- (-4.7,-6.7);
\draw[<->] (-3.3,-7.3) -- (-3.7,-7.7);

\draw[<->] (3.3,-5.3) -- (3.7,-5.7); 
\draw[<->] (2.3,-6.3) -- (2.7,-6.7);
\draw[<->] (1.3,-7.3) -- (1.7,-7.7);

\foreach \y in {0,...,10}
    \foreach \x in {0,..., \y }{
        \pgfmathsetmacro\cx{int(\x + 1)};
        \pgfmathsetmacro\cy{int(10 - \y + \x)};
        \pgfmathsetmacro\z{int(10 - \y - 1)};
        \ifthenelse{ \z > 1 }{ 
        \draw (-\y + 2*\x-2,-\y) node {\huge $\alpha_{\cx,\cy}$};}{}
       }
\end{tikzpicture}
}

\def\excobkcero{\begin{tikzpicture}[baseline={(0,-7)}]

\foreach \y in {3,...,15}
    \foreach \x in {0,..., \y }{
        \pgfmathsetmacro\cx{int(\x + 1)};
        \pgfmathsetmacro\cy{int(15 - \y + \x)};
        \pgfmathsetmacro\z{int(15 - \y - 1)};
        \ifthenelse{ \z > 2 }{ 
        \draw (-\y + 2*\x ,-\y) node {\huge $\alpha_{\cx,\cy}$};}{}
       }
\draw[blue] (-3.5,-2.5) -- (-11.5,-10.5) -- (-6.3,-10.5) -- (-5.3,-11.5) -- (12.5,-11.5) -- (3.5,-2.5) -- (-3.5,-2.5);

\end{tikzpicture}}

\def\excobkuno{\begin{tikzpicture}[baseline={(0,-7)}]

\foreach \y in {3,...,15}
    \foreach \x in {0,..., \y }{
        \pgfmathsetmacro\cx{int(\x + 1)};
        \pgfmathsetmacro\cy{int(15 - \y + \x)};
        \pgfmathsetmacro\z{int(15 - \y - 1)};
        \ifthenelse{ \z > 2 }{ 
        \draw (-\y + 2*\x ,-\y) node {\huge $\alpha_{\cx,\cy}$};}{}
       }
\draw[blue] (-3.5,-2.5) -- (-12.5,-11.5) -- (-10.7,-11.5) -- (-9.7,-10.5) -- (-6.3,-10.5) -- (-5.3,-11.5) -- (12.5,-11.5) -- (3.5,-2.5) -- (-3.5,-2.5);

\end{tikzpicture}}

\def\excobkdos{\begin{tikzpicture}[baseline={(0,-10)}]

\foreach \y in {9,...,15}
    \foreach \x in {0,..., \y }{
        \pgfmathsetmacro\cx{int(\x + 1)};
        \pgfmathsetmacro\cy{int(15 - \y + \x)};
        \pgfmathsetmacro\z{int(15 - \y - 1)};
        \ifthenelse{ \z > 2 }{ 
        \draw (-\y + 2*\x ,-\y) node {\huge $\alpha_{\cx,\cy}$};}{}
       }
\draw[blue] (-9.5,-8) -- (-12.5,-11.5) -- (-10.7,-11.5) -- (-9.7,-10.5) -- (-6.3,-10.5) -- (-5.3,-11.5) -- (12.5,-11.5) -- (9.5,-8);

\draw[dashed,blue] (9.5,-8) -- (-9.5,-8);

\end{tikzpicture}}

\def\excobktres{\begin{tikzpicture}[baseline={(0,-10)}]

\foreach \y in {9,...,15}
    \foreach \x in {0,..., \y }{
        \pgfmathsetmacro\cx{int(\x + 1)};
        \pgfmathsetmacro\cy{int(15 - \y + \x)};
        \pgfmathsetmacro\z{int(15 - \y - 1)};
        \ifthenelse{ \z > 2 }{ 
        \draw (-\y + 2*\x ,-\y) node {\huge $\alpha_{\cx,\cy}$};}{}
       }
\draw[blue] (-9.5,-8) -- (-12.5,-11.5) -- (-8.7,-11.5) --  (-5.3,-11.5) -- (-4.7,-11.5) -- (-3.7,-10.5) -- (-0.3,-10.5) -- (0.7,-11.5) -- (12.7,-11.5) -- (9.5,-8);

\draw[dashed,blue] (9.5,-8) -- (-9.5,-8);

\end{tikzpicture}}

\def\excobkcuatro{\begin{tikzpicture}[baseline={(0,-10)}]

\foreach \y in {9,...,15}
    \foreach \x in {0,..., \y }{
        \pgfmathsetmacro\cx{int(\x + 1)};
        \pgfmathsetmacro\cy{int(15 - \y + \x)};
        \pgfmathsetmacro\z{int(15 - \y - 1)};
        \ifthenelse{ \z > 2 }{ 
        \draw (-\y + 2*\x ,-\y) node {\huge $\alpha_{\cx,\cy}$};}{}
       }
\draw[blue] (-9.5,-8) -- (-12.5,-11.5) -- (-8.7,-11.5) --  (-5.3,-11.5) -- (-2.3,-11.5) -- (1.7,-11.5) -- (2.3,-10.5) -- (5.7,-10.5) -- (6.3,-11.5) -- (7.7,-11.5) -- (10.3,-11.5) --  (12.7,-11.5) -- (9.5,-8);

\draw[dashed,blue] (9.5,-8) -- (-9.5,-8);

\end{tikzpicture}}

\def\excobkcinco{\begin{tikzpicture}[baseline={(0,-10)}]

\foreach \y in {9,...,15}
    \foreach \x in {0,..., \y }{
        \pgfmathsetmacro\cx{int(\x + 1)};
        \pgfmathsetmacro\cy{int(15 - \y + \x)};
        \pgfmathsetmacro\z{int(15 - \y - 1)};
        \ifthenelse{ \z > 2 }{ 
        \draw (-\y + 2*\x ,-\y) node {\huge $\alpha_{\cx,\cy}$};}{}
       }
\draw[blue] (-9.5,-8) -- (-12.5,-11.5) -- (-8.7,-11.5) --  (-5.3,-11.5) -- (-2.3,-11.5) -- (1.7,-11.5) -- (2.3,-10.5) -- (5.7,-10.5) -- (6.3,-11.5) -- (7.7,-11.5) -- (8.3,-10.5) -- (9.7,-10.5) -- (10.3,-11.5) --  (12.7,-11.5) -- (9.5,-8);

\draw[dashed,blue] (9.5,-8) -- (-9.5,-8);

\end{tikzpicture}}

\def\excobkseis{\begin{tikzpicture}[baseline={(0,-10)}]

\foreach \y in {9,...,15}
    \foreach \x in {0,..., \y }{
        \pgfmathsetmacro\cx{int(\x + 1)};
        \pgfmathsetmacro\cy{int(15 - \y + \x)};
        \pgfmathsetmacro\z{int(15 - \y - 1)};
        \ifthenelse{ \z > 2 }{ 
        \draw (-\y + 2*\x ,-\y) node {\huge $\alpha_{\cx,\cy}$};}{}
       }
\draw[blue] (-9.5,-8) -- (-12.5,-11.5) -- (-8.7,-11.5) --  (-5.3,-11.5) -- (3.7,-11.5) -- (4.3,-10.5) -- (5.7,-10.5) -- (6.3,-11.5) -- (7.7,-11.5) -- (8.3,-10.5) -- (9.7,-10.5) -- (10.3,-11.5) --  (12.7,-11.5) -- (9.5,-8);

\draw (12,-12) node {\huge $\alpha_{13,15}$};

\draw[dashed,blue] (9.5,-8) -- (-9.5,-8);

\end{tikzpicture}}

\def\exbonito{\begin{tikzpicture}[rotate=0]

  \foreach \y in {1,...,10}
    \foreach \x in {\y ,..., 10  }{
       \pic[fill=green!0] at (-1/2*\y +\x  ,\y) {square};
     }

 \foreach \y in {5,...,10}
    \foreach \x in {\y ,..., 10  }{
       \pic[fill=green!50] at (-1/2*\y +\x  ,\y) {square};
     }

 \foreach \x in {7,..., 10  }{
       \pic[fill=green!50] at (-1/2*4 +\x  ,4) {square};
     }
     
     \foreach \x in {1,..., 10 }{
\draw[] (\x ,1.5) node {\huge $\x$};
     }

\end{tikzpicture}}

\def\exbonitoB{\begin{tikzpicture}
  \foreach \y in {1,...,10}
    \foreach \x in {\y ,..., 10  }{
       \pic[fill=green!0] at (-1/2*\y +\x  ,\y) {square};
     }

 \foreach \x in {7,..., 9  }{
       \pic[fill=green!50] at (-1/2*4 +\x  ,4) {square};
     }

 \foreach \x in {8,..., 10  }{
       \pic[fill=green!50] at (-1/2*5 +\x  ,5) {square};
     }     

 \foreach \x in {9,..., 10  }{
       \pic[fill=green!50] at (-1/2*6 +\x  ,6) {square};
     }     

 \foreach \x in {10,..., 10  }{
       \pic[fill=green!50] at (-1/2*7 +\x  ,7) {square};
     }     

 \foreach \x in {5,..., 5  }{
       \pic[fill=green!50] at (-1/2*5 +\x  ,5) {square};
     }

 \foreach \x in {4,...,5  }{
       \pic[fill=red!50] at (-1/2*4 +\x  ,4) {square};
     }

 \foreach \x in {3,...,7  }{
       \pic[fill=red!50] at (-1/2*3 +\x  ,3) {square};
     }

 \foreach \x in {3,...,7  }{
       \pic[fill=red!50] at (-1/2*2 +\x  ,2) {square};
     }     
     
     \foreach \x in {1,..., 10 }{
\draw[] (\x ,1.5) node {\huge $\x$};
     }

\end{tikzpicture}}

\def\exLemmaBKA{\begin{tikzpicture}[rotate=0]

  \foreach \y in {1,...,10}
    \foreach \x in {\y ,..., 10  }{
       \pic[fill=green!0] at (-1/2*\y +\x  ,\y) {square};
     }

 \foreach \y in {5,...,10}
    \foreach \x in {\y ,..., 10  }{
       \pic[fill=green!50] at (-1/2*\y +\x  ,\y) {square};
     }

 \foreach \x in {8,..., 10  }{
       \pic[fill=green!50] at (-1/2*4 +\x  ,4) {square};
     }
     
     \foreach \x in {1,..., 10 }{
\draw[] (\x ,1.5) node {\huge $\x$};
     }

 \foreach \x in {4,...,4  }{
       \pic[fill=green!50] at (-1/2*4 +\x  ,4) {square};
     }     

\end{tikzpicture}}

\def\exLemmaBKB{\begin{tikzpicture}[rotate=0]

  \foreach \y in {1,...,10}
    \foreach \x in {\y ,..., 10  }{
       \pic[fill=green!0] at (-1/2*\y +\x  ,\y) {square};
     }

 \foreach \y in {5,...,10}
    \foreach \x in {\y ,..., 10  }{
       \pic[fill=green!50] at (-1/2*\y +\x  ,\y) {square};
     }

 \foreach \x in {9,..., 10  }{
       \pic[fill=green!50] at (-1/2*4 +\x  ,4) {square};
     }
     
     \foreach \x in {1,..., 10 }{
\draw[] (\x ,1.5) node {\huge $\x$};
     }

 \foreach \x in {4,...,5  }{
       \pic[fill=green!50] at (-1/2*4 +\x  ,4) {square};
     }     

\end{tikzpicture}}

\def\exLemmaBKC{\begin{tikzpicture}[rotate=0]

  \foreach \y in {1,...,10}
    \foreach \x in {\y ,..., 10  }{
       \pic[fill=green!0] at (-1/2*\y +\x  ,\y) {square};
     }

 \foreach \y in {5,...,10}
    \foreach \x in {\y ,..., 10  }{
       \pic[fill=green!50] at (-1/2*\y +\x  ,\y) {square};
     }

 \foreach \x in {10,..., 10  }{
       \pic[fill=green!50] at (-1/2*4 +\x  ,4) {square};
     }
     
     \foreach \x in {1,..., 10 }{
\draw[] (\x ,1.5) node {\huge $\x$};
     }

 \foreach \x in {4,...,6  }{
       \pic[fill=green!50] at (-1/2*4 +\x  ,4) {square};
     }     

\end{tikzpicture}}

\def\exLemmaBKD{\begin{tikzpicture}[rotate=0]

  \foreach \y in {1,...,10}
    \foreach \x in {\y ,..., 10  }{
       \pic[fill=green!0] at (-1/2*\y +\x  ,\y) {square};
     }

 \foreach \y in {5,...,10}
    \foreach \x in {\y ,..., 10  }{
       \pic[fill=green!50] at (-1/2*\y +\x  ,\y) {square};
     }

     \foreach \x in {1,..., 10 }{
\draw[] (\x ,1.5) node {\huge $\x$};
     }

 \foreach \x in {4,...,7  }{
       \pic[fill=green!50] at (-1/2*4 +\x  ,4) {square};
     }     

\end{tikzpicture}}

\newcommand{\dibu}[3]{
\pgfmathsetmacro\cx{int(#3) -int(#2) + 2};
\pgfmathsetmacro\cxu{int(\cx)-1};
\pgfmathsetmacro\cxd{int(#3)+ 1};
  \foreach \y in {1,...,#1}
    \foreach \x in {\y ,..., #1  }{
       \pic[fill=green!0] at (-1/2*\y +\x  ,\y) {square};
     }

 \foreach \y in {\cx,...,#1}
    \foreach \x in {\y ,..., #1  }{
       \pic[fill=green!50] at (-1/2*\y +\x  ,\y) {square};
     }

 \foreach \x in {\cxd,..., #1  }{
       \pic[fill=green!50] at (-1/2*\cxu  +\x  ,\cxu) {square};
     }
     
     \foreach \x in {1,..., #1 }{
\draw[] (\x ,1.5) node {\huge $\x$};
     }}

\newcommand{\agregof}[3]{
\pgfmathsetmacro\altx{int(#1)};
 \foreach \x in {#2,..., #3  }{
       \pic[fill=green!50] at (1/2*\altx  +\x-1  ,\altx) {square};
     }
}

\newcommand{\quitof}[3]{
\pgfmathsetmacro\altx{int(#1)};
 \foreach \x in {#2,..., #3  }{
       \pic[fill=white] at (1/2*\altx  +\x-1  ,\altx) {square};
     }
}

\newcommand{\quitoS}[3]{
\pgfmathsetmacro\altx{int(#1)};
\pgfmathsetmacro\xx{int(#2)};
 \foreach \x in {0,..., #3  }{
       \pic[fill=red] at (1/2*\altx  +\xx-1 -1/2*\x,\altx +\x) {square};
     }
}
\newcommand{\agregoA}[3]{
\pgfmathsetmacro\altx{int(#1)};
\pgfmathsetmacro\xx{int(#2)};
 \foreach \x in {0,..., #3  }{
       \pic[line width=4pt] at (1/2*\altx  +\xx-1 +1/2*\x,\altx +\x) {square};
     }
}

\newcommand{\agregoR}[3]{
\pgfmathsetmacro\altx{int(#1)};
 \foreach \x in {#2,..., #3  }{
       \pic[fill=red] at (1/2*\altx  +\x-1  ,\altx) {square};
     }
}

\newcommand{\peso}[9]{
 \foreach \x / \y in {1/1,2/0,3/0,4/0,5/0,6/0,7/0,8/#1,9/#2,10/#3,11/#4,12/#5,13/#6,14/#7,15/#8,16/#9}{
       \draw[] (\x ,0.5) node {\huge $\y$ };
     }
}

\numberwithin{equation}{section}
\numberwithin{theorem}{section}

\title{Positivity of Pre-Canonical Bases for Spherical Hecke Algebras}

\newcommand{\myauthor}[3]{%
  \author{#1}
  \address{#2}
  \email{#3}
}

\myauthor
  {David Plaza}
  {Instituto de Matem\'aticas, Universidad de Talca, Chile}
  {dplaza@utalca.cl}

\myauthor
  {Yamil Sagurie}
  {Instituto de Matem\'aticas, Universidad de Talca, Chile}
  {fsagurie@gmail.com}

\date{\today }

\begin{document}

\maketitle

\begin{abstract}
We introduce a new algorithm for computing Kostka–Foulkes polynomials.
It arises from a combinatorial procedure that computes the transition coefficients between successive pre-canonical bases for spherical Hecke algebras, introduced by Libedinsky, Patimo, and the first author, in order to interpolate between the standard and canonical bases of the spherical Hecke algebras.
Using this algorithm, we prove that the coefficients of the polynomials in the transition matrix from any pre-canonical basis to the next one has nonnegative coefficients. This establishes the positivity conjecture for pre-canonical bases and, moreover, yields a strictly stronger result that uncovers additional positivity phenomena beyond the original scope of the conjecture.
As a by-product, this method yields a combinatorial algorithm for computing Lascoux’s atomic decomposition of canonical basis elements.
%Last but not least
Finally, our results show that, after applying the Satake isomorphism and dualizing with respect to the Hall inner product, pre-canonical bases correspond to a new family of Schur-positive symmetric functions.
\end{abstract}

%\tableofcontents

\section{Introduction}

The computation of weight multiplicities in finite-dimensional simple representations of Lie algebras and algebraic groups is a classical problem in representation theory and algebraic combinatorics, going back to the Weyl character formula. Lusztig~\cite{lusztig1983singularities} introduced a $q$-analogue of these multiplicities, denoted by $K_{\lambda,\mu}(q)$. From a representation-theoretic point of view, the coefficients of $K_{\lambda,\mu}(q)$ record the dimensions of the graded pieces of the Brylinski--Kostant filtration on the $\mu$-weight space of the irreducible representation $L(\lambda)$~\cite{brylinski1989limits}.

In type~$A$, Lusztig's $q$-weight multiplicities are the classical Kostka--Foulkes polynomials. Lascoux and Sch\"utzenberger~\cite{lascoux1978conjecture} gave a celebrated combinatorial formula for these polynomials in terms of the \emph{charge} statistic on semistandard Young tableaux. More recently, Choi, Kim, and Lee~\cite{choilusztig} announced a combinatorial formula for $K_{\lambda,\mu}(q)$ in type~$C$, expressed through energy functions on Kirillov--Reshetikhin crystals, thereby extending the role played by charge in type~$A$.

Several other approaches to $q$-weight multiplicities are available, including Kostant's multiplicity formula~\cite{kostant1958formula}, Kazhdan--Lusztig polynomials~\cite{Lusztig1979}, Littelmann's path model~\cite{littelmann1995paths}, and Sahi's recursive algorithm~\cite{sahi2000new}. 
One of the purposes of this paper is to add a new combinatorial algorithm to this distinguished list.

%%%%%%%%%%%%%%%%%%%%%%%%%%%%%%%%%%%%%%%%%%%%%%%%%%%%%%%%%%%%%%%%%%%%%%%%
\subsection{Atomic decomposition}

Let $\Phi$ be an irreducible root system. 
We fix a system of positive roots $\Phi^+$ and denote by $\Delta$ the corresponding set of simple roots. 
Let $X$ be the corresponding weight lattice and let $X^+$ be the set of dominant weights.
Let $W_f$ and $W_a$ be the finite and affine Weyl groups attached to $\Phi$, respectively.

To each dominant weight $\lambda\in X^+$ we associate an element $\theta(\lambda)\in W_a$. Under the standard identification of elements of $W_a$ with alcoves, $\theta(\lambda)$ corresponds to the alcove obtained by translating by $\lambda$ the alcove associated with the longest element $w_0\in W_f$; see \Cref{def: theta lambda} for the precise definition.

Let $\mathcal{H}=\mathcal{H}(W_a)$ be the Hecke algebra of $W_a$. It is a $\mathbb{Z}[q^{\frac{1}{2}},q^{-\frac{1}{2}}]$-algebra with standard basis $\{\mathbf{H}_w\mid w\in W_a\}$ and canonical basis $\{\underline{\mathbf{H}}_w\mid w\in W_a\}$. For $\lambda\in X^+$ we write $\underline{\mathbf{H}}_\lambda=\underline{\mathbf{H}}_{\theta(\lambda)}$.

The spherical Hecke algebra, denoted by $\mathcal{H}^{\textbf{sph}}$, is a submodule of $\mathcal{H}$. As a $\mathbb{Z}[q^{\frac{1}{2}},q^{-\frac{1}{2}}]$-module, it is free with canonical basis $\{\underline{\mathbf{H}}_\lambda\mid \lambda\in X^+\}$. It also has a standard basis $\{\mathbf{H}_\lambda\mid \lambda\in X^+\}$. Here $\mathbf{H}_\lambda$ is not the standard Hecke basis element $\mathbf{H}_{\theta(\lambda)}$, but rather a sum of standard basis elements indexed by the double coset associated with $\theta(\lambda)$; see \cref{eq: defi standard basis} for details.

The transition coefficients between these two bases are Lusztig's $q$-weight multiplicities:
\begin{equation} \label{eq: intro A}
    \underline{\mathbf{H}}_{\lambda} =\sum_{\mu \in X^+}  K_{\lambda,\mu}(q) \mathbf{H}_\mu . 
\end{equation}

There is a third basis $\{\bfN_\lambda\mid \lambda\in X^+\}$ of $\mathcal{H}^{\textbf{sph}}$, which we call the \emph{atomic basis}. It is defined by
\begin{equation}  \label{eq: intro B}
\mathbf{N}_{\lambda}
= \sum_{\substack{\mu \in X^{+}\\ \mu \le \lambda}}
  q^{\operatorname{ht}(\lambda-\mu)}  \mathbf{H}_{\mu},
\end{equation}
where $\leq$ denotes the dominance order on $X$, and $\operatorname{ht}$ denotes the height of an element of $\mathbb{Z}\Phi$.

We may therefore write
\begin{equation} \label{eq: intro C}
    \underline{\mathbf{H}}_\lambda \,=\,\sum_{\mu\in X^{+}} a_{\mu , \lambda}(q^{\frac{1}{2}})  \bfN_\mu ,
\end{equation}
with $a_{\mu,\lambda}(q^{\frac{1}{2}})\in \mathbb{Z}[q^{\frac{1}{2}}]$. We call these polynomials \emph{atomic polynomials}. We say that $\underline{\mathbf{H}}_\lambda$ admits an atomic decomposition if
\[
    a_{\mu,\lambda}(q^{\frac{1}{2}})\in \mathbb{Z}_{\geq 0}[q^{\frac{1}{2}}]
\]
for all $\mu\in X^+$.

In type~$A$, the existence of an atomic decomposition follows from work of Lascoux~\cite{lascoux1989cyclic}, Shimozono~\cite{shimozono2001multi}, and Lecouvey--Lenart~\cite{lecouvey2021atomic}. More precisely, these sources prove the Satake counterpart of the atomic decomposition used here. The Satake transform is an isomorphism between the spherical Hecke algebra and the ring of symmetric functions; see, for instance, \cite{stembridgekostka}. Under this isomorphism, the canonical basis corresponds to Weyl characters, hence to Schur functions in type~$A$; the standard basis corresponds to Hall--Littlewood polynomials; and the atomic basis corresponds to the atoms introduced by Lascoux in~\cite{lascoux1989cyclic}.

Although these results imply, in principle, that the atomic polynomials can be computed, explicit computation remains difficult in practice. Our first main contribution is the following.

\begin{introtheorem}\label{Teo A}
    In type $A$, we construct a combinatorial algorithm for computing the atomic polynomials in \eqref{eq: intro C}. In particular, we give a new proof of the existence of an atomic decomposition in type $A$.
\end{introtheorem}

Combining \eqref{eq: intro A}, \eqref{eq: intro B}, and \eqref{eq: intro C}, one sees that the atomic polynomials determine the Kostka--Foulkes polynomials. Thus \Cref{Teo A} implies the following consequence.

\begin{introtheorem} \label{Teo B}
     In type $A$, we give a new combinatorial algorithm for computing Kostka--Foulkes polynomials.
\end{introtheorem}

\subsection{Pre-canonical bases}

The main tools in the proof of \Cref{Teo A} are the pre-canonical bases for spherical Hecke algebras introduced by Libedinsky, Patimo, and the first author~\cite{LPP}.

For an integer $i\geq 2$, let $\Phi^{\geq i}$ be the set of positive roots of height at least $i$; that is,
\[
    \Phi^{\geq i}=\{\alpha\in\Phi^+\mid \operatorname{ht}(\alpha)\geq i\}.
\]
For a subset $I\subset \Phi^+$, set $\Sigma_I=\sum_{\alpha\in I}\alpha$.

A weight $\lambda$ is called singular if there exists a reflection $t\in W_f$ such that $t(\lambda+\rho)=\lambda+\rho$, where $\rho$ denotes the half-sum of the positive roots. Otherwise, $\lambda$ is called regular. If $\lambda\in X$ is regular, let $w_\lambda\in W_f$ be the unique element such that $w_\lambda\cdot \lambda\in X^+$, where the dot action is defined by
\[
    w\cdot \lambda=w(\lambda+\rho)-\rho.
\]
For $\lambda\in X$, define
\begin{equation}  \label{tilde de schur}
    \widetilde{\underline{\mathbf{H}}}_\lambda = \begin{cases}
        (-1)^{\ell (w_{\lambda})} \underline{\mathbf{H}}_{w_{\lambda}\cdot \lambda}, & \mbox{if } \lambda \mbox{ is regular;}\\
        0, & \mbox{if } \lambda \mbox{ is singular.}
    \end{cases}
\end{equation}

\begin{definition}\rm \label{defi pre-canonical}
For $i\geq 2$, the $i$-th pre-canonical basis $\mathcal{N}^{i}=\{\bfN_{\lambda }^i\mid \lambda\in X^{+}\}$ is defined by
\begin{equation}  \label{def: pre-canonical}
    \bfN_{\lambda}^{i} = \sum_{I\subset \Phi^{\geq i}} (-q)^{|I|} \widetilde{\underline{\mathbf{H}}}_{\lambda - \Sigma_I}.
\end{equation}  
\end{definition}

The motivation for this definition is the anti-atomic formula~\cite[Theorem 3.1]{LPP}, which asserts that $\bfN_\lambda=\bfN_\lambda^2$ for all $\lambda\in X^+$. Thus the second pre-canonical basis is precisely the atomic basis. This formula expresses atomic basis elements in terms of the canonical basis, and it holds in every affine type.

At the other end of the interpolation, if $h$ is the height of the highest root in $\Phi$, then $\bfN_\lambda^{m}=\underline{\mathbf{H}}_\lambda$ for all $m>h$ and all $\lambda\in X^+$. Hence the pre-canonical bases interpolate between the atomic basis and the canonical basis. %If one also takes \eqref{eq: intro B} into account, they may be viewed as interpolating from the standard basis to the canonical basis through the atomic basis.
Together with \eqref{eq: intro B}, this shows that the sequence of pre-canonical bases interpolates from the standard basis to the canonical basis, passing through the atomic basis at $i=2$.

Based on extensive computations, the authors of~\cite{LPP} proposed the following positivity conjecture.

\begin{introconjecture}   \label{Conj A}
    In type $A$, each element of $\mathcal{N}^{i+1}$ expands positively in the basis $\mathcal{N}^i$.
    
    Equivalently, if
    \begin{equation}
        \bfN_{\lambda}^{i+1} =\sum_{\mu\in X^+} a_{\mu , \lambda}^i (q)  \bfN_{\mu}^i,
    \end{equation}
    then $a_{\mu , \lambda}^i(q)\in\mathbb{Z}_{\geq 0}[q]$ for all $\lambda,\mu\in X^+$ and all $i\geq 2$.
\end{introconjecture}

Our second main contribution is the following theorem.

\begin{introtheorem} \label{Teo C}
    In type $A$, we give a combinatorial algorithm for computing the polynomials $a_{\mu , \lambda}^i(q)$. Using this algorithm, we prove that $a_{\mu , \lambda}^i(q)\in\mathbb{Z}_{\geq 0}[q]$. In particular, we prove \Cref{Conj A}.
\end{introtheorem}

We remark that \Cref{Teo A} and \Cref{Teo B} follow readily from \Cref{Teo C}.

We now outline the proof of \Cref{Teo C}. Fix a root system $\Phi$ of type $A_n$, and let $\Delta=\{\alpha_1,\ldots,\alpha_n\}$ be the set of simple roots. For $1\leq i\leq j\leq n$, set
\begin{equation}
     \alpha_{i,j} =\sum_{k=i}^j \alpha_k \in \Phi^+.
\end{equation}
Thus $\operatorname{ht}(\alpha_{i,j})=j-i+1$. We order $\Phi^{\geq 2}$ by declaring that $\alpha_{i,j}\preceq \alpha_{a,b}$ if either
\begin{itemize}
    \item $\operatorname{ht}(\alpha_{i,j})<\operatorname{ht}(\alpha_{a,b})$, or
    \item $\operatorname{ht}(\alpha_{i,j})=\operatorname{ht}(\alpha_{a,b})$ and $i\leq a$.
\end{itemize}

Define
\[
    \Phi^{\succeq \alpha_{i,j}}=\{\alpha\in\Phi^{\geq 2}\mid \alpha\succeq \alpha_{i,j}\},
    \qquad
    \Phi^{\succ \alpha_{i,j}}=\{\alpha\in\Phi^{\geq 2}\mid \alpha\succ \alpha_{i,j}\}.
\]
We also set
\begin{equation}
    \bfM_{\lambda}^{\succeq \alpha_{i,j}}  =  \sum_{I\subset \Phi^{\succeq \alpha_{i,j} }} (-q)^{|I|} \widetilde{\underline{\mathbf{H}}}_{\lambda - \Sigma_I} 
    \qquad \mbox{and} \qquad 
    \bfM_{\lambda}^{\succ \alpha_{i,j}}  =  \sum_{I\subset \Phi^{\succ \alpha_{i,j} }} (-q)^{|I|} \widetilde{\underline{\mathbf{H}}}_{\lambda - \Sigma_I}.
\end{equation}

Ideally, one would like to express $\bfM_{\lambda}^{\succ \alpha_{i,j}}$ as a positive linear combination of the elements
\[
    \{\bfM_{\mu}^{\succeq \alpha_{i,j}}\mid \mu\in X^+\}.
\]
Iterating such an expression would prove \Cref{Teo C}, since
\begin{equation}
    \Phi^{\succ \alpha_{n-i+1,n}} =\Phi^{\geq i+1}
    \qquad \mbox{and} \qquad
    \Phi^{\succeq \alpha_{1,i}} = \Phi^{\geq i};
\end{equation}
hence
\begin{equation}\label{eq: M es N intro}
    \bfN_{\lambda}^{i+1} = \bfM_{\lambda}^{\succ \alpha_{n-i+1,n}} 
    \qquad \mbox{and} \qquad 
    \bfN_{\lambda}^i = \bfM_{\lambda}^{\succeq \alpha_{1,i}}, \mbox{ for all } 2\leq i\leq n.
\end{equation}

However, this direct strategy does not work: such a decomposition is difficult to obtain and, more importantly, it is not positive in general. The key point of the paper is to replace this unattainable decomposition by a more flexible positive statement.

\begin{introtheorem} \label{Teo D}
Let $\alpha_{i,j} \in \Phi^{\geq 2}$ and $\lambda \in X^+$. Then there exist polynomials $b_{\mu, \lambda }^k(q) \in \mathbb{Z}_{\geq 0}[q]$ such that
\begin{equation} \label{eq: intro descomposition}
    \bfM_{\lambda}^{\succ \alpha_{i,j} } = \sum_{k=0}^{i-1}  \sum_{\substack{\mu \in X^{+}\\ \mu \le \lambda}} b_{\mu, \lambda }^k(q) \bfM_{ \mu }^{\succeq \alpha_{i-k,j-k}}.
\end{equation}
Moreover, we give a combinatorial algorithm for computing the polynomials $b_{\mu,\lambda}^k(q)$.
\end{introtheorem}

Repeatedly applying \Cref{Teo D} yields the following strengthening of \Cref{Teo C}.

\begin{introtheorem}  \label{Teo E}
    Let $\alpha_{i,j}\in\Phi^{\geq 2}$, let $h=j-i+1$, and let $\lambda\in X^+$. Then $\bfM_{\lambda}^{\succ \alpha_{i,j}}$ can be written as a positive linear combination of elements of the $h$-th pre-canonical basis.
    
    In particular, $\bfN^{h+1}_\lambda=\bfM_{\lambda}^{\succ \alpha_{n-h+1,n}}$ can be written as a positive linear combination of the $h$-th pre-canonical basis, which is precisely the assertion of \Cref{Teo C}.
\end{introtheorem}

%%%%%%%%%%%%%%%%%%%%%%%%%%%%%%%%%%%%%%%%%%%%%%%%%%%%%%%%%%%%%%%%%%%%%%%%
\subsection{An illustrative example}

We illustrate \Cref{Teo D} with a concrete example. The purpose of the example is not to present the full technical construction, but rather to show the kind of combinatorics that underlies the proof. The definitions of the operators used below are given later in \Cref{def. p and q,def: words over weights,def: v negrita,def: P negrita}.

Let $n=14$, let $\alpha_{i,j}=\alpha_{8,10}$, and let $h=j-i+1=3$. Consider the dominant weight
\begin{equation}
\lambda = [0,1,0,2,1,1,0,0,0,0,0,0,1,1]_{\varpi},
\end{equation}
where $\lambda=[\lambda_1,\ldots,\lambda_{14}]_{\varpi}$ means that $\lambda=\sum_{k=1}^{14}\lambda_k\varpi_k$, and $\{\varpi_k\}_{1\leq k\leq 14}$ are the fundamental weights.

Our goal is to express $\bfM_{\lambda}^{\succ \alpha_{8,10}}$ as a positive linear combination of terms in
\begin{equation}
 \mathcal{M}\coloneqq \{\bfM_{\mu}^{\succeq \alpha_{8-k,10-k}}\mid \mu\in X^+,\, \mu\leq\lambda, \mbox{ and } 1\leq k\leq 7\}.
\end{equation}

The starting point of our analysis is \Cref{prop: second version refinado}. 
We refer to the formula in that proposition as the \emph{Second Inverse Decomposition} (SID). 
This terminology is justified by the fact that the SID is a refinement of the so-called \emph{First Inverse Decomposition} (FID) in \Cref{prop: second version}, which yields a preliminary decomposition in which some of the weights involved are not necessarily dominant.

In the present example, the SID gives
\begin{equation} \label{eq:ejemplo-intro}
\begin{array}{rlll}
        & \quad\bfM_\lambda^{\succeq \alpha_{8,10}}  + & &  \\
        &   &  &  \\
   &  q^9 \bfM_{\mathbf{P}_{10,3}(\lambda)}^{\succ \alpha_{8,10}} + & & \mathbf{P}\mbox{-term} \\
   \bfM_\lambda^{\succ \alpha_{8,10}} =  & &  &    \\
     &q^{4}\!\left(\bfM_{v_{10,3}^1(\lambda)}^{\succ \alpha_{8,10}} 
- q\, \bfM_{v_{10,3}^1(\lambda)-\alpha_{7,9}}^{\succ \alpha_{8,10}}\right) + & & v^1\mbox{-term}  \\
& & &  \\
     & q^{8}\!\left(\bfM_{\mathbf{v}_{10,3}(\lambda)}^{\succ \alpha_{8,10}} 
- q^5\, \bfM_{\mathbf{P}_{8,3}(\mathbf{v}_{10,3}(\lambda))}^{\succ \alpha_{8,10}}\right).  & & \mathbf{v}\mbox{-term and } \mathbf{Pv}\mbox{-term}  \\
\end{array}
\end{equation}

The weights appearing in this formula are
\begin{equation}
\begin{array}{r}
\mathbf{P}_{10,3}(\lambda)= [1,1,1,1,0,0,0,0,0,0,1,0,1,1]_{\varpi},\\[3pt]
 v_{10,3}^1(\lambda)= [0,1,0,3,0,1,0,0,0,0,0,0,1,0]_{\varpi}, \\[3pt]
\mathbf{v}_{10,3}(\lambda)= [0,1,1,2,0,1,0,0,0,0,0,0,0,1]_{\varpi}, \\[3pt]
\mathbf{P}_{8,3}(\mathbf{v}_{10,3}(\lambda)) = [1, 1, 1, 1, 0, 0, 0, 0, 1, 0, 0, 0, 0, 1]_{\varpi}.
\end{array}
\end{equation}

We now explain how the weights and exponents in \eqref{eq:ejemplo-intro} arise. We identify dominant weights with integer partitions, and integer partitions with Young diagrams. We use partitions with at most $15$ parts. Thus, if $\mu=(\mu_1,\ldots,\mu_{15})$, then the corresponding dominant weight is
\begin{equation}
[\mu_1-\mu_2,\mu_2-\mu_3,\ldots,\mu_{14}-\mu_{15}]_{\varpi}.
\end{equation}

\begin{itemize}
    \item The term $\bfM_\lambda^{\succeq \alpha_{8,10}}$ always appears with coefficient $1$.

\item Consider the $\mathbf{P}$-term $q^9\bfM_{\mathbf{P}_{10,3}(\lambda)}^{\succ \alpha_{8,10}}$.
 
In the partition model, subtracting a root $\alpha_{i,j}$ amounts to moving one box from row $i$ to row $j+1$. The resulting diagram need not be a partition: the box may fail to be removable from row $i$, or its new position in row $j+1$ may fail to be addable. 

The construction of $\mathbf{P}_{10,3}(\lambda)$ is recursive. It is obtained as a composition of certain operators, which we denote by $P$ (in non-boldface), to distinguish them from the boldface objects $\mathbf{P}_{j,h}(\lambda)$.

We begin by considering the  root $\alpha_{8,10}$ of height $h=3$. If row $8$ has a removable box, we set
\begin{equation}
P_{10,3}(\lambda)=\lambda-\alpha_{8,10}.
\end{equation}
If not, we check row $8-h=5$; if row $5$ has a removable box, we set
\begin{equation}
P_{10,3}(\lambda)=\lambda-\alpha_{5,7}-\alpha_{8,10}.
\end{equation}
If this also fails, we continue with row $5-h=2$, and so on. In general, we look for a row $\varrho\equiv i\pmod h$ from which a box can be moved, and define
\begin{equation}
P_{j,h}(\lambda)=\lambda-\alpha_{\varrho,j}.
\end{equation}
If no such row exists, both $P_{j,h}(\lambda)$ and $\mathbf{P}_{j,h}(\lambda)$ are undefined.

In our example,
\begin{equation}
P_{10,3}(\lambda)=\lambda-\alpha_{5,7}-\alpha_{8,10}.
\end{equation}
We then check whether this weight is dominant. If it is, we set $\mathbf{P}_{10,3}(\lambda)=P_{10,3}(\lambda)$. If it is not, we repeat the same procedure starting from $\alpha_{7,9}=\alpha_{8-1,10-1}$; this gives the operator $P_{9,3}$. If $P_{9,3}P_{10,3}(\lambda)$ is defined but not dominant, we continue with the root $\alpha_{6,8}=\alpha_{8-2,10-2}$, and so on.

The process stops either when the resulting weight becomes dominant, in which case this weight is $\mathbf{P}_{10,3}(\lambda)$, or when the process terminates unsuccessfully, in which case $\mathbf{P}_{10,3}(\lambda)$ is undefined. This is illustrated in \Cref{fig:introA}.

Finally, passing from $\lambda$ to $\mathbf{P}_{10,3}(\lambda)$ requires subtracting nine roots of height $3$. This explains the coefficient $q^9$ in front of $\bfM_{\mathbf{P}_{10,3}(\lambda)}^{\succ \alpha_{8,10}}$ in \eqref{eq:ejemplo-intro}.
\begin{figure}[h]
    \centering

\Yboxdim{1cm}
\begin{tikzpicture}[scale=.21]
\tikzset{>={Stealth[length=1.4mm, width=1.1mm]}}

\Yfillopacity{0.3}
\Ylinethick{0.3pt}
\Ylinecolour{blue}
\Yfillcolour{cyan}

% --- Diagram 1 ---
\begin{scope}[local bounding box=Y0]
  \tyng(0cm,3cm,7^2,6^2,4,3,2^7,1)
  \draw[->,red,line width=0.4mm, line cap=round](2.5,-3.5)--(3.5,-3.5)%
    to node[auto,swap]{} (3.5,-6.5)--(2.3,-6.5);
  \draw[->,red,line width=0.4mm, line cap=round](4.5,-0.5)--(5.5,-0.5)%
    to node[auto,swap]{} (5.5,-3.5)--(4,-3.5);
\end{scope}

% --- Diagram 2 ---
\begin{scope}[xshift=9cm,local bounding box=Y1]
  \tyng(0cm,3cm,7^2,6^2,3^2,2^4,3,2^2,1)
  \draw[->,red,line width=0.4mm, line cap=round](2.5,-2.5)--(3.5,-2.5)%
    to node[auto,swap]{} (3.5,-5.5)--(2.3,-5.5);
  \draw[->,red,line width=0.4mm, line cap=round](6.5,0.5)--(7,0.5)%
    to node[auto,swap]{} (7,-2.5)--(4,-2.5);
\end{scope}

% --- Diagram 3 ---
\begin{scope}[xshift=18cm,local bounding box=Y2]
  \tyng(0cm,3cm,7^2,6,5,3^2,2^3,3^2,2^2,1)
  \draw[->,red,line width=0.4mm, line cap=round](3.5,-1.5)--(4,-1.5)%
    to node[auto,swap]{} (4,-4.5)--(2.8,-4.5);
\end{scope}

% --- Diagram 4 ---
\begin{scope}[xshift=27cm,local bounding box=Y3]
  \tyng(0cm,3cm,7^2,6,5,3,2^3,3^3,2^2,1)
  \draw[->,red,line width=0.4mm, line cap=round](3.5,-0.5)--(4,-0.5)%
    to node[auto,swap]{} (4,-3.5)--(2.5,-3.5);
\end{scope}

% --- Diagram 5 ---
\begin{scope}[xshift=36cm,local bounding box=Y4]
  \tyng(0cm,3cm,7^2,6,5,2^3,3^4,2^2,1)
  \draw[->,red,line width=0.4mm, line cap=round](5.5,0.5)--(6,0.5)%
    to node[auto,swap]{} (6,-2.5)--(2.8,-2.5);
\end{scope}

% --- Diagram 6 ---
\begin{scope}[xshift=45cm,local bounding box=Y5]
  \tyng(0cm,3cm,7^2,6,4,2^2,3^5,2^2,1)
  \draw[->,red,line width=0.4mm, line cap=round](6.5,1.5)--(7,1.5)%
    to node[auto,swap]{} (7,-1.5)--(2.8,-1.5);
\end{scope}

% --- Diagram 7 ---
\begin{scope}[xshift=54cm,local bounding box=Y6]
  \tyng(0cm,3cm,7^2,5,4,2,3^6,2^2,1)
  \draw[->,red,line width=0.4mm, line cap=round](7.5,2.5)--(8,2.5)%
    to node[auto,swap]{} (8,-0.5)--(2.8,-0.5);
\end{scope}

% --- Diagram 8 ---
\begin{scope}[xshift=63cm,local bounding box=Y7]
  \tyng(0cm,3cm,7,6,5,4,3^7,2^2,1)
\end{scope}

% Símbolos centrados entre cada par de tableaux
\foreach \A/\B/\Sym in {
  Y0/Y1/{$P_{10,3} $},
  Y1/Y2/{$P_{9,3} $},
  Y2/Y3/{$P_{8,3} $},
  Y3/Y4/{$P_{7,3} $},
  Y4/Y5/{$P_{6,3} $},
  Y5/Y6/{$P_{5,3} $},
  Y6/Y7/{$P_{4,3} $}
}{
  \node[font=\small,anchor=north,yshift=-3pt]
       at ($(\A.south)!0.5!(\B.south)$) {\Sym};
}

\end{tikzpicture}
\caption{The leftmost diagram represents the partition $\lambda$, while the rightmost one corresponds to $\mathbf{P}_{10,3}(\lambda)$. 
Each intermediate diagram is obtained from the previous one by applying the operator displayed between them.}
    \label{fig:introA}
\end{figure}
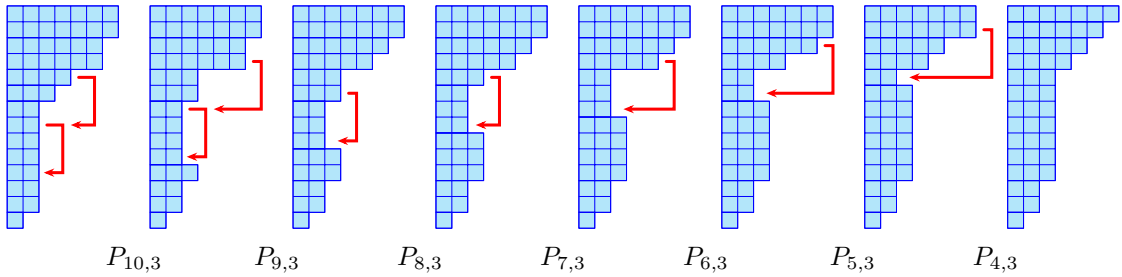

\item We now turn to the $v^1$-term $v^1_{10,3}(\lambda)$.

To obtain this weight, we first repeat the initial step of the previous construction, namely we form $P_{10,3}(\lambda)$. We then move the box in row $11$ by subtracting $\alpha_{11,12}$, a root of height $h-1=2$. If the moved box reaches an addable position, the process stops. Otherwise, we subtract $\alpha_{13,14}=\alpha_{11+2,12+2}$ and continue until either the box reaches an addable position or the resulting weight ceases to exist. This operator is denoted by $Q_{10,3}$, and we define
\begin{equation}
v^1_{10,3}(\lambda)=Q_{10,3}(P_{10,3}(\lambda)).
\end{equation}
The construction is illustrated in \Cref{fig:introB}. The superscript $1$ in $v^1_{10,3}$ indicates that we have used one $P$-operator and one $Q$-operator. The exponent $4$ records the four roots that must be subtracted from $\lambda$ to obtain $v^1_{10,3}(\lambda)$.
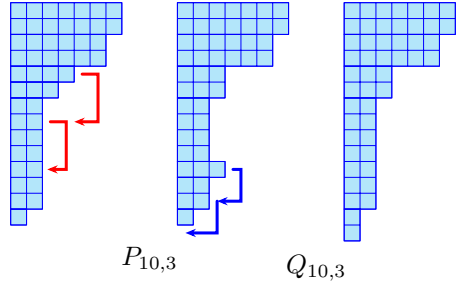
\begin{figure}[h]
    \centering
    \Yboxdim{1cm}
\begin{tikzpicture}[scale=.21]
\tikzset{>={Stealth[length=1.4mm, width=1.1mm]}}

\Yfillopacity{0.3}
\Ylinethick{0.3pt}
\Ylinecolour{blue}
\Yfillcolour{cyan}

% --- Diagram 1 ---
\begin{scope}[local bounding box=Y0]
  \tyng(0cm,3cm,7^2,6^2,4,3,2^7,1)
  \tikzset{>={Stealth[length=1.4mm, width=1.1mm]}}
  \draw[->,red,line width=0.4mm, line cap=round](2.5,-3.5)--(3.5,-3.5)%
    to node[auto,swap]{}%
    (3.5,-6.5)--(2.3,-6.5);
  \draw[->,red,line width=0.4mm, line cap=round](4.5,-0.5)--(5.5,-0.5)%
    to node[auto,swap]{}%
    (5.5,-3.5)--(4,-3.5);
\end{scope}

% --- Diagram 2 ---
\begin{scope}[xshift=10.5cm, local bounding box=Y1]
  \tyng(0cm,3cm,7^2,6^2,3^2,2^4,3,2^2,1)
  \draw[->,blue,line width=0.4mm, line cap=round](3.5,-6.5)--(4,-6.5)%
    to node[auto,swap]{}%
    (4,-8.5)--(2.5,-8.5);
  \draw[->,blue,line width=0.4mm, line cap=round](2.5,-8.5)--(2.5,-8.5)%
    to node[auto,swap]{}%
    (2.5,-10.5)--(0.5,-10.5);
\end{scope}

% --- Diagram 3 ---
\begin{scope}[xshift=21cm, local bounding box=Y2]
  \tyng(0cm,3cm,7^2,6^2,3^2,2^7,1^2)
\end{scope}

% Símbolos centrados entre cada par de tableaux
\foreach \A/\B/\Sym in {
  Y0/Y1/{$P_{10,3} $},
  Y1/Y2/{$Q_{10,3} $}
}{
  \node[font=\small,anchor=north,yshift=-3pt]
       at ($(\A.south)!0.5!(\B.south)$) {\Sym};
}

\end{tikzpicture}

    \caption{The leftmost diagram represents the partition $\lambda$, while the rightmost one corresponds to $v_{10,3}^1(\lambda)$.}
    \label{fig:introB}
\end{figure}

The weight $v_{10,3}^1(\lambda)-\alpha_{7,9}$ also appears in the third row of the right-hand side of \eqref{eq:ejemplo-intro}. This weight need not be dominant. It always occurs with an additional factor of $q$. Together with the external factor $q^4$, this reflects the fact that five roots must be subtracted from $\lambda$ to obtain $v_{10,3}^1(\lambda)-\alpha_{7,9}$.

\item We now consider the $\mathbf{v}$-term $\mathbf{v}_{10,3}(\lambda)$.

In this case, we repeat the first two steps in the construction of $\mathbf{P}_{10,3}(\lambda)$; in general, one repeats $h-1$ steps. Thus we begin with $P_{9,3}P_{10,3}(\lambda)$. We then apply $Q$-operators until the weight becomes dominant: in this example, we apply $Q_{9,3}$, then $Q_{10,3}$, then $Q_{11,3}$, and so on. The process is shown in \Cref{fig:introC}. As before, the exponent $q^8$ records the eight roots subtracted during the process.

\item Finally, consider the $\mathbf{Pv}$-term $\mathbf{P}_{8,3}(\mathbf{v}_{10,3}(\lambda))$.

Starting from $\mathbf{v}_{10,3}(\lambda)$, we proceed as in the construction of the $\mathbf{P}$-term, but now with the root $\alpha_{8,10}=\alpha_{i,j}$ replaced by $\alpha_{6,8}=\alpha_{i-h+1,j-h+1}$. This explains the occurrence of $\mathbf{P}_{8,3}(\mathbf{v}_{10,3}(\lambda))$ and the coefficient $q^{13}$.
\end{itemize}
\begin{figure}
    \centering
    \Yboxdim{1cm}
\begin{tikzpicture}[scale=.21]
\tikzset{>={Stealth[length=1.4mm, width=1.1mm]}}

\Yfillopacity{0.3}
\Ylinethick{0.3pt}
\Ylinecolour{blue}
\Yfillcolour{cyan}

% --- Diagram 1 ---
\begin{scope}[local bounding box=Y0]
  \tyng(0cm,3cm,7^2,6^2,4,3,2^7,1)
  \tikzset{>={Stealth[length=1.4mm, width=1.1mm]}}
  \draw[->,red,line width=0.4mm, line cap=round](2.5,-3.5)--(3.5,-3.5)%
    to node[auto,swap]{}%
    (3.5,-6.5)--(2.3,-6.5);
  \draw[->,red,line width=0.4mm, line cap=round](4.5,-0.5)--(5.5,-0.5)%
    to node[auto,swap]{}%
    (5.5,-3.5)--(4,-3.5);
\end{scope}

% --- Diagram 2 ---
\begin{scope}[xshift=10.5cm, local bounding box=Y1]
  \tyng(0cm,3cm,7^2,6^2,3^2,2^4,3,2^2,1)
  \draw[->,red,line width=0.4mm, line cap=round](2.5,-2.5)--(3.5,-2.5)%
    to node[auto,swap]{}%
    (3.5,-5.5)--(2.3,-5.5);
  \draw[->,red,line width=0.4mm, line cap=round](6.5,0.5)--(7,0.5)%
    to node[auto,swap]{}%
    (7,-2.5)--(4,-2.5);
\end{scope}

% --- Diagram 3 ---
\begin{scope}[xshift=22cm, local bounding box=Y2]
  \tyng(0cm,3cm,7^2,6,5,3^2,2^3,3^2,2^2,1)
  \draw[->,blue,line width=0.4mm, line cap=round](3.5,-5.5)--(4,-5.5)%
    to node[auto,swap]{}%
    (4,-7.5)--(2.5,-7.5);
\end{scope}

% --- Diagram 4 ---
\begin{scope}[xshift=33cm, local bounding box=Y3]
  \tyng(0cm,3cm,7^2,6,5,3^2,2^4,3^2,2^1,1)
  \draw[->,blue,line width=0.4mm, line cap=round](3.5,-6.5)--(4,-6.5)%
    to node[auto,swap]{}%
    (4,-8.5)--(2.5,-8.5);
\end{scope}

% --- Diagram 5 ---
\begin{scope}[xshift=44cm, local bounding box=Y4]
  \tyng(0cm,3cm,7^2,6,5,3^2,2^5,3^2,1)
  \draw[->,blue,line width=0.4mm, line cap=round](3.5,-7.5)--(4,-7.5)%
    to node[auto,swap]{}%
    (4,-9.5)--(1.5,-9.5);
\end{scope}

% --- Diagram 6 ---
\begin{scope}[xshift=55cm, local bounding box=Y5]
  \tyng(0cm,3cm,7^2,6,5,3^2,2^6,3,2)
  \draw[->,blue,line width=0.4mm, line cap=round](3.5,-8.5)--(4,-8.5)%
    to node[auto,swap]{}%
    (4,-10.5)--(0.5,-10.5);
\end{scope}

% --- Diagram 7 ---
\begin{scope}[xshift=66cm, local bounding box=Y6]
  \tyng(0cm,3cm,7^2,6,5,3^2,2^8,1)
\end{scope}

% Símbolos centrados entre cada par de tableaux
\foreach \A/\B/\Sym in {
  Y0/Y1/{$P_{10,3} $},
  Y1/Y2/{$P_{9,3} $},
  Y2/Y3/{$Q_{9,3} $},
  Y3/Y4/{$Q_{10,3} $},
  Y4/Y5/{$Q_{11,3} $},
  Y5/Y6/{$Q_{12,3} $}
}{
  \node[font=\small,anchor=north,yshift=-3pt]
       at ($(\A.south)!0.5!(\B.south)$) {\Sym};
}

\end{tikzpicture}

    \caption{The leftmost diagram represents the partition $\lambda$, while the rightmost one corresponds to $\mathbf{v}_{10,3}(\lambda) $.}
    \label{fig:introC}
\end{figure}
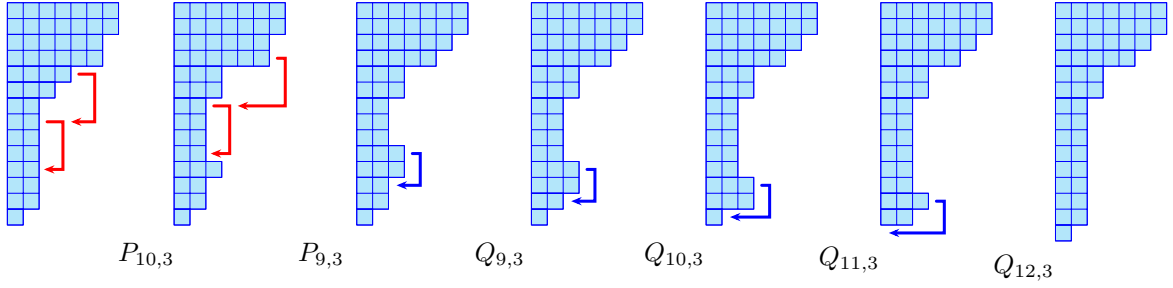

The existence of the weights above depends strongly on $\lambda$. A useful feature of the SID is that, whenever one of these weights is undefined, the corresponding $\bfM$-term is simply omitted and the formula remains valid. For example, if $\lambda$ is the zero weight or a fundamental weight, then none of the auxiliary weights appearing above is defined, and the SID reduces to
\[
    \bfM_{\lambda}^{\succ\alpha_{i,j}}=\bfM_{\lambda}^{\succeq\alpha_{i,j}}.
\]

We now explain how positivity is extracted from \eqref{eq:ejemplo-intro}.

\begin{itemize}
    \item The element $\bfM_{\lambda}^{\succeq\alpha_{8,10}}$ already belongs to $\mathcal{M}$.

    \item The $\mathbf{P}$-term $\bfM_{\mathbf{P}_{10,3}(\lambda)}^{\succ \alpha_{8,10}}$ is handled by induction on the dominance order, since $\mathbf{P}_{10,3}(\lambda)<\lambda$ by construction.

    \item The difference
    \[
        \bfM_{v_{10,3}^1(\lambda)}^{\succ \alpha_{8,10}}
        -q\,\bfM_{v_{10,3}^1(\lambda)-\alpha_{7,9}}^{\succ \alpha_{8,10}}
    \]
    is referred to as a \emph{good pair} in \Cref{section: positivity}.
    %is what we call a \emph{good pair} in \Cref{section: positivity}. 
    In general, the SID may produce $h-2$ good pairs. In this example there is only one, since $h-2=1$.

    The weights appearing in good pairs have the form $v_{j,h}^k(\lambda)$ and $v_{j,h}^k(\lambda)-\alpha_{i-k,j-k}$, with $1\leq k\leq h-2$. The operators $v_{j,h}^k$ are natural generalizations of the operator $v_{10,3}^1$ used above; see \Cref{def: words over weights}. By \Cref{prop: caso bueno}, each good pair admits a positive combinatorial expansion in terms of elements of the form $\bfM^{\succeq\alpha_{i-k,j-k}}$.

    In the present example, $k=1$, and we have
    \begin{equation} \label{eq:intro-expansion-good}
    \bfM_{v_{10,3}^1(\lambda)}^{\succ \alpha_{8,10}} 
    - q\, \bfM_{\,v_{10,3}^1(\lambda)-\alpha_{7,9}}^{\succ \alpha_{8,10}}
    = 
    \sum_{m=1}^{\infty} q^{e_m}\,
    \bfM_{(v^1_{10,3})^{m} (\lambda)}^{\succeq \alpha_{7,9}},
\end{equation}
where $e_m$ is the number of positive roots that must be subtracted from $v^1_{10,3}(\lambda)$ in order to obtain $(v^1_{10,3})^m(\lambda)$ and $(v_{10,3}^{1})^m$ denotes the $m$-fold composition of $v_{10,3}^1$.

With the convention that undefined weights contribute zero, the apparently infinite sum in \eqref{eq:intro-expansion-good} is finite: each application of $v^1_{10,3}$ strictly decreases the weight in the dominance order, and there are only finitely many dominant weights below a fixed dominant weight. In this example,
\begin{equation} \label{eq:intro-expansion-good-concrete}
    \bfM_{v_{10,3}^1(\lambda)}^{\succ \alpha_{8,10}} 
    - q\, \bfM_{\,v_{10,3}^1(\lambda)-\alpha_{7,9}}^{\succ \alpha_{8,10}}
    = 
    \bfM_{v^1_{10,3}(\lambda)}^{\succeq \alpha_{7,9}},
\end{equation}
since $(v_{10,3}^1)^m(\lambda)$ is undefined for all $m\geq 2$.

\item Finally, we consider the difference
\[
    \bfM_{\mathbf{v}_{10,3}(\lambda)}^{\succ \alpha_{8,10}}
    -q^5\,\bfM_{\mathbf{P}_{8,3}(\mathbf{v}_{10,3}(\lambda))}^{\succ \alpha_{8,10}}.
\]
%This is what we call a \emph{bad pair}.
This is referred to as a \emph{bad pair}.
In general, at most one bad pair appears in the SID. Unlike good pairs, bad pairs do not admit an immediate positive expansion, which explains the terminology.

The two weights appearing in a bad pair are decomposed further by applying the SID recursively.
An induction argument, together with commutation rules for the operators on weights, shows that the bad pair can nevertheless be rewritten as a positive linear combination of elements of the form $\bfM^{\succeq\alpha_{i-k,j-k}}$ for $1\leq k\leq i-1$. 
The required commutation rules are introduced in \Cref{section COM RUl}, and the induction is carried out in \Cref{section proof of positivity}.
\end{itemize}

%%%%%%%%%%%%%%%%%%%%%%%%%%%%%%%%%%%%%%%%%%%%%%%%%%%%%%%%%%%%%%%%%%%%%%
\subsection{Symmetric functions}
%%%%%%%%%%%%%%%%%%%%%%%%%%%%%%%%%%%%%%%%%%%%%%%%%%%%%%%%%%%%%%%%%%%%%%

The results of this paper can also be viewed through the language of symmetric functions. As mentioned above, the Satake isomorphism relates the spherical Hecke algebra to the algebra of symmetric functions. We now spell out the symmetric-function analogue of the pre-canonical bases.

As before, we identify dominant weights in type $A$ with partitions. Let $s_\lambda$ denote the Schur function indexed by a partition $\lambda$. We extend Schur functions to arbitrary elements of $\mathbb{Z}^{\ell}$ by the usual straightening rule: for $\gamma\in\mathbb{Z}^{\ell}$, set
\begin{equation}\label{eq:straighten}
  s_\gamma =
  \begin{cases}
    \operatorname{sgn}(\gamma+\rho)\,
    s_{\operatorname{sort}(\gamma+\rho)-\rho},
      & \text{if $\gamma+\rho$ has distinct nonnegative parts,} \\[4pt]
    0, & \text{otherwise.}
  \end{cases}
\end{equation}
Here $\rho=(\ell-1,\ell-2,\dots,0)$, $\operatorname{sort}(\beta)$ denotes the weakly decreasing rearrangement of $\beta$, and $\operatorname{sgn}(\beta)$ is the sign of the shortest permutation sending $\beta$ to $\operatorname{sort}(\beta)$.

This straightening rule is the Schur-function counterpart of \eqref{tilde de schur}. The analogue of \Cref{defi pre-canonical} is the following.

\begin{definition}\rm
The $k$-th pre-canonical Schur function associated with $\lambda$ is
\begin{equation}\label{eq:preSchur}
  s_{\lambda}^{k}
  =  \prod_{\alpha \in \Phi^{\ge k}} (1 - q\,L_{\alpha})\, s_{\lambda},
\end{equation}
where the lowering operator $L_\alpha$ acts on subscripts by
\[
    L_{\alpha}s_{\gamma}=s_{\gamma-(\epsilon_i-\epsilon_{j+1})}
    \qquad \mbox{if } \alpha=\alpha_{i,j}.
\]
\end{definition}

In \eqref{eq:preSchur}, the product is first expanded and then applied to $s_\lambda$. Thus
\begin{equation}\label{eq:preSchurA}
  s_{\lambda}^{k}
  = \sum_{I\subset \Phi^{\geq k}} (-q)^{|I|} s_{\lambda-\Sigma_I},  
\end{equation}
which is the image of \Cref{defi pre-canonical} under the heuristic replacement of canonical basis elements by Schur functions.

The formulation in terms of lowering operators highlights a parallel with Catalan symmetric functions. These are defined by
\begin{equation}
\mathbf{H}_{A,\lambda}
  = \prod_{\alpha\in A} (1 - q R_{\alpha})^{-1} s_{\lambda},
\end{equation}
where $A$ is a root ideal and the raising operator $R_\alpha$ acts by
\[
    R_{\alpha}s_\gamma=s_{\gamma+(\epsilon_i-\epsilon_{j+1})}
    \qquad \mbox{if } \alpha=\alpha_{i,j}.
\]
For our purposes, the relevant point is that each set $\Phi^{\geq k}$ is a root ideal.

Catalan symmetric functions were introduced independently in \cite{chen2010skew} and \cite{panyushev2010generalised}. They interpolate between modified Hall--Littlewood polynomials, when $A=\Phi^+$, and Schur functions, when $A=\varnothing$.

Similarly, after applying the Satake transform and then dualizing with respect to the Hall inner product, the standard and canonical bases of $\mathcal{H}^{\mathbf{sph}}$ correspond to modified Hall--Littlewood and Schur functions, respectively. It is natural to ask whether the pre-canonical Schur functions and the Catalan symmetric functions attached to the ideals $\Phi^{\geq k}$ are dual bases; that is, whether
\begin{equation}
    \langle s_{\lambda}^k , \mathbf{H}_{\Phi^{\ge k},\mu} \rangle = \delta_{\lambda,\mu}.
\end{equation}
This is not the case. For example, when $n=4$ and $k=2$,
\begin{equation}
    s^2_{(2,1,1)} 
      = s_{(2,1,1)} + (-q^{2}-q)\, s_{(1,1,1,1)},
\end{equation}
whereas
\begin{equation}
    \mathbf{H}_{\Phi^{\ge 2}, (1,1,1,1)} 
      = s_{(1,1,1,1)} + q\,s_{(2,1,1)} + q^2 s_{(2,2)}.
\end{equation}
Hence
\begin{equation}
  \langle s^2_{(2,1,1)} , \mathbf{H}_{\Phi^{\ge 2}, (1,1,1,1)} \rangle
  = -q^{2} \neq 0.
\end{equation}

Let $\{\tilde{s}_{\lambda}^{k}\}$ denote the basis dual to $\{s_{\lambda}^k\}$ with respect to the Hall inner product. By \Cref{Teo C}, the expansion of $\tilde{s}_{\lambda}^{k}$ in the basis $\{\tilde{s}_{\mu}^{k+1}\}$ has nonnegative coefficients. It follows, in particular, that $\tilde{s}_{\lambda}^{k}$ is Schur-positive for every $k$ and every $\lambda$.

Blasiak, Morse, and Pun~\cite{blasiak2024demazure} proved the Schur positivity of Catalan symmetric functions for arbitrary root ideals and partitions. Their results provide a powerful refinement of several earlier positivity phenomena. For instance, the Lascoux--Sch\"utzenberger charge formula for $K_{\lambda,\mu}(q)$ appears as a special case of their tableaux formula for Catalan symmetric functions; see \cite[Theorem~2.18]{blasiak2024demazure}.

The functions $\tilde{s}_{\lambda}^{k}$ are more rigid than Catalan symmetric functions, but they have a different advantage: they provide a \emph{global} positive interpolation between modified Hall--Littlewood polynomials and Schur functions. More precisely, an entire basis evolves step by step from the modified Hall--Littlewood basis to the Schur basis. Catalan symmetric functions, by contrast, interpolate one function at a time rather than through a distinguished sequence of full bases.

We also mention the related work~\cite{blasiak2019catalan}, where Blasiak, Morse, Pun, and Summers proved the positivity of the $k$-Schur functions $\mathfrak{s}^k_\lambda$ of Lapointe, Lascoux, and Morse~\cite{luc2003}. They also established the branching property for $k$-Schur functions: the expansion of a $k$-Schur function in the $(k+1)$-Schur basis is positive. Although this statement resembles the positivity satisfied by the dual pre-canonical bases, the two phenomena are different. For $k$-Schur functions, the parameter $k$ controls the allowed size of the first part of the partitions involved; in our setting, the parameter $k$ controls the root ideal used to define the $k$-th pre-canonical Schur function.

We close this discussion by pointing out that readers more familiar with symmetric functions than with Hecke algebras can read much of the paper by applying this dictionary: replace dominant weights by partitions and canonical basis elements by Schur functions.

\subsection{Organization of the paper and acknowledgments}

The paper is organized as follows.
In \Cref{sec: prelimnares}, we introduce the pre-canonical bases and establish the notation used throughout the paper.
In \Cref{section identities to prove first inverse decomposition}, we develop the identities leading to the First Inverse Decomposition, whose proof is given in \Cref{Section: descomposición inversa}.
Likewise, \Cref{sec: Towards the Second Inverse Decomposition} is devoted to developing the tools required for the proof of the Second Inverse Decomposition, which is carried out in \Cref{sec: second inverse decomp}.
Finally, in \Cref{section: positivity}, we prove \Cref{Teo D}, thereby establishing the positivity conjecture for pre-canonical bases. 

\medskip

The authors would like to warmly thank Damian de la Fuente and Nicolas Libedinsky for their comments and suggestions, which greatly improved the presentation of this paper.
The first author would also like to thank Cedric Lecouvey, Cristian Lenart, and Leonardo Patimo for many stimulating discussions related to the topics of this paper during the early stages of the project.

\medskip

David Plaza was partially supported by ANID--FONDECYT Regular Grant 1240199.
Yamil Sagurie was partially supported by ANID-beca doctorado nacional 21201961 and ANID--FONDECYT Regular Grant 1240199.

\section{Preliminaries}  \label{sec: prelimnares}
In this section we introduce the spherical Hecke algebra and their Pre-canonical bases.

\subsection{Affine Weyl groups}

From now on we fix a positive integer $n$.
Let $\Phi$ be a root system of type $A_n$. 
We use the following realization of $\Phi$. 
Let $V\subset \mathbb{R}^{n+1}$ be the subspace of vectors with coordinates adding up to zero. 
In this setting we set $\Phi=\{ \epsilon_i -\epsilon_j \mid 1\leq i\neq j \leq n+1\}$. 
The simple roots are $\Delta = \{ \alpha_i\coloneqq \epsilon_i-\epsilon_{i+1} \mid 1\leq i \leq n \}$ and the set of positive roots is given by $\Phi^+=\{   \alpha_{i,j} \coloneqq  \epsilon_i -\epsilon_{j+1} \mid 1\leq i\leq j \leq n  \}$. 
The fundamental weights $\{\varpi_1 , \ldots , \varpi_n\}$ are given by the rule $\pint{\varpi_i}{\alpha_j} = \delta_{i,j}$. 
The weight lattice is $X=\operatorname{span}_{\mathbb{Z}} \{\varpi_1 , \ldots , \varpi_n\}$ and the dominant weights are $X^+=\operatorname{span}_{\mathbb{N}} \{\varpi_1 , \ldots , \varpi_n\}$.

\begin{remark}\rm \label{remark producto interno}
  Let $1\leq i < j \leq n$ and $1\leq k \leq n$. Then, we have
  \begin{equation}
      \pint{\alpha_{i,j}}{\alpha_k} = \begin{cases}
          1, & \mbox{if } k=i \mbox{ or } k=j;\\
          -1, & \mbox{if } k=i-1 \mbox{ or } k=j+1; \\
          0, & \mbox{otherwise.}
       \end{cases}
  \end{equation}
\end{remark}

\begin{definition}\rm 
  We  define a total order $\preceq $ on  $\Phi^{\geq 2}$ as follows. 
  We write $\alpha_{i,j}\preceq \alpha_{k,r}$ if either $j-i < r-k$,  or $j-i = r-k$ and $i\leq k$.   
\end{definition}

Notice that the root $\alpha_{1,n}$ (resp. $\alpha_{1,2}$) is the maximal (resp. minimal) element for this order.
Given $a\preceq b\in \Phi^{\geq 2}$, we define the  closed interval $[a,b]$ to be  the set of elements $x \in \Phi^{\geq 2}$ such that $a\preceq x \preceq b$. 
Open and half-open intervals are defined in a similar fashion.
For $\alpha_{i,j}\in \Phi^{\geq 2}$ we define
\begin{equation}
    \Phi^{\succeq \alpha_{i,j}} =\{ \alpha \in \Phi^{\geq 2} \mid \alpha \succeq \alpha_{i,j}  \} \qquad \mbox{and} \qquad
     \Phi^{\succ \alpha_{i,j}} =\{ \alpha \in \Phi^{\geq 2} \mid \alpha \succ \alpha_{i,j}  \}.
\end{equation}

Henceforth we represent  $\Phi^+$ as an array of boxes in the plane where each box corresponds to a unique positive root.
We draw a ``pyramid'' of boxes with $n$ boxes at the bottom. 
The boxes on the bottom represent the simple roots and they are represented by the corresponding index. 
The other positive roots are obtained as follows. 
The box associated to $\alpha_{i,j}\in \Phi^{\geq 2}$ is the unique box $B$ such that the pyramid with apex $B$ has as corners of its base the boxes associated to the simple roots $\alpha_i$ and $\alpha_j$. 
Given $K\subseteq \Phi^+$ we color with green a box if its corresponding root belongs to $K$. 
In this setting, it is quite easy to compare to roots $\alpha, \beta \in \Phi^{\geq 2}$.
Namely, let $A$ and $B$ be the boxes associated to $\alpha$ and $\beta$, respectively.
Then, $\alpha \succ \beta$ if and only if either $A$ is located above $B$, or $A$ and $B$ are located at the same level and $A$ is on the right of $B$.
For instance, the array of boxes in  \Cref{fig: rootsA} corresponds to $\Phi^{\succ \alpha_{3,6}}=\Phi^{\succeq \alpha_{4,7}}$ for $n=10$.

\begin{figure}[H]
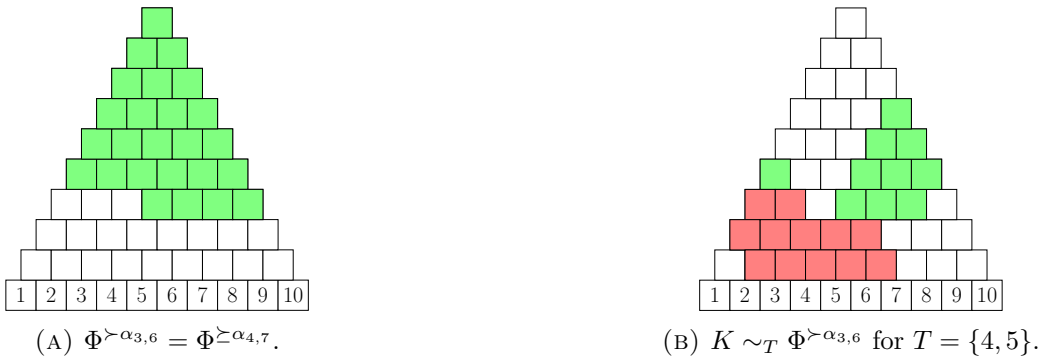

     \centering
     \begin{subfigure}[b]{0.45\textwidth}
         \centering
        {\scalebox{.4}{\exbonito}}
    \caption{ $\Phi^{\succ \alpha_{3,6}}=\Phi^{\succeq \alpha_{4,7}}$.  }
    \label{fig: rootsA}
     \end{subfigure}
     \hfill
     \begin{subfigure}[b]{0.45\textwidth}
         \centering
        {\scalebox{.4}{\exbonitoB}}
    \caption{$ K \sim_T \Phi^{\succ \alpha_{3,6}} $ for $T=\{4,5\}$.}
    \label{fig: rootsB}
     \end{subfigure}
        \caption{Roots in a pyramid for $n=10$.}
        \label{fig: ejemplos de roots}
\end{figure}

Let $\alpha\in \Phi^+$. We define $s_{\alpha}: V\rightarrow V$ by $s_\alpha(\lambda) = \lambda - \pint{\lambda}{\alpha}\alpha $. 
The (finite) Weyl group $W_f$ of $\Phi$ is the group generated by $\{s_{\alpha}\mid \alpha \in \Phi^+ \}$. 
Let $s_i = s_{\alpha_i}$ and $S_{f}=\{s_i \mid 1\leq i \leq n\}$. 
Then $W_f$ is a Coxeter system with generators $S_f$. 
As usual, we denote by $\ell $ and $\leq$ its length function and Bruhat order, respectively. 
Also, we denote by $w_0$ the longest element of  $W_f$. 
For $\alpha=\alpha_{i,j}\in \Phi^{+}$, the corresponding reflection $s_{\alpha_{i,j}}$ is given by
\begin{equation}\label{eq:reflection}
s_{\alpha_{i,j}} = s_i s_{i+1}\cdots s_{j-1}s_j s_{j-1}\cdots s_{i+1}s_i.
\end{equation}

Besides the natural action of $W_f$ on $V$ we have the \emph{dot} action that is defined as follows. 
Let $\rho $ be the half-sum of the positive roots. 
We have $\rho = \sum_{i=1}^n \varpi_i$. 
The dot action is given by $w\cdot \lambda = w(\lambda +\rho) -\rho$, for $w\in W_f$ and $\lambda \in V$. 

Let $m\in\mathbb{Z}$ and $\alpha\in\Phi^+$. 
Define the (affine) hyperplane
$H_{\alpha,m}:=\{\lambda\in V \mid \pint{\lambda}{\alpha}=m\}$,
and let $s_{\alpha,m}$ denote the reflection along $H_{\alpha,m}$.
In formulas, $s_{\alpha,m}(\lambda)=\lambda-(\pint{\lambda}{\alpha}-m)\,\alpha$.
The affine Weyl group $W_a$ is the group generated by $\{s_{\alpha,m } \mid \alpha\in \Phi^+, m\in \mathbb{Z}\}$. 

The connected components of the complement of the union of all hyperplanes
$H_{\alpha , m}$ are called \emph{alcoves}. 
We denote by $\mathcal{A}$ the set of all alcoves. 
We call 
$$\mathcal{A}_0 = \{ \lambda \in V \mid -1 < \pint{\lambda}{\alpha} < 0  \mbox{ for any } \alpha \in \Phi^+  \}$$ the fundamental alcove. 
The map $W_a \rightarrow \mathcal{A}$ given by $w\mapsto w(\mathcal{A}_0)$ is a bijection. 

The walls of $\mathcal{A}_0$ are supported in the hyperplanes $H_{\alpha , 0}$ for $\alpha \in \Delta$ and $H_{\beta , -1}$ for $\beta = \alpha_{1,n}$. 
We set $s_{0} = s_{\beta , -1}$, the reflection along $H_{\beta , -1}$ and $S_a = S_f \cup \{s_0\}$. 
Then, $W_a$ is a Coxeter group with generators $S_a$. 

\begin{definition}\rm \label{def: theta lambda}
    Let $\lambda \in X^+$. 
    The set $w_0(\mathcal{A}_0)+\lambda$ belongs to $\mathcal{A}$. 
    Then we define $\theta(\lambda) \in W_a$ to be the unique element in $W_a$ such that $\theta(\lambda) (\mathcal{A}_0) =  w_0(\mathcal{A}_0)+\lambda$. 
\end{definition}

\subsection{Spherical Hecke algebras}

Let $\mathcal{H} = \mathcal{H}(W_{a}, S_a)$  be the affine Hecke algebra of $(W_a,S_a)$. 
It is the $\mathbb{Z}[q^{\frac{1}{2}},q^{-\frac{1}{2}}]$-algebra  with generators $\{H_{s_i} \mid s_i \in S_a\}$ and relations 
$H_{s_i}^2 = (q^{-\frac{1}{2}} -q^{\frac{1}{2}})H_{s_i}+1$, $H_{s_i}H_{s_j}=H_{s_j}H_{s_i}$ if $s_i$ and $s_j$ commute, and   $H_{s_i}H_{s_j}H_{s_{i}} =H_{s_j}H_{s_i}H_{s_j}$ if $s_i$ and $s_j$ do not commute. 
The Hecke algebra $\mathcal{H}$ comes equipped with a  standard basis $\{\mathbf{H}_w \mid w\in W_a\}$ and a canonical basis $\{ \underline{\mathbf{H}}_{w} \mid w\in W_a \ \}$ 
(see \cite{Lusztig1979,Soergel1997KazhdanLusztigPA}).
The coefficients in the transition matrix between the canonical and standard bases of $\mathcal{H}$ are the Kazhdan-Lusztig polynomials that we denote by $h_{x,y}(q^{1/2})\in \mathbb{Z}[q^{1/2}]$ for $x, y \in W_a$. 
In formulas, we have
\begin{equation}
    \underline{\mathbf{H}}_y = \sum_{x\in W_a} h_{x,y}(q^{1/2}) \mathbf{H}_x. 
\end{equation}

\begin{definition}\rm
   For $\lambda \in X^+$ we set 
$\underline{\mathbf{H}}_{\lambda} = \underline{\mathbf{H}}_{\theta (\lambda)}$. 
The spherical Hecke algebra $\sph $ is defined as
\begin{equation}
    \sph = \operatorname{span}_{\mathbb{Z}[q^{\frac{1}{2}},q^{-\frac{1}{2}}]} \{ \underline{\mathbf{H}}_{\lambda} \mid \lambda \in X^{+}  \}. 
\end{equation} 
We refer to $ \{ \underline{\mathbf{H}}_{\lambda} \mid \lambda \in X^{+}  \}$ as the canonical basis of $\sph$. 
\end{definition}

\begin{remark}\rm
As its name suggests, $\sph$ is an algebra.
However, this is not a subalgebra of $\mathcal{H}$.
In this paper we use only the linear structure of $\sph$, so we do not recall the product. 
For the ring  structure, see \cite{LPP}.
\end{remark}

The spherical Hecke algebra also has a standard basis $\{ \mathbf{H}_{\lambda} \mid  \lambda \in X^+ \}$, where
\begin{equation}\label{eq: defi standard basis}
     \mathbf{H}_{\lambda} = \sum_{x \in D_{\lambda}} q^{-\frac{1}{2}(\ell(\theta(\lambda)) - \ell(x))} \mathbf{H}_{x},   
\end{equation}
 $D_\lambda= W_f  \,\theta (\lambda) \, W_\lambda$ and 
$W_{\lambda}$ is the maximal parabolic subgroup of $W_{a}$ generated by $S_a\setminus \{s_k\}$ and $0\leq k \leq n$ is the unique integer such that $\lambda \equiv -\varpi_k \mod \mathbb{Z}\Phi$ ($\varpi_0\coloneqq 0$ by convention). 

For the sake of completeness, we recall here the definition of the pre-canonical bases given in \Cref{defi pre-canonical}. 

\begin{definition}\rm
   For $i\geq 2 $ the $i$-th pre-canonical basis, $\mathcal{N}^{i}=\{ \bfN_{\lambda }^i  \mid  \lambda \in X^{+} \}$, is given by 
\begin{equation}  \label{def: pre-canonical}
    \bfN_{\lambda}^{i} = \sum_{I\subset \Phi^{\geq i}} (-q)^{|I|} \widetilde{\underline{\mathbf{H}}}_{\lambda - \Sigma_I}, 
\end{equation}  
As its name suggests, for each $i$, $\mathcal{N}^{i}$ is a basis for $\sph$.
\end{definition}

The main result of this paper is a proof of the positivity statement in \Cref{Conj A}. 
The key idea is that pre-canonical basis elements are too rigid for our purposes: they are defined only for dominant weights, and their construction is constrained to  sets of the form $\Phi^{\ge i}$ for some $i\ge 2$. 
We address these two limitations as follows.

\begin{definition}\rm
      For $A\subset \Phi$ and $\mu \in X$ we define
      \begin{equation}
          \bfM_{\mu}^{A}\coloneqq \sum\limits_{I\subset A} (-q)^{\labs I \rabs}\bfH_{\mu - \Sigma_{I}}
      \end{equation}
    In particular, we define $\bfM_{\lambda}^{\succeq \alpha_{i,j}}= \bfM_{\lambda}^{\Phi^{\succeq \alpha_{i,j}}}$ and  $\bfM_{\lambda}^{\succ \alpha_{i,j}}= \bfM_{\lambda}^{\Phi^{\succ \alpha_{i,j}}}$.
\end{definition}

The following lemma is the key to rewrite $\bfM$-elements.  It is proved in \cite[Proposition 4.3]{LPP}.

\begin{lemma}\rm\label{lem: weyl group over M elements}
Let $A\subset \Phi^+$ and $\lambda\in X$. Then, for all $w\in W_f$ we have 
\begin{equation}\label{eq: lemma act on M-elements}
  \bfM_{\lambda}^{A} = (-1)^{\ell(w)}\,\bfM_{w\cdot \lambda}^{\,w(A)}.
\end{equation}
In particular, if $w(A)=A$, $w\cdot \lambda=\lambda$, and $\ell(w)$ is odd, then $\bfM_{\lambda}^{A}=0$.
Moreover, if $w=s_{\alpha}$ for some $\alpha\in\Phi^{+}$ and $\pint{\lambda+\rho}{\alpha}=0$, then $\bfM_{\lambda}^{A}=0$. 
\end{lemma}

The above lemma has two easy corollaries that shows how we can modify the set indexing a $\bfM$-element, while keeping the same weight. 

\begin{corollary}\label{coro puedo agregar}
    Let $A\subset \Phi^+$ and $\lambda \in X$.  
    Let $ 1\leq r \leq n$ and $\beta \in \Phi^+$.
    Suppose that $s_r(A)=A$,  $\pint{\lambda}{\alpha_r}=0$ and $\pint{\beta}{\alpha_r} =1$ then 
    \begin{equation}
        \bfM_{\lambda}^{A} = \bfM_{\lambda}^{A\cup \{\beta\} }.
    \end{equation}
\end{corollary}

\begin{proof}
By the definition of $\bfM$-elements we have $\bfM_{\lambda}^{A\cup \{\beta\} } = \bfM_{\lambda}^{A} -q \bfM_{\lambda -\beta }^{A}  $. 
Since $\pint{\lambda -\beta +\rho}{\alpha_r}=0$ we conclude by \Cref{lem: weyl group over M elements} that $\bfM_{\lambda -\beta }^{A}=0$, and the result follows. 
\end{proof}

\begin{corollary}\rm\label{lem: sacar una poner otra}
    Let $A\subseteq \Phi^{\geq 2} $ and $1\leq r \leq n$. Suppose that $A\setminus s_r(A)=\{\alpha \}$ and that $ \beta \notin A$. Furthermore, assume that $\langle  \alpha, \alpha_r \rangle = \langle   \beta , \alpha_r  \rangle =1$. If $\lambda \in X$ and $\langle \lambda , \alpha_r \rangle =0$ then 
    \begin{equation}
        \bfM_{\lambda}^A = \bfM_{\lambda}^{(A\setminus \{\alpha\}) \cup \{\beta\}  }. 
    \end{equation}
\end{corollary}

\begin{proof}
   We begin by noticing that $s_r(A\setminus \{ \alpha  \}) = A\setminus \{ \alpha \}$ and $\pint{\lambda -\alpha +\rho}{\alpha_r}=0$. Therefore, \Cref{lem: weyl group over M elements} implies that $\bfM_\lambda^A =\bfM_{\lambda}^{A\setminus \{\alpha\}}$. 
   On the other hand, by applying \Cref{coro puedo agregar} we obtain $\bfM_{\lambda}^{A\setminus \{\alpha \}} = \bfM_{\lambda}^{(A\setminus \{\alpha\}) \cup \{\beta\}  } $.
   By comparing the two expressions we get the desired equality. 
\end{proof}

% Henceforth, we will say that we apply \Cref{lem: sacar una poner otra} at position $r$ to the triple $(A,\alpha, \beta)$ if we are under the hypothesis of the lemma.  

\section{Towards the First Inverse Decomposition}
\label{section identities to prove first inverse decomposition}

In this section, we collect various results concerning $\bfM$-elements. Specifically, given $\bfM_{\lambda}^A$ where $A \subseteq \Phi^{\geq 2}$ and $\lambda \in X$, we present several straightening rules that enable us to rewrite this element by modifying either the set $A$ or the weight $\lambda$. At this point, many of these rules lack clear motivation, and their proofs involve delicate combinatorial arguments. Consequently, the initial reading might appear somewhat dense. If this is the case, we encourage the reader to skip the proofs in this section during their first pass and proceed directly to the next section where these results are applied. Nonetheless, it is important for readers to keep in mind that our ultimate goal is to rewrite elements of the form $\bfM_{\lambda}^A$, where $\lambda$ is non-dominant, in terms of elements $\bfM_{\mu}^B$ for some subset $B$ and dominant weights $\mu$.

The results in this section pave the way for the First Inverse Decomposition (proved in the next section) which expresses $\bfM_{\lambda}^{\succeq \alpha_{i,j}}$ as a linear combination of elements of the form $\bfM_{\mu}^{\succ \alpha_{i,j}}$.

\begin{definition}\rm\label{def: locally greater than alpha} 
    Let $T\subset [1,n]$ and $\alpha_{i,j} \in \Phi^{\geq 2}$. We say that  $K\subseteq \Phi^{\geq 2}$ is $T$-congruent  to  $\Phi^{\succ\alpha_{i,j}}$ (we write $K\sim_{T} \Phi^{\succ\alpha_{i,j}}$) if the following conditions are satisfied:

    \begin{enumerate}[(a)]
        \item  \label{condition a in T-congruent}$\displaystyle\bigcup_{t\in T}  \{  \alpha \in \Phi^{\succ \alpha_{i,j}} \mid s_{t} (\alpha )\neq \alpha  \} \subseteq K $. 
        \item \label{condition b in T-congruent} If $\displaystyle\alpha\in K\setminus  \Phi^{\succ \alpha_{i,j}}$ then $s_{t}(\alpha)=\alpha$ for all $t\in T$. 
    \end{enumerate} 
\end{definition}

\begin{example}\rm 
Let us illustrate \Cref{def: locally greater than alpha}. Let $ n=10 $, $ T=\{4, 5\} $, and $ K \subseteq \Phi^{\geq 2} $ such that $ K \sim_T \Phi^{\succ \alpha_{3,6}} $. \Cref{condition a in T-congruent} in \Cref{def: locally greater than alpha} requires that the roots associated with the green boxes in \Cref{fig: rootsB} belong to $ K $. On the other hand, \Cref{condition b in T-congruent} requires that the roots associated with the red boxes in \Cref{fig: rootsB} do not belong to $ K $. All the remaining roots may optionally belong to $K$.

\end{example}

   We stress that $\Phi^{\succ \alpha_{i,j}} \sim_T \Phi^{\succ \alpha_{i,j}}  $ for all $T\subset [1,n]$. The motivation behind Definition \ref{def: locally greater than alpha} is that we need to work with $\bfM$-elements with superscripts $K$ that behave like $\Phi^{\succ \alpha_{i,j}}$ under the action of $s_t$ for $t \in T$. The following lemma makes this statement precise. 
   
\begin{lemma}\rm\label{lem: las que se salen}
    Let $T\subset [1,n]$, $\alpha_{i,j} \in \Phi^{\geq 2}$ and $K\sim_{T} \Phi^{\succ\alpha_{i,j}}$. Then 
     $K\setminus s_t(K) = \Phi^{\succ \alpha_{i,j}} \setminus  s_t(\Phi^{\succ\alpha_{i,j}}) $, for all $t\in T$. 
\end{lemma}
   
\begin{proof}
The result follows directly from Definition \ref{def: locally greater than alpha}. 
\end{proof}

For further reference, we write out the set $\Phi^{\succ \alpha_{i,j}} \setminus  s_t(\Phi^{\succ\alpha_{i,j}})$ explicitly. Let  $h=j-i+1$. Since $\hgt(s_t(\alpha_{a,b})) < \hgt(\alpha_{a,b})$ if and only if $a=t$ or $b=t$, we have

\begin{equation} \label{eq: las que se salen}
\Phi^{\succ \alpha_{i,j}} \setminus  s_t(\Phi^{\succ\alpha_{i,j}}) =  \left\{ \begin{array}{ll}
    \{  \alpha_{t-h, t} , \alpha_{t,t+h}\} , & \mbox{ if } t < i ; \\
      \{ \alpha_{t-h, t} \} , & \mbox{ if } t= i ; \\
  \{  \alpha_{ t-h,t}, \, \alpha_{t,t+h-1}     \} ,   & \mbox{ if } i < t \leq j ;\\
 \{  \alpha_{t-h, t}, \alpha_{t-h+1, t}, \alpha_{t,t+h-1}   \},    &  \mbox{ if } t=j+1; \\ 
     \{  \alpha_{t-h+1,t}, \, \alpha_{t,t+h-1}  \},    &   \mbox{ if } j+1< t.
 \end{array}  \right.   
\end{equation}  
In the above, if a root does not exist (for example, if $t-h<1$) it is neglected.

\medskip
In Lemmas~\ref{bk}, \ref{co: bk corollary}, \ref{lem: right solving}, \ref{Claim: free Q}, and~\ref{Claim: free Q2}
below, the set $Y$ is obtained from $K$ by adjoining a single root of the form
$\alpha_{r-h,r-1}$ (in Lemma~\ref{Claim: free Q2} the role of $r$ is played by $i+1$). The
hypotheses of these lemmas allow the case $r=h$: here the pair $(r-h,r-1)=(0,h-1)$ does not
index an actual positive root, since its first coordinate is not a valid row index. As a
convention, this root is simply neglected in that case, and we set $Y=K$.

The proofs of these five lemmas are essentially the same whether $r=h$ or $r\ge h+1$. For this reason, below we only give the proofs assuming $r\ge h+1$, the case $r=h$
being entirely analogous.

\begin{lemma}\rm\label{bk}
    Let  $K\subseteq\Phi^{\geq 2}$, $\alpha_{i,j}\in \Phi^{\geq 2}$ and $h=j-i+1$.  Let  $(r,p)$ be a pair of integers such that $i < r \leq p \leq j $, $r\geq h$ and $ p+h-1\leq n$. Suppose that $K\sim_{T} \Phi^{\succ\alpha_{i,j}}$  where $T=[r,p]$. If  $\lambda \in X$ satisfies   $\pint{\lambda}{\alpha_t} =0$  for all $t\in T$ then 
    \begin{equation} \label{eq: lem bk}
        \bfM_{\lambda}^{Y} = \bfM_{\lambda}^{Z},
    \end{equation}
    where $Y=K\cup \{\alpha_{r-h,r-1}\}$ and $Z=\left(Y\setminus [\alpha_{r,r+h-1},\alpha_{p,p+h-1}]\right) \cup [\alpha_{r-h+1,r},\alpha_{p-h+1,p}]$.
\end{lemma}

\begin{proof}
We begin by noticing that  $\alpha_{r-h,r-1}\notin K$, $\alpha_{t,t+h-1}\in K$ and $\alpha_{t-h+1,t}\notin K$, for all $r\leq t \leq p$. 
Therefore, the roots added and removed in the definitions of $Y$ and $Z$ are indeed present or absent as required.
%Therefore, we do add and eliminate the roots that appear in the definition of sets $Y$ and $Z$. 

We fix $r$ and proceed by induction on $p $.
Suppose $p=r$.  Notice that in this case we have $T=\{r\}$ and $Z= (Y\setminus\{\alpha_{r,r+h-1} \}) \cup \{ \alpha_{r-h+1,r}\}$. 
By combining \Cref{lem: las que se salen} and \eqref{eq: las que se salen} we get   
\begin{equation} \label{eq: KminusK}
    K\setminus s_t(K)=\{ \alpha_{t-h,t}, \, \alpha_{t,t+h-1} \},
\end{equation}
for all $r\leq t \leq p$.

By applying \eqref{eq: KminusK}  for $t=r$ we get $Y\setminus s_r(Y) =\{ \alpha_{r,r+h-1} \}$. 
By hypothesis we have $\langle \lambda  , \alpha_r \rangle =0$. 
Then, we can apply \Cref{lem: sacar una poner otra} at position $r$ to the triple $(Y, \alpha_{r,r+h-1}, \alpha_{r-h+1,r})$ in order to obtain \eqref{eq: lem bk} for the case $p=r$. 
This completes the base of our induction.

We now suppose that $p>r $ and assume that the result holds for $p-1$. That is, we have  $\bfM_{\lambda}^{Y} = \bfM_{\lambda}^{Z^{\prime}}$ where
 \begin{equation}\label{eq: proof bk 2}
        Z^{\prime} = \left(Y\setminus [\alpha_{r,r+h-1},\alpha_{p-1,p+h-2}]\right) \cup [\alpha_{r-h+1,r},\alpha_{p-h,p-1}].
    \end{equation}

  On the other hand, by applying \eqref{eq: KminusK} for $t=p$  we get $K\setminus s_p(K) = \{  \alpha_{p-h,p}, \, \alpha_{p,p+h-1} \}$.
  Since we are assuming $p>r$, the definition of $Y$ yields $Y\setminus s_p(Y) = \{  \alpha_{p-h,p}, \, \alpha_{p,p+h-1} \}$.
  Furthermore, the sets $[\alpha_{r,r+h-1},\alpha_{p-1,p+h-2}]$ and $[\alpha_{r-h+1,r},\alpha_{p-h-1,p-2}]$ are $s_p$-invariant.
  As $s_p(\alpha_{p-h,p}) = \alpha_{p-h,p-1}\in Z' $, we can conclude that
  \begin{equation}\label{eq: Zprime}
      Z'\setminus s_p(Z') =\{ \alpha_{p,p+h-1}  \}.
  \end{equation}
 We stress  that $\langle \lambda , \alpha_p \rangle =0$ and 
 \begin{equation}
  Z= (Z'\setminus \{ \alpha_{p,p+h-1}  \}) \cup \{  \alpha_{p-h+1,p} \}. 
\end{equation}
Thus, we can apply \Cref{lem: sacar una poner otra} at position $p$ to the triple $(Z',\alpha_{p,p+h-1}, \alpha_{p-h+1,p})$ to obtain $\bfM_{\lambda}^{Z'} = \bfM_{\lambda}^{Z}$.
It follows that $\bfM_{\lambda}^{Y} = \bfM_{\lambda}^{Z}$. 
This completes our inductive step and the proof of the lemma.
\end{proof}

\begin{example}\rm

  We illustrate \Cref{bk} with an example. 
  Consider $ i = 4 $ and $ j = 7 $, so that $ h = 4 $. 
  Let $ K = \Phi^{\succ \alpha_{4,7}} $ and $ r = 5 $.
  In the notation of the lemma, the set $ Y $ is shown in \Cref{fig: ExaLemma3.15A}. 
  Note that the set $ Y $ does not depend on the value of $ p $. 
  We have three possible values for $ p $, namely $ p = 5, 6, 7 $. 
  For each choice of $ p $, the corresponding set $ Z $ is displayed in \Cref{fig: ExaLemma3.15B}, \Cref{fig: ExaLemma3.15C}, and \Cref{fig: ExaLemma3.15D}. 
  Then, \Cref{bk} asserts that for any $ \lambda \in X $ such that $ \langle \lambda, \alpha_t \rangle = 0 $ for each $ t = 5, 6, 7 $, we have
    \begin{equation}
        \bfM_\lambda^Y = \bfM_\lambda^{Z_5} =  \bfM_\lambda^{Z_6}  =  \bfM_\lambda^{Z_7} .
    \end{equation}

\begin{figure}[H]
     \centering
     \begin{subfigure}[b]{0.22\textwidth}
         \centering
        {\scalebox{.3}{\exLemmaBKA }}
    \caption{ $Y$  }
    \label{fig: ExaLemma3.15A}
     \end{subfigure}
     \hfill
     \begin{subfigure}[b]{0.22\textwidth}
         \centering
        {\scalebox{.3}{\exLemmaBKB}}
    \caption{$Z_5 $}
    \label{fig: ExaLemma3.15B}
     \end{subfigure}
          \hfill
     \begin{subfigure}[b]{0.22\textwidth}
         \centering
        {\scalebox{.3}{\exLemmaBKC}}
    \caption{ $Z_6 $}
    \label{fig: ExaLemma3.15C}
     \end{subfigure}
               \hfill
     \begin{subfigure}[b]{0.22\textwidth}
         \centering
        {\scalebox{.3}{\exLemmaBKD}}
    \caption{$ Z_7$}
    \label{fig: ExaLemma3.15D}
     \end{subfigure}
        \caption{An example of \Cref{bk}.}
        \label{fig: EJEMPLO Lemma 3.15}
\end{figure}
\end{example}

\begin{lemma}\rm\label{co: bk corollary}
    Let $K\subset\Phi^{\geq 2}$, $\alpha_{i,j}\in \Phi^{\geq 2}$ and $h=j-i+1$. Let $(r,p)$ be a pair of integers such that either $(r,p)=(i+1,j)$ or  $i+1 < r \leq p \leq j$.  Furthermore, assume that 
    $r\geq h$. 
    Let $$T=\bigcup\limits_{m=0}^{d} [r+m(h-1),p + m(h-1)], $$
    where $d$ is a non-negative integer such that $p+(d+1)(h-1) \leq n$. Suppose  that $K\sim_{T} \Phi^{\succ \alpha_{i,j}}$. If $\lambda\in X$ satisfies $\pint{\lambda}{\alpha_t} =0$ for all $t\in T$, then 
    \begin{equation} \label{eq: bk corollary}
        \bfM_{\lambda}^{Y} = \bfM_{\lambda}^{Y_{d}},
    \end{equation}
    where $Y=K\cup \{\alpha_{r-h,r-1}\}$ and
    $$Y_{d} = (Y\setminus [\alpha_{r+d(h-1),r+(d+1)(h-1)},\alpha_{p+d(h-1),p+(d+1)(h-1)}]) \cup [\alpha_{r-h+1,r},\alpha_{p-h+1,p}].$$
\end{lemma}

\begin{proof}
% We first treat the case when $i+1<r$. 
    We proceed by induction on $d$. The  case $d=0$ is covered by  \Cref{bk}. Let $d\geq 1$ and suppose  the result holds for $d-1$. Then,   we have  $\bfM_{\lambda}^{Y} = \bfM_{\lambda}^{Y_{d-1}}$.
\begin{claim} \label{claim prime}
Let $r'=r+d(h-1)$, $p'=p+d(h-1)$, $i'=p+(d-1)(h-1)$, $j'=p'$, $T'=[r',p']$ and $K^{\prime}= Y_{d-1}\setminus \{\alpha_{r'-h,r'-1}\}$. Then $K'\sim_{T'} \Phi^{\succ \alpha_{i',j'}}$. 
\end{claim}

\begin{proof}
Throughout the proof of the claim we fix $t\in [r',p']$. Suppose that $\alpha \in   \Phi^{\succ \alpha_{i',j'}}$ and  that $s_t(\alpha) \neq \alpha$.  Since $\alpha_{i,j} \prec \alpha_{i',j'}$ we have $\alpha \in \Phi^{\succ \alpha_{i,j}}$. As $[r',p']\subseteq T$ and $K\sim_T \Phi^{\succ\alpha_{i,j}}$ we conclude that $\alpha\in K$. We stress that 
\begin{equation}\label{eq: Y_{d-1}}
\begin{array}{rl} 
  Y_{d-1} =   & (Y\setminus [\alpha_{r'-h+1,r'},\alpha_{p'-h+1,p'}]) \cup [\alpha_{r-h+1,r},\alpha_{p-h+1,p}] \\
     = & (K\setminus [\alpha_{r'-h+1,r'},\alpha_{p'-h+1,p'}]) \cup [\alpha_{r-h,r-1},\alpha_{p-h+1,p}] .
\end{array}
\end{equation}
Since $[\alpha_{r'-h+1,r'},\alpha_{p'-h+1,p'}]\cap \Phi^{\succ \alpha_{i',j'}}  = \emptyset $ we conclude that $\alpha \in Y_{d-1}$.  As $K^{\prime}= Y_{d-1}\setminus \{\alpha_{r'-h,r'-1}\}$ and $\alpha_{r'-h,r'-1} \notin  \Phi^{\succ \alpha_{i',j'}}$ we get $\alpha \in K'$. This proves condition \ref{condition a in T-congruent} in \Cref{def: locally greater than alpha} for $K'$, $T'$ and  $\Phi^{\succ \alpha_{i',j'}}$. 

\smallskip
We now prove condition \ref{condition b in T-congruent}  in \Cref{def: locally greater than alpha} for $K'$, $T'$ and  $\Phi^{\succ \alpha_{i',j'}}$. Let $\alpha \in K' \setminus \Phi^{\succ \alpha_{i',j'}} $. Then $\alpha \in Y_{d-1}$ and $\alpha \neq \alpha_{r'-h,r'-1}$. We must show that $s_t(\alpha) = \alpha$.   Suppose that $\alpha \in [\alpha_{r-h,r-1},\alpha_{p-h+1,p}]$. If $d>1$ then $r+ (h-1) <r' \leq t$. But $p <r+ (h-1) $ holds independently of the value of $d$. We conclude that $p+1<t$  and therefore we have $s_t(\alpha) = \alpha$. Similarly, if  $d=1$ and $i+1<r $ then we have $p+1<t $. Thus in this case we still have $s_t(\alpha) =\alpha$. Finally, suppose that $d=1$, $r=i+1$ and $p=j$. In this case we can only guarantee $p+1\leq t$. Therefore, $s_t(\alpha) = \alpha$ with the exception of the case when $t=p+1=r'=j+1$ and $\alpha =\alpha_{p-h+1,p} =\alpha_{i,j} $. However, we have $ \alpha_{r'-h,r'-1} = \alpha_{i,j}$. Thus the above case must not be considered as $\alpha_{r'-h,r'-1} \notin K'$. Summing up, if $\alpha \in K' \setminus \Phi^{\succ \alpha_{i',j'}} $ and $\alpha \in [\alpha_{r-h,r-1},\alpha_{p-h+1,p}]$ we must have $s_t(\alpha) = \alpha$. 

\smallskip
By the previous paragraph and  \eqref{eq: Y_{d-1}} we can assume that $\alpha \in K\setminus [\alpha_{r'-h+1,r'},\alpha_{p'-h+1,p'}] $. If $\alpha \notin \Phi^{\succ \alpha_{i,j}}$ then $s_{t}(\alpha) =\alpha$ since $K\sim_T \Phi^{\succ \alpha_{i,j}}$. Therefore we can assume $ \alpha_{i,j} \prec \alpha$. As $\alpha\notin  \Phi^{\succ \alpha_{i',j'}}$ and $\alpha \notin [\alpha_{r'-h+1,r'},\alpha_{p'-h+1,p'}]$ we must have $\alpha \in [\alpha_{i+1,j+1}, \alpha_{r'-h,r'-1}]$ (we recall that $i'=p'-h+1$ and $j'=p'$). Furthermore, we can assume $\alpha$ belongs to the half-open interval $[\alpha_{i+1,j+1}, \alpha_{r'-h,r'-1})$ as $\alpha_{r'-h,r'-1}\notin K'$ by definition. Since $r' \leq t$ we have that in the interval  $[\alpha_{i+1,j+1}, \alpha_{r'-h,r'-1})$ all the roots are fixed by $s_t$. This finishes the verification of condition \ref{condition b in T-congruent} in \Cref{def: locally greater than alpha} and the proof of the claim. 
\end{proof}

Let us return to the proof of the lemma. Applying \Cref{bk} to $K'$, $\alpha_{i',j'}$, $r'$ and $p'$ yields $\bfM_{\lambda}^{Y_{d-1}} = \bfM_{\lambda}^{Z}$, where
\begin{equation}\label{eq: Y_d}
    Z= ( Y_{d-1}  \setminus [  \alpha_{r',r'+h-1},\alpha_{p',p'+h-1}  ] ) \cup [ \alpha_{r'-h+1,r'}, \alpha_{p'-h+1,p'} ].
\end{equation}
By combining \eqref{eq: Y_{d-1}} and \eqref{eq: Y_d} it is easy to see that $Z=Y_d$. Therefore, $ \bfM_{\lambda}^{Y} = \bfM_{\lambda}^{Y_{d-1}} = \bfM_{\lambda}^Z  $ as we wanted to show. 
\end{proof}

\begin{example}\rm
We illustrate \Cref{co: bk corollary} with an example. 
Consider $i = 4$ and $j = 7$, so that $h = 4$. 
Let $K = \Phi^{\succ \alpha_{4,7}}$. 
The sets involved in the two cases are displayed in \Cref{fig: EJEMPLO Lemma 3.7 1} and \cref{fig: EJEMPLO Lemma 3.7 2}, depending on the value of $r$. 
For this example, we choose $(r,p) = (6,7)$, which satisfies $i + 1 < r$, whereas the alternative case corresponds to $(r,p) = (5,7)$. 
Using the notation of the lemma, the set $Y$ is depicted in \Cref{fig: ExaLemma3.7A1} and \Cref{fig: ExaLemma3.7A2}. 
If we set $n = 13$, then $d$ has two possible values, namely $d = 0$ and $d = 1$. 
For each choice of $d$, the corresponding set $Y_{d}$ is shown in \Cref{fig: ExaLemma3.7B1} and \Cref{fig: ExaLemma3.7C1} for the first case, and in \Cref{fig: ExaLemma3.7B2} and \Cref{fig: ExaLemma3.7C2} for the second case. 
Finally, by \Cref{co: bk corollary}, for any $\lambda \in X$ such that $\langle \lambda, \alpha_t \rangle = 0$ for all $t \in \{6,7,9,10\}$ in case $i+1 < r$ or $\langle \lambda, \alpha_t \rangle = 0$ for all $t \in [5,10]$ in case $(i+1,j)$, it follows that
\begin{equation}
    \bfM_{\lambda}^{Y} = \bfM_{\lambda}^{Y_{0}} = \bfM_{\lambda}^{Y_{1}}.
\end{equation}
    \begin{figure}[H]
     \centering
     \begin{subfigure}[b]{0.22\textwidth}
         \centering
        {\scalebox{.25}{\begin{tikzpicture}
    \dibu{13}{4}{7}
    \agregof{4}{2}{2}
    \end{tikzpicture}}}
    \caption{ $Y$  }
    \label{fig: ExaLemma3.7A1}
     \end{subfigure}
     %\hfill
     \begin{subfigure}[b]{0.22\textwidth}
         \centering
        {\scalebox{.25}{\begin{tikzpicture}
    \dibu{13}{4}{7}
    \agregof{4}{2}{4}
    \quitof{4}{6}{7}
    \end{tikzpicture}}}
    \caption{$Y_{0}$}
    \label{fig: ExaLemma3.7B1}
     \end{subfigure}
          %\hfill
     \begin{subfigure}[b]{0.22\textwidth}
         \centering
        {\scalebox{.25}{\begin{tikzpicture}
    \dibu{13}{4}{7}
    \agregof{4}{2}{4}
    \quitof{4}{9}{10}
    \end{tikzpicture}}}
    \caption{ $Y_{1} $}
    \label{fig: ExaLemma3.7C1}
     \end{subfigure}
        \caption{Case $i+1 < r\leq p\leq j$ of \Cref{co: bk corollary}.}
        \label{fig: EJEMPLO Lemma 3.7 1}
\end{figure}
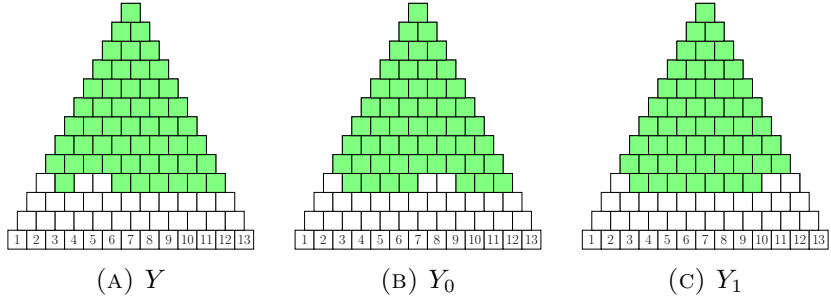
 \begin{figure}[H]
     \centering
     \begin{subfigure}[b]{0.22\textwidth}
         \centering
        {\scalebox{.25}{\begin{tikzpicture}
    \dibu{13}{4}{7}
    \agregof{4}{1}{1}
    \end{tikzpicture}}}
    \caption{ $Y$  }
    \label{fig: ExaLemma3.7A2}
     \end{subfigure}
     %\hfill
     \begin{subfigure}[b]{0.22\textwidth}
         \centering
        {\scalebox{.25}{\begin{tikzpicture}
    \dibu{13}{4}{7}
    \agregof{4}{1}{4}
    \quitof{4}{5}{7}
    \end{tikzpicture}}}
    \caption{$Y_{0} $}
    \label{fig: ExaLemma3.7B2}
     \end{subfigure}
          %\hfill
     \begin{subfigure}[b]{0.22\textwidth}
         \centering
        {\scalebox{.25}{\begin{tikzpicture}
    \dibu{13}{4}{7}
    \agregof{4}{1}{4}
    \quitof{4}{8}{10}
    \end{tikzpicture}}}
    \caption{ $Y_{1} $}
    \label{fig: ExaLemma3.7C2}
     \end{subfigure}
        \caption{Case $(r,p) = (i+1,j)$ of \Cref{co: bk corollary}.}
        \label{fig: EJEMPLO Lemma 3.7 2}
\end{figure}
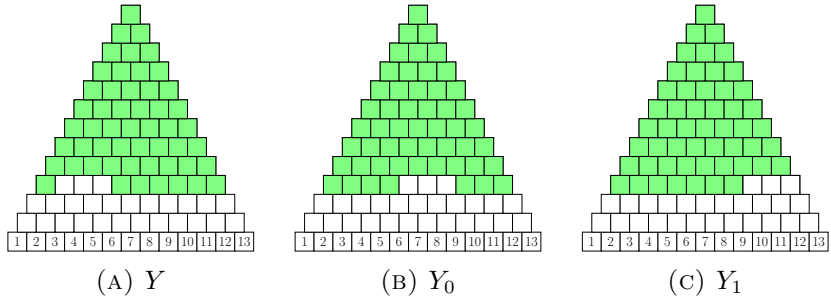   
\end{example}

The following result is a slight variation of \Cref{co: bk corollary}.
The key distinction between both results lies in the $p'$-th coordinate of $\lambda$: in \Cref{co: bk corollary}, it is $0$, whereas in \Cref{lem: right solving}, it is $-1$.
Another difference is that \Cref{co: bk corollary} modifies the superscript of the corresponding $\bfM$-element while maintaining the same weight.
In contrast, \Cref{lem: right solving} preserves the superscript but changes the weight indexing the relevant $\bfM$-element.

\begin{lemma}\rm\label{lem: right solving}
Let $K \subset \Phi^{\geq 2}$ and $\alpha_{i,j} \in \Phi^{\geq 2}$ and $h = j - i + 1$.
Let $(r,p)$ be integers  such that  $i + 1 < r \leq p  \leq j$.
Additionally, assume that $r\geq h$.
Let $$T =  \bigcup_{m = 0}^{d} [r + m(h - 1), p + m(h - 1)],$$
where  $d$ is a non-negative integer such that $p'\coloneqq p + d(h - 1) \leq n$.
Let $\lambda \in X$ satisfying $\pint{\lambda}{\alpha_t} = 0$ for all $t \in T \setminus \{p'\}$, and $\pint{\lambda}{\alpha_{p'}} = -1$.
Then, if  $K \sim_{T} \Phi^{\succ \alpha_{i,j}}$ we have
\begin{equation}  \label{eq: M element with a -1 in p'}
    \bfM_{\lambda}^{Y} = 
\begin{cases} 
q \bfM_{\lambda - \alpha_{{p'} + 1, {p'} + h - 1}}^{Y}, & \text{if } {p'} + h - 1 \leq n; \\
0, & \text{otherwise},
\end{cases}
\end{equation}
where $Y = K \cup \{\alpha_{r - h, r - 1}\}$.
\end{lemma}

\begin{proof}
    Let 
    \begin{equation}
        T_0=  \bigcup_{m = 0}^{d-1} [r + m(h - 1), p + m(h - 1)].
    \end{equation}
    Since $T_0\subset T$ we have $K\sim_{T_0} \Phi^{\succ \alpha_{i,j}}$.  
    It follows from \Cref{co: bk corollary} that $\bfM_{\lambda}^Y = \bfM_{\lambda}^{Y_{d-1}}$  , where $$Y_{d-1} = (Y\setminus [\alpha_{r'-h+1,r'},\alpha_{p'-h+1,p'}]) \cup [\alpha_{r-h+1,r},\alpha_{p-h+1,p}],$$ and $r'=r+d(h-1)$. 
    
    We split the proof into three cases based on whether $ r' + h - 1 $ and $ p' + h - 1 $ are less than, equal to, or greater than $ n $. We stress that $r'+h-1 \leq p'+h-1$.

\bigskip
   \noindent \textbf{Case 1. }  $p'+h-1\leq n$. 
 Let $\mu = \lambda -\alpha_{{p'}+1,{p'}+h-1}$. 
 We have  $\langle  \mu , \alpha_{t} \rangle =\langle \lambda ,\alpha_{t}  \rangle =0 $ for all $t\in T_0$.
Then, \Cref{co: bk corollary} implies that   $\bfM_{\mu}^Y = \bfM_{\mu}^{Y_{d-1}} $.
Let $ i' = p + (d-1)(h-1) $, $ j' = p' $, $ T_1 = [r', p'] $, and $ K' = Y_{d-1} \setminus \{ \alpha_{r'-h, r'-1} \} $. \Cref{claim prime} implies that $ K' \sim_{T_1} \Phi^{\succ \alpha_{i', j'}} $. We notice that we cannot apply \Cref{bk} to $ \lambda $ in this setting since $ \pint{\lambda}{\alpha_{p'}} = -1 $ and $ p' \in T_1 $. However, we can consider $ T_2 = T_1 \setminus \{ p' \} = [r', p'-1] $. Since $ T_2 \subset T_1 $, we have $ K' \sim_{T_2} \Phi^{\succ \alpha_{i', j'}} $. On the other hand, $  \pint{\mu}{\alpha_t} =\pint{\lambda}{\alpha_t} = 0 $ for all $ t \in T_2 $. We are now in a position to apply \Cref{bk}. Consequently, we obtain $ \bfM_{\lambda}^{Y_{d-1}} = \bfM_{\lambda}^{Z} $ and $ \bfM_{\mu}^{Y_{d-1}} = \bfM_{\mu}^{Z} $, where
\begin{equation}
    \begin{array}{rl}
        Z & = \left( Y_{d-1} \setminus  [\alpha_{r', r'+h-1}, \alpha_{p'-1, p'+h-2}]\right) \cup [\alpha_{r'-h+1, r'}, \alpha_{p'-h, p'-1}].  \\
         & = \left( Y \setminus \left( [\alpha_{r', r'+h-1}, \alpha_{p'-1, p'+h-2}] \cup \{ \alpha_{p'-h+1, p'} \} \right) \right) \cup [\alpha_{r-h+1, r}, \alpha_{p-h+1, p}].
    \end{array}
\end{equation}

This yields $ \bfM_{\lambda}^{Y} = \bfM_{\lambda}^{Z} $ and $ \bfM_{\mu}^{Y} = \bfM_{\mu}^{Z} $.

On the other hand, since $ K \sim_T \Phi^{\succ \alpha_{i, j}} $ and $ j < p' \in T $, we conclude via \Cref{lem: las que se salen} and \eqref{eq: las que se salen} that $ K \setminus s_{p'}(K) = \{ \alpha_{p'-h+1, p'}, \alpha_{p', p'+h-1} \} $. Furthermore, the set $ Z $ is obtained from $ K $ by adding and eliminating a number of roots. Among these roots, the only one that is not fixed by $ s_{p'} $ is $ \alpha_{p'-h+1, p'} $, and this root has been eliminated from $ K $ to obtain $ Z $. Thus, $ Z \setminus s_{p'}(Z) = \{ \alpha_{p', p'+h-1} \} $. Since $ s_{p'}(\alpha_{p', p'+h-1}) = \alpha_{p'+1, p'+h-1} $, the set $ Z \cup \{ \alpha_{p'+1, p'+h-1} \} $ is $ s_{p'} $-invariant. Therefore, the fact that  $\langle \lambda ,\alpha_{p'} \rangle = -1 $ and  \Cref{lem: weyl group over M elements} imply that
\begin{equation}
    0 = \bfM_{\lambda}^{Z \cup \{ \alpha_{p'+1, p'+h-1} \}} = \bfM_{\lambda}^{Z} - q \bfM_{\lambda - \alpha_{p'+1, p'+h-1}}^{Z}  = \bfM_{\lambda}^{Z} - q \bfM_{\mu}^{Z}.
\end{equation}

We conclude that $ \bfM_\lambda^Y = q \bfM_\mu^Y $, as desired.

\bigskip
\noindent \textbf{Case 2. }  $p'+h-1 > n$ and $r'+h-1 \leq n$. 

 Let $p''= n-h+1$ and $T_{3} = [r',p'']$. Also, consider  $i'$, $j'$,  $T_{1}$ and $K'$ as defined in \textbf{Case 1}. By \Cref{claim prime} we obtain  $K'  \sim_{T_{1}} \Phi^{\succ \alpha_{i',j'}}$. As $T_{3} \subset T_{1}$, we conclude that $K' \sim_{T_{3}} \Phi^{\succ \alpha_{i',j'}}$.  By applying \Cref{bk} with respect to $K'$, $\Phi^{\succ \alpha_{i',j'}}$, $(r', p'')$, and $\lambda$, we get $\bfM_{\lambda}^{Y_{d-1}} = \bfM_{\lambda}^{Z_1}$, where 
\begin{equation}
\begin{array}{rl}
    Z_1 =& \left(Y_{d-1} \setminus [\alpha_{r',r'+h-1}, \alpha_{p'',n}] \right) \cup [\alpha_{r'-h+1,r'}, \alpha_{p''-h+1,p''}] \\
     =& \left(Y \setminus \left([\alpha_{p''-h+2,p''+1}, \alpha_{p'-h+1,p'}] \cup [\alpha_{r',r'+h-1}, \alpha_{p'',n}] \right) \right) \cup [\alpha_{r-h+1,r}, \alpha_{p-h+1,p}].
\end{array}
\end{equation}
Thus,  $\bfM_{\lambda}^{Y} = \bfM_{\lambda}^{Z_1}$.

\bigskip
Let $f=p^{\prime} - p^{\prime \prime}$. We notice that $f\geq 1$.  For $2\leq k \leq f$ we inductively define
\begin{equation}
    Z_k= Z_{k-1} \cup \{  \alpha_{ p^{\prime \prime}+(k-1)-(h-1) , p^{\prime \prime}+(k-1)}  \}. 
\end{equation}
A precise inspection of the definition of set  $Z_k$ shows that it is $s_{p^{\prime \prime}+k}$-invariant. Thus \Cref{lem: weyl group over M elements} and $\pint{\lambda -\alpha_{ p^{\prime \prime}+(k-1)-(h-1) , p^{\prime \prime}+(k-1)}}{\alpha_{p^{\prime \prime}+k-1}}=-1$ imply
\begin{equation}
    \bfM_{\lambda}^{Z_k}  =   \bfM_{\lambda}^{Z_{k-1}}-q  \bfM_{\lambda -\alpha_{ p^{\prime \prime}+(k-1)-(h-1) , p^{\prime \prime}+(k-1)} }^{Z_{k-1}}   =  \bfM_{\lambda}^{Z_{k-1}}. 
\end{equation}

We conclude that $\bfM^{Y}_{\lambda} = \bfM^{Z_1}_{\lambda} = \bfM^{Z_2}_{\lambda} = \cdots = \bfM^{Z_f}_{\lambda}$. Since $Z_f$ is $s_{p'}$-invariant ($p'' + f = p'$) and $\langle \lambda, \alpha_{p'} \rangle = -1$, it follows from \Cref{lem: weyl group over M elements} that $\bfM^{Z_f}_{\lambda} = 0$. Therefore, $\bfM^{Y}_{\lambda} = 0$, which completes the proof in this case. 

\bigskip
\noindent \textbf{Case 3. }   $r'+h-1 > n$. 
Let $f=p'-r'$. We notice that $f\geq 0$.  We define $Z_0=Y_{d-1}$ and for $0<k\leq f$ we inductively define 
\begin{equation}
    Z_k = Z_{k-1} \cup \{  \alpha_{r'+(k-1)-(h-1), r'+(k-1)} \}.
\end{equation}
The set $Z_{k}$ is $s_{r'+k}$-invariant for all $0\leq k \leq f$. Thus,   $\pint{\lambda -\alpha_{r'+(k-1)-(h-1), r'+(k-1)}}{\alpha_{r'+(k-1)}}  =-1 $ and \Cref{lem: weyl group over M elements}  imply
\begin{equation}
    \bfM_{\lambda}^{Z_k}  =   \bfM_{\lambda}^{Z_{k-1}}-q  \bfM_{\lambda - \alpha_{r'+(k-1)-(h-1), r'+(k-1)}  }^{Z_{k-1}}   =  \bfM_{\lambda}^{Z_{k-1}}, 
\end{equation}
for all  $1\leq k\leq f$. As before we conclude that $\bfM^{Y}_{\lambda} = \bfM^{Z_0}_{\lambda} = \bfM^{Z_1}_{\lambda} = \cdots = \bfM^{Z_f}_{\lambda}$. Since $Z_f$ is $s_{p'}$-invariant ($r' + f = p'$) and $\langle \lambda, \alpha_{p'} \rangle = -1$, it follows from \Cref{lem: weyl group over M elements} that $\bfM^{Z_f}_{\lambda} = 0$. Therefore, $\bfM^{Y}_{\lambda} = 0$, which completes the proof in this case. 
\end{proof}

\begin{example}\rm
For the reader's convenience we reproduce some relevant steps of the proof of \Cref{lem: right solving} in an example. 
Consider $i = 5$ and $j = 9$, so that $h = 5$. 
Let $K = \Phi^{\succ \alpha_{5,9}}$,  $r = 7$, $p = 9$ and $d = 1$. 
In this case we have
$$  T=[7,9] \cup [11,13],  \qquad T_0=[7,9] \qquad \mbox{and} \qquad T_2=[11,12].  $$
Take $\lambda = -\varpi_{13}$. 
We stress that $\langle \lambda, \alpha_t \rangle =0$ for $t\in T\setminus\{13\}$ and $\langle \lambda , \alpha_{13}\rangle =-1$.

To determinate which case of the proof applies, we must specify the value of $n$.
%To decide in which case of the proof we are we need to specify the value of $n$.
We set $n=17$, so that we are in Case 1 of the proof. 
We set $\mu = \lambda-\alpha_{14,17}$. 
We have $Y = \Phi^{\succ \alpha_{5,9}} \cup \{\alpha_{2,6}\}$.
This set is illustrated in \Cref{fig: ExaLemma3.10A}. 
Then we apply \Cref{co: bk corollary} (or even \Cref{bk} since in this case $d=1$) to obtain the equalities $\bfM_\lambda^Y=\bfM_\lambda^{Y_0}$ and $\bfM_\mu^Y=\bfM_\mu^{Y_0}$, where $Y_0$ is illustrated in \Cref{fig: ExaLemma3.10B}. 
Graphically, the set $Y_0$ is obtained from  $Y$ by adding the three boxes to the right of the box associated to $\alpha_{2,6}$ and by eliminating the three boxes to the right of the box associated to $\alpha_{6,10}$.

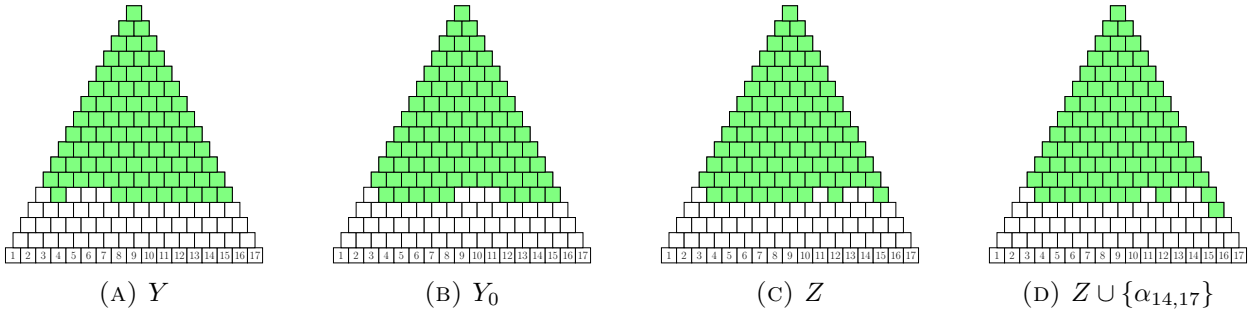
\begin{figure}[H]
     \centering
     \begin{subfigure}[b]{0.22\textwidth}
         \centering
        {\scalebox{.2}{\begin{tikzpicture}
    \dibu{17}{5}{9}
    \agregof{5}{2}{2}
    \end{tikzpicture}}}
    \caption{ $Y$  }
    \label{fig: ExaLemma3.10A}
     \end{subfigure}
     \hfill
     \begin{subfigure}[b]{0.22\textwidth}
         \centering
        {\scalebox{.2}{\begin{tikzpicture}
    \dibu{17}{5}{9}
    \agregof{5}{2}{5}
    \quitof{5}{7}{9}
    \end{tikzpicture}}}
    \caption{ $Y_{0}$  }
    \label{fig: ExaLemma3.10B}
     \end{subfigure}
     \hfill
     \begin{subfigure}[b]{0.22\textwidth}
         \centering
        {\scalebox{.2}{\begin{tikzpicture}
    \dibu{17}{5}{9}
    \agregof{5}{2}{5}
    \quitof{5}{9}{9}
    \quitof{5}{11}{12}
    \end{tikzpicture}}}
    \caption{$Z$}
    \label{fig: ExaLemma3.10C}
     \end{subfigure}
     \hfill
     \begin{subfigure}[b]{0.22\textwidth}
         \centering
        {\scalebox{.2}{\begin{tikzpicture}
    \dibu{17}{5}{9}
    \agregof{5}{2}{5}
    \agregof{4}{14}{14}
    \quitof{5}{9}{9}
    \quitof{5}{11}{12}
    \end{tikzpicture}}}
    \caption{$Z\cup\{\alpha_{14,17}\}$}
    \label{fig: ExaLemma3.10D}
     \end{subfigure}
        \caption{Case 1 in the proof of \Cref{lem: right solving}.}
        \label{fig: EJEMPLO Lemma 3.10}
\end{figure}

We now apply \Cref{bk} to obtain the equalities  $\bfM_\lambda^{Y_0}=\bfM_\lambda^Z$ and $\bfM_\mu^{Y_0}=\bfM_\mu^Z$, where $Z$ is depicted in \Cref{fig: ExaLemma3.10C}. We notice that $Z$ is obtained from $Y_0$ by adding the two boxes to the right of $\alpha_{6,10}$ and by eliminating the two roots to the right of $\alpha_{10,14}$. We finally add the root $\alpha_{14,17}$ to the set $Z$. This set is depicted in \Cref{fig: ExaLemma3.10D}. The important thing here is that this set, $Z\cup \{\alpha_{14,17}\}$, is invariant under $s_{13}$. Thus we apply \Cref{lem: weyl group over M elements} to conclude that $\bfM_{\lambda}^{Z\cup \{\alpha_{14,17}\}}=0$. Therefore, $\bfM_\lambda^Y = q\bfM^Y_{\lambda-\alpha_{14,17}}$. 

\bigskip
To replicate the arguments from case 2, we select $n = 15$ to satisfy the conditions specific to this case. The set $Y = \Phi^{\succ \alpha_{5,9}} \cup \{\alpha_{2,6}\}$ is depicted in \Cref{fig: ExaLemma3.10Acase2} for this particular choice of $n$. 
By applying \Cref{co: bk corollary}, we deduce that $\bfM_{\lambda}^{Y} = \bfM_{\lambda}^{Y_{0}}$, similar to the previous case, where the set $Y_{0}$ is illustrated in \Cref{fig: ExaLemma3.10Bcase2}. 

\begin{figure}[H]
     \centering
     \begin{subfigure}[b]{0.22\textwidth}
         \centering
        {\scalebox{.2}{\begin{tikzpicture}
    \dibu{15}{5}{9}
    \agregof{5}{2}{2}
    \end{tikzpicture}}}
    \caption{ $Y$  }
    \label{fig: ExaLemma3.10Acase2}
     \end{subfigure}
     \hfill
     \begin{subfigure}[b]{0.22\textwidth}
         \centering
        {\scalebox{.2}{\begin{tikzpicture}
    \dibu{15}{5}{9}
    \agregof{5}{2}{5}
    \quitof{5}{7}{9}
    \end{tikzpicture}}}
    \caption{ $Y_{0}$  }
    \label{fig: ExaLemma3.10Bcase2}
     \end{subfigure}
     \hfill
     \begin{subfigure}[b]{0.22\textwidth}
         \centering
        {\scalebox{.2}{\begin{tikzpicture}
    \dibu{15}{5}{9}
    \agregof{5}{2}{5}
    \quitof{5}{8}{9}
    \quitof{5}{11}{11}
    \end{tikzpicture}}}
    \caption{$Z_{1}$}
    \label{fig: ExaLemma3.10Ccase2}
     \end{subfigure}
     \hfill
     \begin{subfigure}[b]{0.22\textwidth}
         \centering
        {\scalebox{.2}{\begin{tikzpicture}
    \dibu{15}{5}{9}
    \agregof{5}{2}{5}
    \quitof{5}{9}{9}
    \quitof{5}{11}{11}
    \end{tikzpicture}}}
    \caption{$Z_{2}$}
    \label{fig: ExaLemma3.10Dcase2}
     \end{subfigure}
        \caption{Case 2 in the proof of \Cref{lem: right solving}.}
        \label{fig: EJEMPLO Lemma 3.10case2}
\end{figure}
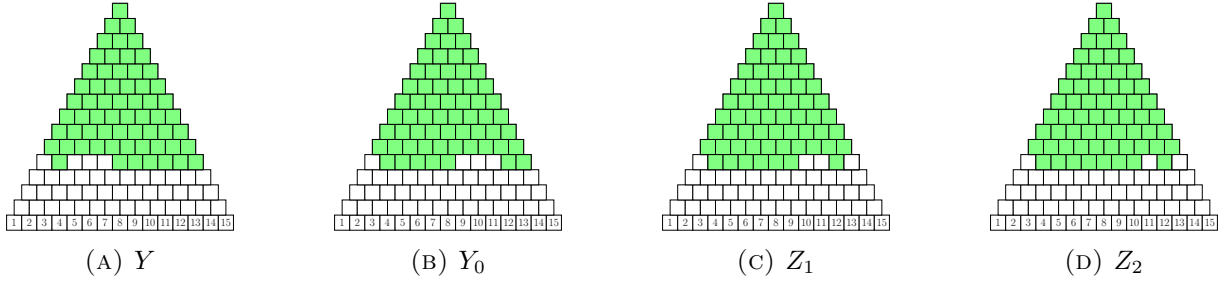

At this stage, it is not possible to directly apply \Cref{bk} with $T_{2}$, as it does not satisfy the condition $12 + h - 1 = 16 \leq n$. Instead, we select the largest subset of $T_{2}$ that meets the required criteria. In this instance, we choose $T_{3} = \{11\} \subset T_{2}$. Consequently, it follows that $\bfM_{\lambda}^{Y_{0}} = \bfM_{\lambda}^{Z_{1}}$, where the set $Z_{1}$ is represented in \Cref{fig: ExaLemma3.10Ccase2}. 
By the construction of $Z_{1}$ and our choice of $n$, we observe that it is $s_{12}$-invariant. Thus, applying \Cref{lem: weyl group over M elements}, we obtain $\bfM_{\lambda}^{Z_{1}} = \bfM_{\lambda}^{Z_{2}}$, where $Z_{2}$ is derived from $Z_{1}$ by adding the root $\alpha_{8,12}$. 
Furthermore, the construction of $Z_{2}$ ensures that it is $s_{13}$-invariant. Combined with the condition $\langle \lambda, \alpha_{13} \rangle = -1$, this allows us to conclude that $\bfM_{\lambda}^{Z_{2}} = 0$. Therefore $\bfM_{\lambda}^{Y} =0$. 

\bigskip
The arguments in case 3 are very similar to those in case 2, so we omit its example.
\end{example}

\Cref{lem: right solving} marks the first instance where our philosophy begins to emerge: attempting to eliminate negative coordinates in weights indexing $\bfM$-elements. Specifically, for $\lambda$ as defined in the lemma, we observe that $\pint{\lambda}{\alpha_{p'}} = -1$ and $\pint{\lambda - \alpha_{{p'} + 1, {p'} + h - 1}}{\alpha_{p'}} = 0$ (where $\lambda - \alpha_{{p'} + 1, {p'} + h - 1}$ is the weight appearing on the right-hand side of \eqref{eq: M element with a -1 in p'}). However, if $\pint{\lambda}{\alpha_{p'+h-1}} = 0$, this would introduce a $-1$ in the ${p'} + h - 1$ coordinate of $\lambda - \alpha_{{p'} + 1, {p'} + h - 1}$. The significance here is that this new $-1$ coordinate is positioned to the right of the original one, suggesting a potential resolution to this issue.

\begin{lemma}\rm\label{Claim: free Q}
        Let $K \subset \Phi^{\geq 2}$, $\alpha_{i,j} \in \Phi^{\geq 2}$ and  $h = j - i + 1$.
        Let $(r,p)$ be integers with $i + 1 < r \leq p \leq j$. 
        Let  $d$ be a positive integer such that $p+(d+1)(h-1)\leq n$. 
        Let 
        $$T = \bigcup_{m = 0}^{d} [r + m(h - 1), p + m(h - 1)].$$
        Let $0\leq k \leq d$ and set $p'=p+k(h-1)$. 
        Let $\lambda \in X$ satisfying $\pint{\lambda}{\alpha_t} = 0$ for all $t \in T \setminus \{p'\}$ and $\pint{\lambda}{\alpha_{p'}} = -1$. 
        Then, if $K \sim_{T} \Phi^{\succ \alpha_{i,j}}$, we have
    \begin{equation}
        \bfM_{\lambda}^{Y} = q^{d-k+1}\bfM_{\lambda-\alpha_{p'+1,p'+(d-k+1)(h-1)}}^{Y},
    \end{equation}
    where $Y = K \cup \{\alpha_{r-h, r-1}\}$.
    \end{lemma}
    \begin{proof}
We fix $d$ and proceed by induction on $d-k$. 
The base case $d=k$ is \Cref{lem: right solving}. We now assume that  $0<d-k$ and that the result holds for all $k'$ such that $d-k'<d-k$.   We apply \Cref{lem: right solving} using the set 
         $$T' = \bigcup_{m = 0}^{k} [r + m(h - 1), p + m(h - 1)].$$
We obtain $\bfM_{\lambda}^{Y}=q\bfM_{\mu}^{Y}$, where $\mu=\lambda - \alpha_{p'+1,p'+(h-1)}$. Then, we apply our inductive hypothesis to  $k'=k+1$, $T$ and $\mu$ to obtain 
       \begin{equation}
           \bfM_{\mu}^{Y}=q^{d-k'+1}\bfM_{\mu-\alpha_{p'+h,p'+h-1+(d-k'+1)(h -1)}}^{Y} =q^{d-k}\bfM_{\lambda-\alpha_{p'+1,p'+(d-k+1)(h -1)}}^{Y} .
       \end{equation}
Therefore, $\bfM_{\lambda}^{Y} = q^{d-k+1}\bfM_{\lambda-\alpha_{p'+1,p'+(d-k+1)(h -1)}}^{Y}$ as we wanted to show. 
   \end{proof} 

\begin{example}\rm
    We proceed to illustrate \Cref{Claim: free Q} with an example. Consider $i = 3$, $j = 6$, $(r, p) = (5, 6)$, and $k = 1$.  
Additionally, let $K = \Phi^{\succ \alpha_{3,6}}$.  
The value of $d$ must be chosen to be at least $k$, with its upper limit determined by the value of $n$. Suppose $n=18$, and let us illustrate, the different results obtained for various values of $d$.  
Under this setting, $d = 1, 2, 3$. Thus, for $T = \{5,6, 8,9,11,12,14,15\}$ and $\lambda \in X$ satisfying $\langle \lambda, \alpha_{t} \rangle = 0$ for $t \in T\setminus \{9\}$ and $\langle \lambda, \alpha_{9} \rangle = -1$, we have, by \Cref{Claim: free Q},  
\begin{equation}
    \bfM_{\lambda}^{Y} = q\bfM_{\lambda - \alpha_{9,12}}^{Y} = q^2\bfM_{\lambda - \alpha_{9,15}}^{Y} = q^3\bfM_{\lambda - \alpha_{9,18}}^{Y}.
\end{equation}

\end{example}

The following lemma represents a minor variation of \Cref{Claim: free Q}. In the notation of this lemma, we set $r=i+1$ and adjust the set $T$ to allow the application of \Cref{co: bk corollary} in this context. The proof follows a similar approach to that of \Cref{Claim: free Q}; therefore, we omit it for brevity.

\begin{lemma}\rm\label{Claim: free Q2}
Let $K \subset \Phi^{\geq 2}$, $\alpha_{i,j} \in \Phi^{\geq 2}$, and set $h:=j-i+1$.
Let $p$ be an integer with $i+1 \le p \le j$, and assume $i\geq h$.
Let $d$ be a positive integer such that $p+d(h-1)\le n$, and define
\begin{equation}
T = [\,i+1+d(h-1),\, p+d(h-1)\,] \;\cup\;
\bigcup_{m=0}^{d-1} [\,i+1+m(h-1),\, j+m(h-1)\,].
\end{equation}
Let $0\leq k\leq d$ and set $p':=p+k(h-1)$.
Let $\lambda \in X$ satisfy
\begin{equation}
\langle \lambda,\alpha_{p'}\rangle=-1
\quad\text{and}\quad
\langle \lambda,\alpha_t\rangle=0 \ \text{for all } t\in T\setminus\{p'\}.
\end{equation}
If $k<d$, additionally assume $\langle \lambda,\alpha_{p'+1}\rangle=1$.
Suppose $K \sim_T \Phi^{\succ \alpha_{i,j}}$.
Then
\begin{equation}
\bfM_{\lambda}^{Y} =
\begin{cases}
q^{\,d-k+1}\,
\bfM_{\lambda - \alpha_{\,p'+1,\, p'+(d-k+1)(h-1)}}^{Y}, &
\text{if } p+(d+1)(h-1)\le n,\\[4pt]
0, & \text{otherwise},
\end{cases}
\end{equation}
where $Y = K \cup \{\alpha_{\,i-h+1,\, i}\}$.
\end{lemma}

\begin{lemma}\rm\label{lem: left solving}
Let $K \subset \Phi^{\geq 2}$, $\alpha_{i,j} \in \Phi^{\geq 2}$ and $h = j - i + 1$. 
Let $p$ be an integer such that $i -h +1\leq p \leq i$ and 
$$T =  \bigcup_{m = 0}^{d} [p-mh,i-mh],$$
where $d$ is a positive integer such that $p'\coloneqq p - dh \geq 1$.
Let $\lambda \in X$ satisfying $\pint{\lambda}{\alpha_t} = 0$ for all $t \in T \setminus \{p'\}$, and $\pint{\lambda}{\alpha_{p'}} = -1$. Then, if  $K \sim_{T} \Phi^{\succ \alpha_{i,j}}$ we have
\begin{equation}\label{eq: left solving}
    \bfM_{\lambda}^{K} = 
\begin{cases} 
q \bfM_{\lambda - \alpha_{{p'} -h, {p'} - 1}}^{K}, & \text{if } {p'} - h \geq 1; \\
0, & \text{otherwise}.
\end{cases}
\end{equation}
\end{lemma} 

\begin{proof}
 We first treat the case  $p'-h\geq 1$.  In this case $\alpha_{p'-h,p'-1}$ exists. Furthermore, since $K\sim_T \Phi^{\succ\alpha_{i,j}}$ and $p'\in T$ we have $\alpha_{p'-h,p'-1}\not \in K$. Therefore, 
 \begin{equation}
     \bfM_{\lambda}^{K\cup \{  \alpha_{p'-h,p'-1}   \} } = \bfM_{\lambda }^K- q\bfM_{\lambda - \alpha_{p'-h,p'-1}}^{K}. 
 \end{equation}
Thus it suffices to show that  $\bfM_{\lambda}^{K\cup \{  \alpha_{p'-h,p'-1}   \} }=0$. Set $K_0 =K\cup \{  \alpha_{p'-h,p'-1}   \} $  and
\begin{equation}
\beta_m = \begin{cases}
   \alpha_{i,j}, & \mbox{if } m=0; \\
   \alpha_{i-mh,i-(m-1)h},    &  \mbox{if } 1\leq m \leq d+1.  
\end{cases}
\end{equation}
 We also define $K_m=(K_{m-1}\setminus \{\beta_m\}) \cup \{ \beta_{m-1} \}$ for $1\leq m \leq d+1$.  
 By combining the fact that $K \sim_{T} \Phi^{\succ\alpha_{i,j}}$ together with \Cref{lem: las que se salen} and \eqref{eq: las que se salen} we conclude that $K_m \setminus s_{i-mh}(K_m)  =\{  \beta_{m+1} \}$ for $0\leq m < d$.   Therefore, a successive application of \Cref{lem: sacar una poner otra} for $m=0,1,\ldots , d-1$  at position $i-mh$ to  the triple $(K_m,\beta_{m+1},\beta_m )$ yields 
\begin{equation}  \label{eq: K0 equal to Kd}
    \bfM_{\lambda}^{K_0} =\bfM_{\lambda}^{K_d}.  
\end{equation}

We proceed by induction on $i-p$. 
If $p=i$ then  $p'=i-dh$ and  $K_d=  s_{p'}(K_d)$.  Since $\pint{\lambda}{\alpha_{p'}}=-1$ it follows by \Cref{lem: weyl group over M elements} that $\bfM_{\lambda}^{K_d}=0$, which proves the lemma when $p=i$.

 We now assume that $p<i$.
 In this case, $K_d\setminus s_{i-dh}(K_d)= \{ \beta_{d+1} \}$. 
 Then another application of \Cref{lem: sacar una poner otra} at position $i-dh$ to the triple $(K_d,\beta_{d+1}, \beta_d)$ yields 
\begin{equation}  \label{eq: K0 equal to Kd}
    \bfM_{\lambda}^{K_0} =\bfM_{\lambda}^{K_d}= \bfM_{\lambda}^{K_{d+1}}.  
\end{equation}
We stress that $K_{d+1} = (K_0 \setminus \{ \beta_{d+1} \}) \cup \{\beta_0\}  = (K \setminus \{ \alpha_{i-(d+1)h,i-dh} \}) \cup \{\alpha_{p'-h,p'-1},\alpha_{i,j}\}$. 
Using this description and the fact that $K\sim_T \Phi^{\succ\alpha_{i,j}}$ it is easy to see  that $K_{d+1}\setminus \{ \alpha_{p'-h,p'-1} \} \sim_{T'} \Phi^{\succ\alpha_{i-1,j-1}}$, where 
$$T' =  \bigcup_{m = 0}^{d} [p-mh,i-1-mh].$$
Since $(i-1)-p < i-p$ we can apply our inductive hypothesis to conclude that $\bfM_\lambda^{K_{d+1}} =0$. This finishes the proof in the case $p'-h\geq 1$. 

The case $p'-h<1$ is handled similarly.
Indeed, we repeat the same argument but with $K$ playing the role of  $K_0=K\cup \{  \alpha_{p'-h,p'-1}   \}  $, as the root $\alpha_{p'-h,p'-1} $ does not exist in this case. For the sake of brevity we leave the details to the reader. 
\end{proof}

\begin{lemma}\rm\label{Claim: free P}
Let $K \subset \Phi^{\geq 2}$, $\alpha_{i,j} \in \Phi^{\geq 2}$ and $h = j - i + 1$. Let $p$ be an integer such that $i -h+1 < p \leq i$. Let $d$ be a positive integer such that $p-(d+1)h\geq 1$. Let 
    \begin{equation}
        T=\bigcup_{m=0}^{d}[p-mh,i-mh].
    \end{equation}
    Let $1\leq k \leq d$ and set $p'= p-kh$. 
    Let $\lambda\in X$ satisfying $\pint{\lambda}{\alpha_{t}}=0$ for all $t\in T\setminus\{p'\}$ and $\pint{\lambda}{\alpha_{p'}}=-1$. Then, if $K \sim_{T} \Phi^{\succ \alpha_{i,j}}$, we have
    \begin{equation}
        \bfM_{\lambda}^{K} = q^{d-k+1}\bfM_{\lambda - \alpha_{p'-(d-k+1)h,p'-1}}^{K}.
    \end{equation}
\end{lemma}

\begin{proof}
    For a fixed $d$, we proceed by induction on $d - k$.  
The base case $d = k$ corresponds to \Cref{lem: left solving}. Now, suppose that $d - k > 0$ and that the result holds for all $k'$ such that $d - k' < d - k$.  
We apply \Cref{lem: left solving} to $K$, $\lambda$, $\alpha_{i,j}$, $p$, and $k$, obtaining $\bfM_{\lambda}^{K} = q\bfM_{\mu}^{K}$, where $\mu = \lambda - \alpha_{p' - h, p' - 1}$.  
Next, using the inductive hypothesis on $K$, $\mu$, $\alpha_{i,j}$, $k' = k + 1$, and $d$, we have  
\begin{equation}
    \bfM_{\mu}^{K} = q^{d - k' + 1} \bfM_{\mu - \alpha_{p' - h - (d - k' + 1)h, p' - h - 1}}^{K} = q^{d - k} \bfM_{\lambda - \alpha_{p' - (d - k + 1)h, p' - 1}}^{K}.
\end{equation}  
Thus, we conclude that $\bfM_{\lambda}^{K} = q^{d - k + 1} \bfM_{\lambda - \alpha_{p' - (d - k + 1)h, p' - 1}}^{K}$, as required.
    
\end{proof}

\begin{definition}\rm \label{def: admissible}
Let $Z\subseteq \Phi^{\geq 2}$ and $I = (i_{1},i_{2},\ldots , i_{p})$ be a sequence of positive integers such that $1\leq i_{1}< i_{2} < \cdots < i_{p}\leq n$. Let 
\begin{equation}
    I_\alpha = \bigcup\limits_{a=1}^{p-1} \{\alpha_{i_{a},i_{a+1}}\}.
\end{equation}
For $1\leq b\leq p$, we say that $I$ is $(Z,i_{b})$-admissible if the following conditions hold.
\begin{enumerate}[(a)]
\item \label{def: admissible uno} $I_\alpha\subseteq Z$.
\item \label{def: admissible dos} For $1\leq a \leq p-1$ if $i_{a+1}-i_a=1$ then we have $a=p-1$ and $b\neq p$. 
\item  \label{def: admissible tres} $s_{i_{q}}(Z\setminus I_{\alpha}) = Z\setminus I_{\alpha}$ for all $2\leq q \leq p-1$ and $q=b$.
\end{enumerate}
\end{definition}

\begin{remark}
    Notice that the condition  $q=b$ in \Cref{def: admissible}\ref{def: admissible tres} is redundant unless $b=1$ or $b=p$.
\end{remark}

\begin{lemma}\rm\label{lem: desc admissible set.}
    Let $Z \subseteq \Phi^{\geq 2}$ and $I=(i_{1},i_{2},\ldots i_{p})$ be a sequence of positive integers such that $I$ is $(Z,b)$-admissible for some $1\leq b \leq p$. Let $\lambda \in X$ such that $\langle \lambda,\alpha_{i_{q}}\rangle = 0$ for all $2\leq q \leq p-1$ with $q\neq b$, and $\langle\lambda,\alpha_{i_b}\rangle = -1$. Then, we have
    \begin{equation} \label{eq: lema admisible set}
        \bfM_{\lambda}^{Z} = \left\{\begin{array}{ll}
            q^{b-1}\bfM_{\lambda - \alpha_{i_{1},i_{b}-1}}^{Z} + q^{p-b}\bfM_{\lambda - \alpha_{i_{b}+1,i_{p}}}^{Z} - q^{p-1}\bfM_{\lambda - \alpha_{i_{1},i_{b}-1} - \alpha_{i_{b}+1,i_{p}}}^{Z},  & \text{if } 2\leq b \leq p-1; \\
            & \\
            q^{p - 1 }\bfM_{\lambda - \alpha_{i_{1},i_{p}-1}}^{Z}, & \text{if } b=p;\\
            & \\
            q^{p - 1 }\bfM_{\lambda - \alpha_{i_{1}+1,i_{p}}}^{Z}, & \text{if } b=1.\\
        \end{array}\right.
    \end{equation}
\end{lemma}

\begin{proof}
By definition of $\bfM$-elements we have
 \begin{equation}\label{eq: defM admissible 1}
        \bfM_{\lambda}^{Z} = \sum_{J\subset Y} (-q)^{\mid J\mid } \bfM_{\lambda - \Sigma_{J}}^{Z\setminus Y} . 
    \end{equation}

We first treat the case  $2\leq b \leq p-1$. Set $Y = I_{\alpha}$.    
We will study each one of the terms occurring in  the sum of the right-hand side of \eqref{eq: defM admissible 1}. We claim that

\begin{equation}\label{eq: sum proof admisssible}
  \bfM_{\lambda - \Sigma_{J}}^{Z\setminus Y}  = \left\{ 
  \begin{array}{ll}
     (-1)^{p-2}\bfM_{\lambda - \alpha_{i_{1},i_{b}-1} -\alpha_{i_{b} +1 , i_{p}}}^{Z\setminus Y},   & \mbox{if }  J=Y; \\
       &   \\
    (-1)^{b-1}\bfM_{\lambda - \alpha_{i_{1},i_{b}-1}}^{Z\setminus Y},   & \mbox{if }  J = \bigcup\limits_{a=1}^{b-1} \{\alpha_{i_{a},i_{a+1}}\};  \\
    &  \\
    (-1)^{p-b}\bfM_{\lambda - \alpha_{i_{b}+1,i_{p}}}^{Z\setminus Y} , & \mbox{if } J = \bigcup\limits_{a=b}^{p-1} \{\alpha_{i_{a},i_{a+1}}\};\\
    & \\
    0,& \mbox{otherwise.}
  \end{array}
  \right.   
\end{equation}

    % For $2\leq a \leq p-1 $ with $a\neq b$, since $\lambda_{i_{a}} = 0$ thus 

    % \begin{equation}\label{eq: }
    % \begin{array}{rl}
    %     s_{i_a}\cdot(\lambda-\alpha_{i_{a-1},i_a} - \alpha_{i_a,i_{a+1}}) =& \lambda - \alpha_{i_{a-1},i_{a+1}} \\
    %     s_{i_a}\cdot(\lambda-\alpha_{i_{a-1},i_a}) =& \lambda - \alpha_{i_{a-1},i_a}\\
    %     s_{i_a}\cdot(\lambda-\alpha_{i_a,i_{a+1}}) =& \lambda - \alpha_{i_a,i_{a+1}}
    % \end{array}
    % \end{equation}

    % if $a=b$ hence
    % \begin{equation}
    %     \begin{array}{rl}
    %         s_{i_b}\cdot(\lambda-\alpha_{i_{b-1},i_b}-\alpha_{i_b,i_{b+1}}) =& \lambda-\alpha_{i_{b-1},i_b-1}-\alpha_{i_b+1,i_{b+1}} \\
    %         s_{i_b}\cdot(\lambda-\alpha_{i_{b-1},i_b}) =& \lambda - \alpha_{i_{b-1},i_b-1}\\
    %         s_{i_b}\cdot(\lambda-\alpha_{i_b,i_{b+1}}) =& \lambda - \alpha_{i_{b} +1,i_{b+1}}\\
    %         s_{i_b}\cdot(\lambda) =& \lambda
    %     \end{array}.
    % \end{equation}

    % Consequently, 
    If $J=Y$ then by applying  \Cref{lem: weyl group over M elements} for $w=s_{i_{2}}s_{i_{3}}\cdots s_{i_{p-1}}$ we get

    \begin{equation}\label{eq: proof admissible 1}
        \bfM_{\lambda - \sum_{Y}}^{Z\setminus Y} = (-1)^{p-2}\bfM_{w\cdot(\lambda - \sum_{Y})}^{w(Z\setminus Y)} = (-1)^{p-2}\bfM_{w\cdot(\lambda - \sum_{Y})}^{Z\setminus Y},
    \end{equation}
    where the last equality follows by the fact that $I$ is a  $(Z,b)$-admissible set. 
    
    We now compute  $w\cdot(\lambda - \Sigma_{Y})$. We notice that for $2\leq j\leq p-1$ we have
    \begin{equation}
       \left\langle \lambda - \Sigma_Y,\alpha_{i_{j}} \right\rangle  = \left\{
       \begin{array}{cc}
            -2, & \mbox{if }j\neq b; \\
            -3, & \mbox{if }j=b. 
       \end{array} \right.
    \end{equation}
    Therefore, 
    \begin{equation}
        s_{i_j}\cdot (\lambda - \textstyle\Sigma_Y) =     \left\{  \begin{array}{ll}
          \lambda - \Sigma_{Y} + \alpha_{i_j},   & \mbox{if } j\neq b; \\
         \lambda - \Sigma_{Y} + 2\alpha_{i_{b}},    & \mbox{if } j=b. 
        \end{array}  \right.
    \end{equation}
On the other hand, by \Cref{def: admissible dos} of \Cref{def: admissible} we have $s_{i_{j}} (\alpha_{i_{k}}) = \alpha_{i_{k}}$ for $j\neq k$. Then, we obtain 

\begin{equation}\label{eq: proof admissible 2}
\displaystyle
 w\cdot (\lambda - \Sigma_{Y}) = \lambda -\Sigma_{Y}+\alpha_{i_b} +\sum_{j=2}^{p-1} \alpha_{i_j} =   \lambda - \alpha_{i_{1},i_{b}-1} -\alpha_{i_{b} +1 , i_{p}}.    
\end{equation}

By combining \eqref{eq: proof admissible 1} and  \eqref{eq: proof admissible 2} we obtain the first row of \eqref{eq: sum proof admisssible}.

For $J = \bigcup\limits_{a=1}^{b-1} \{\alpha_{i_{a},i_{a+1}}\}$ we apply \Cref{lem: weyl group over M elements} for $w=s_{i_{2}}\cdots s_{i_{b}}$ in order to obtain
    \begin{equation}
         \bfM_{\lambda - \Sigma_{J}}^{Z\setminus Y} =(-1)^{b-1}\bfM_{w\cdot(\lambda - \Sigma_{J})}^{w(Z\setminus Y)}= (-1)^{b-1}\bfM_{w\cdot(\lambda - \Sigma_{J})}^{Z\setminus Y}.
    \end{equation}
In this case notice that $\langle\lambda -\Sigma_J,\alpha_{i_j}\rangle =-2$ for all $2\leq j \leq b$. Thus the second row in \eqref{eq: sum proof admisssible} follow from the equality
    \begin{equation}
        w\cdot(\lambda - \Sigma_{J}) = \lambda -\Sigma_J +\sum_{j=2}^b\alpha_{i_j} = \lambda -\alpha_{i_1,i_b-1}. 
    \end{equation}
The case  $J = \bigcup\limits_{a=b}^{p-1} \{\alpha_{i_{a},i_{a+1}}\}$  is dealt with similarity. Indeed, we apply  \Cref{lem: weyl group over M elements} for 
 $w=s_{i_{b}}s_{i_{b+1}}\cdots s_{i_{p-1}}$ and compute  
\begin{equation}
w\cdot (\lambda-{\Sigma_{J}})  = \lambda + \sum_{j=b}^{p-1} (-\alpha_{i_{j},i_{j+1}} + \alpha_{i_{j}}) = \lambda - \alpha_{i_{b}+1,i_{p}},
\end{equation}
which gives the third row in \eqref{eq: sum proof admisssible}.

It remains to show the last row of  \eqref{eq: sum proof admisssible}. Let $J\subset Y$ be such that $\alpha_{i_{j-1},i_{j}}\in J$ but $\alpha_{i_{j},i_{j+1}} \not \in J$ for some $j\neq b$. We have $\langle\lambda - \Sigma_{J},\alpha_{i_{j}}\rangle = -1$.  Thus \Cref{lem: weyl group over M elements} and $s_{i_{j}}(Z\setminus Y) = Z\setminus Y$ implies 
    \begin{equation}
        \bfM_{\lambda-\Sigma_{J}}^{Z\setminus Y} = 0.
    \end{equation}
   Finally, if  $J=\emptyset$ then $\langle\lambda,\alpha_{i_{b}}\rangle = -1$, so $\bfM_{\lambda}^{Z\setminus Y} = 0$ by \Cref{lem: weyl group over M elements}. This finishes the proof of \eqref{eq: sum proof admisssible}. 

By combining \eqref{eq: defM admissible 1} and \eqref{eq: sum proof admisssible} we obtain 

\begin{equation} \label{eq: proof admissible uno}
\bfM_{\lambda }^{Z}  = 
q^{b-1}\bfM_{\lambda - \alpha_{i_{1},i_{b}-1}}^{Z\setminus Y}+
    q^{p-b}\bfM_{\lambda - \alpha_{i_{b}+1,i_{p}}}^{Z\setminus Y} -
     q^{p-1}\bfM_{\lambda - \alpha_{i_{1},i_{b}-1} -\alpha_{i_{b} +1 , i_{p}}}^{Z\setminus Y}.
\end{equation}

 Although   \eqref{eq: proof admissible uno} looks like very similar to the first row of \eqref{eq: lema admisible set}, we still have a big discrepancy, namely,  the superindex of the $\bfM$-elements. 
 To deal with this we expand $\bfM^Z_\mu$ in terms of $\bfM^{Z\setminus Y}$-elements for $\mu$ being each one of the three subindexes occurring in \eqref{eq: proof admissible uno}.

\begin{claim} \label{Claim admissible}
The following equalities hold.
\begin{align}
\hspace{-10pt} \label{claim1} \bfM_{\lambda - \alpha_{i_{1},i_{b}-1}}^{Z}   & = \bfM_{\lambda - \alpha_{i_{1},i_{b}-1}}^{Z\setminus Y} - q^{p-1} \bfM_{\lambda - \alpha_{i_{1},i_{b}-1}-\alpha_{i_{1},i_{p}}}^{Z\setminus Y}. \\[5pt]
\hspace{-10pt} \label{claim2} \bfM_{\lambda - \alpha_{i_{b}+1,i_{p}}}^{Z}    &  =  \bfM_{\lambda - \alpha_{i_{b}+1,i_{p}}}^{Z\setminus Y}-
 q^{p-1}\bfM_{\lambda - \alpha_{i_{b}+1,i_{p}} - \alpha_{i_{1},i_{p}}}^{Z\setminus Y}. \\[5pt]
\hspace{-10pt} \label{claim3}\bfM_{\lambda - \alpha_{i_{1},i_{b}-1} - \alpha_{i_{b}+1,i_{p}}}^{Z} & = \bfM_{\lambda - \alpha_{i_{1},i_{b}-1} - \alpha_{i_{b}+1,i_{p}}}^{Z\setminus Y} - q^{b-1}\bfM_{\lambda - \alpha_{i_{1},i_{b}-1} - \alpha_{i_{1},i_{p}}}^{Z\setminus Y} -q^{p-b}\bfM_{\lambda - \alpha_{i_{1},i_{p}} - \alpha_{i_{b}+1,i_{p}}}^{Z\setminus Y}.
\end{align}
\end{claim}

\begin{proof}[Proof of \Cref{Claim admissible}.]
We first prove \eqref{claim1}. We begin by noticing that 
\begin{equation}
    \bfM_{\lambda - \alpha_{i_{1},i_{b}-1}}^{Z}  = \sum_{J\subseteq Y} (-q)^{|J|} \bfM^{Z\setminus Y}_{\lambda - \alpha_{i_{1},i_{b}-1} -\Sigma_J}.
\end{equation}  
Thus,  \eqref{claim1} follows from 
  \begin{equation}\label{eq: sum proof admissible 2}
        \bfM_{\lambda - \alpha_{i_{1},i_{b}-1}-\Sigma_{J}}^{Z\setminus Y} = \left\{ \begin{array}{ll}
            \bfM_{\lambda - \alpha_{i_{1},i_{b}-1}}^{Z\setminus Y}, & \text{if } J=\emptyset; \\[5pt]
            (-1)^{p-2}\bfM_{\lambda - \alpha_{i_{1},i_{b}-1}-\alpha_{i_{1},i_{p}}}^{Z\setminus Y}, & \text{if } J = Y ;\\[5pt]
            0,&\text{otherwise}.
        \end{array}\right.
    \end{equation}

If $J=\emptyset$ there is nothing to prove. For $J=Y$ we notice that $\langle\lambda - \alpha_{i_{1},i_{b}-1}-\Sigma_{Y},\alpha_{i_{j}}\rangle = -2$ for all $2\leq j\leq p-1$. Let $w=s_{i_2}\cdots s_{i_{p-1}}$. Arguing as before, we obtain
\begin{equation}
      w\cdot(\lambda - \alpha_{i_{1},i_{b}-1} -{\Sigma_{Y}}) =\lambda - \alpha_{i_{1},i_{b}-1} -{\Sigma_{Y}} +\sum_{j=2}^{p-1} \alpha_{i_j} =
       \lambda - \alpha_{i_{1},i_{b}-1} -\alpha_{i_{1},i_{p}}.
\end{equation}
Therefore, by applying \Cref{lem: weyl group over M elements} for the element $w$ we get
\begin{equation}
    \bfM^{Z\setminus Y}_{\lambda - \alpha_{i_{1},i_{b}-1}-\Sigma_{Y}}  = (-1)^{p-2} \bfM^{w(Z\setminus Y)}_{ w\cdot(\lambda - \alpha_{i_{1},i_{b}-1} -{\Sigma_{Y}})} = (-1)^{p-2} \bfM^{Z\setminus Y}_{\lambda - \alpha_{i_{1},i_{b}-1} -\alpha_{i_{1},i_{p}}}.  
\end{equation}
Now if $J\neq \emptyset $ and $J\neq Y$ there is an index $2\leq j \leq p-1$ such that $\alpha_{i_{j-1},i_j}\in J$ and $\alpha_{i_j,i_{j+1}}\notin J$. It follows that 
\begin{equation}
    \langle\lambda - \alpha_{i_{1},i_{b}-1}-\Sigma_{J},\alpha_{i_j}\rangle =-1.
\end{equation}
Therefore, by applying  \Cref{lem: weyl group over M elements} to the element $s_{i_j}$ we conclude that $\bfM^{Z\setminus Y}_{\lambda - \alpha_{i_{1},i_{b}-1}-\Sigma_{J}}=0$. This finishes the proof of \eqref{eq: sum proof admissible 2} and \eqref{claim1}.

The proofs of \eqref{claim2} and \eqref{claim3} follow by the identities

  \begin{equation}\label{eq: sum proof admissible 3}
        \bfM_{\lambda - \alpha_{i_{b}+1,i_{p}}-{\Sigma_{J}}}^{Z\setminus Y} = \left\{ \begin{array}{ll}
            \bfM_{\lambda - \alpha_{i_{b}+1,i_{p}}}^{Z\setminus Y}, &\text{if }J=\emptyset ; \\[5pt]
            (-1)^{p-2}\bfM_{\lambda - \alpha_{i_{b}+1,i_{p}} - \alpha_{i_{1},i_{p}}}^{Z\setminus Y}, &\text{if } J=Y ;\\[5pt]
            0, &\text{otherwise; }
        \end{array}\right.
    \end{equation}

and  

 \begin{equation}\label{eq: sum proof admissible 4}
        \bfM_{\lambda - \alpha_{i_{1},i_{b}-1} - \alpha_{i_{b}+1,i_{p}} - {\Sigma_{J}}}^{Z\setminus Y} = \left\{\begin{array}{ll}
            \bfM_{\lambda - \alpha_{i_{1},i_{b}-1} - \alpha_{i_{b}+1,i_{p}}}^{Z\setminus Y}, & \text{if } J=\emptyset ;\\[5pt]
            (-1)^{b-2}\bfM_{\lambda - \alpha_{i_{1},i_{b}-1} - \alpha_{i_{1},i_{p}}}^{Z\setminus Y}, & \text{if } J=\bigcup\limits_{a=1}^{b-1}\{\alpha_{i_{a},i_{a+1}}\}; \\[10pt]
            (-1)^{p-b-1}\bfM_{\lambda - \alpha_{i_{1},i_{p}} - \alpha_{i_{b}+1,i_{p}}}^{Z\setminus Y}, & \text{if } J=\bigcup\limits_{a=b}^{p-1}\{\alpha_{i_{a},i_{a+1}}\}; \\[10pt]
            0, &\text{otherwise.}
        \end{array}\right. 
    \end{equation}

Both \eqref{eq: sum proof admissible 3} and \eqref{eq: sum proof admissible 4} are proved using computations similar to the ones used in the proof of \eqref{eq: sum proof admissible 2}. For the sake of brevity we omit the details. This completes the proof of \Cref{Claim admissible}. 
\end{proof}

Using \Cref{Claim admissible} it is now straightforward to see that \eqref{eq: proof admissible uno} is equivalent to the first row of \eqref{eq: lema admisible set}. 
This finishes the proof of this case.

\medskip
We now prove the case $b=p$.    \Cref{eq: defM admissible 1} still holds but the role of \eqref{eq: sum proof admisssible} is played by the following
\begin{equation}\label{eq: sum proof admissible 5}
    \bfM_{\lambda-\Sigma_{J}}^{Z\setminus Y}=\left\{\begin{array}{lr}
        (-1)^{p-1}\bfM_{\lambda - \alpha_{i_{1},i_{p}-1}}^{Z\setminus Y}, &\text{if } J=Y;\\[5pt]
        0, & \text{otherwise.}
    \end{array}\right.
\end{equation}
Let us prove this equality.  If $J=Y$, consider $w=s_{i_{2}}\cdots s_{i_{p}}$ and that in this case $\langle \lambda -\Sigma_{Y},\alpha_{i_j}\rangle=-2$ for all $2\leq j \leq p$, we can compute
\begin{equation}
    w\cdot (\lambda-\Sigma_{Y}) = \lambda - \alpha_{i_{1},i_{p}-1}.
\end{equation}
This calculation give us by \Cref{lem: weyl group over M elements} the first row of \eqref{eq: sum proof admissible 5}.

We now assume that $J\neq Y$.
If $\alpha_{i_{p-1},i_{p}}\not\in J$, then $\langle\lambda-\Sigma_{J},\alpha_{i_{p}}\rangle =-1$. 
Thus, $\bfM_{\lambda -\Sigma_J}^{Z\setminus Y}=0$ by \Cref{lem: weyl group over M elements}. 
Then we can assume that $\alpha_{i_{p-1},i_{p}}\in J$. 
Since $J\neq Y$ there is an index $2\leq j <p$ such that $\alpha_{i_{j-1},i_{j}} \not\in J$ but $\alpha_{i_{j},i_{j+1}}\in J$. 
In this situation we have $\langle \lambda-\Sigma_{J},\alpha_{i_{j}}\rangle =-1$. 
Once again \Cref{lem: weyl group over M elements} implies that $\bfM_{\lambda -\Sigma_J}^{Z\setminus Y}=0$. 
This proves the second row of \eqref{eq: sum proof admissible 5}. 
A combination of \eqref{eq: defM admissible 1} and \eqref{eq: sum proof admissible 5} yields 
\begin{equation}\label{eq: admissible b=p 1}
    \bfM_{\lambda}^Z   =  q^{p-1}\bfM_{\lambda-\alpha_{i_{1},i_{p}-1}}^{Z\setminus Y}. 
\end{equation}

On the other hand, we claim that
\begin{equation}\label{eq: admissible b=p 2}
    \bfM_{\lambda - \alpha_{i_{1},i_{p}-1}-\Sigma_{J}}^{Z\setminus Y} = \left\{\begin{array}{lr}
        \bfM_{\lambda - \alpha_{i_{1},i_{p}-1}}^{Z\setminus Y}, &\text{if } J=\emptyset;  \\
        0, &\text{otherwise}.
    \end{array}\right.
\end{equation}
We notice that $\pint{\lambda - \alpha_{i_{1},i_{p}-1}}{\alpha_{i_j}}=0$ for all $2\leq j \leq p $. If $J=\emptyset$, there is nothing to prove. 
If $J\neq \emptyset$ then there exist some index $2\leq j\leq p$ such that $\langle\lambda - \alpha_{i_{1},i_{p}-1} - \Sigma_{J},\alpha_{i_{j}}\rangle = -1$. 
By \Cref{lem: weyl group over M elements}  it follows that $\bfM_{\lambda - \alpha_{i_{1},i_{p}-1} - \Sigma_{J}}^{Z\setminus Y} = 0$.
This proves \eqref{eq: admissible b=p 2}. 

\medskip
By applying \eqref{eq: defM admissible 1} to the weight $ \lambda - \alpha_{i_{1},i_{p}-1}$ and \eqref{eq: admissible b=p 2} we obtain 
\begin{equation}\label{eq: admissible b=p 3}
    \bfM_{\lambda - \alpha_{i_{1},i_{p}-1}}^Z  =    \bfM_{\lambda - \alpha_{i_{1},i_{p}-1}}^{Z\setminus Y}.
\end{equation}
Finally, the case $b=p$ in   \eqref{eq: lema admisible set} follows by combining \eqref{eq: admissible b=p 1} and \eqref{eq: admissible b=p 3}. 

\medskip
Case $b=1$ is similar to case $b=p$, which we shall omit for brevity.
\end{proof}

\section{First inverse decomposition} \label{Section: descomposición inversa}

Let $\lambda\in X^+$. Our goal in this section is to express $\bfM_{\lambda}^{\succeq \alpha_{i,j}}$ as a linear combination of terms of the form $\bfM_{\mu}^{\succ \alpha_{i,j}}$. 
Since $\Phi^{\succeq \alpha_{i,j}}=\Phi^{\succ \alpha_{i,j}}\cup\{\alpha_{i,j}\}$, we have
\begin{equation}
   \bfM_{\lambda}^{\succeq \alpha_{i,j}}
   \;=\;
   \bfM_{\lambda}^{\succ \alpha_{i,j}}
   \;-\;
   q\,\bfM_{\lambda-\alpha_{i,j}}^{\succ \alpha_{i,j}}.
\end{equation}
If $\lambda-\alpha_{i,j}\in X^+$ then this equality already gives the desired decomposition. 
Otherwise, a refinement is required for proving the positivity conjecture.
In \Cref{prop: second version} we provide an initial decomposition, and a sharper form is established in \Cref{sec: second inverse decomp} (see \Cref{prop: second version refinado}). 
The latter is the version best suited for proving the positivity conjecture.

\subsection{$P$ and $Q$ operators}
In this section, we introduce the notation needed to express the first of the aforementioned decompositions. The main tools are the $P$ and $Q$ operators, which act on the set of weights $X$.

\begin{definition}\rm 
Let $\lambda\in X$ and $h $ be an  integer greater than or equal to $2$. 
We say that 
\begin{enumerate}[(a)]
    \item  $\lambda$ is $L$-dominant at position $r$ and height $h$ if $\langle \lambda , \alpha_k \rangle \geq 0$ for all  $1\leq k\leq r$ with $k\equiv r \mod{h}$. 
    \item  $\lambda$ is $R$-dominant at position $r$ and height $h$ if $\langle \lambda , \alpha_k \rangle \geq 0$ for all $r \leq k \leq n$ with $k\equiv r \mod{h-1}$.
    \item $\lambda$ is $L$-dominant at position $r$  if $\langle \lambda , \alpha_k \rangle \geq 0$ for all  $1\leq k\leq r$. 
    \item $\lambda$ is $R$-dominant at position $r$ if $\langle \lambda , \alpha_k \rangle \geq 0$ for all $r \leq k \leq n$.
\end{enumerate}
We denote by $X_{r,h}^+(L) $ (resp. $X_{r,h}^+(R)$) the set of all weights $L$-dominant (resp. $R$-dominant) at position $r$ and height $h$.
Similarly, we denote by $X_{r}^+(L) $ (resp. $X_{r}^+(R)$) the set of all weights $L$-dominant (resp. $R$-dominant) at position $r$.
\end{definition}

\begin{definition}\rm \label{def. p and q}
Let $\lambda \in X$, $1\leq i<j \leq n$ and $h=j-i+1$. 
We define
\begin{equation}
\begin{array}{ll}
   \varrho = \varrho_{j,h}(\lambda )  = &   \displaystyle  \max \{   1\leq k \leq i \mid k\equiv i \mod h \quad \mbox{ and } \quad \lambda -\alpha_{k,j} \in X_{i,h}^+(L)     \}, \\[3pt]
 \vartheta=   \vartheta_{j,h}(\lambda)  = &   \displaystyle \min \{ j< k \leq n \mid k\equiv j \mod (h -1) \quad \mbox{ and } \quad  \lambda -\alpha_{j+1,k } \in X_{j,h}^+(R)     \}.
\end{array}
\end{equation}

\begin{enumerate}[(a)]
    \item If $\varrho$ exists we define $\OP{j,h}(\lambda)= \lambda-\alpha_{\varrho,j}$. Otherwise, we say $\OP{j,h}(\lambda)$ is not defined. 
    
    We call $\varrho$ the integer associated to $\OP{j,h}(\lambda)$.
    \item If $\vartheta$ exists we define $\OQ{j,h}(\lambda) =\lambda-\alpha_{j+1,\vartheta}$. Otherwise, we say $\OQ{j,h}(\lambda)$ is not defined. 
    
    We call $\vartheta $ the integer associated to $\OQ{j,h}(\lambda)$.
\end{enumerate}
In order to simplify notation, if $h$ is clear for the context, we just write
$ \OP{j}(\lambda) =\OP{j,h}(\lambda)  $ and $\OQ{j}(\lambda)=\OQ{j,h}(\lambda) $.
\end{definition}

\begin{example}\rm
Let $n=9$, $\lambda=[1,1,0,0,0,0,0,1,1]$, $j=7$ and  $h=3$. 
Then, $\varrho =2$, $\vartheta = 9$, 
    $P_{7,3}(\lambda) = \lambda -\alpha_{2,7} = [2,0,0,0,0,0,-1,2,1] $ and  $Q_{7,3}(\lambda) = \lambda - \alpha_{8,9}= [1,1,0,0,0,0,1,0,0]$.
\end{example}

\begin{definition}\rm\label{def: words over weights} 
    Let $\lambda \in X$, $1\leq j \leq n $ and $h\geq 2$. 
    For $k\geq 1$ we define
    \begin{equation}
    \begin{array}{rl} 
        P^k_{j,h} (\lambda) = &\OP{j-k+1,h}\OP{j-k+2,h}\cdots \OP{j-1,h}\OP{j,h}(\lambda).\\ [3pt]
        Q^{k}_{j,h} (\lambda) =& \OQ{j,h}\OQ{j-1,h}\cdots \OQ{j-k+1,h}(\lambda).\\[3pt]
        Q^{-k}_{j,h} (\lambda) =& \OQ{j+k-1,h}\OQ{j+k-2,h}\cdots \OQ{j+1,h}\OQ{j,h}(\lambda).\\[3pt]
        v^k_{j,h}(\lambda) =&  Q^{k}_{j,h} P^k_{j,h} (\lambda).
        % (v_{k}P^{m})_{j,h}(\lambda) =& (v_{k})_{(j-m,h)}(\OP{}^{m})_{(j,h)}(\lambda).\\[2pt]
        % (P^{m}v_{k})_{j,h}(\lambda) =& (\OP{}^{m})_{(j-k,h)}(v_{k})_{(j,h)}(\lambda).
    \end{array}
    \end{equation}
  By convention, we define all these operators for $k=0$ simply as $\lambda$. 
\end{definition}
\begin{remark}\rm
    Observe that in \Cref{def: words over weights}, the definitions of the operators $Q^{k}_{j,h}$ and $Q^{-k}_{j,h}$ are equivalent up to a shift in the index. More precisely, one has
    \begin{equation}
        Q^{k}_{j,h}(\lambda) = Q^{-k}_{j-k+1,h}(\lambda).
    \end{equation}
    Throughout the text, we will freely use both notations depending on convenience of reading or context.
\end{remark}

We stress that the operators defined before correspond to a sequence of compositions of $P$ and $Q$ operators. Throughout these sequence it may well be the case that one of the elements involved in the composition is not defined. In this case, we simply say that the whole operator is not defined in a given weight. 

\begin{example}\rm
Let $(i,j)=(5,8)$ and $h = 4$. Consider the weight $\lambda = [2,0,1,2,0,0,0,0,0,0,1,0,1,0]_{\varpi}$. 
Applying our operators, we obtain:
\begin{equation}
\begin{aligned}
    P_{8,4}^2(\lambda) &= [1, 0, 2, 1, 0, 0, -1, 0, 1, 0, 1, 0, 1, 0]_{\varpi}, \\[5pt]
    v_{8,4}^2(\lambda) = Q_{8,4}^2P_{8,4}^2(\lambda) &= [1, 0, 2, 1, 0, 0, 0, 0, 0, 0, 0, 1, 0, 1]_{\varpi}, \\[5pt]
    P_{6,4}^2v_{8,4}^2(\lambda) &= [2, 0, 1, 1, -1, 0, 1, 0, 0, 0, 0, 1, 0, 1]_{\varpi}.
\end{aligned}
\end{equation}
This allows us to conclude that the operators $P_{8,4}^2$, $v_{8,4}^2$, and $P_{6,4}^2v_{8,4}^2$ are all defined at $\lambda$.

\medskip
On the other hand, if we change the initial weight to $\lambda = [2,0,1,2,2,0,0,0,0,0,0,0,1,0]_{\varpi}$, we compute:
\begin{equation}
    P_{8,4}^2(\lambda) = [2, 0, 2, 2, 1, 0, -1, 0, 1, 0, 0, 0, 1, 0]_{\varpi}.
\end{equation}
Thus, $P_{8,4}^2(\lambda)$ is defined. However, notice that $\pint{P_{8,4}^2(\lambda)}{\alpha_{11}} = \pint{P_{8,4}^2(\lambda)}{\alpha_{14}} = 0$. This implies that $Q_{7,4}(P_{8,4}^2(\lambda))$ is undefined, and consequently, neither $v_{8,4}^2(\lambda)$ nor $P_{6,4}^2v_{8,4}^2(\lambda)$ are defined.
\end{example}

The following lemma shows certain relationship between the integers associated to the different $P$ operators that occur in the composed operator $P^k$.

\begin{lemma} \label{lem: desigualdad de varrhos}
    Let $\lambda \in X^+_{j}(L)$, $\alpha_{i,j}\in \Phi^{\geq 2}$ and $h=j-i+1$. 
    For $r\geq 0$ we set $\lambda_{r} = P^r_{j,h}   (\lambda) $ and  $\varrho_{r}=\varrho_{j-r+1,h}(\lambda_{r-1})   $.
    Suppose that  $\lambda_{k}$ is defined for some  $k\geq 1$ and that  $\varrho_{k}\geq 2$. 
    Then, $\lambda_{k+1}$ is defined and $\varrho_{k+1} \geq \varrho_{k} - 1$.
    In particular, if $\lambda_{k}$ is defined but $\lambda_{k+1}$ is not then $\varrho_k=1$.   
\end{lemma}

\begin{proof}
By definition $\varrho_k\equiv i-k +1  \mod h $. 
It follows that $\varrho_k -1\equiv i-k   \mod h $.
Since $\lambda \in X_j^+(L) $  and $\lambda_{k-1} $ is defined we have $\lambda_{k-1} \in X^{+}_{j-k+1}(L)$. 
In particular, $\pint{\lambda_{k-1}}{\alpha_{\varrho_k-1}}\geq 0$. It follows that
\begin{equation}
    \pint{\lambda_{k}}{\alpha_{\varrho_k-1}} = \pint{\lambda_{k-1}-\alpha_{\varrho_k,j-k+1}}{\alpha_{\varrho_k-1}} = \pint{\lambda_{k-1}}{\alpha_{\varrho_k-1}} -\pint{\alpha_{\varrho_k,j-k+1}}{\alpha_{\varrho_k-1}} = \pint{\lambda_{k-1}}{\alpha_{\varrho_k-1}} +1 \geq 1. 
\end{equation}
Thus  $\lambda_k - \alpha_{\varrho_k-1, j-k+1}\in X^+_{j-k,h}(L) $. 
We conclude that  $\lambda_{k+1}$ is defined and $\varrho_{k+1} \geq \varrho_{k} - 1$.
\end{proof}

\begin{corollary} \label{lem: ceros por Ps}
 Let $\lambda \in X^+_{j}(L)$ and $\alpha_{i,j}\in \Phi^{\geq 2}$ and $h=j-i+1$.
 For $r\geq 0$ we set $\lambda_{r} = P^r_{j,h}   (\lambda) $ and  $\varrho_{r}=\varrho_{j-r+1,h}(\lambda_{r-1})  $.
    Suppose that  $\lambda_{k}$ is defined for some  $1\leq k \leq h$. 
    Then, we have
    $\pint{\lambda_{k-1}}{\alpha_{t}} = 0$,
    for all $\displaystyle  t\in U_d\coloneqq \bigcup_{m=0}^{d} [i-k+1-mh,i-mh]$, where $d$ is the unique integer such that $\varrho_k = i-k+1 -(d+1)h$.
\end{corollary}

\begin{proof}
If $d=-1$ then $U_d=\emptyset$ and the statement is empty, so assume $d\geq 0$.

We argue by contradiction. Suppose that there exists $t\in U_d$ with 
$\pint{\lambda_{k-1}}{\alpha_t}\neq 0$.
Since $U_d\subseteq [\varrho_k+h,i]$, we have
\begin{equation} \label{eq: t is in  between}
    \varrho_k+h \le t \le i.
\end{equation}

Because $\lambda\in X_j^+(L)$, we have $\lambda_{k-1} \in X_{j-k+1}^+(L)$.  
Moreover, since $k\le h$, it follows that $t\le i\le j-k+1$, and hence
$\pint{\lambda_{k-1}}{\alpha_t}>0$.
The inequality $\varrho_k+h \le t$ and the maximality of $\varrho_k$ imply that 
$t\not\equiv i-k+1 \pmod{h}$.  
Since $t\in U_d$, there exists $0\le s<k-1$ such that 
$t\equiv i-s \pmod{h}$.
By maximality of $\varrho_{s+1}$ we have $t\le \varrho_{s+1}$.
Applying \Cref{lem: desigualdad de varrhos} repeatedly gives
\begin{equation}
    \varrho_{s+1} \le \varrho_k + (k-s-1) < \varrho_k + h.
\end{equation}
Therefore,
    $t < \varrho_k + h$,
contradicting \eqref{eq: t is in  between}.
This contradiction proves the lemma.
\end{proof}

Both \Cref{lem: desigualdad de varrhos} and \Cref{lem: ceros por Ps} have their $Q$-counterpart that we now enunciate. 
In both cases we omit the proof as it is almost identical to the one of its $P$-counterpart.

\begin{lemma} \label{lem: desigualdad de varthetas}
    Let $\mu \in X^+$ and $\alpha_{i,j}\in \Phi^{\geq 2}$ and $h=j-i+1$. 
    For $r\geq 0$ we set $\mu_{r} = Q^{r}_{i,h}(\mu)$ and $\vartheta_{r}=\vartheta_{i+r,h}(\mu_{r-1})$.
    % Let $\vartheta_{r}$ be the integer associated to $Q_{i+r,h}(\mu_{r-1})$.
    Suppose that $\mu_{k}$ is defined for some integer $k\geq 1$ and $\vartheta_{k} \leq n-1$.
    Then, $\mu_{k+1}$ is defined and $\vartheta_{k+1} \leq \vartheta_{k} + 1$.
% \begin{equation}\label{eq: desigualdad de los varthetas}
%     \vartheta_{k+1} \leq \vartheta_{k} + 1.
% \end{equation}
In particular, if $\mu_{k}$ is defined but $\mu_{k+1}$ is not then $\vartheta_k =n$. 
\end{lemma}

\begin{corollary} \label{lem: ceros por Qs}  
% \David{Revisar si el caso $k=h$ funciona aqu\'i y en caso de ser as\'i, incluirlo aunque no se use mas adelante (por completitud)}
    Let $\alpha_{i,j}\in \Phi^{\geq 2}$ and $h=j-i+1$. 
    Let $r$ be an integer such that $i+1\leq r\leq j$.
    Let $\mu \in X^+_{r+1}(R)$.
    For $x\geq 0$, we set $\mu_{x} = Q^{-x}_{r,h}(\mu)$.
    % Set $\mu_{0} = \lambda$ and define $\mu_{r}\coloneqq Q_{p+r-1,h}(\mu_{r-1})$ for $r\geq 1$.
    Suppose $\mu_k$ is defined for some $1\leq k\leq h$.
    Then, we have $\pint{\mu_{k-1}}{\alpha_{t}} = 0$, for all $t\in U_d$, where
    \begin{equation}\label{eq: Ud de zeros en Q}
     U_{d} =   \bigcup_{m=1}^{d} [r + m(h-1), r + k -1 + m(h-1)]
    \end{equation}
  and  $d$ is the unique integer satisfying $\vartheta_{r+k-1,h}(\mu_{k-1}) = r+k-1 + (d+1)(h-1)$.  
\end{corollary}

The following lemma provides the integers associated to the $Q$ operators involved in the  operator $v^{h-1}_{j,h}$.

\begin{lemma}\label{lem: thetas para el caso h-1}
    Let $\lambda \in X^+$ and $\alpha_{i,j}\in \Phi^{\geq 2}$ and $h=j-i+1$. 
    Suppose $v^{h-1}_{j,h}(\lambda)$ is defined.
    We define
    \begin{equation}\label{eq: varthetas en h-1}
        \vartheta_{r}= \left\{  \begin{array}{ll}
          \vartheta_{i+1,h}(P^{h-1}_{j,h}(\lambda)),   & \mbox{if } r=1;  \\[3pt]
        \vartheta_{i+r,h}(Q^{r-1}_{i+r-1,h}(P^{h-1}_{j,h}(\lambda))),     &  \mbox{if } 2\leq r \leq h-1.
        \end{array}  \right.
    \end{equation}
    Then,  $\vartheta_{r}=j+r$, for $1\leq r\leq h-1 $.
\end{lemma}

\begin{proof}
    We begin by noticing that 
    \begin{equation}
        P^{h-1}_{j,h}(\lambda) = \lambda - \sum_{r=1}^{h-1}\alpha_{\varrho_{r},j-r+1},
    \end{equation}
    where $\varrho_{r} = \varrho_{j-r+1,h}(P^{r-1}_{j,h}(\lambda))$,  for $1\leq r\leq h-1$.
    Using this, we can compute the following: 
    \begin{equation}\label{eq: calculo ph-1}
        \pint{P^{h-1}_{j,h}(\lambda)}{\alpha_{j+1}} = \pint{\lambda}{\alpha_{j+1}} - \sum_{r=1}^{h-1} \pint{\alpha_{\varrho_{r},j-r+1}}{\alpha_{j+1}} = \pint{\lambda}{\alpha_{j+1}} + 1 \geq 1,
    \end{equation}
    where the last inequality follows since $\lambda \in X^+$.

   We recall that $\vartheta_1$ is, by definition, the minimal integer satisfying $P^{h-1}_{j,h}(\lambda)-\alpha_{i+1,\vartheta_1} \in X_{i+1,h}^+(R)$ and $\vartheta_1\equiv i+1 \mod h-1$.
   As $j+1 = i+1 + h-1$, \eqref{eq: calculo ph-1} implies that   $\vartheta_{1} = j+1$, which is our claim for the case $r=1$.

On the other hand,  \Cref{lem: desigualdad de varthetas} yields $\vartheta_{r} \leq \vartheta_{r-1}+1$ for all $2\leq r\leq h-1$.
Furthermore, by \cref{eq: varthetas en h-1},    $ i+r+h-1  \leq \vartheta_r$. 
By combining these two inequalities with the definition of $h$, we obtain 
\begin{equation}
    j+r\leq \vartheta_r\leq \vartheta_{r-1}+1. 
\end{equation}
Finally, an inductive argument with the equality $\vartheta_1=j+1$ yields $\vartheta_r=j+r$, as we wanted to show. 
\end{proof}

We now define the \textbf{degree} of the action of a $P$ or $Q$ operator on a weight $\lambda$.

\begin{definition}\rm  \label{def: degree of P and Q}
Let $\lambda \in X$, $1\leq j \leq n$ and $h \geq 2$. 
\begin{itemize}
    \item If $P_{j,h}(\lambda)$ is defined, then $\D{P,j,h}(\lambda)=d+1$, where $d$ is the unique integer satisfying  $\varrho_{j,h}(\lambda) = i - dh$.
    \item If $Q_{j,h}(\lambda)$ is defined, then $\D{Q,j,h}(\lambda)=d$, where $d$ is the unique integer satisfying $\vartheta_{j,h}(\lambda) = j + d(h-1)$.
\end{itemize}
We adopt the convention that \(\D{P,j,h}(\lambda)\) (resp. \(\D{Q,j,h}(\lambda)\)) is defined only when \(P_{j,h}(\lambda)\) (resp. \(Q_{j,h}(\lambda)\)) is.

\smallskip
\noindent
We refer to $\D{P,j,h}(\lambda)$ (resp. $\D{Q,j,h}(\lambda)$) as the degree of the operator $P_{j,h}$ (resp. $Q_{j,h}$) evaluated at $\lambda$.
\end{definition}

\begin{remark}\rm
Roughly speaking, $\D{P,j,h}(\lambda)$ measures the cost of passing from $\lambda$ to $P_{j,h}(\lambda)$ in the following sense. 
Recall that
$P_{j,h}(\lambda)=\lambda-\alpha_{\varrho,j}$,
where $\varrho =\varrho_{j,h}(\lambda)$. 
Thus, obtaining $P_{j,h}(\lambda)$ from $\lambda$ amounts to subtracting $\alpha_{\varrho,j}$. 
We can write 
\begin{equation} \label{descomponer raiz larga en suma de height h}
    \alpha_{\varrho ,j}=\sum_{k=0}^{d}\alpha_{i-kh,\,j-kh},
\end{equation}
with  $d$ as in \Cref{def: degree of P and Q}. 
Consequently, the degree corresponds to the number of summands in \cref{descomponer raiz larga en suma de height h}; equivalently, it quantifies how many roots of height $h$  must be subtracted from $\lambda $ to reach $P_{j,h}(\lambda)$.
Similar considerations hold for $\D{Q,j,h}$, with the role of $h$ replaced by $h-1$. 

Lastly, we have a notational remark. 
To simplify notation one might be tempted to write the degree as $D(P_{j,h}(\lambda))$ rather than $D(P_{j,h})(\lambda)$. 
We avoid this, because it would misleadingly suggest that we are assigning a degree to the weight $P_{j,h}(\lambda)$ itself, rather than to the ``passage''  $\lambda \to P_{j,h}(\lambda)$ encoded by $\D{P,j,h}(\lambda)$.
\end{remark}

We can extend the definition of degree for any composition of $P$ and $Q$ operators. 

\begin{definition} \rm
    Let $\lambda \in X$. 
    Let $Y=(Y_1,\ldots ,Y_{\ell})$ be a sequence of $P$ and $Q$ operators. 
    Suppose that $Y_{\ell} \cdots Y_1(\lambda)$ is defined.
    We set $\lambda_0=\lambda$ and $\lambda_k =Y_k\cdots Y_{1}(\lambda)$, for $1\leq k < \ell$.
    Then, we define the degree of the operator $Y$ evaluated at $\lambda$ is defined as
    \begin{equation}
        D(Y)(\lambda) = \sum_{k=1}^{\ell} D(Y_k)(\lambda_{k-1}). 
    \end{equation}
\end{definition}

We are now in position to state the decomposition mentioned at the beginning of this section.

\begin{proposition}\rm\label{prop: second version}
    Let $\lambda\in X^{+}$,  $\alpha_{i,j}\in \Phi^{\geq 2}$ and $h=j-i+1$ .  Let $k$ be an integer with $1\leq k <h$ satisfying  $\langle \lambda , \alpha_r \rangle  =0$, for all $j-k+1\leq r \leq j$. Then,  we have    
    \begin{equation}\label{eq: second version}
     \!  \bfM_{\lambda}^{\succeq \alpha_{i,j}}\! \!  = \! \bfM_{\lambda}^{\succ \alpha_{i,j}} \!  - q^{D(P^{k}_{j,h})(\lambda)+1} \bfM_{P^k_{j,h}(\lambda) - \alpha_{i-k,j-k}}^{\succ \alpha_{i,j}}\! \!  -  \displaystyle\sum_{r=1}^{k} q^{D(v^{r}_{j,h})(\lambda)} \! \left( \bfM_{v^{r}_{j,h}(\lambda)}^{\succ \alpha_{i,j}}\!  -\!  q\bfM_{v^{r}_{j,h}(\lambda)-\alpha_{i-r,j-r}}^{\succ \alpha_{i,j}}\right).
    \end{equation}
\end{proposition}

% \David{No s\'e si este convencion la hicimos antes}
\begin{remark}\rm\label{remark operator not defined}
       We emphasize that if any $P$-operator or $v^r$-operator is not defined, the corresponding $\bfM$-element is set  equal to zero. 
\end{remark}

%%%%%%%%%%%%%%%%%%%%%%%%%%%%%%%%%%%%%%%%%%%%%%%%%%%%%%%%%%%%%%%%%%%%%%%%%%%%%%%%%%%%%%%
\subsection{Identities involving $P$ and $Q$ operators}

In this section, building on the results of \Cref{section identities to prove first inverse decomposition}, we develop the tools required to prove \Cref{prop: second version}. 
That proof is presented at  the end of this section.

\begin{lemma}\rm \label{lem: free Q1} 
    Let $K \subset \Phi^{\geq 2}$, $\alpha_{i,j} \in \Phi^{\geq 2}$, and  $h = j - i + 1$. 
    Let $(r,p)$ be integers such that $i + 1 < r \leq p \leq j$  and $h+1 \leq r$. 
    Let $k\geq 0$ be  such that $p'\coloneqq p+k(h-1) \leq n$.
    Let  $Y = K \cup \{\alpha_{r-h, r-1}\}$ and $\lambda \in X^+_{p'+h-1,h}(R)$. 
\begin{enumerate}[(a)]
    \item   \label{item a} Suppose that $\OQ{p',h}(\lambda)$ is defined and let  $\vartheta \coloneqq \vartheta_{p',h}(\lambda)$ be the integer associated to $\OQ{p',h}(\lambda)$. 
    
    We define $d$ to be the unique integer satisfying $\vartheta = p +(d+1)(h-1)$.
    \item  \label{item b} Suppose that  $\OQ{p',h}(\lambda)$ is not defined.
    
    We define $d$ to be  the largest  integer such that $p + d(h - 1) \leq n$. 
\end{enumerate}
    In both cases we define
    $$T = \bigcup_{m = 0}^{d} [r + m(h - 1), p + m(h - 1)],$$
    and assume that $\pint{\lambda}{\alpha_t} = 0$ for all $t \in T \setminus \{p'\}$ and $\pint{\lambda}{\alpha_{p'}} = -1$. 
    Then, if $K \sim_{T} \Phi^{\succ \alpha_{i,j}}$, we have
    \begin{equation}
        \bfM_{\lambda}^{Y} = \left\{ \begin{array}{ll}
         q^{\D{Q,p',h}(\lambda)}\bfM_{Q_{p',h}(\lambda)}^{Y},    & \mbox{if } \OQ{p',h}(\lambda) \mbox{ is defined;}    \\
          0,    & \mbox{if } \OQ{p',h}(\lambda) \mbox{ is not defined.} 
        \end{array} \right.
    \end{equation}
\end{lemma}

\begin{proof}
We first address case \ref{item a}.  
That is, we assume that $\OQ{p',h}(\lambda)$ is defined. 
By definition, $\OQ{p',h}(\lambda) = \lambda - \alpha_{p'+1,\vartheta}$, and $\lambda - \alpha_{p'+1,\vartheta}\in X^{+}_{p',h}$
% is $R$-dominant at position $p'$ and height $h$. 
Moreover, by the definition of $\D{Q,p',h}(\lambda)$, we obtain 
\begin{equation}
    \D{Q,p',h}(\lambda) = \frac{\vartheta -p'}{h-1} = d - k + 1.
\end{equation}  
Applying \Cref{Claim: free Q} to $K$, $\alpha_{i,j}$, $(r,p,k)$ and $d$, it follows that 
\begin{equation}
    \bfM_{\lambda}^Y = q^{d-k+1} \bfM_{\lambda - \alpha_{p'+1,\vartheta}}^{Y} = q^{\D{Q,p',h}(\lambda)}\bfM_{Q_{p',h}(\lambda)}^{Y},
\end{equation}

We now treat case \ref{item b}.  
Thus, we assume that $\OQ{p',h}(\lambda)$ is not defined. 
By definition, $k \leq d$. If $k = d$, then the maximality of $d$ and the second case of \eqref{eq: M element with a -1 in p'} in \Cref{lem: right solving} imply that $\bfM_{\lambda}^{Y} = 0$.

Thus, we may assume that $k < d$.  
In this situation, we apply \Cref{Claim: free Q} to $K$, $\lambda$, $k$, and $d - 1$, yielding $\bfM_{\lambda}^{Y} = q^{d-k}\bfM_{\mu}^{Y}$, where $\mu = \lambda - \alpha_{p'+1,p'+(d-k)(h-1)}$ and $p'+(d-k)(h-1) = p+d(h-1)$ .  
Observe that via a simple calculation we have that  $\pint{\mu}{\alpha_{p+d(h-1)}} = -1$ and $\pint{\mu}{\alpha_{t}} = 0$ for $t \in T\setminus \{p + d(h-1)\}$.

By the maximality of $d$ and the second case of  \Cref{lem: right solving} applied to $K$, $\alpha_{i,j}$, $(r,p)$, $\mu$ and $T$, we obtain $\bfM_{\mu}^{Y} = 0$. 
Therefore, $\bfM_{\lambda}^Y = 0$, as required.
\end{proof}

\begin{lemma}\rm \label{lem: free Q2}
Let $K \subset \Phi^{\geq 2}$, $\alpha_{i,j} \in \Phi^{\geq 2}$ and let $h = j - i + 1$.
Suppose that $1 \leq i+1 - h$.
Let $p$ be an integer such that $i + 1 \leq p \leq j$. 
Let $k\geq 0$ be such that $p'\coloneqq p+k(h-1) \leq n$. 
Let $Y = K \cup \{\alpha_{i-h+1, i}\}$ and  $\lambda \in X_{p'+h-1,h}^+(R)$.

\begin{enumerate}[(a)]
    \item  \label{item a Q} Suppose that $\OQ{p',h}(\lambda)$ is defined. 
    Let $\vartheta\coloneqq \vartheta_{p',h}(\lambda)$ be the integer associated to $\OQ{p',h}(\lambda)$.  
    
    We define $d$ to be the unique integer satisfying $\vartheta = p+ (d+1)(h-1)$.
    
    \item  \label{item b Q} Suppose that $\OQ{p',h}(\lambda)$ is not defined.
    
    We define $d$ to be the largest integer such that $p + d(h - 1) \leq n$.
\end{enumerate}

In both cases we define
\begin{equation}
    T = \left( \bigcup_{m = 0}^{d-1} [i+1 + m(h - 1), j + m(h - 1)] \right) \cup [i+1+d(h-1),p+d(h-1)] 
\end{equation}
and assume that $\pint{\lambda}{\alpha_t} = 0$ for all $t \in T \setminus \{p'\}$, $\pint{\lambda}{\alpha_{p'}} = -1$.

Furthermore, if $k<d$ we also assume that $\pint{\lambda}{\alpha_{p'+1}} =1$.
If  $K \sim_{T} \Phi^{\succ \alpha_{i,j}}$  then we have
\begin{equation}
        \bfM_{\lambda}^{Y} = \left\{ \begin{array}{ll}
         q^{\D{Q,p',h}(\lambda)}\bfM_{Q_{p',h}(\lambda)}^{Y},    & \mbox{if } \OQ{p',h}(\lambda) \mbox{ is defined;}    \\
          0,    & \mbox{if } \OQ{p',h}(\lambda) \mbox{ is not defined.} 
        \end{array} \right.
\end{equation}
\end{lemma}

\begin{proof}
   The result follows by mimicking the proof of  \Cref{lem: free Q1},  using \Cref{Claim: free Q2} instead of \Cref{Claim: free Q}.
\end{proof}

\begin{lemma}\rm\label{lem: free P}  
Let $K \subset \Phi^{\geq 2}$, $\alpha_{i,j} \in \Phi^{\geq 2}$ and $h = j - i + 1$. 
Let $p$ be an integer such that $i -h +1< p \leq i$. 
Let $k\geq 0$ be such that $p'\coloneqq p-kh\geq 1$. 
Let $\lambda\in X_{p'-h,h}^+(L)$.
\begin{enumerate}[(a)]
    \item \label{item a P} Suppose that $P_{p'-1,h}(\lambda)$ is defined and let $\varrho= \varrho_{p'-1,h}(\lambda)$ be the integer associated to $P_{p'-1,h}(\lambda)$.
    
    We define $d$ to be the unique integer satisfying $\varrho = p-(d+1)h$.
 \item \label{item b P} Suppose that $P_{p'-1,h}(\lambda)$ is not defined.
    
     We define $d$ to be the largest integer such that $p-dh\geq 1$.
 \end{enumerate}
In both cases, we define $\displaystyle T=\bigcup_{m=0}^{d}[p-mh,i-mh]$
and assume that $\pint{\lambda}{\alpha_{t}}=0$ for all $t\in T\setminus\{p'\}$, $\pint{\lambda}{\alpha_{p'}}=-1$. 

If  $K \sim_{T} \Phi^{\succ \alpha_{i,j}}$ then we have
   \begin{equation}
        \bfM_{\lambda}^{K} = \left\{ 
        \begin{array}{ll}
         q^{\D{P,p'-1,h}(\lambda)}\bfM_{P_{p'-1,h}(\lambda)}^{K},    &  \mbox{if } P_{p'-1,h}(\lambda)  \mbox{ is defined;} \\
           0,  & \mbox{if } P_{p'-1,h}(\lambda)  \mbox{ is  not defined.}
        \end{array}
        \right.
   \end{equation}
\end{lemma}

\begin{proof}
We first address case \ref{item a P}.  
Assume that $\OP{p'-1,h}(\lambda)$ is defined. 
By definition, $\OP{p'-1,h}(\lambda) = \lambda - \alpha_{\varrho,p'-1}$, and $\lambda - \alpha_{\varrho,p'-1} \in X_{p'-h+1,h}(L)$
% is $L$-dominant at position $p'$ and height $h$. 
Moreover, by the definition of $d$, we deduce that $\D{P,p'-1,h}(\lambda) = d - k + 1$.  
Applying \Cref{Claim: free P}, it follows that 
\begin{equation}
    \bfM_{\lambda}^K = q^{d-k+1} \bfM_{\lambda - \alpha_{\varrho,p'-1}}^{K}.
\end{equation}
% \David{NO hay $Y$ aqu\'i!! No deber\'ia ser $K$. Revisar en toda la demo.}

\medskip
We now consider case \ref{item b P}.  
Assume that $\OP{p'-1,h}(\lambda)$ is not defined. 
By definition, $k \leq d$. 
If $k = d$ then the maximality of $d$ and the second case of \eqref{eq: left solving} in \Cref{lem: left solving} applied to $K$, $\alpha_{i,j}$, $p$, $\lambda$ and $d$ imply that $\bfM_{\lambda}^{K} = 0$.
Thus, we may assume that $k < d$.  
In this setting we apply \Cref{Claim: free P} to $K$, $\alpha_{i,j}$, $\lambda$, $k$, and $d - 1$, yielding $\bfM_{\lambda}^{K} = q^{d-k}\bfM_{\mu}^{K}$,
% \David{Creo que aca es el mismo error que en el caso (b) del lemma 4.17. Revisar!!!} 
where $\mu = \lambda - \alpha_{p'-(d-k)h,p'-1}$ and $p'-(d-k)h = p-dh$.  
Observe that via a simple calculation we have that $\pint{\mu}{\alpha_{p-dh}} = -1$ and $\pint{\mu}{\alpha_{t}} = 0$ for $t \in T\setminus \{p - d(h-1)\}$.
By maximality of $d$ and the second case of \Cref{lem: left solving} applied to $K$, $\alpha_{i,j}$, $p$ $\mu$, $d$ and $T$ it follows that
\begin{equation}
    \bfM_{\mu}^{K} = 0.
\end{equation}
Therefore, $\bfM_{\lambda}^K = 0$, as required.
\end{proof}

\begin{lemma}\rm\label{lem: pre completar v}
    Let $\alpha_{i,j} \in \Phi^{\geq 2}$ and $h=j-i+1$.
    Let $r$ and $x$ be integers such that $i+1 \leq r \leq j$, $h+1\leq r$ and $1\leq x \leq j-r+1$. 
    Let  $\lambda\in X^+_{r+1}(R)$ be such that $\pint{\lambda}{\alpha_{t}} =0$ for $t\in [r+1,j]$, $\pint{\lambda}{\alpha_{r}} =-1$ and $\pint{\lambda}{\alpha_{j+1}} \geq 1$. 
    If  $Y= \Phi^{\succ \alpha_{i,j}}\cup \{\alpha_{r-h,r-1}\}$ then
        \begin{equation}\label{eq: pre completar v}
            \bfM_{\lambda}^{Y} = \left\{\begin{array}{ll}
                q^{D(Q^{-x}_{r,h})(\lambda)}\bfM_{Q^{-x}_{r,h}(\lambda)}^{Y}, &\mbox{if } Q^{-x}_{r,h}(\lambda) \mbox{ is defined}; \\
                0, & \mbox{if } Q^{-x}_{r,h}(\lambda)\mbox{ is not defined}.
            \end{array}\right.
        \end{equation}
\end{lemma}

\begin{proof}
We only prove the statement when $Q^{-x}_{r,h}(\lambda)$ is defined, the other case being analogous. 

For $0\leq m \leq x$ we define $\mu_{m} = Q^{-m}_{r,h}(\lambda)$.
For $1\leq m \leq x$ we define $\vartheta_m=\vartheta_{r+m-1,h}(\mu_{m-1})$ and $d_m$ to be the unique integer satisfying $\vartheta_m =r+m-1+(d_m+1)(h-1) $.  

We split the proof in two cases. 

\begin{description}
    \item[Case A]  $r>i+1 $. 

    We proceed by induction on $x$. 
    If $x=1$ then $Q^{-x}_{r,h}(\lambda)=Q_{r,h}(\lambda)$. 
    Thus, the result reduces to \Cref{lem: free Q1} applied to 
    $K=\Phi^{\succ \alpha_{i,j}}$ and $p'=p=r$.
    Indeed, using the notation in that lemma we have 
    \begin{equation}
     T= \bigcup_{m=0}^{d_1}        \{ r+m(h-1)   \}.   
    \end{equation}
     Since $\lambda \in X^{+}_{r+1}(R)$,  the minimality of $\vartheta_1$ implies  $\pint{\lambda}{\alpha_{t}}=0$ for all $t\in T\setminus \{r\}$. 
     Thus we are under the hypothesis of \Cref{lem: free Q1}, which allows us to conclude.

We now assume that $x\geq 2$ and suppose that \eqref{eq: pre completar v} holds for $x-1$. 
This is, 
   \begin{equation}\label{eq: hp completar v hasta p}
     \displaystyle   \bfM_{\lambda }^{Y} = q^{D(Q^{-(x-1)}_{r,h})(\lambda)}\bfM_{\mu_{x-1}}^{Y}.
    \end{equation}
Let 
 \begin{equation}
        T  = \bigcup_{m=0}^{d_x}[r+m(h-1),r+x-1+m(h-1)].
\end{equation}

By definition,  we have  $\displaystyle \mu_{x-1} = \lambda - \sum_{m=1}^{x-1}\alpha_{r+m,\vartheta_{m}}$.
A direct computation, using \Cref{remark producto interno}, together with \Cref{lem: ceros por Qs} applied to $\mu_{x}$, 
imply that   $\pint{\mu_{x-1}}{\alpha_{r+x-1}} = -1$ and $\pint{\mu_{x-1}}{\alpha_{t}} = 0$, for all $t\in T\setminus\{r+x-1\}$.
Then, we can apply \Cref{lem: free Q1} to $K=\Phi^{\succ \alpha_{i,j}}$, $r$, $p=r+x-1=p'$ and $\mu_{x-1}$ to get
    \begin{equation}\label{eq: completar v hasta p}
        \bfM_{\mu_{x-1}}^{Y} = q^{D(Q_{r+x-1,h})(\mu_{x-1})}\bfM_{Q_{r+x-1,h}(\mu_{x-1})}^{Y}  = q^{D(Q_{r+x-1,h})(\mu_{x-1})}\bfM_{\mu_x}^{Y}.
    \end{equation}
    We conclude by combining \eqref{eq: hp completar v hasta p}, \eqref{eq: completar v hasta p} and  $D(Q^{-x}_{r,h})(\lambda) = D(Q^{-(x-1)}_{r,h})(\lambda) + D(Q_{r+x-1,h})(\mu_{x-1})$.

\item[Case B]  $r=i+1 $. 

Using the assumption $\pint{\lambda}{\alpha_{j+1}}\ge 1$ and an inductive argument (similar to the one in the proof of \Cref{lem: thetas para el caso h-1}), we obtain $\vartheta_m=j+m$ for all $1\le m\le x$.
With this in hand, the result follows by the same inductive method as in the proof of \textbf{Case A}, replacing \Cref{lem: free Q1} with \Cref{Claim: free Q2}.
\end{description}
\end{proof}

\begin{corollary}\rm \label{lem: v existe o no existe}
    Let $\lambda\in X^{+}$,  $\alpha_{i,j}\in \Phi^{\geq 2}$ and $h=j-i+1$ .  
    Let $s$ be an integer with $1 \leq s \leq h-1$ satisfying  $\langle \lambda , \alpha_t \rangle  = 0$, for all $j-s+1\leq t \leq j$. 
    Set $\lambda_{s} = P^s_{j,h}(\lambda)$.
   If $\lambda_{s}$ is defined then we have
    \begin{equation}\label{eq: v existe o no existe}
        \hspace{-10pt}\bfM_{\lambda_{s}}^{\succ \alpha_{i,j}} = \left\{\begin{array}{ll}
        q\bfM_{\lambda_{s}-\alpha_{i-s,j-s}}^{\succ \alpha_{i,j}} + q^{D(Q^{-s}_{j-s+1,h})(\lambda_{s})}\left(\bfM_{v^s_{j,h}(\lambda)}^{\succ \alpha_{i,j}} - q\bfM_{v^s_{j,h}(\lambda)-\alpha_{i-s,j-s}}^{\succ \alpha_{i,j}}\right), & \mbox{if } v^s_{j,h}(\lambda) \mbox{ is defined};\\
        q\bfM_{\lambda_{s}-\alpha_{i-s,j-s}}^{\succ\alpha_{i,j}},  & \mbox{otherwise}.
        \end{array}\right.
    \end{equation}    
\end{corollary}

\begin{proof}
By definition of $\lambda_s$ and the hypothesis on $\lambda$  we have   $\pint{\lambda_{s}}{\alpha_{j-s+1}} = -1$ and $\pint{\lambda_{s}}{\alpha_{t}} = 0$ for all $t\in [j-s+2,j]$. 
Furthermore, since $\lambda \in X^+$ it follows that $\lambda_{s}\in X^+_{j-s+2}(R)$.

We recall that
\begin{equation}
  v^{s}_{j,h}(\lambda) = Q^s_{j,h}P^{s}_{j,h}(\lambda) = Q^s_{j,h}(\lambda_s)= Q^{-s}_{j-s+1,h}(\lambda_s ).   
\end{equation}

In particular,  $Q^{-s}_{j-s+1,h}(\lambda_s )$ is defined if and only if   $v^{s}_{j,h}(\lambda)$ is. 

Let  $Y = \Phi^{\succ \alpha_{i,j}}   \cup \{ \alpha_{i-s,j-s} \}$.
By applying \Cref{lem: pre completar v} to $\alpha_{i,j}$, $(r,p) = (j-s+1,s)$ and $\lambda_{s}$, it follows that
\begin{equation}\label{eq: 4.21 pre completar v}
            \bfM_{\lambda_s}^{Y} = \left\{\begin{array}{ll}
               q^{D(Q^{-s}_{j-s+1,h})(\lambda_{s})}\bfM_{v^s_{j,h}(\lambda )}^{Y}, &\mbox{if } v^s_{j,h}(\lambda) \mbox{ is defined}; \\
                0, & \mbox{if } v^s_{j,h}(\lambda) \mbox{ is not  defined}.
            \end{array}\right.
        \end{equation}
Then, the corollary follows by expanding according to the rule $\bfM_{\mu}^{Y} = \bfM_{\mu}^{\succ \alpha_{i,j}} - q \bfM_{\mu -\alpha_{i-s,j-s}}^{\succ \alpha_{i,j}}$
\end{proof}

\begin{lemma}\rm \label{lem: los 3 casos de second version}
    Let $\lambda\in X^{+}$,  $\alpha_{i,j}\in \Phi^{\geq 2}$ and $h=j-i+1$ .  
    Let $s$ be an integer with $0 \leq s < h-1$ satisfying  $\langle \lambda , \alpha_t \rangle  = 0$, for all $j-s\leq t \leq j$. 
    Let $\lambda_{s} = P^{s}_{j,h}(\lambda)$ and $\lambda_{s+1}=P^{s+1}_{j,h}(\lambda)$.
    We also define 
    \begin{equation}
       u=\D{P,j-s,h}(\lambda_{s})+1 \mbox{ and } v = D(Q^{-s}_{j-s+1,h})(\lambda_s) + \D{P,j-s,h}(\lambda_{s}).\footnote{If the right-hand side is defined. } 
    \end{equation}
    If $\lambda_s$ is defined then 
\begin{equation}  \label{eq: casicasi First}
    q\bfM_{\lambda_{s} - \alpha_{i-s,j-s}}^{\succ \alpha_{i,j}} =   \begin{cases}
      q^{u}\bfM_{\lambda_{s+1}-\alpha_{i-(s+1),j-(s+1)}}^{\succ \alpha_{i,j}} + q^{v}\left(\bfM_{v^{s+1}_{j,h}(\lambda)}^{\succ \alpha_{i,j}} - q\bfM_{v^{s+1}_{j,h}(\lambda)-\alpha_{i-(s+1),j-(s+1)}}^{\succ \alpha_{i,j}}\right)  &  \\
     q^{u}\bfM_{\lambda_{s+1}-\alpha_{i-(s+1),j-(s+1)}}^{\succ\alpha_{i,j}}   &   \\
     0, &  
    \end{cases}
\end{equation}
where the three cases correspond to:  $v^{s+1}_{j,h}(\lambda)$ is defined,  $\lambda_{s+1}$ is defined but $v^{s+1}_{j,h}(\lambda)$ is not, and $\lambda_{s+1}$  is not defined, respectively.
\end{lemma}

\begin{remark}\rm
    Since $v^{s+1}_{j,h}(\lambda)= Q^{s+1}_{j,h}P^{s+1}_{j,h}(\lambda) =Q^{s+1}_{j,h}(\lambda_{s+1})$, it follows that if $\lambda_{s+1}$ is not defined then $v^{s+1}_{j,h}(\lambda )$ is not defined. 
\end{remark}

\begin{proof}
    Suppose $\lambda_{s+1} = P_{j-s,h}(\lambda_{s})$ is defined. 
    We begin by noticing that if $\pint{\lambda_{s}}{\alpha_{i-s}} \geq 1$, then by definition of $P$-operator, it follows that $\lambda_{s+1}= \lambda_{s} - \alpha_{i-s,j-s}$.
    Therefore, $\bfM_{\lambda_{s} - \alpha_{i-s,j-s}}^{\succ \alpha_{i,j}} = \bfM_{\lambda_{s+1}}^{\succ \alpha_{i,j}}$ and the result follows by \Cref{lem: v existe o no existe}. 
    
    Thus, we assume that $\pint{\lambda_{s}}{\alpha_{i-s}} = 0$.
    By \Cref{lem: ceros por Ps} applied to $\lambda\in X^{+}$, $\alpha_{i,j}$ and $k=s+1<h$ we get
    \begin{equation}
        \pint{\lambda_{s}}{\alpha_{t}} = 0 \quad \mbox{for all}\; t\in T=\bigcup_{m=0}^{d} [i-s-mh,i-mh],
    \end{equation}
    where $d$ is the unique integer satisfying $\varrho_{j-s,h}(\lambda_{s}) = i-s-(d+1)h$. 
    It follows from \Cref{remark producto interno} that
    \begin{equation}  \label{eq: lemma casicasi 0 y-1}
        \pint{\lambda_{s} - \alpha_{i-s,j-s}}{\alpha_{t}}  = \begin{cases}
            0, & \mbox{if } t \in T \setminus \{i-s\};\\
           -1, & \mbox{if } t=i-s.
        \end{cases}
    \end{equation}
    Hence, we can apply \Cref{lem: free P} to $K=\Phi^{\succ \alpha_{i,j}}$, $\alpha_{i,j}$, $p=p'=i-s$, $T$ and $\lambda_{s}-\alpha_{i-s,j-s}$, to get
    \begin{equation}\label{eq: p existe en second version 1}
        q\bfM_{\lambda_{s} - \alpha_{i-s,j-s}}^{\succ \alpha_{i,j}} = q^{\D{P,i-s-1,h}(\lambda_{s} - \alpha_{i-s,j-s})+1}\bfM_{P_{i-s-1,h}(\lambda_{s}-\alpha_{i-s,j-s})}^{\succ \alpha_{i,j}}
        = q^{\D{P,j-s,h}(\lambda_{s})}\bfM_{\lambda_{s+1}}^{\succ \alpha_{i,j}},
    \end{equation}
    where the second equality holds since 
    \begin{equation}
        P_{i-s-1,h}(\lambda_{s}-\alpha_{i-s,j-s}) = P_{j-s,h}(\lambda_{s})\quad \mbox{ and } \quad \D{P,i-s-1,h}(\lambda_{s} - \alpha_{i-s,j-s})+1 = \D{P,j-s,h}(\lambda_{s}).
    \end{equation}
   Thus, using  \Cref{lem: v existe o no existe} for $s+1$ to expand the right-hand side of  \eqref{eq: p existe en second version 1} we obtain the first two cases of \cref{eq: casicasi First}. 

   We now  suppose that $\lambda_{s+1}$ is not defined.  
    By \Cref{lem: ceros por Ps} applied to $\lambda$, $\alpha_{i,j}$ and $k=s$ it follows that $\pint{\lambda_{s-1}}{\alpha_{t}} = 0$ for all $t\in T$, where
    \begin{equation}
        T=\bigcup_{m=0}^{d}[i-s+1-mh,i-mh],
    \end{equation}
    and $d$ is the unique integer satisfying $\varrho_{s} = \varrho_{j-s+1,h}(\lambda_{s-1}) = i-s+1-(d+1)h$.
    On the other hand, since $\lambda_{s+1}$ is not defined, it follows by \Cref{lem: desigualdad de varrhos} that $\varrho_{s} = 1$.
    In other words, $\lambda_{s} = \lambda_{s-1} -\alpha_{1,j-s+1}$.
   By \Cref{remark producto interno} we conclude that $\pint{\lambda_{s}}{\alpha_{t}}=0$ for all $t\in T$.
    Moreover, given that $\lambda_{s+1}$ is not defined, by the definition of the  $P$ operator we have  $\pint{\lambda_{s}}{\alpha_{t}} =0$ for all $t\in \{i-s-mh\mid 0\leq m\leq d\}$.
    Thus, we conclude that $\pint{\lambda_{s}}{\alpha_{t}} =0$ for all $t\in T'$, where
    \begin{equation}
        T'=\bigcup_{m=0}^{d} [i-s-mh,i-mh].
    \end{equation} 
 Using \Cref{remark producto interno} we obtain \eqref{eq: lemma casicasi 0 y-1} for all $t\in T'$. 
 As before, we can apply \Cref{lem: free P} to $K=\Phi^{\succ \alpha_{i,j}}$, $\alpha_{i,j}$, $p=p'=i-s$, $T'$ and $\lambda_{s}-\alpha_{i-s,j-s}$, to get $\bfM_{\lambda_{s}-\alpha_{i-s,j-s}}^{\succ \alpha_{i,j}}=0$, thus proving the third case in \Cref{eq: casicasi First}. 
\end{proof}

We are now in position to prove the First Inverse Decomposition given in \Cref{prop: second version}.

\begin{myproof}[Proof of \Cref{prop: second version}]
We begin by emphasizing that in \cref{eq: second version} the $\bfM$-elements indexed by undefined weights are set equal to $0$ by convention.
    Let $1\leq k< h$ satisfy the hypothesis of the \Cref{prop: second version}, i.e. $\pint{\lambda}{\alpha_r}=0$ for all $j-k+1\leq r \leq j$.  
    We  proceed  by  induction on $k$.
    Consider first the case $k = 1$.
   % Notice that in this case we must have $r=j$. 
    By the definition of $\bfM$-elements, we have
    \begin{equation}\label{eq: second version -1}
    \bfM_{\lambda}^{\succeq \alpha_{i,j}} = \bfM_{\lambda}^{\succ \alpha_{i,j}} - q\bfM_{\lambda-\alpha_{i,j}}^{\succ \alpha_{i,j}}.
    \end{equation}
  By applying  \Cref{lem: los 3 casos de second version} for $s=0$ to rewrite $q\bfM_{\lambda-\alpha_{i,j}}^{\succ \alpha_{i,j}}$ we obtain \eqref{eq: second version} for $k=1$,  as we wanted.
 
We now assume that \cref{eq: second version} holds for some $1\leq k < h-1$.
Let $\lambda_{s} = P^s_{j,h}(\lambda)$ for $1\leq s\leq k+1$. 
We also let $\mathcal{RH}_k$ and $\mathcal{RH}_{k+1}$ be the right-hand side of \eqref{eq: second version} for $k$ and $k+1$, respectively. 
We must show that $\bfM_{\lambda}^{\succeq \alpha_{i,j}}=\mathcal{RH}_{k+1}$. 
By our inductive hypothesis, this amounts to show that $\mathcal{RH}_k=\mathcal{RH}_{k+1}$. 
This turns out to be  equivalent to

\begin{equation} \label{eq: turns out to be equivalent}
  q \bfM_{\lambda_{k} - \alpha_{i-k,j-k}}^{\succ \alpha_{i,j}} = q^{u} \bfM_{\lambda_{k+1} - \alpha_{i-(k+1),j-(k+1)}}^{\succ \alpha_{i,j}} 
  + q^{ v   }  
  \left( \bfM_{v^{k+1}_{j,h}(\lambda)}^{\succ \alpha_{i,j}} - q \bfM_{v^{k+1}_{j,h}(\lambda) - \alpha_{i-(k+1),j-(k+1)}}^{\succ \alpha_{i,j}} \right),
\end{equation}
where $u=D(P_{j-k})(\lambda_k) + 1$ and $v=D(Q^{k+1}_{j,h})(\lambda_{k+1}) +D(P_{j-k,h} )(\lambda_k)  $.

Due to the conventions mentioned at the beginning of this proof,  \cref{eq: turns out to be equivalent} follows by \Cref{lem: los 3 casos de second version}. 
This completes the induction and the proof of the First Inverse Decomposition.
\end{myproof}

\section{Towards the Second Inverse Decomposition}\label{sec: Towards the Second Inverse Decomposition}

In this section we develop the tools needed for refining the decomposition  in \Cref{prop: second version}. 

We begin by introducing some notation. 

\begin{definition} \rm
Let $u \geq 1$, $h \geq 2$ and $(v,x)$ be a pair of non-negative integers. 
For $\lambda \in X$ we define
\begin{equation}
 R_u(\lambda) = -\pint{\lambda}{\alpha_u},
    \qquad R_{u,v}(\lambda) = -\pint{\lambda}{\alpha_{u,u+v}},
    \qquad  R_{u,v,x}(\lambda) = \displaystyle \sum_{m=0}^{x} R_{u,v+hm}(\lambda).   
\end{equation}
We stress that $R_{u,0}(\lambda) = R_u(\lambda)$,  $R_{u,v,0}(\lambda) = R_{u,v}(\lambda)$ and $R_{u,0,0}(\lambda) = R_u(\lambda)$. 
\end{definition}

\begin{definition}\rm \label{def: Trapecio}
    Let $h\geq 2$ and $A \subset \Phi^{\geq h}$. 
    We say that $A$ is a trapezoid at height $h$ if $A$ satisfies the following conditions:
    \begin{enumerate}
        \item  If $\alpha, \beta \in A$ are roots of the same height and $\alpha \prec \beta$, then every root $\gamma$ with $\alpha \prec \gamma \prec \beta$ also belongs to $A$.
        
        \item  If $\alpha_{i,j} \in A$ and $\operatorname{ht}(\alpha_{i,j}) > h$, then the adjacent roots $\alpha_{i,j-1}$ and $\alpha_{i+1,j}$ also lie in $A$.
    \end{enumerate}
    The name  trapezoid reflects the fact that the set of squares associated with the roots in $A$ forms a shape resembling a trapezoid, as illustrated in \Cref{fig: tpz T}.
\end{definition}

   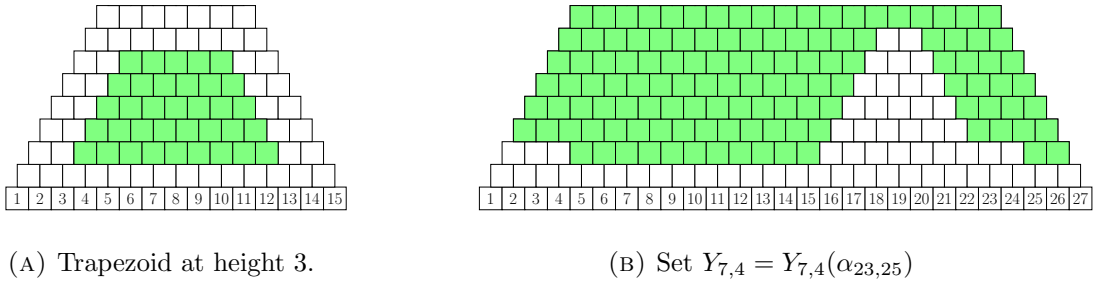
\begin{figure}[H]
        \centering
        \begin{subfigure}[t]{0.3\textwidth}
        \scalebox{0.3}{\begin{tikzpicture}
        \clip (-1,0) rectangle (28,10);
            \dibu{15}{1}{14}
            \agregof{3}{3}{11}
            \agregof{4}{3}{10}
            \agregof{5}{3}{9}
            \agregof{6}{3}{8}
            \agregof{7}{3}{7}
        \end{tikzpicture}}
        \caption{Trapezoid at height 3.}
        \label{fig: tpz T}
        \end{subfigure}
        \hspace{1cm}
        \begin{subfigure}[t]{0.5\textwidth}
        \scalebox{0.3}{\begin{tikzpicture}
        \clip (-1,0) rectangle (28,10);
            \dibu{27}{23}{25}
            \quitof{4}{15}{20}
            \quitof{5}{15}{19}
            \quitof{6}{15}{18}
            \quitof{7}{15}{17}
            \quitof{8}{15}{16}
            \agregof{3}{4}{14}
        \end{tikzpicture}}
        \caption{Set  $Y_{7,4}=Y_{7,4}(\alpha_{23,25})$}
        \label{fig: ejemplo y43}
        \end{subfigure}
        \caption{Examples of sets in \Cref{def: Trapecio} and \Cref{def: SSS and YYY}. }
    \end{figure}

\begin{definition}\rm\label{def: SSS}
Let  $k\geq 0$ and $\alpha_{i,j} \in \Phi^+$. We define 
\begin{equation}
    S_k(\alpha_{i,j}) = \{   \alpha_{i-m,j}  \mid 0 \leq m \leq k     \}.
\end{equation}
\end{definition}

\begin{lemma}\rm\label{lem: second string lemma} 
Let   $Y \subset \Phi^{+}$, $\alpha_{i,j}\in \Phi^{\geq 2}$ and $k\geq 0$.
Assume that $Y \setminus S_{k}(\alpha_{i,j})$ is $s_t$-invariant for every $  i-k\leq  t \leq i$.
Let $\lambda \in X$ be such that $R_t(\lambda)=0$ for $i-k \leq t < i$ and that $R_i(\lambda)=r$ for some $0 \leq r \leq k+1$. Then we have
\begin{equation}
\bfM_{\lambda}^{Y} = q^{r}\bfM_{\lambda - r\alpha_{i+1,j}}^{Y\setminus S_{k}(\alpha_{i,j})}.
\end{equation}
\end{lemma}

\begin{proof}
Let $S= S_{k}(\alpha_{i,j})$.
By definition of $\bfM$-elements, we have
\begin{equation}\label{eq: proof second string lemma 0}
\bfM_{\lambda}^{Y} = \sum\limits_{I \subset S} (-q)^{|I|} \bfM_{\lambda - \Sigma_I}^{Y \setminus  S}.
\end{equation}

We first treat the case $r=0$. 
Let $\emptyset \neq I \subset S$. 
Let $m_0 =\min \{ 0\leq m \leq k \mid \alpha_{i-m,j}\in I  \}$.
By the minimality of $m_0$ we have $R_{i-m_0}(\lambda - \Sigma_I)=1$. 
As $Y\setminus S$ is $s_{i-m_0}$-invariant, \Cref{lem: weyl group over M elements} yields  $\bfM_{\lambda - \Sigma_I}^{Y \setminus S}=0$.
Therefore, the only term that contributes to the sum in   \eqref{eq: proof second string lemma 0}  is the one corresponding to  $I=\emptyset$. 
Thus,  $\bfM_{\lambda}^{Y} = \bfM_{\lambda }^{Y\setminus S}$, which is the claim of the lemma for $r=0$.

We now treat the case $r \geq 1$. 
Arguing as in the previous paragraph, it is easy to see that $\bfM_{\lambda - \Sigma_I}^{Y \setminus S}=0$ unless $I=S_u(\alpha_{i,j})$ for some $0 \leq u \leq k$. 
We claim that in fact $u = r-1$.
Indeed, if $u>r-1$ then $\pint{\lambda-\Sigma_I+\rho}{\alpha_{i-r,i}}=0$. 
Therefore, $s_{\alpha_{i-r,i}}\cdot ( \lambda-\Sigma_I )= \lambda-\Sigma_I $, and \Cref{lem: weyl group over M elements} implies that $\bfM_{\lambda - \Sigma_I}^{Y \setminus S} = 0$.
On the other hand, if $u<r-1$ then $\pint{\lambda-\Sigma_I+\rho}{\alpha_{i-r+1,i}}=0$, and once again, \Cref{lem: weyl group over M elements} would force $\bfM_{\lambda - \Sigma_I}^{Y \setminus S} = 0$.
All in all, we have shown that the only term that contributes to the sum in \eqref{eq: proof second string lemma 0} is the one associated to $J\coloneqq S_{r-1}(\alpha_{i,j})$.

It follows that
\begin{equation}
    \bfM_{\lambda}^{Y} = (-q)^r\bfM_{\lambda - \Sigma_{J}}^{Y \setminus S} = (-q)^r(-1)^{r} \bfM_{\lambda - \Sigma_{J} +\mu}^{Y \setminus S}
\end{equation}
where $\displaystyle \mu = \sum_{m=0}^{r-1} (r-m) \alpha_{i-m}$. 
We stress that to obtain the last equality we have applied \Cref{lem: weyl group over M elements} for $w=s_{i-(r-1)} \cdots   s_{i-1}s_i$.
Finally, the result follows by noticing that $   \Sigma_{J} - \mu  =  r\alpha_{i+1,j}$. 
\end{proof}

\begin{corollary} \label{coro: push pyramid to the right}
    Let $\alpha_{i,j} \in \Phi^{\geq 2}$ and set $h = j - i + 1$. 
    Let $A \subset \Phi^{\succ \alpha_{i,j}}$ be a trapezoid at height $h$. 
    Denote by $\alpha_{x,y}$ and $\alpha_{u,v}$ the roots corresponding to the top-left and bottom-right vertices of  $A$, respectively.
    Let $B$ be the trapezoid obtained by pushing $A$ one unit to the right. 
   Let $h' = \operatorname{ht}(\alpha_{x,y})$ and notice that $h=\operatorname{ht}(\alpha_{u,v})$. 
    Let $\lambda \in X$ be such that $R_{t}(\lambda) = 0$ for $y-(h'-h) \leq t < y$ and $R_{y}(\lambda) = r$ for some $0 \leq r \leq h' - h +1$.
   If  $y=u+1$ then 
    \begin{equation}
        \mathbf{M}^{ \Phi^{\succ \alpha_{i,j}} \setminus A}_\lambda
        =
        q^{r} \,
        \mathbf{M}^{\Phi^{\succ \alpha_{i,j}} \setminus B}_{\lambda - r \alpha_{y+1,y +(h-1)}}.
    \end{equation} 
\end{corollary}

\begin{proof}
The condition $y=u+1$ guarantees that the set $(\Phi^{\succ \alpha_{i,j}} \setminus  A) \setminus S_{h'-h}(\alpha_{y,y+(h-1)})$ is $s_t$-invariant for   $   y-(h'-h)  \leq t \leq y$. 
Thus, \Cref{lem: second string lemma} yields
\begin{equation}\label{eq: intermedio trapecio}
      \mathbf{M}^{ \Phi^{\succ \alpha_{i,j}} \setminus A}_\lambda = q^r   \mathbf{M}^{ (\Phi^{\succ \alpha_{i,j}} \setminus A) \setminus S_{h'-h}(\alpha_{y,y+(h-1)})}_{\lambda -  r\alpha_{y+1,y+(h-1)}} . 
\end{equation}
On the other hand,  we have
\begin{equation}
    B= (A\cup S_{h'-h}(\alpha_{y,y+(h-1)})) \setminus \{\alpha_{x,y-m} \mid  0\leq  m\leq h'-h \}  .
\end{equation}
Therefore, the result follows by a repeated application of \Cref{coro puedo agregar} and \eqref{eq: intermedio trapecio}.
\end{proof}

\begin{example}\rm
We now illustrate \Cref{coro: push pyramid to the right} with a concrete example.
Let $\alpha_{i,j}=\alpha_{1,4}$, so that $h=4$. 
Consider the trapezoid $A$ at height $4$ with top-left corner $\alpha_{3,8}$ and bottom-right corner $\alpha_{7,10}$.
Then $h'=\operatorname{ht}(\alpha_{3,8})=6$, so that $h'-h=2$. 

Let $\lambda = [1,-2,1,3,2,0,0,-3,0,-1,5,-2,0]_{\varpi}$. 
We observe that $R_{6}(\lambda)=R_{7}(\lambda)=0$ and $R_{8}(\lambda)=3 \le h'-h+1=3$.
Applying \Cref{coro: push pyramid to the right}, we obtain
\begin{equation}
    \bfM^{\Phi^{\succ \alpha_{1,4}} \setminus A}_\lambda
    =
    q^{3}\,
    \bfM^{\Phi^{\succ \alpha_{1,4}} \setminus B}_{\,\lambda - 3 \alpha_{9,11}},
\end{equation}
where $B$ is obtained from $A$ by shifting it one unit to the right. 
This is depicted in Figures~\ref{fig:mayor-tpzA} and \ref{fig:mayor-tpzB}, where the corresponding weights are shown below the pyramids.

\medskip

We now apply \Cref{coro: push pyramid to the right} once more to shift the trapezoid an additional unit to the right. 
The height conditions remain unchanged, and the required vanishing conditions are automatically satisfied.
The only point to verify is that
\begin{equation}
R_{9}(\lambda - 3\alpha_{9,11}) \le h'-h+1.
\end{equation}
Indeed,
\begin{equation}
\lambda - 3\alpha_{9,11}
=
[1,-2,1,3,2,0,0,0,\mathbf{-3},-1,2,1,0]_{\varpi},
\end{equation}
so $R_{9}(\lambda - 3\alpha_{9,11})=3$, and therefore the inequality holds.
Applying the corollary again yields

\begin{equation}
    \bfM^{\Phi^{\succ \alpha_{1,4}} \setminus B}_{\,\lambda - 3 \alpha_{9,11}}
    =
    q^{3}\,
    \bfM^{\Phi^{\succ \alpha_{1,4}} \setminus C}_{\,\lambda - 3 \alpha_{9,11} - 3 \alpha_{10,12}},
\end{equation}
where $C$ is obtained from $B$ by shifting it one further unit to the right.

\medskip

We could continue applying the corollary as long as the inequality in the $10$-th coordinate remains valid. 
However, in this case
\begin{equation}
    R_{10}(\lambda - 3 \alpha_{9,11} - 3 \alpha_{10,12}) = 4 > h'-h+1 = 3,
\end{equation}
so the corollary can no longer be applied.
\end{example}

    \begin{figure}[H]
        \centering
        \begin{subfigure}[t]{0.23\textwidth}
        \scalebox{0.3}{
        \begin{tikzpicture}
        \dibu{13}{1}{4}
        \agregoR{4}{3}{7}
        \agregoR{5}{3}{6}
        \agregoR{6}{3}{5}
        \foreach \x / \y in {1/1,2/-2,3/1,4/3,5/2,6/0,7/0,8/-3,9/0,10/-1,11/5,12/-2,13/0}{
       \draw[] (\x ,0.5) node {\huge $\y$ };
     }
        \end{tikzpicture}}
        \caption{ Set $\Phi^{\succ \alpha_{1,4}}\setminus A$}
        \label{fig:mayor-tpzA}
        \end{subfigure}
        \hspace{1cm}
        \begin{subfigure}[t]{0.23\textwidth}
        \scalebox{0.3}{
        \begin{tikzpicture}
        \dibu{13}{1}{4}
        \agregoR{4}{4}{8}
        \agregoR{5}{4}{7}
        \agregoR{6}{4}{6}
        \foreach \x / \y in {1/1,2/-2,3/1,4/3,5/2,6/0,7/0,8/0,9/-3,10/-1,11/2,12/1,13/0}{
       \draw[] (\x ,0.5) node {\huge $\y$ };
     }
        \end{tikzpicture}}
        \caption{ Set $\Phi^{\succ \alpha_{1,4}}\setminus B$}
         \label{fig:mayor-tpzB}
        \end{subfigure}
        \hspace{1cm}  
        \begin{subfigure}[t]{0.23\textwidth}
        \scalebox{0.3}{
        \begin{tikzpicture}
        \dibu{13}{1}{4}
        \agregoR{4}{5}{9}
        \agregoR{5}{5}{8}
        \agregoR{6}{5}{7}
        \foreach \x / \y in {1/1,2/-2,3/1,4/3,5/2,6/0,7/0,8/0,9/0,10/-4,11/2,12/-2,13/3}{
       \draw[] (\x ,0.5) node {\huge $\y$ };
     }
        \end{tikzpicture}}
        \caption{ Set $\Phi^{\succ \alpha_{1,4}}\setminus C$}
        \label{fig:mayor-tpzC}
        \end{subfigure}
        \caption{Illustrating \Cref{coro: push pyramid to the right}}
        \label{fig: tpz}
    \end{figure}

\begin{definition}\rm   \label{def: SSS and YYY}
Let $\alpha_{i,j}\in \Phi^{\geq 2}$ and $h=j-i+1$.
Let $a$ and $k$ be integers satisfying
\begin{equation}
 0\leq k \leq \frac{i}{h}-2    \qquad \mbox{and} \qquad   1\leq a\leq i-(k+1)h+1.
\end{equation}

For $x\geq 0$ we define $ b_{x} = a+(x+1)h-2$ and $a_{x} = b_x-x$.
Furthermore,% we also define the set:
\begin{itemize}
    \item If $a > h$, we define the set
\begin{equation}\label{eq: def YYY}
    \mathrm{Y}_{a,k} = Y_{a,k}(\alpha_{i,j})=
    \Biggl(
        \Phi^{\succ \alpha_{i,j}}
        \setminus
        \Biggl(
            \bigcup_{m=0}^{k}
            [\alpha_{a_{k}-h+2,a_{k}+2+m},\alpha_{b_{k}-m,b_{k}+h}]
        \Biggr)
    \Biggr)
    \cup
    [\alpha_{a-h,a-1},\alpha_{a_{k}-h+1,a_{k}}].
\end{equation}
    \item Otherwise, we define
\begin{equation}\label{eq: def YYYB}
    \mathrm{Y}_{a,k} = \mathrm{Y}_{a,k}(\alpha_{i,j})=
    \Biggl(
        \Phi^{\succ \alpha_{i,j}}
        \setminus
        \Biggl(
            \bigcup_{m=0}^{k}
            [\alpha_{a_{k}-h+2,a_{k}+2+m},\alpha_{b_{k}-m,b_{k}+h}]
        \Biggr)
    \Biggr)
    \cup
    [\alpha_{1,h},\alpha_{a_{k}-h+1,a_{k}}].
\end{equation}
\end{itemize}
\end{definition}

% \begin{remark}\rm
% \Cref{def: SSS and YYY} introduces a definition (e.g., for the set $Y_{a,k}$) that is partitioned into two similar cases depending on the value of $a$.
% The difference between the two resulting definitions corresponds to the interval $[\alpha_{a-h,a-1},\alpha_{a_{k}-h+1,a_{k}}]$ in the case $a>h$, or $[\alpha_{1,h},\alpha_{a_{k}-h+1,a_{k}}]$ in the case $a \leq h$.

% In the remainder of \Cref{sec: Towards the Second Inverse Decomposition}, most results will be proved only for the $a>h$ case, since the $a \leq h$ case follows from analogous arguments.
% \end{remark}

\begin{example}\rm 
     Let $\alpha_{i,j}=\alpha_{23,25}$, so that $h=3$.
     We choose the parameters $(a,k)=(7,4)$. 
     We have $a_4=16$ and $b_4=20$. 
     The set $Y_{7,4}$ is depicted in \Cref{fig: ejemplo y43}.
     To save space, we only displayed roots of height at most $9$. 
     All the non-displayed roots belong to $Y_{7,4}$.
\end{example}

\begin{lemma}\label{lem: fifth string lemma} 
Let $\alpha_{i,j} \in \Phi^{\geq 2}$ and set $h = j-i+1$.
Let $a$ be an integer such that $1\leq a \leq i+1-h$.
Let $\lambda \in X$ and suppose that $R_{a,t}(\lambda)\in \{0,1\}$ for all $t\in [0,h-2]$. 
Furthermore, we define 
\begin{equation}
    d=\sum_{t=0}^{h-2}R_{a,t}(\lambda), \qquad \mu = \sum_{t=0}^{h-2}R_{a,t}(\lambda)\alpha_{a+t+1,a+t+h}, \qquad \mathrm{Y} = \begin{cases}
        \Phi^{\succ \alpha_{i,j}}\cup \{\alpha_{a-h,a-1}\},&\text{if } a>h;\\%, \qquad \mbox{and}
        \Phi^{\succ \alpha_{i,j}},&\text{otherwise}.
    \end{cases}
\end{equation}

Then, we have
\begin{equation}
        \bfM_{\lambda}^{\mathrm{Y}} = q^{d}\bfM_{\lambda - \mu}^{\mathrm{Y}_{a,0}}.
\end{equation}
\end{lemma}

\begin{proof}
We only prove the case $a>h$, the case $1\leq a \leq h$ is proved in a similar fashion. 

For $0\leq t \leq h-2$ we define 
\begin{equation}
\begin{array}{rl}
    C_{t} =  & \left(\Phi^{\succ \alpha_{i,j}}\setminus [\alpha_{a,a+h},\alpha_{a+t,a+t+h}]\right)\cup [\alpha_{a-h,a-1},\alpha_{a+t-h+1,a+t}], \\[5pt]
    \mu_{t} =  & \displaystyle  \sum_{m=0}^{t}R_{a,m}(\lambda)\alpha_{a+m+1,a+m+h}.
\end{array}
\end{equation}
We first prove the following:
\begin{equation}\label{eq: inicio}
    \bfM_{\lambda}^{Y} =q^{R_a(\lambda)}\bfM_{\lambda - \mu_0 }^{C_0}.  
\end{equation}

Since $s_{a}(\alpha_{a-h,a-1}) = \alpha_{a-h,a}$ and $a<i$, it follows from \cref{eq: las que se salen} that $Y \setminus \{\alpha_{a,a+h}\}$ is $s_{a}$-invariant.
Then, \Cref{lem: second string lemma} applied to $S_{0}(\alpha_{a,a+h})=\{\alpha_{a,a+h}\}$ yields $\bfM_{\lambda}^{Y} = q^{R_{a}(\lambda)}\bfM_{\lambda - R_{a}(\lambda)\alpha_{a+1,a+h}}^{Y\setminus \{\alpha_{a,a+h}\}}$.
On the other hand,  since $R_a(\lambda - R_{a}(\lambda)\alpha_{a+1,a+h})=0$ we can apply \Cref{coro puedo agregar}  to obtain 
\begin{equation}\label{eq: fifth string lemma 1}
\bfM_{\lambda - R_{a}(\lambda)\alpha_{a+1,a+h}}^{Y\setminus \{\alpha_{a,a+h}\}} =\bfM_{\lambda - R_{a}(\lambda)\alpha_{a+1,a+h}}^{(Y\setminus \{\alpha_{a,a+h}\})\cup \{\alpha_{a-h+1,a}\}}= \bfM_{\lambda - \mu_0}^{C_0} .
\end{equation}
This proves \eqref{eq: inicio}. 
We now prove the following. 

\begin{claim}\label{claimclaimclaim1}
    For all $0\leq t < h-2$ we have $\bfM_{\lambda - \mu_{t}}^{C_{t}} = q^{R_{a,t+1}(\lambda)}\bfM_{\lambda - \mu_{t+1}}^{C_{t+1}}$.
\end{claim}

\begin{proof}
We fix $0\leq t < h-2$ and set $\alpha = \alpha_{a+t+1,a+t+h+1} $ and $\beta=\alpha_{a+t+2,a+t+h+1}$.
We first notice that one can use \cref{eq: las que se salen} to see that $C_{t}\setminus \{\alpha \}$ is $s_{a+t+1}$-invariant.
On the other hand,  by using \Cref{remark producto interno} and the fact that $t<h-2$, it is easy to see that $R_{a+t+1}(\lambda -\mu_t) = R_{a,t+1}(\lambda) $. 
By \Cref{lem: second string lemma} applied to $S_0(\alpha)$ we obtain 
\begin{equation}\label{eq: fifth string lemma 2}
\bfM_{\lambda - \mu_{t}}^{C_{t}} = q^{R_{a,t+1}(\lambda )} \bfM_{\lambda - \mu_{t} - R_{a,t+1}(\lambda) \beta }^{C_{t} \setminus \{\alpha\}} = q^{R_{a,t+1}(\lambda )} \bfM_{\lambda - \mu_{t+1}}^{C_{t} \setminus \{\alpha\}}.
\end{equation}
Another application of \Cref{remark producto interno} yields $R_{a+t+1}(\lambda-\mu_{t+1})=0 $. 
Thus, \Cref{coro puedo agregar} and the definition of the sets $C_t$ yield 
$\bfM_{\lambda - \mu_{t+1}}^{C_{t} \setminus \{\alpha\} }=\bfM_{\lambda - \mu_{t+1}}^{C_{t+1}}  $. 
Finally, a combination of this with \cref{eq: fifth string lemma 2} proves the claim.
\end{proof}
On the other hand, we have  $\mathrm{Y}_{a,0}=C_{h-2}$ and $\mu = \mu_{h-2}$. 
Therefore, the lemma follows by \eqref{eq: inicio} and a repeated application of \Cref{claimclaimclaim1}.
\end{proof}

\begin{example}\rm
Let $\lambda = [2,1,0,0,-1,0,0,0,0,1,0,0]_\varpi$, $a=5$ and $\alpha_{i,j} = \alpha_{8,11}$.
We have $h=4$ and $R_{5,0}(\lambda) =R_{5,1}(\lambda) =R_{5,2} (\lambda)= 1$.
Hence, $\mu = \alpha_{6,9} + \alpha_{7,10} + \alpha_{8,11}$.
Then, we have $\bfM_{\lambda}^{\mathrm{Y}} = q^{3}\bfM_{\lambda - \mu}^{\mathrm{Y}_{5,0}}$, where the sets $\mathrm{Y}$ and $\mathrm{Y}_{5,0}$ are depicted in \Cref{fig: y en y40}, moreover the weights are presented on the bottom of the pyramids.
We stress that the main difference between $\lambda$ and $\lambda - \mu$ is the position of the first negative component. 
\end{example}

\begin{figure}[H]
        \centering
        \begin{subfigure}[t]{0.23\textwidth}
        \scalebox{0.3}{\begin{tikzpicture}
            \dibu{12}{8}{11}
            \agregof{4}{1}{1}
            \foreach \x / \y in {1/2,2/1,3/0,4/0,5/-1,6/0,7/0,8/0,9/0,10/1,11/0,12/0}{
       \draw[] (\x ,0.5) node {\huge $\y$ };
     }
        \end{tikzpicture}}
        \label{fig: y lem}
        \caption{ $\mathrm{Y}$}
        \end{subfigure}
        \hspace{2cm}
        \begin{subfigure}[t]{0.23\textwidth}
        \scalebox{0.3}{\begin{tikzpicture}
            \dibu{12}{8}{11}
            \agregof{4}{1}{4}
            \quitof{5}{5}{7}
            \foreach \x / \y in {1/2,2/1,3/0,4/0,5/0,6/0,7/0,8/-1,9/-1,10/1,11/0,12/1}{
       \draw[] (\x ,0.5) node {\huge $\y$ };
     }
        \end{tikzpicture}}
        \label{fig: lem y40}
        \caption{ $\mathrm{Y}_{5,0}$}
        \end{subfigure}
        \caption{Illustrating \Cref{lem: fifth string lemma}}
        \label{fig: y en y40}
\end{figure}

The next lemma is of a technical nature. Its proof is straightforward, but it reduces to a somewhat tedious case-by-case inspection of the indices of the roots in the relevant subsets, and we therefore omit the details. Instead, we provide a geometric example that makes the underlying idea more transparent.

\begin{lemma}\label{claim: third string lemma 1}
Let  $\alpha_{i,j} \in \Phi^{\geq 2}$, $h=j-i+1$,  $a\geq 1$ and $k\geq 0$.    
Let $I=[a_k+1, b_k+1]$.
Then,  
\begin{enumerate}
    \item\label{item (1) claim third string lemma 1}  If  $b_{k}<i-1$ then the set $\mathrm{Y}_{a,k}\setminus S_k(\alpha_{b_k+1,b_k+h+1}) $ is $s_t$-invariant  for all $t\in I$. 
    \item\label{item (2) claim third string lemma 1}  If  $b_k=i-1$ then the set $Y_{a,k}$ is $s_t$-invariant for all $t\in I =[i-k,i]$. 
\end{enumerate}
\end{lemma}

\begin{example}\rm
In this example, we illustrate the invariance stated in \Cref{claim: third string lemma 1}. 
Set $a=7$ and $k=4$, which gives the corresponding integers $(a_{4},b_{4}) = (16,20)$. 

\medskip
To understand the invariance, it is helpful to recall the geometric action of the simple reflections $s_t$ on the roots: applying $s_t$ moves a root $\alpha_{t,x}$ diagonally down-right to $\alpha_{t+1,x}$, and a root $\alpha_{x,t}$ diagonally down-left to $\alpha_{x,t-1}$. Conversely, it moves $\alpha_{t+1,x}$ diagonally up-left to $\alpha_{t,x}$, and $\alpha_{x,t-1}$ diagonally up-right to $\alpha_{x,t}$.

\medskip
We first address the invariance presented in \Cref{item (1) claim third string lemma 1} of \Cref{claim: third string lemma 1}. The set $Y_{7,4}(\alpha_{23,25})\setminus S_{4}(\alpha_{21,24})$ is depicted in \Cref{fig: bk<i-1}, where the roots of $S_{4}(\alpha_{21,24})$ are highlighted in red. 

\medskip
In this case, it suffices to analyze the roots located at the interior borders of the green region. Recall that $t\in [17,21]$. If we look along the left border of the white trapezoid, the roots that are not fixed by $s_t$ move either up-right or down-left, thus remaining strictly within the green region. At the bottom of the green region, all roots up to $\alpha_{15,16}$ are fixed, while the root $\alpha_{15,16}$ can only move up-right, safely staying inside the set. Similarly, the roots modified by the reflections along the right border of the trapezoid are moved either up-left or down-right, which also keeps them inside the green area. Finally, the two roots $\alpha_{15,22}$ and $\alpha_{16,23}$ directly above the trapezoid are fixed by these reflections. 
This  proves the invariance.

\medskip
Similarly, we can observe the invariance for \Cref{item (2) claim third string lemma 1} of \Cref{claim: third string lemma 1} using the set $Y_{7,4}(\alpha_{21,23})$ depicted in \Cref{fig: bk=i-1}. In this configuration, the previously missing roots (such as $\alpha_{22,24}$) are now included, effectively extending the border of the set. Any boundary root moved outward by the reflections simply lands on these newly included existing roots within the green region. Therefore, the set remains invariant without the need to subtract the red section $S_4(\alpha_{21,24})$.

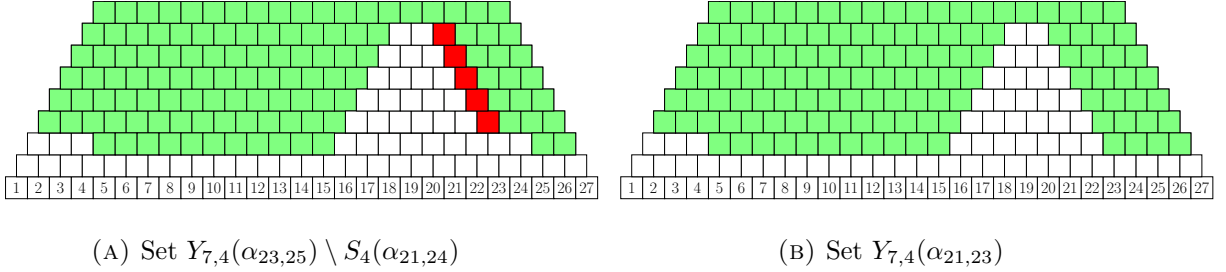
\begin{figure}[H]
        \centering
        \begin{subfigure}[t]{0.48\textwidth}
        \scalebox{0.29}{\begin{tikzpicture}
        \clip (-1,0) rectangle (28,10);
            \dibu{27}{23}{25}
            \quitof{4}{15}{20}
            \quitof{5}{15}{19}
            \quitof{6}{15}{18}
            \quitof{7}{15}{17}
            \quitof{8}{15}{16}
            \agregof{3}{4}{14}
            \quitoS{4}{21}{4}
        \end{tikzpicture}}
        \caption{Set $Y_{7,4}(\alpha_{23,25})\setminus S_{4}(\alpha_{21,24})$}
        \label{fig: bk<i-1}
        \end{subfigure}
        % \hspace{0.5cm}
        \begin{subfigure}[t]{0.48\textwidth}
        \scalebox{0.29}{\begin{tikzpicture}
        \clip (-1,0) rectangle (28,10);
            \dibu{27}{21}{23}
            \quitof{4}{15}{20}
            \quitof{5}{15}{19}
            \quitof{6}{15}{18}
            \quitof{7}{15}{17}
            \quitof{8}{15}{16}
            \agregof{3}{4}{14}
        \end{tikzpicture}}
        \caption{Set  $Y_{7,4}(\alpha_{21,23})$}
        \label{fig: bk=i-1}
        \end{subfigure}
        \caption{Examples of invariance of sets in \Cref{claim: third string lemma 1}. }
    \end{figure}

\end{example}

\begin{lemma}\label{lem: third string lemma}  
Let  $\alpha_{i,j} \in \Phi^{\geq 2}$, $h=j-i+1$, $a\geq 1$, $k\geq 0$ and $\lambda \in X$. 
Suppose that 
\begin{enumerate}[(a)]
\item $b_{k}<i-h$. \label{itemA}  
\item $R_{t}(\lambda)=0$, for $a_{k}+1\leq t \leq b_{k}$. \label{itemB}
\item $ 0\leq R_{b_k+1}(\lambda) \leq k+1$. \label{itemC}
\item $0\leq R_{b_{k}+1,t}(\lambda) \leq k+2$, for  $1\leq t<h$.   \label{itemD}
\end{enumerate}
Furthermore, we define
$\displaystyle \mu =\sum_{t=0}^{h-1} R_{b_{k}+1,t}(\lambda )  \alpha_{b_{k}+t+2,b_{k}+t+h+1}$ and 
$\displaystyle R=\sum_{t=0}^{h-1} R_{b_{k}+1,t} (\lambda )$.
Then, we have 
\begin{equation}
    \bfM_{\lambda}^{\mathrm{Y}_{a,k}} = q^{R}\bfM_{\lambda - \mu}^{ \mathrm{Y}_{a,k+1}}.
    \end{equation}
\end{lemma}

\begin{proof}
By combining \cref{itemB}, \cref{itemC} and  \Cref{claim: third string lemma 1}, we can apply  \Cref{lem: second string lemma}  to $S_{k}(\alpha_{b_k+1,b_k+1+h})$    in order to obtain
\begin{equation}\label{eq: third string lemma 1}
    \bfM_{\lambda}^{\mathrm{Y}_{a,k}} =  q^{R_{b_{k}+1}(\lambda)}\bfM_{\lambda - \mu_{0}}^{\mathrm{Z}_{0}},
\end{equation}
where $\mu_{0} \coloneq R_{b_{k}+1}(\lambda) \alpha_{b_{k}+2,b_{k}+1+h}$ and  $\mathrm{Z}_{0} \coloneq \mathrm{Y}_{a,k}\setminus S_{k}(\alpha_{b_k+1,b_k+h+1})$. 

\medskip
We proceed to iteratively transform $\bfM_{\lambda-\mu_{0}}^{\mathrm{Z}_{0}}$ until reaching $\bfM_{\lambda-\mu}^{\mathrm{Y}_{a,k+1}}$. 
To do this we need to introduce a bit more of notation. 
We define 
\begin{equation}
\begin{array}{rclll}
\mathrm{A}_{s+1} &  = & \displaystyle \{ \alpha_{a_k+s-(h-2),a_k+s+m+1} \mid 0\leq m \leq k+1  \}  , &  & (0\leq s < h-1)  \\[15pt]
\mathrm{Z}_{s+1} &= & \left(\mathrm{Z}_{s}\setminus  S_{k+1}(\alpha_{b_k+s+2,b_k+s+h+2 })  \right)\cup \mathrm{A}_{s+1}, &  & (0\leq s <h-1)\\[5pt]
\mu_{s} &=  & \displaystyle \sum_{u=0}^s R_{b_k+1,u} (\lambda) \alpha_{b_k+2+u,b_k+h+1+u}, &   &  (1\leq s \leq h-1). 
\end{array}
\end{equation}

\begin{claim}\label{claim: third string lemma 2}
For all  $0\leq s \leq h-1$ we have
\begin{enumerate}[(a)]
    \item\label{item: a claim 2} $\mathrm{Z}_{s}$ is $s_{t}$-invariant for $a_{k}+s+1\leq t \leq b_{k}+s+1$.
    \item\label{item: b claim 2} $\mathrm{Z}_{s}\setminus S_{k+1}(\alpha_{b_k+s+2,b_k+s+h+2 })$ is $s_{t}$-invariant for $a_{k}+s+1  \leq t \leq  b_{k}+s+2$.
\end{enumerate}
\end{claim}
\begin{proof}
    The proof uses induction on $s$. 
    The case $s=0$ for \cref{item: a claim 2} is \Cref{claim: third string lemma 1}. 
    Then, we prove that if \cref{item: a claim 2} holds for $s$ then also \cref{item: b claim 2} holds for $s$, and that if \cref{item: b claim 2} holds for $s$  then \cref{item: a claim 2} holds for $s+1$.  
    For the sake of brevity we omit the details. 
\end{proof}

\begin{claim}\label{claim: third string lemma 3}
For all $0\leq s < h-1$ we have
\begin{enumerate}[(a)]
    \item \label{itemAA} $R_t(\lambda - \mu_s) = 0$, for $t\in [a_{k}+s+1,b_{k}+s+1]$.
     \item \label{itemBB} $ R_{b_{k}+s+2} (\lambda -\mu_s) =  R_{b_{k}+1,s+1}(\lambda)$.
    \item \label{itemCC} $ R_{t}(\lambda -\mu_{s+1}) = 0$, 
    for $t\in [a_{k}+s+1,b_{k}+s+2]$.
\end{enumerate}
\end{claim}

\begin{proof}
We fix $0\leq s < h-1$ and  $t\in [a_{k}+s+1,b_{k}+s+1]$. 
Let us first suppose that $a_{k}+s+1 \leq  t\leq b_k$. 
 \Cref{remark producto interno} implies that $\pint{\alpha_{b_k+2+u,b_k+h+1+u}}{\alpha_t} =0$ for all $0\leq u \leq s$. 
Then,  \cref{itemB} in the hypothesis of the lemma allows us to conclude  $R_{t}(\lambda - \mu_s)=0$, as we wanted to show.

We now assume that $t=b_{k}+1$. 
By using \Cref{remark producto interno} we obtain
\begin{equation}
 R_{b_{k}+1}(\lambda - \mu_{s}) =   R_{b_{k}+1}(\lambda)-  R_{b_{k}+1}( \mu_{s})=R_{b_{k}+1}(\lambda)-  R_{b_{k}+1,0}( \lambda) =0.
\end{equation}

Finally, we suppose that $b_{k}+2\leq  t \leq b_k+s+1 $. 
In this case we have 
\begin{equation}
  R_{t}(\lambda -\mu_s) =   
  R_{t}(\lambda) - R_{t}(\mu_{s}) =
  R_{t}(\lambda)+ \sum_{u=0}^s R_{b_k+1,u} \pint{\alpha_{b_k+2+u,b_k+h+1+u}}{\alpha_t}.
\end{equation}

Since $s<h-1$, \Cref{remark producto interno} implies that the only two terms that contribute to the above sum are for $u=t-b_k-2$ and $u=t-b_k-1$. 
Rewriting this  and using the definition of the numbers $R$ we get 
\begin{equation}
R_{t}(\lambda -\mu_s) =   
R_{t}(\lambda)+R_{b_k+1,t-b_k-2}(\lambda) -R_{b_k+1,t-b_k-1}(\lambda) = 
R_{t}(\lambda)- R_{t}(\lambda) =
0.
\end{equation}
This finishes the proof of \cref{itemAA}.

The two remaining items are dealt with similarity, so we omit the details. 
We finish by highlighting that the difference between \cref{itemAA} and \cref{itemCC} is that we do not know the value of $  R_{a_{k}}(\lambda -\mu_0)$.
\end{proof}

By combining  \cref{item: b claim 2} in \Cref{claim: third string lemma 2},   \crefrange{itemAA}{itemBB} in \Cref{claim: third string lemma 3}  and \cref{itemD} in the hypothesis of the lemma, we can apply \Cref{lem: second string lemma} to $S_{k+1}(\alpha_{b_k+s+2,b_k+s+h+2 })$,  for each $0\leq s<h-1$. 
We recall that by definition we have
$\mu_{s+1}= \mu_s +R_{b_k+1,s+1}(\lambda) \alpha_{b_k+s+3,b_k+s+2+h}$.
Thus we get
\begin{equation}\label{eq: third string lemma 2}
    \bfM_{\lambda -\mu_{s}}^{\mathrm{Z}_{s}} = q^{R_{b_{k}+1,s+1}(\lambda)}\bfM_{\lambda -\mu_{s+1}}^{\mathrm{Z}_{s}\setminus S_{k+1}(\alpha_{b_k+s+2,b_k+s+h+2 })}.
\end{equation} 

We now aim to replace the set indexing the $\bfM$-element on the right-hand side of \eqref{eq: third string lemma 2} with $Z_{s+1}$. This can indeed be done by invoking   \cref{item: b claim 2} of \Cref{claim: third string lemma 2} together with \cref{itemCC} in \Cref{claim: third string lemma 3}.
With this in place, we may apply \Cref{coro puedo agregar} $(k+2)$-times in order to add each root of $A_{s+1}$ to the set $\mathrm{Z}_{s}\setminus S_{k+1}(\alpha_{b_k+s+2,b_k+s+h+2 })$.
It is important to note that the addition of these roots must be carried out in order—from the highest root to the lowest root in $A_{s+1}$. Otherwise, the invariance would be affected, preventing us from applying \Cref{coro puedo agregar} as intended. 
In summary, we obtain
\begin{equation}\label{eq: third string lemma 6}
\bfM_{\lambda-\mu_{s}}^{\mathrm{Z}_{s}} = q^{R_{b_{k}+1,s+1}(\lambda)}\bfM_{\lambda-\mu_{s+1}}^{\mathrm{Z}_{s+1}}.
\end{equation}

 Finally,  by combining \eqref{eq: third string lemma 1} and \eqref{eq: third string lemma 6} we get $\bfM_{\lambda}^{\mathrm{Y}_{a,k}} =q^R \bfM_{\lambda - \mu_{h-1}}^{Z_{h-1}}$ and the lemma follows since $\mu =\mu_{h-1}$ and $\mathrm{Z}_{h-1} = \mathrm{Y}_{a,k+1}$.
\end{proof}

\begin{example}\rm \label{ejemplo piramides}
We illustrate \Cref{lem: third string lemma} and its proof for
$\lambda=[1,0,0,0,0,0,0,0,-1,-1,0,1]_\varpi$, $\alpha_{i,j}=\alpha_{13,15}$, and $(a,k)=(4,1)$ in \Cref{fig: y41 y42}.
Here $a_{1}=7$ and $b_{1}=8$.
To save space we draw only positive roots of height $\le 7$ (higher ones belong to any set but omitted).
Roots \emph{deleted} at a step are shown in \textcolor{red}{red}; roots \emph{inserted} are outlined in thick black.

\textbf{(a)} Start with $Y_{4,1}$, which is $s_8$-stable.
\textbf{(b)} To also get $s_9$-stability, delete $\alpha_{9,12}$; this breaks $s_8$-stability, so delete $\alpha_{8,12}$ as well. The result is $Z_0$.
\textbf{(c)} Enforce $s_8,s_9,s_{10}$-stability by passing to $Z_0\setminus S_{10,2}$ (further deletions).
\textbf{(d)} Insert successively $\alpha_{6,10}$, $\alpha_{6,9}$, and $\alpha_{6,8}$ (by \Cref{coro puedo agregar}); after each insertion the set remains stable as required, yielding $Z_1$ (stable under $s_9,s_{10}$).
\textbf{(e)} Delete $\alpha_{9,14}$, $\alpha_{10,14}$, and $\alpha_{11,14}$ to also achieve $s_{11}$-stability, obtaining $Z_1\setminus S_{11,2}$.
\textbf{(f)} Finally, insert $\alpha_{7,11}$, $\alpha_{7,10}$, and $\alpha_{7,9}$ to obtain $Y_{4,2}$.

The weight rows below the pyramids record these changes:
$(a)\to(b)$: subtract $\alpha_{10,12}$;
$(b)\to(c)$: subtract $2\alpha_{11,13}$ (the 10th coordinate equals $-2$);
$(d)\to(e)$: subtract $2\alpha_{12,14}$.
\end{example}

\begin{figure}[htbp]
\centering
% --- Fila 1 ---
\begin{subfigure}[t]{0.32\textwidth}
\centering
\scalebox{0.3}{
        \begin{tikzpicture}
        \clip (-1,0) rectangle (17,8);
            \dibu{16}{13}{15}
            \agregof{3}{1}{5}
            \quitof{4}{6}{8}
            \quitof{5}{6}{7}
            \peso{0}{-1}{-1}{0}{1}{0}{0}{0}{0}
        \end{tikzpicture}
        }
\caption{$Y_{4,1} $}
\label{fig:a}
\end{subfigure}\hfill
\begin{subfigure}[t]{0.32\textwidth}
\centering
\scalebox{0.3}{
        \begin{tikzpicture}
        \clip (-1,0) rectangle (17,8);
            \dibu{16}{13}{15}
            \agregof{3}{1}{5}
            \quitof{4}{6}{8}
            \quitof{5}{6}{7}
            \quitoS{4}{9}{1}
            \peso{0}{0}{-2}{0}{0}{1}{0}{0}{0}
        \end{tikzpicture}
        }
\caption{$Z_0 $}
\label{fig:b}
\end{subfigure}\hfill
\begin{subfigure}[t]{0.32\textwidth}
\centering
\scalebox{0.3}{
        \begin{tikzpicture}
        \clip (-1,0) rectangle (17,8);
            \dibu{16}{13}{15}
            \agregof{3}{1}{5}
            \quitof{4}{6}{9}
            \quitof{5}{6}{8}
            \quitoS{4}{10}{2}
            \peso{0}{0}{0}{-2}{0}{-1}{2}{0}{0}
        \end{tikzpicture}
        }
\caption{$Z_0\setminus S_{10,2}  $}
\label{fig:c}
\end{subfigure}

\par\medskip % salto de línea entre filas

% --- Fila 2 ---
\begin{subfigure}[t]{0.32\textwidth}
\centering
\scalebox{0.3}{
        \begin{tikzpicture}
        \clip (-1,0) rectangle (17,8);
            \dibu{16}{13}{15}
            \agregof{3}{1}{6}
            \quitof{4}{7}{10}
            \quitof{5}{7}{9}
            \quitof{6}{8}{8}
            \agregoA{3}{6}{2}
            \peso{0}{0}{0}{-2}{0}{-1}{2}{0}{0}
        \end{tikzpicture}
        }
\caption{$Z_1$}
\label{fig:d}
\end{subfigure}\hfill
\begin{subfigure}[t]{0.32\textwidth}
\centering
\scalebox{0.3}{
        \begin{tikzpicture}
        \clip (-1,0) rectangle (17,8);
            \dibu{16}{13}{15}
            \agregof{3}{1}{6}
            \quitof{4}{7}{10}
            \quitof{5}{7}{9}
            \quitof{6}{8}{8}
            \quitoS{4}{11}{2}
            \peso{0}{0}{0}{0}{-2}{-1}{0}{2}{0}
        \end{tikzpicture}
        }
\caption{$Z_1\setminus S_{11,2} $}
\label{fig:e}
\end{subfigure}\hfill
\begin{subfigure}[t]{0.32\textwidth}
\centering
\scalebox{0.3}{\begin{tikzpicture}
        \clip (-1,0) rectangle (17,8);
            \dibu{16}{13}{15}
            \agregof{3}{1}{7}
            \quitof{4}{8}{11}
            \quitof{5}{8}{10}
            \quitof{6}{8}{9}
            \agregoA{3}{7}{2}
            \peso{0}{0}{0}{0}{-2}{-1}{0}{2}{0}
        \end{tikzpicture}}
\caption{ $\mathrm{Y}_{4,2}$}
\label{fig: y42}
\end{subfigure}
\caption{Illustrating the proof of \Cref{lem: third string lemma}.}
\label{fig: y41 y42}
\end{figure}
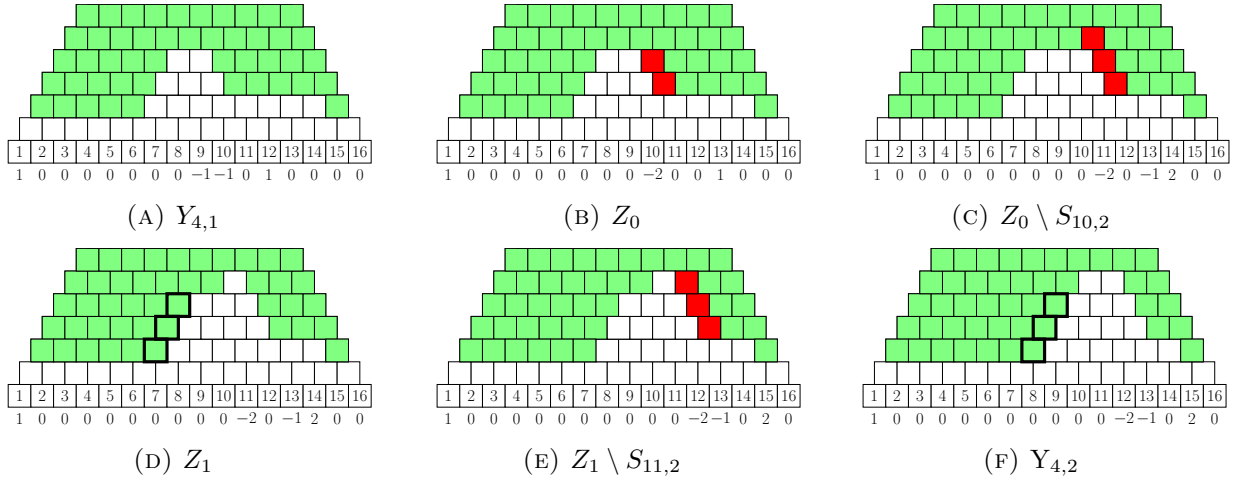

\begin{lemma}\label{lem: fourth string lemma}  
Let  $\alpha_{i,j} \in \Phi^{\geq 2}$, $h=j-i+1$, $a\geq 1$, $k\geq 0$ and $\lambda \in X$.  
We suppose that
\begin{enumerate}[(a)]
\item $b_{k}< i$. \label{itemA4}  
\item $0\leq R_{a,t}(\lambda)\leq 1$, for $0\leq t \leq h-2$. \label{itemB4}
\item $ 0\leq R_{a,h-1,x}(\lambda) \leq x+1$, for $0\leq x\leq k-1$. \label{itemC4}
\item $0\leq R_{a,t,x}(\lambda) \leq x+1$, for $0\leq x\leq k$ and $0\leq t<h-1$.     \label{itemD4}
\end{enumerate}
Furthermore, we define 
\begin{equation}
    \begin{array}{rll}
        \mu_x & = & \displaystyle \left(  \sum_{\ell =0}^x \sum_{m=0}^{h-1} R_{a,m,\ell }(\lambda)\alpha_{b_{\ell }+2 +m -(h -1) ,b_{\ell}+2+m}\right)  - R_{a,h-1,x}  (\lambda)\alpha_{b_{x}+2,b_{x}+2+(h-1)}.  \\[15pt]
        \mathfrak{R} _{k} &= & \displaystyle\left (\sum_{\ell =0}^{k}\sum_{m=0}^{h-1} R_{a,m,\ell }(\lambda)\right) - R_{a,h-1,k}(\lambda)\\[15pt]
        \mathrm{Y} & = & \begin{cases}
            \Phi^{\succ \alpha_{i,j}}\cup \{\alpha_{a-h,a-1}\},&\text{if } a>h;\\
            \Phi^{\succ \alpha_{i,j}},&\text{otherwise}.
        \end{cases}
    \end{array}
\end{equation}
Then, we have 
\begin{equation}
    \bfM_{\lambda}^{\mathrm{Y}} = q^{\mathfrak{R}_k}\bfM_{\lambda - \mu_{k}}^{ \mathrm{Y}_{a,k}}.
\end{equation}
\end{lemma}

\begin{proof}
We begin by proving the following.
 \begin{claim}\label{claim: fourth string claim}
  We have
    \begin{equation}\label{eq: fourth string claim}
       R_{t}(\lambda - \mu_k)  = \left\{\begin{array}{ll}
            0, &  \text{if}\quad t\in [a,b_{k}];\\[5pt]
            R_{a,h-1,k}(\lambda) , &  \text{if}\quad t=b_{k}+1;\\[5pt]
          \displaystyle  \sum_{m=0}^{k+1}R_{t - mh}(\lambda),& \text{if}\quad t\in [b_{k}+2,b_{k+1}].
        \end{array}\right.
    \end{equation}
    \end{claim}
\begin{proof}
By \Cref{remark producto interno} the only term in $\mu_k$ that contributes to the computation of $R_{a}(\lambda - \mu_k)$  is the one associated to $\ell =0$ and $m=0$. Therefore, we have
    \begin{equation}
R_{a}(\lambda -\mu_{k}) = R_a(\lambda ) -R_a(\mu_k) = R_{a}(\lambda)-R_{a,0,0}(\lambda)=0. 
    \end{equation}
    We now assume that $a+1\leq t\leq b_k$ and write $t-a-1=\ell h+m$ for unique non-negative integers $0\leq \ell \leq k$ and $0\leq m \leq h-1$. 
    By \Cref{remark producto interno} we get
    \begin{equation}
        R_t(\mu_k) = \left\{ \begin{array}{ll}
       R_{a,m+1,0}(\lambda)-R_{a,m,0}(\lambda),      &  \mbox{if } \ell = 0 \mbox{ and } m<h-1; \\[3pt]
       R_{a,0,1}(\lambda)-R_{a,0,0}(\lambda)-R_{a,h-1,0}(\lambda) ,     & \mbox{if } \ell = 0 \mbox{ and } m=h-1; \\[3pt]
       R_{a,m+1,\ell}(\lambda)-R_{a,m,\ell}(\lambda) -R_{a,m+1,\ell -1}(\lambda) +R_{a,m,\ell -1}(\lambda), & \mbox{if } \ell > 0 \mbox{ and } m<h-1; \\[3pt]
       R_{a,0,\ell +1}(\lambda)-R_{a,h-1,\ell}(\lambda) -R_{a,0,\ell }(\lambda) +R_{a,h-1,\ell -1}(\lambda), & \mbox{if } \ell > 0 \mbox{ and } m=h-1. \\
        \end{array} \right.
    \end{equation}
 By looking at the definition of the $R$-numbers we can see that for all $\ell $ and $m$ we have
 \begin{equation}
      R_t(\mu_k) = R_{a+\ell h +m+1,0,0} (\lambda)=R_{a+\ell h +m+1} (\lambda) =R_{t}(\lambda) .
 \end{equation}
Consequently, $ R_t(\lambda - \mu_k)=0$ and the first case in \cref{eq: fourth string claim} follows.

We now assume that $t = b_k + 1$.
Observe that, a priori, the term corresponding to $\ell = k$ and $m = h-1$ in the sum defining $\mu_k$ contributes to  $R_t(\mu_k)$.
However, this term is canceled by the one outside the sum.
We introduce $\mu_k$ in this way solely for the sake of a more compact expression.
With this in mind, and using \Cref{remark producto interno}, we obtain
\begin{equation}
\begin{array}{lll}
R_{b_k+1}(\lambda - \mu_k)    &  = & R_{b_{k}+1}(\lambda)-R_{a,h-2,k-1}(\lambda) + R_{a,h-1,k-1}(\lambda) + R_{a,h-2,k}(\lambda) \\
     & = & R_{b_{k}+1}(\lambda)-R_{a,h-2,k-1}(\lambda) + R_{a,h-1,k-1}(\lambda) + R_{a,h-2,k-1}(\lambda) +R_{a,h-2+kh}(\lambda) \\
     &  =  &  R_{b_{k}+1} (\lambda) +R_{a,h-1,k-1}(\lambda)  +R_{a,h-2+kh}(\lambda) \\
     & =  &  R_{a,h-1+hk}(\lambda) + R_{a,h-1,k-1}(\lambda) \\
     &  =  &  R_{a,h-1,k}(\lambda),
\end{array}
\end{equation}
as we wanted to show.

We now assume that $t=b_{k}+2$.
By \Cref{remark producto interno} and by recalling that $b_k+2=a+(k+1)h$, we obtain
\begin{equation}
  R_{b_k+2}(\lambda - \mu_k)= R_{b_{k}+2}(\lambda) -R_{a,h-1,k-1}(\lambda)+R_{a,0,k}(\lambda) = R_{b_{k}+2}(\lambda) + \sum_{m=0}^{k} R_{a+mh}(\lambda) =  \sum_{m=0}^{k+1} R_{a+mh}(\lambda).
\end{equation}

Finally, we assume  $b_k+2 < t \leq b_k+h$. 
In this case we have
\begin{equation}
   R_t(\lambda - \mu_k)  = R_{t}(\lambda)+R_{a,t-b_k-2,k}(\lambda)-R_{a,t-b_k-3,k}(\lambda) =  R_{t}(\lambda)+ \sum_{m=0}^{k} R_{a+t -b_k-2+mh}(\lambda) =  \sum_{m=0}^{k+1} R_{t-mh}(\lambda).
\end{equation}
This finishes the proof of the claim. 
\end{proof}

We now return to the proof of the lemma. 
The argument proceeds by induction on $k$. 
% First, observe that 
% \begin{equation}
%    \mathrm{Y}_{a,0} =\begin{cases}    
%    ( \Phi^{\succ \alpha_{i,j}}  \setminus [\alpha_{a,a+h},\alpha_{a-h+2,a+2}]) \cup [\alpha_{a-h,a-1},\alpha_{a-1,a+h-2}], &\text{if } a>h;\\
%    ( \Phi^{\succ \alpha_{i,j}}  \setminus [\alpha_{a,a+h},\alpha_{a-h+2,a+2}]) \cup [\alpha_{1,h},\alpha_{a-1,a+h-2}],&\text{otherwise}.
%    \end{cases}
% \end{equation}
By \cref{itemA4} and  \cref{itemB4} we have that $a\leq i+1-h$ and that $R_{a,t}(\lambda)=\{0,1\}$ for all $ 0\leq t \leq h-2$, respectively.
Thus, we are in position to apply \Cref{lem: fifth string lemma} to $\alpha_{i,j}$, $a$,  $\lambda$ and $\mu=\mu_0$.
We obtain 
\begin{equation}
   \bfM_{\lambda}^{Y} = q^{ \mathfrak{R}_0} \bfM_{\lambda-\mu_0}^{Y_{a,0}}.  
\end{equation}
This establishes the initial step of the induction.

Assume now that the lemma holds for some $k$, and let us verify the case $k+1$. 
By the inductive hypothesis, we obtain
\begin{equation}\label{eq: poof fourth string lemma 1}
    \bfM_{\lambda}^{Y} = q^{   \mathfrak{R}_{k}}\bfM_{\lambda - \mu_{k}}^{\mathrm{Y}_{a,k}}.
\end{equation}

We now want to apply \Cref{lem: third string lemma} in order to rewrite $\bfM_{\lambda-\mu_k}^{Y_{a,k}}$. 
To do this we need to check that  $\alpha_{i,j}$, $a$, $k$ and $\lambda -\mu_k$  satisfy the hypothesis of that lemma, assuming that $\alpha_{i,j}$, $a$, $k+1$ and $\lambda$ satisfy the hypothesis of the current lemma. 
We need to verify
 
\begin{enumerate}[(I)]
\item $b_{k}<i-h$.
\label{itemAAA}
\item $R_{t}(\lambda-\mu_k)=0$ for $a_{k}+1\leq t \leq b_{k}$. \label{itemBBB}
\item $0\leq R_{b_k+1}(\lambda -\mu_k) \leq k+1$.
\label{itemCCC}
\item $0\leq R_{b_k+1,t}(\lambda -\mu_k) \leq k+2$ for $1\leq t<h$.
\label{itemDDD}
\end{enumerate}

 \Cref{itemAAA} follows from \cref{itemA} applied to $k+1$, since $b_{k+1} = b_{k} + h < i$.
 \Cref{itemBBB} is a direct consequence of  \Cref{claim: fourth string claim}. 
On the other hand, the same claim implies that  $R_{b_k+1}(\lambda-\mu_k) =R_{a,h-1,k}(\lambda)$. 
This together with \cref{itemC4} in the hypothesis (applied to $k+1$) verifies \cref{itemCCC}. 
Finally, using  \Cref{claim: fourth string claim} and the definition of the numbers $R$, we get
\begin{equation}
   R_{b_k+1,t}(\lambda -\mu_k) =   R_{a,t-1,k+1}(\lambda),
\end{equation}
for all $1\leq t < h$. 
Therefore,  by using \cref{itemD4} in the hypothesis (applied to $k+1$) we obtain  \cref{itemDDD}.

This verify that $\alpha_{i,j}$, $(a,k)$, $\lambda - \mu_{k}$ satisfy the hypothesis of \Cref{lem: third string lemma}. By applying this lemma in this context we get
\begin{equation}\label{eq: fourth string lemma 2}
    \bfM_{\lambda - \mu_{k}}^{\mathrm{Y}_{a,k}} = q^{R}\bfM_{\lambda - \mu_{k} - \mu}^{\mathrm{Y}_{a,k+1}},
\end{equation}
where 
\begin{equation}
    R = \sum_{t=0}^{h-1} R_{b_{k}+1,t}(\lambda - \mu_{k}) \qquad \mbox{and} \qquad    \mu = \sum_{t=0}^{h-1}R_{b_{k}+1,t}(\lambda - \mu_{k})\alpha_{b_{k}+t+2,b_{k}+t+h+1}. 
\end{equation}
A direct computation, using \Cref{claim: fourth string claim} and the definition of the numbers $R$,   shows that $\mu_{k+1}=\mu_k +\mu $ and $\mathfrak{R}_{k+1}=\mathfrak{R}_k + R   $.
Therefore, a combination of \eqref{eq: poof fourth string lemma 1} and \eqref{eq: fourth string lemma 2}   yields the result.
\end{proof}

\begin{corollary}
\label{lem: primer paso PPPPPP}
Let $\alpha_{i,j} \in \Phi^{\geq 2}$ and $h=j-i+1$.
Let $(a,b)$ be a pair of integers such that $h+1\leq a < b < i$ with $a\equiv b\equiv i+1 \mod{h}$. 
Let  $\lambda\in X$ be such that  $R_t(\lambda) = 0$ for $t \in [a,i]\setminus \{b\}$ and $R_b(\lambda) = 1$.
Furthermore, we define  $Y=\Phi^{\succ\alpha_{i,j}}\cup \{\alpha_{a-h,a-1}\}$.
Then, we have $\bfM_{\lambda}^{Y} = 0$.
\end{corollary}

\begin{proof}
Let 
\begin{equation}
       k = \frac{i+1-a}{h}-1 \geq 0 . 
\end{equation}
We begin by checking that the triple $(a,k, \lambda)$ satisfies the hypothesis of  \Cref{lem: fourth string lemma}. 
This is, we need to verify \crefrange{itemA4}{itemD4} in that lemma.

We first notice that $b_k=a+(k+1)h-2 = i-1<i  $. This gives \cref{itemA4}. 

We now notice that $a+h \leq b < i$, since $a<b$ and $a \equiv b \pmod{h}$. 
It follows that for  $0 \leq t \leq h-2$ we have
\begin{equation}
    R_{a,t}(\lambda) = \sum_{m=a}^{a+t}R_{m}(\lambda)= 0, 
\end{equation}
which gives \cref{itemB4}. 
To check \cref{itemC4},  for $0\leq x < k$, we compute 
\begin{equation}\label{EEEEQQQQ}
  \displaystyle  0\leq  R_{a,h-1,x}(\lambda) 
    = \sum_{m=0}^{x} \left(\sum_{t=a}^{a+h(m+1)-1}R_{t}(\lambda)  \right)
    \leq  \sum_{m=1}^{x} 1 = x
\end{equation}
Let us explain the  inequality. 
By the hypothesis on the numbers $R_{t}(\lambda)$, the internal sum in \cref{EEEEQQQQ} is $1$ if $b$ occurs in the sum range and $0$ otherwise. 
Furthermore, if $m=0$ then $b$ does not occur in the corresponding internal sum. 
This verifies \cref{itemC4}. 

Finally, \cref{itemD4} is verified in a similar fashion. 
Concretely, for $0\leq x \leq k$ and $0\leq t < h-1$,  we expand $R_{a,t,x}$ as in \cref{EEEEQQQQ}.
The only difference in this case is that if $t\geq 1$ then $b$ might occur in the internal sum corresponding to $m=0$. 
It follows that  $0\leq  R_{a,h-1,x}(\lambda) \leq x +1$, as we wanted to show. 

Having checked the hypothesis of \Cref{lem: fourth string lemma} for the triple $(a,k,\lambda)$, we can apply it in order to obtain
\begin{equation}
    \bfM_{\lambda}^{Y} = q^{\mathfrak{R}_{k}}\bfM_{\lambda - \mu_{k}}^{\mathrm{Y}_{a,k}},
\end{equation}
where $\mu_{k}$ and $\mathfrak{R}_{k}$ are defined as in \Cref{lem: fourth string lemma}.

We claim that $\bfM_{\lambda - \mu_{k}}^{\mathrm{Y}_{a,k}}=0$, and this would give  the result. 
We recall that our choice of  $k$ gives $b_k=i-1$. 
Therefore, by  \Cref{claim: third string lemma 1} we conclude that $Y_{a,k}$ is invariant under $s_t$ for $t\in [i-k,i]$. 
On the other hand, \Cref{claim: fourth string claim}  and a direct computation yield $R_{b_k+1}(\lambda-\mu_k) = R_{a,h-1,k} = k-c+1$, where $c\coloneqq (b-a)/h \geq 1$. 
Thus, using  \Cref{claim: fourth string claim} once again, we deduce that 
\begin{equation}
    \langle \lambda -\mu_{k}  , \alpha_{a_{k}+c,b_{k}+1} \rangle = - (k- (c-1)). 
\end{equation}
Finally, by applying \Cref{lem: weyl group over M elements} with $w = s_{\alpha_{a_{k}+c,b_{k}+1}}$ we obtain that $\bfM_{\lambda - \mu_{k}}^{\mathrm{Y}_{a,k}}=0$.
This proves our claim and the corollary. 
\end{proof}

\section{Second Inverse Decomposition}
\label{sec: second inverse decomp}

In this section we refine the First Inverse Decomposition from \Cref{prop: second version}.
This version is best suited for establishing \Cref{Conj A} in the following section.

\subsection{The $\mathbf{v}$-operator}

\begin{definition}\rm \label{def: v negrita}
    Let $\lambda \in X^+$ and  $(j,h)$ be a pair of  integers such that $2\leq h \leq j$.
    Suppose that $\pint{\lambda}{\alpha_t}=0$ for all $j-(h-2)\leq t \leq j$.
    Let $\varkappa$ be the smallest non-negative integer such that $Q^{-\varkappa}_{j+1,h}(v^{h-1}_{j,h}(\lambda)) \in X^+$.
    \begin{itemize}
        \item If such an integer exists then we define $\mathbf{v}_{j,h}(\lambda) =  Q^{-\varkappa}_{j+1,h}(v^{h-1}_{j,h}(\lambda))$.
        \item Otherwise, we say that $\mathbf{v}_{j,h}(\lambda)$ is not defined. 
    \end{itemize}
     In the first case we call $\varkappa =\varkappa_{j,h}(\lambda) $ the integer associated to $\mathbf{v}_{j,h}(\lambda)$. 
\end{definition}

\begin{remark}\rm \label{rem: varkappa =0}
    We stress that if $v^{h-1}_{j,h}(\lambda)\in X^+$  then $\varkappa=0$ and $\mathbf{v}_{j,h}(\lambda) = v^{h-1}_{j,h}(\lambda)$.
    On the other hand, if $v^{h-1}_{j,h}(\lambda)$ is not defined, then $\mathbf{v}_{j,h}(\lambda)$ is automatically not defined.
\end{remark}

The following lemma asserts that, under repeated application of  $Q$-operators, the weight $v^{h-1}_{j,h}(\lambda)$ either transforms into the dominant weight $\mathbf{v}_{j,h}(\lambda)$, yielding an identity between the corresponding $\bfM$-elements, or else $v^{h-1}_{j,h}(\lambda)$ fails to reach the dominant chamber, in which case the associated $\bfM$-element vanishes.

\begin{lemma}\rm\label{lem: completar Q en caso h-1}
    Let  $\alpha_{i,j}\in \Phi^{\geq 2}$ and $h=j-i+1$.
    Let $\lambda\in X^{+}$ satisfying $\langle \lambda , \alpha_r \rangle  =0$, for all $i+1\leq r \leq j$.
    Suppose that $v^{h-1}_{j,h}(\lambda)$ is defined and set $Y= \Phi^{\succ \alpha_{i,j}} \cup \{ \alpha_{i-h+1,i}  \} $.
    Then, we have
\begin{equation}
    \bfM_{v^{h-1}_{j,h}(\lambda)}^{Y} =
    \begin{cases}
        q^{\varkappa}\bfM_{\mathbf{v}_{j,h}(\lambda)}^{Y}, & \mbox{if } \mathbf{v}_{j,h}(\lambda) \mbox{ is defined;} \\
        0, & \mbox{otherwise.}
    \end{cases}
\end{equation}
\end{lemma}

\begin{proof}
We first assume that $\mathbf{v}_{j,h}(\lambda)$ is defined.
If $\varkappa=0$ then  there is nothing to prove.
So that  we can assume $\varkappa>0$.
For $0 \leq k \leq \varkappa$ we define $\mu_k = Q^{-k}_{j+1,h}(v^{h-1}_{j,h}(\lambda))$. 
By combining \Cref{lem: desigualdad de varthetas} and \Cref{lem: thetas para el caso h-1}  we obtain for all $1\leq k \leq \varkappa$ that 
\begin{equation} \label{eq:S6A}
    \mu_k = \mu_{k-1} - \alpha_{j+k+1,j+k+h-1}, 
\end{equation}
or in  words, the $Q$-operator needed to pass from $\mu_{k-1} $ to $\mu_k$ consists of a unique root of height $h-1$. 
On the other hand,  \Cref{lem: thetas para el caso h-1} also implies that
\begin{equation}\label{eq:S6B}
  \mu_0= v^{h-1}_{j,h}(\lambda) = P^{h-1}_{j,h}(\lambda) - \sum_{r=1}^{h-1}\alpha_{i+r+1,j+r} =  P^{h-1}_{j,h}(\lambda) +\varpi_{i+1} -2\varpi_{j+1} +\varpi_{j+h}.
\end{equation}

Furthermore, by definition of $P^{h-1}_{j,h}(\lambda)$  and the assumptions on $\lambda$ we have 
\begin{equation}\label{eq:S6C}
    \pint{P_{j,h}^{h-1}(\lambda)}{\alpha_t} = \begin{cases}
    -1, & \mbox{if } t=i+1;\\
    0, & \mbox{if } t \in [i+2,j];\\
    \pint{\lambda}{\alpha_{j+1}} +1, & \mbox{if }  t=j+1.\\
        \pint{\lambda}{\alpha_t}, & \mbox{if } t\in [j+2,n];\\ 
    \end{cases}
\end{equation}                                                                                                                                                                                                                                                                                                             

Thus, combining \eqref{eq:S6A}, \eqref{eq:S6B}, and \eqref{eq:S6C} with \Cref{remark producto interno}, we obtain, for all 
$0\leq k<\varkappa$, that

\begin{equation}\label{eqS6 mu versus alpha}
\pint{\mu_k}{\alpha_t} = \begin{cases}
    0, & \mbox{if } t\in [i+1, j+k];\\
    -1,& \mbox{if } t=j+k+1.
\end{cases}
\end{equation}

Hence, for $0\leq k \leq \varkappa - 1$, we can apply \Cref{lem: free Q2} to $K=\Phi^{\succ\alpha_{i,j} } $, $\alpha_{i,j}$, $T=T_{k}= [i+1,j+k+1]$, $\mu_{k}$, $p' = j+k+1$ and $p$ be the unique integer such that $i+1\leq p\leq j$ and $p\equiv p' \mod{h-1}$ to get
    \begin{equation}
        \bfM_{\mu_{k}}^{Y} = q\bfM_{\mu_{k+1}}^{Y}.
    \end{equation}
    
Therefore,  $\bfM_{v^{h-1}_{j,h}(\lambda)}^{Y} =  \bfM_{\mu_{0}}^{Y} =q^{\varkappa}  \bfM_{\mu_{\varkappa}}^{Y} =  q^{\varkappa}\bfM_{\mathbf{v}_{j,h}(\lambda)}^{Y}$, as we wanted to show.

We now assume that $\mathbf{v}_{j,h}(\lambda)$ is not defined. 
Although this weight is not defined, the weight $\mu_{n-j-h+1}  $ does exist. 
Furthermore, it can be computed using \cref{eq:S6A} and \cref{eqS6 mu versus alpha} still holds. 
Thus, arguing as in the previous case we obtain
\begin{equation}\label{eq:S6D}
    \bfM_{v^{h-1}_{j,h}(\lambda)}^Y = q^{n-j-h+1} \bfM_{\mu_{n-j-h+1}}^{Y}.
\end{equation}

 On the other hand, by  \Cref{lem: free Q2} applied to $K=\Phi^{\succ \alpha_{i,j} } $, $\alpha_{i,j}$, $T=[i+1,n-h+2]$, $\mu_{n-j-h+1}\in X$, $p'=n-h+2$, and $p$ be the unique positive integer such that $i+1\leq p\leq j$ and $p\equiv p'\mod{h-1}$, we conclude that
    \begin{equation}\label{eq: final completar Q caso h-1}
        \bfM_{\mu_{n-j-h+1}}^{Y} =0.
    \end{equation}
 Therefore, the result follows by combining  \cref{eq:S6D} and  \eqref{eq: final completar Q caso h-1}. 
\end{proof}

%%%%%%%%%%%%%%%%%%%%%%%%%%%%%%%%%%%%%%%%%%%%%%%%%%%%%%%%%%%%
%%%%%%%%%%%%%%%%%%%%%%%%%%%%%%%%%%%%%%%%%%%%%%%%%%%%%%%%%%%%
\subsection{The $\mathbf{P}$-operator}
%%%%%%%%%%%%%%%%%%%%%%%%%%%%%%%%%%%%%%%%%%%%%%%%%%%%%%%%%%%%
%%%%%%%%%%%%%%%%%%%%%%%%%%%%%%%%%%%%%%%%%%%%%%%%%%%%%%%%%%%%

\begin{definition}\rm  \label{def: P negrita}
  Let $(h,k)$ be a pair of integers such that $2\leq h \leq k \leq n$.
  Let $\lambda \in X_{k}^+(L)$. 
  Let $x$ be the smallest positive integer such that $P^{x}_{k,h}(\lambda)\in X_{k}^+(L)$.
  If such an integer exists, then we define $\mathbf{P}_{k,h}(\lambda) =  P^{x}_{k,h}(\lambda)$. 
  Otherwise, we say that  $\bfP_{k,h}(\lambda)$ is not defined. 
\end{definition}

The main goal of this section is to prove the next lemma. 

\begin{lemma}\rm\label{lem: bfP}
  Let $\alpha_{i,j}\in \Phi^{\geq 2}$ and $h=j-i+1$.
  Let  $\lambda \in X_{j}^+(L)$ and  $\mu = P^{h}_{j,h}(\lambda)$.
  If $\pint{\mu}{\alpha_{i}} = -1$, then
    \begin{equation}\label{eq: bfP}
        \bfM_{\mu}^{\succ \alpha_{i,j}} = \left\{ \begin{array}{ll}
            q^{D(\mathbf{P}_{i-1,h})(\mu)}\bfM_{\mathbf{P}_{i-1,h}(\mu)}^{\succ \alpha_{i,j}}, & \text{if } \mathbf{P}_{i-1,h}(\mu) \text{ is defined;} \\[2pt]
            0, & \text{otherwise.}
        \end{array}\right.
    \end{equation}
\end{lemma}

Before embarking in the  proof of \Cref{lem: bfP} we need some preparatory results. 

\begin{lemma}\rm\label{lem: blackP caso 0}
    Let $\alpha_{i,j}\in \Phi^{\geq 2}$  and  $h=j-i+1$.
    Suppose that  $i+1\equiv 0 \mod h$.  
    Let $b$ be an integer such that $h \leq b \leq i$ and $ b\equiv 0 \mod{h}$.
    Let $\mu\in X_{h-1}^{+}(L) $  be such that $\pint{\mu}{\alpha_t}=-\delta_{t,b}$ for $t\in [h,i]$. 
    Then we have
    \begin{equation}\label{eq: first caso 0 blackP}
        \bfM_{\mu - \alpha_{i,j}}^{\succ \alpha_{i,j}} = 0.
    \end{equation}
\end{lemma}

\begin{proof}
We proceed by induction on $m = \pint{\mu}{\alpha_{h-1}}$.
If $m=0$ then $\pint{\mu-\alpha_{i,j}}{\alpha_t}=0$ for all $1\leq t \leq i-h$ and $t\equiv i-h \equiv -1 \mod h$. 
By definition of the $P_{i-1,h}$-operator, it follows that  $P_{i-1,h}(\mu-\alpha_{i,j})$ is not defined.
Then,  by \Cref{lem: free P} applied to $K =\Phi^{\succ \alpha_{i,j}}$, $\alpha_{i,j}$, $p = i = p'$ and the weight $\mu- \alpha_{i,j}$ we obtain $\bfM_{\mu-\alpha_{i,j}}^{\succ \alpha_{i,j}} = 0$.

We now suppose that $m\geq 1$ and assume that \eqref{eq: first caso 0 blackP} holds for $m-1$.
In this case $P_{i-1,h}(\mu-\alpha_{i,j})$ is defined.
In fact, we have
\begin{equation}
    \mu'\coloneqq P_{i-1,h}(\mu-\alpha_{i,j}) = \mu-\alpha_{i,j} -\alpha_{h-1,i-1}=\mu-\alpha_{h-1,j}= P_{j,h}(\mu) .
\end{equation}
Then, applying \Cref{lem: free P} with the same parameters as before, we get 
\begin{equation}\label{eq: first caso 0 blackP proof 1}
    \bfM_{\mu-\alpha_{i,j}}^{\succ \alpha_{i,j}} = q^{D(P_{i-1,h})(\mu  - \alpha_{i,j})}\bfM_{\mu'}^{\succ\alpha_{i,j}}=  q^{D(P_{j,h})(\mu  )-1}\bfM_{\mu'}^{\succ\alpha_{i,j}}.
\end{equation}

On the other hand, by the conditions imposed on $b$ and $\mu$ it is easy to see that $P_{b-1,h}(\mu')$ is not defined.
 Therefore, another application of  \Cref{lem: free P}, but now for the parameters
\begin{equation}
\displaystyle    K =\Phi^{\succ\alpha_{i-1,j-1}}, \quad \alpha_{i-1,j-1}, \qquad p = i-h+1, \qquad p'= b, \qquad k = \frac{i-b+1}{h}-1 \quad \mbox{ and } \quad \mu',
\end{equation}
 yields $\bfM_{\mu'}^{\succ \alpha_{i-1,j-1}} =0$.
 By the definition of $\bfM$-elements  we have
 \begin{equation} \label{eq: first caso 0 blackP proof 1 AA}
     0 = \bfM_{\mu'}^{\succ \alpha_{i-1,j-1}}  = \bfM_{\mu'}^{\succ \alpha_{i,j}} -q \bfM_{\mu' -\alpha_{i,j}}^{\succ \alpha_{i,j}}.
 \end{equation}
Thus, by combining \eqref{eq: first caso 0 blackP proof 1} and \eqref{eq: first caso 0 blackP proof 1 AA}, we obtain

\begin{equation}\label{eq: first caso 0 blackP proof 1 AAA}
    \bfM_{\mu-\alpha_{i,j}}^{\succ \alpha_{i,j}} = q^{D(P_{j,h})(\mu)}\bfM_{\mu' -\alpha_{i,j}}^{\succ\alpha_{i,j}}.
\end{equation}

Finally, since $\mu'$ satisfies the hypothesis of the lemma and $\pint{\mu'}{\alpha_{h-1}} =m-1$ we can apply our inductive hypothesis to conclude that the right-hand side of \cref{eq: first caso 0 blackP proof 1 AAA} vanishes, which completes the proof. 
\end{proof}

% \begin{lemma}\rm \label{lem: 0s de PPPPPPP}
%     Let  $\alpha_{i,j}\in \Phi^{\geq 2}$ and $h=j-i+1$.
%     Let $\lambda\in X_{j}^+(L)$ and for $x\geq 0$ we set  $\lambda_{x} = P^x_{j,h}(\lambda)$.
%     Let $k\geq h+1$ be an integer such that  $\lambda_{k}$ is defined.
%     Suppose that $\pint{\lambda_{x}}{\alpha_{j-(x-1)}} = -1$ for all $h\leq x<k$.
%     Then, we have
%     \begin{equation}\label{eq: 0s de PPPPPPP}
%         \pint{\lambda_{k-1}}{\alpha_{t}} =0,\quad\text{for all }\quad t\in [\varrho_{k}+h,i]\setminus \{j-(k-2)\},
%     \end{equation}
%    where $\varrho_{k} = \varrho_{j-(k-1),h}(\lambda_{k-1})$.
   
%   Additionally, if $\varrho_{k}+h \leq j-k+1 \leq i$, then $\pint{\lambda_{k}}{\alpha_{j-k+1}} = -1$.
% \end{lemma}

\begin{lemma}\rm \label{lem: 0s de PPPPPPP}
    Let  $\alpha_{i,j}\in \Phi^{\geq 2}$ and $h=j-i+1$.
    Let $\lambda\in X_{j}^+(L)$ and for $x\geq 0$ we set  $\lambda_{x} = P^x_{j,h}(\lambda)$.
    Let $k\geq h+1$ be an integer such that  $\lambda_{k}$ is defined and $\lambda_{u}\notin X_{j}^+(\lambda) $  for $h\leq u<k$.
    Then, we have
    \begin{equation}\label{eq: 0s de PPPPPPP}
        \pint{\lambda_{k-1}}{\alpha_{t}} =0,\quad\text{for all }\quad t\in [\varrho_{k}+h,i]\setminus \{j-(k-2)\},
    \end{equation}
   where $\varrho_{k} = \varrho_{j-(k-1),h}(\lambda_{k-1})$.
\end{lemma}

\begin{proof}
This result  follows by an inductive argument  on $k$ that  combines  \Cref{lem: desigualdad de varrhos},  \Cref{lem: ceros por Ps} and the definition of the $P$-operator. 
We left the details to the reader.
\end{proof}

% \begin{lemma}\rm\label{lem: first blackP}
%     Let $\alpha_{i,j}\in \Phi^{\geq 2}$ and $h=j-i+1$. 
%     Let  $b$ be an integer such that $h+1 \leq b \leq i$.
%     Let $\lambda\in X_j^+(L)$ and set  $\lambda_{x} = P^{x}_{j,h}(\lambda)$ for $x\geq 0$.
%     Suppose that $\lambda_{j-b+1}$ is defined and $\pint{\lambda_{x}}{\alpha_{j-x+1}}=-1$ for $h\leq x\leq j-b+1$.
    
%     Then
%     \begin{equation}\label{eq: first blackP}
%         \bfM_{\lambda_{j-b+1}}^{\succ \alpha_{i,j}} = \left\{\begin{array}{ll}
%             q^{\D{P,b-1,h}(\lambda_{j-b+1})}\bfM_{\lambda_{j-b+2}}^{\succ \alpha_{i,j}},&\text{if}\quad \lambda_{j-b+2}\quad\text{is defined};\\
%             0, &\text{otherwise}.
%         \end{array}\right. 
%     \end{equation}
% \end{lemma}

\begin{lemma}\rm\label{lem: first blackP} 
    Let $\alpha_{i,j}\in \Phi^{\geq 2}$ and $h=j-i+1$. 
    Let $\lambda\in X_j^+(L)$ and set  $\lambda_{x} = P^{x}_{j,h}(\lambda)$ for $x\geq 0$.
    Let  $u$ be an integer such that $h \leq u \leq i-1$.
    Suppose that $\lambda_{u}$ is defined and $\lambda_{x}\notin X_{j}^+(\lambda) $  for $h\leq x<u$. 
    Then, we have
    \begin{equation}\label{eq: first blackP}
        \bfM_{\lambda_{u}}^{\succ \alpha_{i,j}} = \left\{\begin{array}{ll}
            q^{\D{P,j-u,h}(\lambda_{u})}\bfM_{\lambda_{u+1}}^{\succ \alpha_{i,j}},&\text{if } \lambda_{u+1}\text{ is defined};\\
            0, &\text{otherwise}.
        \end{array}\right. 
    \end{equation}
\end{lemma}

\begin{proof}
Let $b= i-u+h$. 
We remark that $b-1=j-u$. 
We split the proof in four cases.

\medskip
\textbf{Case A. }   $\lambda_{u+1}$ is defined and $b\equiv i+1 \mod h$.
\medskip

We recall that $\lambda_{u+1}=P_{b-1,h}(\lambda_{u})$. 
Let $\varrho_{u+1} = \varrho_{b-1,h}(\lambda_{u})$ be the integer associated to $P_{b-1,h}(\lambda_{u})$.

By \Cref{lem: 0s de PPPPPPP}  applied to $k = u+1$ we obtain
 $\pint{\lambda_{u}}{\alpha_{t}} = 0$ for all 
 $t\in T\coloneq[\varrho_{u+1}+h,i]\setminus \{b\}$.
    Furthermore, since $\lambda_{u}\notin X_{j}^+(L)$ we have
   $\pint{\lambda_{u}}{\alpha_{b}} = -1$.

 We  proceed by induction on $D = \D{P,b-1,h}(\lambda_{u})$.

Suppose that $D=1$.
In this case we have $\varrho_{u+1}=i-u=b-h$ and $T=[b+1,i] $.  
Therefore, we can apply  \Cref{lem: left solving} to $\alpha_{i,j}$,  $K=\Phi^{\succ \alpha_{i,j}}$, $p=i-h+1$,  $p^{\prime} = b$ and $\lambda_{u}$ to get 
   \begin{equation}
      \bfM_{\lambda_{u}}^{\succ \alpha_{i,j}}= q\bfM_{\lambda_{u} -\alpha_{b-h,b-1}}^{\succ \alpha_{i,j}} = q\bfM_{\lambda_{u+1}}^{\succ \alpha_{i,j}},  
   \end{equation}
   which proves the lemma for $D=1$.    
   
 We now suppose that  $D>1 $ and that the lemma holds for $D-1$. 
Let $a \coloneqq  \varrho_{u+1} +h$. 
We notice that $a\equiv i-u \mod h$ and therefore $a\equiv b \mod h$. 
Furthermore, since $D>1$ we have $a<b$ and since $b\neq i$ we have $b<i$. 
Thus,  we can apply \Cref{lem: primer paso PPPPPP} to $\alpha_{i,j}$, $a$ and $b$ to obtain $\bfM_{\lambda_{u}}^{Y}=0$, where $Y=\Phi^{\succ \alpha_{i,j}} \cup \{ \alpha_{a-h,a-1}  \}$.
By definition of $\bfM$-elements we have
\begin{equation}
    \bfM_{\lambda_{u}}^{\succ \alpha_{i,j}} = q\bfM_{\lambda_{u} - \alpha_{a-h,a-1}}^{\succ \alpha_{i,j}}.
\end{equation}
On the other hand, we recall that 
\begin{equation}
    \lambda_{u+1} =P_{b-1,h}(\lambda_{u})=\lambda_{u}-\alpha_{\varrho_{u+1}
    ,b-1} =\lambda_{u}-\alpha_{a-h,b-1} =(\lambda_{u}-\alpha_{a-h,a-1})-\alpha_{a,b-1}.
\end{equation}
Since $ \pint{\lambda_{u}-\alpha_{a-h,a-1}}{\alpha_{a}}= 1$ we conclude that 
$ P_{b-1,h}(\lambda_{u}-\alpha_{a-h,a-1}) =  \lambda_{u+1}$.
Furthermore, we have $D(P_{b-1,h})(\lambda_{u} - \alpha_{a-h,a-1}) = D-1$. 
Therefore, our inductive hypothesis gives
   \begin{equation}
       \bfM_{\lambda_{u}}^{\succ \alpha_{i,j}} = q\bfM_{\lambda_{u} - \alpha_{a-h,a-1}}^{\succ \alpha_{i,j}} = qq^{D-1}\bfM_{\lambda_{u+1}}^{\succ \alpha_{i,j}}= q^{D}\bfM_{\lambda_{u+1}}^{\succ \alpha_{i,j}}.
   \end{equation}

\medskip
\textbf{Case B. }  $\lambda_{u+1}$ is not defined and $b\equiv i+1 \mod h$.
\medskip

Since $\lambda_{u}$ is defined and  $\lambda_{u+1}$ is not, \Cref{lem: desigualdad de varrhos} implies  $\varrho_{u} = \varrho_{b,h}(\lambda_{u-1}) = 1$. 
In particular, we have $b\equiv 0 \mod h$.
In addition, by definition of the $P$-operator we have $\pint{\lambda_{u}}{\alpha_{t}}=0$ for all $1\leq t \leq b-h $ with $t\equiv b-h \equiv 0 \mod h$. 
In particular, $\pint{\lambda_{u}}{\alpha_{h}}=0$. 

On the other hand, by  \Cref{lem: 0s de PPPPPPP} applied to $k=u$ we obtain 
$\pint{\lambda_{u-1}}{\alpha_{t}}$ for all $t\in [h+1, i] \setminus \{b+1\}$. 
Furthermore, $\lambda_{u-1}\notin X_j^{+}(L)$ implies $\pint{\lambda_{u-1}}{\alpha_{b+1}}=-1$. 
Since $\lambda_u= \lambda_{u-1}- \alpha_{1,b}$ we obtain  $\pint{\mu_u}{\alpha_t}=0$ for all $t\in  [h+1,i]\setminus\{b\}$ and $\pint{\lambda_{u}}{\alpha_{b}} = -1$.

Summing up, we obtain that $\pint{\lambda_u}{\alpha_t}=0$ for all $t\in  [h,i]\setminus\{b\}$ and $\pint{\lambda_{u}}{\alpha_{b}} = -1$.
Then, we can apply  \Cref{lem: free P} to $\alpha_{i-1,j-1}$, $K=\Phi^{\succ \alpha_{i-1,j-1}}$, $p =i+1-h$, $p'=b$ and $\lambda_{u}$ to obtain $\bfM_{\lambda_{u}}^{\succ \alpha_{i-1,j-1}} = 0$. 
We remark that our use of \Cref{lem: free P} is justified since we use $\alpha_{i-1,j-1}$ rather than $\alpha_{i,j}$, so that our $p$ belongs to the admissible range.
By decomposing $\bfM_{\lambda_{u}}^{\succ \alpha_{i-1,j-1}} = 0$ we get
    \begin{equation}\label{eq: blackP proof 1}
        \bfM_{\lambda_{u}}^{\succ \alpha_{i,j}} =q \bfM_{\lambda_{u} - \alpha_{i,j}}^{\succ \alpha_{i,j}}.
    \end{equation}

Since $b\equiv 0 \mod h$, \Cref{lem: blackP caso 0}  applied to $\mu = \lambda_{u}$ yields  $\bfM_{\lambda_{u} - \alpha_{i,j}}^{\succ \alpha_{i,j}} = 0$. 
Therefore, the result follows by replacing this in \eqref{eq: blackP proof 1}.

\medskip
\textbf{Case C.}  $\lambda_{u+1}$ is defined and $b \not \equiv i+1 \mod h$.
\medskip

As in Case A we obtain  $\pint{\lambda_{u}}{\alpha_{t}} = 0$ for all $t\in [\varrho_{u+1}+h,i]\setminus \{b\}$ and  $\pint{\lambda_{u}}{\alpha_{b}} = -1$.
    Let $p$ be  the unique integer such that  $i-h+1< p\leq i$ and $p\equiv b \mod h$.
    Let $d$ be the integer defined by $\varrho_{u+1} = p-(d+1)h$ and define
    \begin{equation}
        T = \bigcup_{m = 0}^{d}[p-mh,i-mh]\subset [\varrho_{u+1} +h,i].
    \end{equation}
   We apply  \Cref{lem: free P} to  $K=\Phi^{\succ \alpha_{i,j}}$, $\alpha_{i,j}$, $T$, $\lambda_{j-b+1}$, $p' = b$ and $p$ to obtain
\begin{equation}
    \bfM_{\lambda_u}^{\succ\alpha_{i,j}} =q^{D(P_{b-1,h})(\lambda_{u})}  \bfM_{P_{b-1,h}(\lambda_{j-b+1})}^{\succ \alpha_{i,j}} =q^{D(P_{b-1,h})(\lambda_{u})}  \bfM_{\lambda_{u+1}}^{\succ \alpha_{i,j}} ,
\end{equation}
as we wanted to show. 

\medskip
\textbf{Case D.} $\lambda_{u+1}$ is not defined and  $b \not \equiv i+1 \mod h$.
\medskip

As in Case B we get $\varrho_{u}=1$,    $\pint{\lambda_u}{\alpha_t}=0$ for all $t\in  [h+1,i]\setminus\{b\}$ and $\pint{\lambda_{u}}{\alpha_{b}} = -1$.   
Let $p$ be the unique integer such that  $i-h+1\leq p\leq i$ and $p\equiv b \mod h$. 
Then,  by \Cref{lem: free P} applied to $\alpha_{i,j}$, $K=\Phi^{\succ \alpha_{i,j}}$,  $\lambda_{u}$, $p' = b$ and $p$ we obtain that $\bfM_{\lambda_{u}}^{\succ\alpha_{i,j}}=0$.
\end{proof}

We are now in position of proving the main result in this section.

\begin{myproof}[Proof of \Cref{lem: bfP}]
For $x \geq 0$, set $\lambda_x = P^x_{j,h}(\lambda)$.
We stress that $\mu=\lambda_h $.

We first consider the case when $\mathbf{P}_{i-1,h}(\mu)$ is defined.
By hypothesis we have $\lambda_{h} \notin X_{j}^+(L)$.
Let $s\geq h+1$ be the integer such that $\lambda_s = \mathbf{P}_{i-1,h}(\mu) $.
Let $h+1\leq u <s$. 
By the minimality of $s$ we have $\lambda_{u} \notin X_{j}^+(L)$.

Then, by  \Cref{lem: first blackP} we have
\begin{equation}
    \bfM_{\lambda_{u}}^{\succ \alpha_{i,j}}
    = q^{D(P_{j-u,h})(\lambda_{u})}
      \bfM_{\lambda_{u+1}}^{\succ \alpha_{i,j}}
\end{equation}
for all $h \leq u <s$. It follows that
\begin{equation}
 \displaystyle   \bfM_{\mu}^{\succ \alpha_{i,j}} =   \bfM_{\lambda_h}^{\succ \alpha_{i,j}}
    = q^{D}
      \bfM_{\lambda_s}^{\succ \alpha_{i,j}}
    = q^{D(\mathbf{P}_{i - 1,h})(\mu)}
      \bfM_{\mathbf{P}_{i - 1,h}(\mu)}^{\succ \alpha_{i,j}},
\end{equation}
where $\displaystyle D=\sum_{u =h}^{s-1} D(P_{j-u,h})(\lambda_{u})$.
This proves the lemma when $\mathbf{P}_{i - 1,h}(\mu)$ is defined.

We now address the case when $\mathbf{P}_{i-1,h}(\mu)$ is not defined.
Let $s\geq h+1$ be the smallest  integer such that $\lambda_s$ is not defined.
Arguing as in the previous paragraph we obtain
\begin{equation}
    \bfM_{\mu}^{\succ \alpha_{i,j}} =   \bfM_{\lambda_h}^{\succ \alpha_{i,j}} = q^E \bfM_{\lambda_{s-1}}^{\succ \alpha_{i,j}}, 
\end{equation}
for some integer $E$. 
Then another application of \Cref{lem: first blackP} yields $\bfM_{\lambda_{s-1}}^{\succ \alpha_{i,j}}=0$, which proves the lemma in this case. 
\end{myproof}

%%%%%%%%%%%%%%%%%%%%%%%%%%%%%%%%%%%%%%%%%%%%%%%%%%%%%%%%%%%%%%%%%%%%
\subsection{Mixed operators}
%%%%%%%%%%%%%%%%%%%%%%%%%%%%%%%%%%%%%%%%%%%%%%%%%%%%%%%%%%%%%%%%%%%%

In this section we establish a relation between the weights 
$P^{h-1}_{j,h}(\lambda)$, 
$\mathbf{v}_{j,h}(\lambda)$,
$P^{h}_{j,h}(\lambda)$ and
$P_{i,h}\mathbf{v}_{j,h}(\lambda)$
(when all or some of them exist).
This is the key to pass from the First Inverse Decomposition to the Second Inverse Decomposition. 

\begin{lemma}\rm\label{lem: ultimo P}
    Let $\lambda\in X^{+}$,  $\alpha_{i,j}\in \Phi^{\geq 2}$ and $h=j-i+1$. 
    Suppose that $\langle \lambda , \alpha_r \rangle  =0$, for all $i\leq r \leq j$.
    and let 
\begin{equation}
    \begin{array}{lll}
      d_{1} =D(P^{h-1}_{j,h})(\lambda)+1   & \mbox{ }\qquad \mbox{ }\qquad    & d_{2}=D(\mathbf{v}_{j,h})(\lambda)+1 \\
      d_{3} = D(P^{h}_{j,h})(\lambda)   & \mbox{ }\qquad  \mbox{ }\qquad    &  d_{4}=D(P_{i,h}\mathbf{v}_{j,h})(\lambda)
    \end{array}
\end{equation}
    \begin{enumerate}[(a)]
        \item Suppose that $P^{h}_{j,h}(\lambda)$ and $\mathbf{v}_{j,h}(\lambda)$ are defined, then\label{item P existe}
    \begin{equation} \label{eq: ultimo P existe}
      q^{d_{1}} \bfM_{P^{h-1}_{j,h}(\lambda) - \alpha_{i-h+1,i}}^{\succ \alpha_{i,j}} - q^{d_{2}} \bfM_{\mathbf{v}_{j,h}(\lambda)-\alpha_{i-h+1,i}}^{\succ \alpha_{i,j}} = q^{d_{3}} \bfM_{P^{h}_{j,h}(\lambda)}^{\succ \alpha_{i,j}} - q^{d_{4}} \bfM_{P_{i,h}\mathbf{v}_{j,h}(\lambda)}^{\succ \alpha_{i,j}}.
    \end{equation}
        \item Suppose that $\mathbf{v}_{j,h}(\lambda)$ is defined, but $P^{h}_{j,h}(\lambda)$ is not, then\label{item P no existe}
    \begin{equation}\label{eq: ultimo P no existe}
        q^{d_{1}} \bfM_{P^{h-1}_{j,h}(\lambda) - \alpha_{i-h+1,i}}^{\succ \alpha_{i,j}} - q^{d_{2}} \bfM_{\mathbf{v}_{j,h}(\lambda)-\alpha_{i-h+1,i}}^{\succ \alpha_{i,j}} = 0.
    \end{equation}
        \item Suppose that $P^{h}_{j,h}(\lambda)$ is defined, but $\mathbf{v}_{j,h}(\lambda)$ is not, then\label{item V no existe}
    \begin{equation}\label{eq: ultimo P  y v no existe}
        \bfM_{P^{h-1}_{j,h}(\lambda) - \alpha_{i-h+1,i}}^{\succ \alpha_{i,j}} = q^{D(P_{i,h})(P_{j,h}^{h-1}(\lambda))} \bfM_{P^{h}_{j,h}(\lambda)}^{\succ \alpha_{i,j}}.
    \end{equation}
        \item Suppose that $P^{h}_{j,h}(\lambda)$ and $\mathbf{v}_{j,h}(\lambda)$are not defined, then\label{item p y v no existen}
    \begin{equation}\label{eq: ultimo P, ambos no existen}
        \bfM_{P^{h-1}_{j,h}(\lambda) - \alpha_{i-h+1,i}}^{\succ \alpha_{i,j}} = 0.
    \end{equation}
    \end{enumerate}
\end{lemma}

\begin{proof}
    We only prove \cref{eq: ultimo P existe}, the other three cases being analogous. 

We  notice that $P_{j,h}^{h}(\lambda) = P_{i,h}(P_{j,h}^{h-1} (\lambda) ) $. 
Let  $\varrho=\varrho_{i,h}(P_{j,h}^{h-1} (\lambda)) $. 
By definition if $d=D(P_{i,h} )(P_{j,h}^{h-1} (\lambda))$ then $\varrho = i -dh+1$. 
Furthermore, we have the 
\begin{equation}\label{eq:degreesA}
    d_1+d=d_3+1.
\end{equation}
We also notice that $\pint{P_{j,h}^{h-1} (\lambda)}{\alpha_t} =     \pint{\mathbf{v}_{j,h}(\lambda)}{\alpha_t}  $  for all $1\leq t \leq i$. 
It follows that $\varrho = \varrho_{i,h} (  \mathbf{v}_{j,h}(\lambda) )$.

On the other hand, let $\varkappa = \varkappa_{j,h}(\lambda) $.
Then, by definition of $\mathbf{v}_{j,h}(\lambda)$ we have 
\begin{equation}\label{eq: relation P con v negrita}
    \mathbf{v}_{j,h}(\lambda) = P^{h-1}_{j,h}(\lambda) -\sum_{u=1}^{h-1+\varkappa}  \alpha_{i+1+u,i+1+u+(h-2)}.
\end{equation}
It follows that 
\begin{equation}\label{eq:degreesB}
    d_2=d_1 + \varkappa +h -1  \qquad \mbox{and} \qquad d_4=d_3 + \varkappa +h -1.
\end{equation}

We prove the following.

\begin{claim}\label{claim complicado}
For all  $\varrho < a \leq i-h+1$ with $a\equiv i+1 \mod h$ we have
\begin{equation}
   \bfM_{P_{j,h}^{h-1}(\lambda) - \alpha_{a,i}}^{\succ \alpha_{i,j}} - q^{\varkappa+h-1} \bfM_{\mathbf{v}_{j,h}(\lambda)-\alpha_{a,i}}^{\succ \alpha_{i,j}}  = q  \left( \bfM_{P_{j,h}^{h-1}(\lambda)-\alpha_{a-h,i}}^{\succ \alpha_{i,j}} - q^{\varkappa+h-1} \bfM_{\mathbf{v}_{j,h}(\lambda)- \alpha_{a-h,i}}^{\succ \alpha_{i,j}} \right).
\end{equation}   
\end{claim}

\begin{proof}
    Let $\mu= P_{j,h}^{h-1}(\lambda) - \alpha_{a,i}$ and $\mu' = \mathbf{v}_{j,h}(\lambda)-\alpha_{a,i}$. 
By combining \Cref{lem: 0s de PPPPPPP}, the hypothesis on $\lambda$ in the lemma, the inequality $\varrho \leq a-h$ and the definition of $\varkappa$ we obtain 
\begin{equation}
    R_t(\mu) = \left\{\begin{array}{rl}
        1, & \text{if } t=a;\\
        0, & \text{if } a+1\leq t < i;\\
        1, & \text{if } t=i;\\
        0, & \text{if } i+1\leq t\leq j+\varkappa, \text{ and } t\neq j+1;\\
        -1, &\text{if } t = j+1\leq j+\varkappa.
    \end{array}\right.
    \qquad  R_t(\mu') = \left\{\begin{array}{rl}
        1, & \text{if } t=a;\\
        0, & \text{if } a+1\leq t < i;\\
        1, & \text{if } t=i;\\
        -1, & \text{if } t= i+1;\\
        0, & \text{if } i+2\leq t\leq j+\varkappa.
    \end{array}\right.
\end{equation}

    Let $Y=\Phi^{\succ\alpha_{i,j}} \cup \{ \alpha_{a-h,a-1} \}$ and $k\geq 0$ defined by the equality $i-h+1-a=kh$.  
    We notice that both $\mu$ and $\mu'$ satisfies the hypothesis of \Cref{lem: fourth string lemma} with $a$ and $k$ as above. Therefore,  using the notation in that lemma, we obtain 
    \begin{equation} \label{eq: alternativa 1}
        \bfM_{\mu}^{Y} = q^{\mathfrak{R}_k}\bfM_{\mu- \mu_k}^{Y_{a,k}}  \qquad \mbox{and} \qquad   \bfM_{\mu'}^{Y} = q^{\mathfrak{R}_k}\bfM_{\mu'- \mu_k}^{Y_{a,k}}.
    \end{equation}
    We stress that both the weight $\mu_k$ and the integer $\mathfrak{R_k}$ depends heavily on some of the components of $\mu$ and $\mu'$. 
    However, in our setting the components of $\mu$ and $\mu'$ involved in the computation of   $\mu_k$ and $\mathfrak{R}_k$ coincide. 
    More precisely, we have 
    \begin{equation}
        \mu_0 = \sum_{t=0}^{h-2} \alpha_{i-t,j-t}, \qquad \mu_{s} = \mu_{s-1} + \sum_{t=0}^{(s+1)h-2}    \alpha_{i-t,j-t}    \qquad \mbox{and} \qquad \mathfrak{R}_{k} =  \frac{k+1}{2}(i-a+h-1) .
    \end{equation}
In particular, we have
\begin{equation}
    \pint{\mu-\mu_k}{\alpha_{i-s}}= \pint{\mu'-\mu_k}{\alpha_{i-s}} = \left\{ \begin{array}{rl}
        -(k+2) , & \mbox{if } s=0;  \\
         0 ,& \mbox{if } 0<s\leq k. 
    \end{array}
    \right.
\end{equation}
    
Then, using  \eqref{eq: alternativa 1}  and \Cref{lem: weyl group over M elements} for the element $s_{i-k}\cdots s_{i-1}s_i $, we obtain 

        \begin{equation} \label{eq: alternativa 2}
        \bfM_{\mu}^{Y} = (-1)^{k+1}q^{\mathfrak{R}_k}\bfM_{\mu- \mu_k  +\gamma_k}^{Y_{a,k}}  \qquad \mbox{and} \qquad   \bfM_{\mu'}^{Y} = (-1)^{k+1}q^{\mathfrak{R}_k}\bfM_{\mu'- \mu_k +\gamma_k}^{Y_{a,k}}, 
    \end{equation}
where $\gamma_{k} =\displaystyle \sum_{s=0}^k (k+1-s) \alpha_{i-s}$. 

To simplify notation we set $ \Upsilon_k \coloneqq \mu- \mu_k  +\gamma_k$ and $\Upsilon_k'\coloneqq\mu'- \mu_k  +\gamma_k$. 
We have
\begin{equation} \label{eq: upsilon }
    \pint{\Upsilon_k}{\alpha_t} = \left\{  
    \begin{array}{rl}
      0,    & \mbox{if } i-k \leq t \leq i; \\
      -(k+2),    & \mbox{if } t=i+1; \\
      0,  & \mbox{if }  i+2\leq  t \leq j+\varkappa \mbox{ and } t\neq j+1;\\
      k+2, & \mbox{if }  t=j+1 \leq j+\varkappa.
    \end{array} \right.
\end{equation}
and 
\begin{equation} \label{eq: upsilon prime}
  \pint{\Upsilon_k'}{\alpha_t} = \left\{  
    \begin{array}{rl}
      0,    & \mbox{if } i-k \leq t \leq i; \\
      -(k+1),    & \mbox{if } t=i+1; \\
      0,  & \mbox{if }  i+2\leq  t \leq j+\varkappa \mbox{ and } t\neq j+1;\\
      k+1, & \mbox{if }  t=j+1 \leq  j+\varkappa.
    \end{array} \right.
\end{equation}

On the other hand, we observe that  since $b_k=i-1$ we have $Y_{a,k}=\Phi^{\succ \alpha_{a-h,a-1}} \setminus A_k$, where $A_k$ is the trapezoid at height $h$ with top-left corner $\alpha_{i-(h+k)+1,i+1}$ and right-bottom corner $\alpha_{i,j}$. 
We emphasize that $h' \coloneq \operatorname{ht} (\alpha_{i-(h+k)+1,i+1}) = h+k+1$. 
So that $h'-h=k+1 $. 

Therefore we are in position to apply \Cref{coro: push pyramid to the right}. 
We obtain
\begin{equation}
    \bfM_{\Upsilon_k}^{Y_{a,k}} =q^{k+2} \bfM_{\Upsilon_k-(k+2)\alpha_{i+2,i+2+(h-2)}}^{\Phi^{\succ \alpha_{a-h,a-1}} \setminus A_{k}(1)} \qquad \mbox{and} \qquad  \bfM_{\Upsilon_k'}^{Y_{a,k}} =q^{k+1} \bfM_{\Upsilon_k'-(k+1)\alpha_{i+2,i+2+(h-2)}}^{\Phi^{\succ \alpha_{a-h,a-1}} \setminus A_{k}(1)},
\end{equation}
where $A_k(1)$ is the trapezoid $A_k$ shifted by one unit to the right. 

We can repeat this process. 
Let us be more precise. 
If we set $\Upsilon_k(1)= \Upsilon_k-(k+2)\alpha_{i+2,i+2+(h-2)}$ and $\Upsilon_k'(1)= \Upsilon_k'-(k+2)\alpha_{i+2,i+2+(h-2)}$ then the conditions in \eqref{eq: upsilon } and \eqref{eq: upsilon prime} are ``shifted by one unit to the right''. 
Concretely, we have
\begin{equation} \label{eq: upsilon A}
    \pint{\Upsilon_k(1)}{\alpha_t} = \left\{  
    \begin{array}{rl}
      0,    & \mbox{if } i+1-k \leq t \leq i+1; \\
      -(k+2),    & \mbox{if } t=i+2; \\
      0,  & \mbox{if }  i+3\leq  t \leq j +\varkappa \mbox{ and } t\neq j+2;\\
      k+2, & \mbox{if }  t=j+2 \leq j+\varkappa.
    \end{array} \right.
\end{equation}
and 
\begin{equation} \label{eq: upsilon prime A}
    \pint{\Upsilon_k'(1)}{\alpha_t} = \left\{  
    \begin{array}{rl}
      0,    & \mbox{if } i+1-k \leq t \leq i+1; \\
      -(k+1),    & \mbox{if } t=i+2; \\
      0,  & \mbox{if }  i+3\leq  t \leq j +\varkappa \mbox{ and } t\neq j+2;\\
      k+1, & \mbox{if }  t=j+2 \leq j+\varkappa.
    \end{array} \right.
\end{equation}
So that we can apply \Cref{coro: push pyramid to the right} again. 
We get
\begin{equation}
\begin{array}{l}
     \bfM_{\Upsilon_k(1)}^{\Phi^{\succ \alpha_{a-h,a-1}} \setminus A_{k}(1)} = q^{k+2} \bfM_{\Upsilon_k(1)-(k+2)\alpha_{i+3,i+3+(h-2)}}^{\Phi^{\succ \alpha_{a-h,a-1}} \setminus A_{k}(2)},   \\[10pt]
  \bfM_{\Upsilon_k'(1)}^{\Phi^{\succ \alpha_{a-h,a-1}} \setminus A_{k}(1)} =q^{k+1} \bfM_{\Upsilon_k'(1)-(k+1)\alpha_{i+3,i+3+(h-2)}}^{\Phi^{\succ \alpha_{a-h,a-1}} \setminus A_{k}(2)},
\end{array}
\end{equation}
where $A_k(2)$ is the trapezoid $A_k$ shifted by two units to the right. 

By continuing this way, we eventually obtain 
\begin{equation}
    \begin{array}{l}
         \bfM_{\mu}^Y =(-1)^{k+1}q^{\mathfrak{R}_k}q^{(k+2)(h-1+\varkappa)} \bfM_{\Upsilon_k-\zeta_k}^{\Phi^{ \succ \alpha_{a-h,a-1}  } \setminus A_k(h-1+\varkappa) }, \\[10pt]
         \bfM_{\mu'}^Y =(-1)^{k+1}q^{\mathfrak{R}_k} q^{(k+1)(h-1+\varkappa)}\bfM_{\Upsilon_k'-\zeta_k'}^{\Phi^{ \succ \alpha_{a-h,a-1}  } \setminus A_k(h-1+\varkappa) },
    \end{array}
\end{equation}
where $A_k(h-1+\varkappa)$ is the trapezoid $A_k$ shifted by $h-1+\varkappa$ units to the right,  
\begin{equation}
    \zeta_k=  (k+2) \sum_{u=1}^{h-1+\varkappa}  \alpha_{i+1+u,i+1+u+(h-2)}  \qquad \mbox{and} \qquad   \zeta_k'=  (k+1) \sum_{u=1}^{h-1+\varkappa}  \alpha_{i+1+u,i+1+u+(h-2)} .
\end{equation}
Equation  \eqref{eq: relation P con v negrita} yields
\begin{equation}
\begin{array}{rl}
   ( \Upsilon_k-\zeta_k) -(\Upsilon_k'-\zeta_k')   & \displaystyle = \mu -\mu' - \sum_{u=1}^{h-1+\varkappa}  \alpha_{i+1+u,i+1+u+(h-2)} \\[10pt]
     &\displaystyle =  P_{j,h}^{h-1}(\lambda) -\mathbf{v}_{j,h}(\lambda) - \sum_{u=1}^{h-1+\varkappa}  \alpha_{i+1+u,i+1+u+(h-2)} \\[10pt]
     & =0.
\end{array}
\end{equation}
Thus, $\bfM_{\mu}^Y= q^{h-1+\varkappa} \bfM_{\mu'}^Y$.
Finally, by the  definition of $Y=\Phi^{\succ \alpha_{i,j}} \cup \{  \alpha_{a-h,a-1} \} $ and the equality $\alpha_{a-h,i}=\alpha_{a-h,a-1}+\alpha_{a,i}  $, we get
\begin{equation}
       \bfM_{P_{j,h}^{h-1}(\lambda) - \alpha_{a,i}}^{\succ \alpha_{i,j}} - q\bfM_{P_{j,h}^{h-1}(\lambda)-\alpha_{a-h,i}}^{\succ \alpha_{i,j}} =
       q^{h-1 + \varkappa } \left( \bfM_{\mathbf{v}_{j,h}(\lambda)-\alpha_{a,i}}^{\succ \alpha_{i,j}}  -q \bfM_{\mathbf{v}_{j,h}(\lambda)- \alpha_{a-h,i}}^{\succ \alpha_{i,j}} \right),
\end{equation}
which is equivalent to our claim.
\end{proof}

We recall that
\begin{equation}
    \varrho = \varrho_{i,h}(P_{j,h}^{h-1} (\lambda)) = \varrho_{i,h} (  \mathbf{v}_{j,h}(\lambda) ) = i -dh+1.
\end{equation}
Therefore, a repeated application of \Cref{claim complicado} starting with $a=i-h+1$ yields
\begin{equation} \label{eq: consequence claim}
\begin{array}{l}
  \qquad\quad\quad \bfM_{P_{j,h}^{h-1}(\lambda) - \alpha_{i-h+1,i}}^{\succ \alpha_{i,j}} - q^{\varkappa+h-1} \bfM_{\mathbf{v}_{j,h}(\lambda)-\alpha_{i-h+1,i}}^{\succ \alpha_{i,j}}  
     \\ = q^{d-1}  \left( \bfM_{P_{j,h}^{h-1}(\lambda)-\alpha_{i-dh+1,i}}^{\succ \alpha_{i,j}} - q^{\varkappa+h-1} \bfM_{\mathbf{v}_{j,h}(\lambda)- \alpha_{i-dh+1,i}}^{\succ \alpha_{i,j}} \right) \\
     = q^{d-1}  \left( \bfM_{P_{j,h}^{h}(\lambda)}^{\succ \alpha_{i,j}} - q^{\varkappa+h-1} \bfM_{P_{i,h}\mathbf{v}_{j,h}(\lambda)}^{\succ \alpha_{i,j}} \right).
\end{array}
\end{equation}  
Thus, \eqref{eq: ultimo P existe} follows by multiplying \eqref{eq: consequence claim} by $q^{d_1}$ and using the identities in \eqref{eq:degreesA} and \eqref{eq:degreesB}.
\end{proof}

\subsection{The second inverse decomposition}

We now turn to the proof of the \emph{Second Inverse Decomposition (SID)}, the central result of this section and the key to prove the positivity of the pre-canonical bases in the next section.

Before stating the result,  let us explain why the SID is needed in the present situation.
If we apply the First Inverse Decomposition in \Cref{prop: second version} to a weight $\lambda$, we can find that some of the weights indexing the $\bfM$-elements in the decomposition may fail to be dominant (disregarding the subtracted roots).
More precisely, the problematic terms are $P_{j,h}^{h-1}(\lambda)$ and $v_{j,h}^{h-1}(\lambda)$.
The SID addresses this issue by replacing these non-dominant weights with dominant ones.
This ensures that all weights involved remain dominant, which will allows us to use an inductive argument  in the proof of \Cref{Conj A}.

\begin{proposition}\label{prop: second version refinado}
Let $\lambda\in X^{+}$, $\alpha_{i,j}\in \Phi^{\geq 2}$ and $h=j-i+1$.  
Define $\beta_m=\alpha_{i-m,j-m}$ and 
\begin{equation}
    k=k_{j,h}(\lambda) \coloneq \min\{ 0 \leq s \leq h-2 \mid \pint{\lambda}{\alpha_{j-s}}>0\} .
\end{equation}

Then, we have
\begin{equation}\label{eq: SID ecuacion combi}
 \bfM_{\lambda}^{\succeq \alpha_{i,j}} = \left\{ \begin{array}{ll}
 \displaystyle    \bfM_{\lambda}^{\succ \alpha_{i,j}}
     - q^{d}\,\bfM_{\mathbf{P}_{j,h}(\lambda)}^{\succ \alpha_{i,j}}
     - \sum_{m=1}^{k}   q^{d_m} X_m , 
       & \mbox{if } k \mbox{ exists;}  \\
 \displaystyle   \bfM_{\lambda}^{\succ \alpha_{i,j}}
     - q^{d}\,\bfM_{\mathbf{P}_{j,h}(\lambda)}^{\succ \alpha_{i,j}}
     - \sum_{m=1}^{h-2}   q^{d_m} X_m
     - q^{e_1} \left(
        \bfM_{\mathbf{v}_{j,h}(\lambda)}^{\succ \alpha_{i,j}}
          - q^{e_2}\,  \bfM_{\mathbf{P}_{i,h}\mathbf{v}_{j,h}(\lambda)}^{\succ \alpha_{i,j}}
       \right) ,    & \mbox{otherwise;}
  \end{array}   \right.  
\end{equation}
where 
$d = D(\mathbf{P}_{j,h})(\lambda)$, 
$d_m = D(v_{j,h}^m)(\lambda)$, 
$e_1 = D(\mathbf{v}_{j,h})(\lambda)$, $ e_2 = D(\mathbf{P}_{i,h})(\mathbf{v}_{j,h}(\lambda))$ and 
\begin{equation} \label{eq: defi X in SID}
  X_m  =  \bfM_{v^{m}_{j,h}(\lambda)}^{\succ \alpha_{i,j}}
         -   q\,\bfM_{v^{m}_{j,h}(\lambda)-\beta_m}^{\succ \alpha_{i,j}}.   
\end{equation}
\end{proposition}

\begin{remark}\rm\label{remark operator not defined refinada}  
      We recall our convention that if some $\mathbf{P}$-operator, $v^m$-operator or $\mathbf{v}$-operator is not defined, then the corresponding $\bfM$-element is set equal to zero. 
\end{remark}

\begin{proof}
We divide the proof into two cases according to whether $k$ exists.

\textbf{Case I. } The integer $k$ exists. 

Applying \Cref{prop: second version} to $\lambda$, $\alpha_{i,j}$ and $k$, we obtain:
\begin{equation}\label{eq: demo first refined version 1}
    \!  \bfM_{\lambda}^{\succeq \alpha_{i,j}}\! \!  = \! \bfM_{\lambda}^{\succ \alpha_{i,j}} \!  - q^{D(P^{k}_{j,h})(\lambda)+1} \bfM_{P^{k}_{j,h}(\lambda) - \beta_{k}}^{\succ \alpha_{i,j}}\! \!  -  \displaystyle\sum_{m=1}^{k} q^{D(v^{m}_{j,h})(\lambda)} \! \left( \bfM_{v^{m}_{j,h}(\lambda)}^{\succ \alpha_{i,j}}\!  -\!  q\bfM_{v^{m}_{j,h}(\lambda)-\beta_m}^{\succ \alpha_{i,j}}\right).
\end{equation}

Comparing this with the target equality  in \eqref{eq: SID ecuacion combi}, it is enough to show the following identity:
\begin{equation}\label{eq: demo first refined version 2}
    \bfM_{P^{k}_{j,h}(\lambda) - \beta_{k}}^{\succ \alpha_{i,j}}  =  \begin{cases}
        q^{D(P_{j-k,h})(P^{k}_{j,h}(\lambda))-1}\bfM_{\mathbf{P}_{j,h}(\lambda)}^{\succ \alpha_{i,j}},& \text{if }\mathbf{P}_{j,h}(\lambda) \text{ is defined};\\
        0,&\text{otherwise}.
    \end{cases}
\end{equation}

We proceed by analyzing these two subcases.

\medskip
\noindent
\textbf{Case IA.} The weight $\mathbf{P}_{j,h}(\lambda)$ is defined.

We first notice that 
\begin{equation}\label{eq: Pk+1 = P negrita}
    \pint{P^{k+1}_{j,h}(\lambda)}{\alpha_{j-k}} = \pint{\lambda}{\alpha_{j-k}} - \sum_{m=j-k}^{j}\pint{\alpha_{\varrho_{m},m}}{\alpha_{j-k}} = \pint{\lambda}{\alpha_{j-k}} - 1 \geq 0,
\end{equation}
where $\varrho_{m} = \varrho_{m,h}(P^{j-m}_{j,h}(\lambda))$. The final equality follows from \Cref{remark producto interno} and the definition of the integer $k = k_{j,h}(\lambda)$. 
From \eqref{eq: Pk+1 = P negrita}, we conclude that $\mathbf{P}_{j,h}(\lambda) = P^{k+1}_{j,h}(\lambda)$.

If $\pint{\lambda}{\alpha_{i-k}} \geq 1$, then by the definition of the $P$-operator, $P^{k+1}_{j,h}(\lambda)= P^{k}_{j,h}(\lambda) - \beta_{k}$.
The result follows via substituting this directly into \eqref{eq: demo first refined version 2} and noticing that $D(P_{j-k,h})(P^{k}_{j,h}(\lambda)) = 1$.

We now prove the equality under the assumption $\pint{\lambda}{\alpha_{i-k}} = 0$. 
By applying \Cref{lem: ceros por Ps} to $\lambda$, $\alpha_{i,j}$, and $k+1$, we get
\begin{equation}\label{eq: demo first refined version 3}
    \pint{P^{k}_{j,h}(\lambda)}{\alpha_{t}} = 0 \quad \mbox{for all}\; t\in T=\bigcup_{m=0}^{d} [i-k-mh,i-mh],
\end{equation}
 where $d$ is the unique integer satisfying $\varrho_{j-k,h}(P^{k}_{j,h}(\lambda)) = i-k-(d+1)h$.
From this, it follows that
\begin{equation} \label{eq: demo first refined version 4}
    \pint{P^{k}_{j,h}(\lambda) - \beta_{k}}{\alpha_{t}}  = \begin{cases}
            0, & \mbox{if } t \in T \setminus \{i-k\};\\
            -1, & \mbox{if } t=i-k.
    \end{cases}
\end{equation}
We can now apply \Cref{lem: free P} to $K=\Phi^{\succ \alpha_{i,j}}$, $\alpha_{i,j}$, $p=p'=i-k$, the set $T$, and $P^{k}_{j,h}(\lambda)-\beta_{k}$.
We obtain
\begin{equation}\label{eq: demo first refined version 5}
    \bfM_{P^{k}_{j,h}(\lambda) - \beta_{k}}^{\succ \alpha_{i,j}} = q^{D(P_{i-k-1,h})(P^{k}_{j,h}(\lambda) - \beta_{k})}\bfM_{P_{i-k-1,h}(P^{k}_{j,h}(\lambda)-\beta_{k})}^{\succ \alpha_{i,j}}
    = q^{\D{P,j-k,h}(P^{k}_{j,h}(\lambda))-1}\bfM_{P^{k+1}_{j,h}(\lambda)}^{\succ \alpha_{i,j}},
\end{equation}
where the second equality holds because
\begin{equation}\label{eq: demo first refined version 6}
\begin{array}{lcl}
    P_{i-k-1,h}(P^{k}_{j,h}(\lambda)-\beta_{k}) &=& P_{j-k,h}(P^{k}_{j,h}(\lambda)) = P^{k+1}_{j,h}(\lambda) \\[10pt]
    D(P_{i-k-1,h})(P^{k}_{j,h}(\lambda) - \beta_{k})+1 &=& \D{P,j-k,h}(P^{k}_{j,h}(\lambda)).
\end{array}
\end{equation}
This establishes the first case of \eqref{eq: demo first refined version 2}.

\medskip
\noindent
\textbf{Case IB.  } The weight $\mathbf{P}_{j,h}(\lambda)$ is not defined.

We must show that $\bfM_{P^{k}_{j,h}(\lambda) - \beta_{k}}^{\succ \alpha_{i,j}} = 0$. 

First, if $P^{k}_{j,h}(\lambda)$ is not defined, then by convention  $\bfM_{P^{k}_{j,h}(\lambda) - \beta_{k}}$  vanishes and we are done.  

Second, assume $P^{k}_{j,h}(\lambda)$ is defined. 
Since $\mathbf{P}_{j,h}(\lambda)$ is not defined, the definition of $k$ forces  
that $P^{k+1}_{j,h}(\lambda)$ is not defined.
Applying \Cref{lem: ceros por Ps} to $\lambda$, $\alpha_{i,j}$, and $k$ gives $\pint{P^{k-1}_{j,h}(\lambda)}{\alpha_{t}} = 0$ for all $t\in T$, where
\begin{equation}
    T=\bigcup_{m=0}^{d}[i-k+1-mh,i-mh],
\end{equation}
and $d$ is the unique integer satisfying $\varrho_{k} = \varrho_{j-k+1,h}(P^{k-1}_{j,h}(\lambda)) = i-k+1-(d+1)h$.
Since $P^{k+1}_{j,h}(\lambda)$ is not defined, \Cref{lem: desigualdad de varrhos} implies $\varrho_{k} = 1$. In other words, $P^{k}_{j,h}(\lambda) = P^{k-1}_{j,h}(\lambda) -\alpha_{1,j-k+1}$.
By \Cref{remark producto interno}, we conclude that $\pint{P^{k}_{j,h}(\lambda)}{\alpha_{t}}=0$ for all $t\in T$.

Furthermore, the condition that $P^{k+1}_{j,h}(\lambda)$ is not defined also implies, by the definition of the $P$-operator, that $\pint{P^{k}_{j,h}(\lambda)}{\alpha_{t}} =0$ for all $t\in \{i-k-mh\mid 0\leq m\leq d\}$.
Combining these zero-conditions, we conclude that $\pint{P^{k}_{j,h}(\lambda)}{\alpha_{t}} =0$ for all $t\in T'$, where
\begin{equation}
    T'=\bigcup_{m=0}^{d} [i-k-mh,i-mh].
\end{equation}
Using \Cref{remark producto interno} again, we obtain
\begin{equation}
    \pint{P^{k}_{j,h}(\lambda) - \beta_{k}}{\alpha_{t}} = \begin{cases}
            0,&\text{if } t\in T'\setminus \{i-k\};\\
            -1,&\text{if } t = i-k.
    \end{cases}
\end{equation}
As in the first case, we apply \Cref{lem: free P} to $K=\Phi^{\succ \alpha_{i,j}}$, $\alpha_{i,j}$, $p=p'=i-k$, the set $T'$, and $P^{k}_{j,h}(\lambda)-\beta_{k}$, to obtain $\bfM_{P^{k}_{j,h}(\lambda)-\beta_{k}}^{\succ \alpha_{i,j}}=0$, as we wanted to show. 

This completes the proof of \textbf{Case I}.

\bigskip
\textbf{Case II. } The integer $k$ does not exist. 

We notice that the hypothesis in this case implies that   $\pint{\lambda}{\alpha_{j-s}}=0$ for all $0\leq s \leq h-2$.

We only prove the case when all the weights involved in \eqref{eq: SID ecuacion combi} exist, the other cases follow in a similar fashion. 

By \Cref{prop: second version} applied to $ h-1$ we have
\begin{equation} \label{eq:loquedaFID}
     \bfM_{\lambda}^{\succeq \alpha_{i,j}} = \bfM_{\lambda}^{\succ \alpha_{i,j}} - q^{D(P^{h-1}_{j,h})(\lambda)+1} \bfM_{P^{h-1}_{j,h}(\lambda) - \beta_{h-1}}^{\succ \alpha_{i,j}} - \displaystyle\sum_{m=1}^{h-1} q^{d_m} \left( \bfM_{v^{m}_{j,h}(\lambda)}^{\succ \alpha_{i,j}} - q\bfM_{v^{m}_{j,h}(\lambda)-\beta_m}^{\succ \alpha_{i,j}}\right).
\end{equation}
By comparing \eqref{eq: SID ecuacion combi} and \eqref{eq:loquedaFID} it is easy to note that both formulas differ by three terms. 
We work with these terms. 
We have

\begin{equation}\label{eq: pasar de FID to SID}
\begin{array}{l}
      \quad   q^{D(P^{h-1}_{j,h})(\lambda)+1} \bfM_{P^{h-1}_{j,h}(\lambda) - \beta_{h-1}}^{\succ \alpha_{i,j}}
     + q^{d_{h-1}} \left( \bfM_{v^{m}_{j,h}(\lambda)}^{\succ \alpha_{i,j}} - q\bfM_{v^{m}_{j,h}(\lambda)-\beta_{h-1}}^{\succ \alpha_{i,j}} \right)    \\[15pt]
     = q^{D(P^{h-1}_{j,h})(\lambda)+1} \bfM_{P^{h-1}_{j,h}(\lambda) - \beta_{h-1}}^{\succ \alpha_{i,j}}
     + q^{e_1} \left( \bfM_{\mathbf{v}_{j,h}(\lambda)}^{\succ \alpha_{i,j}} - q\bfM_{\mathbf{v}_{j,h}(\lambda)-\beta_{h-1}}^{\succ \alpha_{i,j}} \right)   \\[15pt]
     = q^{D(P^{h}_{j,h})(\lambda)} \bfM_{P^{h}_{j,h}(\lambda)}^{\succ \alpha_{i,j}}
     + q^{e_1} \bfM_{\mathbf{v}_{j,h}(\lambda)}^{\succ \alpha_{i,j}} - q^{D(P_{i,h}\mathbf{v}_{j,h})(\lambda)}\bfM_{P_{i,h}\mathbf{v}_{j,h}(\lambda)}^{\succ \alpha_{i,j}}     \\[15pt]
     = q^{d} \bfM_{\mathbf{P}_{j,h}(\lambda)}^{\succ \alpha_{i,j}}
     + q^{e_1} \bfM_{\mathbf{v}_{j,h}(\lambda)}^{\succ \alpha_{i,j}} - q^{e_1+e_2}\bfM_{\mathbf{P}_{i,h}\mathbf{v}_{j,h}(\lambda)}^{\succ \alpha_{i,j}}  . 
\end{array}
\end{equation}

Let us briefly indicate the origin of these equalities.
The first equality is a consequence of \Cref{lem: completar Q en caso h-1}.
The second equality follows from \Cref{lem: ultimo P}.
The last equality is implied by \Cref{lem: bfP}.
Finally, let us note that the equalities appearing in the exponents come directly from the definition of the operators involved and the additivity in the  corresponding degrees.

Then, \eqref{eq: SID ecuacion combi} is obtained by replacing \eqref{eq: pasar de FID to SID} in \eqref{eq:loquedaFID}.
\end{proof}

\section{Positivity}\label{section: positivity}

In this section we prove \Cref{Teo D} from the introduction, and therefore verify \Cref{Conj A}.
More concretely, we prove that for any $\alpha_{i,j}\in \Phi^{\geq 2}$ and any $\lambda \in X^+$, the element $\bfM_{\lambda }^{\succ \alpha_{i,j}}$ can be written as a positive linear combination of elements of the form $\bfM_{\mu}^{\succeq \alpha_{i-k,j-k}}$ for $k\geq 0$, which is the content of \Cref{Teo D}.
From this, \Cref{Conj A} follows easily.

Let us briefly explain our strategy.
Our starting point is the Second Inverse Decomposition in \Cref{prop: second version refinado}.
We begin by isolating the term $\bfM_{\lambda}^{\succ \alpha_{i,j}}$ in \eqref{eq: SID ecuacion combi}.
We then have three cases to consider:
\begin{itemize}
\item The term $\bfM_{\mathbf{P}_{j,h}(\lambda)}^{\succ \alpha_{i,j}}$.

Since $\mathbf{P}_{j,h}(\lambda)< \lambda$ in the dominance order, we may assume by induction that \Cref{Teo D} holds for this weight.

\item The terms $X_m$ (defined in \eqref{eq: defi X in SID}). 

We refer to these terms as good pairs. 
We study these pairs in \Cref{section: good pairs}.
In particular, in \Cref{prop: caso bueno} we provide a positive combinatorial expansion of a term $X_m$ in terms of $\bfM_{\mu}^{\succeq \alpha_{i-m,j-m}}$-elements. 

\item The difference
\[
\bfM_{\mathbf{v}_{j,h}(\lambda)}^{\succ \alpha_{i,j}}
      - q^{e_2}\,  \bfM_{\mathbf{P}_{i,h}\mathbf{v}_{j,h}(\lambda)}^{\succ \alpha_{i,j}}.
\]

We observe that both weights involved in this difference are dominant. Therefore, we can apply \Cref{prop: second version refinado} again to these weights, and using an inductive argument, we conclude that this difference can also be written positively in terms of $\bfM$-elements.      
\end{itemize}

\subsection{Good Pairs} \label{section: good pairs}
In this section we obtain a positive combinatorial formula for the terms $X_m$ occurring in   \Cref{prop: second version refinado}. 
Our first step is the following technical lemma that will serve as the main tool in the proof of the aforementioned formula.

\begin{lemma}\label{lem: lamda - alpha  a v}
    Let  $\alpha_{i,j}\in \Phi^{\geq 3}$  and $h=j-i+1$. Let $1\leq k \leq h-2$ and $Y=\Phi^{\succ \alpha_{i,j} } \cup \{ \alpha_{i-k,j-k} \}$.
    Let  $B\subsetneq [\alpha_{i-(k-1),j-(k-1)},\alpha_{i,j}]$.
    We define
    \begin{equation} \label{eq: defin a ,b , mu}
    \begin{array}{lll}
       a  & = & \min   \{  j-(k-1)\leq u \leq j  \mid \alpha_{u-(h-1),u} \notin B \}. \\[5pt]
       b  & = & \min \{  a<  u \leq j+1  \mid  \alpha_{u-(h-1), u} \in Y\cup B \}.\\[3pt]
       \mu& = & \displaystyle\sum_{m=a}^{b-1} \alpha_{m-h+1,m}
    \end{array}
    \end{equation}
    
    Let $\lambda \in X^{+}$ be such that  $\pint{\lambda}{\alpha_{t}} =0$ for all $t\in [a,b-1]$.
    Then, we have
    \begin{equation}\label{eq: alpha a v}
        \bfM_{\lambda - \mu}^{Y\cup B} = \left\{\begin{array}{ll}
            q^{D(v^{b-a}_{b-1,h})(\lambda) - (b-a)}\bfM_{v^{b-a}_{b-1,h}(\lambda)}^{Y\cup B}, & \text{if}\; v^{b-a}_{b-1,h}(\lambda) \;\text{is defined}; \\[5pt]
            0, & \text{otherwise.}
        \end{array}\right.
    \end{equation}
\end{lemma}

\begin{proof}
We begin with the following.
\begin{claim}\label{claim: alpha a v}
        Let $S = [a,n] \setminus [b,j+1]$ and $J = [1,b-h] \setminus [i-h+1,a-h+1]$. We have
\begin{enumerate}
        \item \label{item a claim alpha a v} $(Y\cup B)\setminus \{\alpha_{a-h,a-1}\}\sim_{S} \Phi^{\succ \alpha_{b-h,b-1}}$. 
        \item \label{item b claim alpha a v} $(Y\cup B)\sim_{J}\Phi^{\succ \alpha_{b-h,b-1}}$.
\end{enumerate}
    \end{claim}
\begin{proof}
    This follows by a direct computation that involves \eqref{eq: las que se salen}.
    We omit the details.
\end{proof}

Suppose that $v_{b-1,h}^{b-a}(\lambda)$ is defined. 
We claim the following.
\begin{claim}\label{claim: pass tov stop in P}
    For all $1\leq u \leq b-a$ we have
    \begin{equation} \label{eq:  pass to v stop in P}
        \bfM^{Y\cup B}_{\lambda -\mu} = q^{D(P_{b-1,h}^u)(\lambda) -u}\bfM_{P_{b-1,h}^u(\lambda)-\mu_u}^{Y\cup B},
    \end{equation}
    where $\displaystyle \mu_u=\sum_{m=a}^{b-u-1}\alpha_{m-h+1,m}$. 
\end{claim}
\begin{proof}
    We prove the claim by induction on $u$.
    Suppose that $u=1$. 
    If $\pint{\lambda}{\alpha_{b-h}}>0$ then $P_{b-1,h}^1(\lambda)= \lambda -\alpha_{b-h,b-1}$. 
    Therefore,   $P_{b-1,h}^1(\lambda) - \mu_1=\lambda -\mu $ and $D(P_{b-1,h}^{1})(\lambda)=1$.
    Thus, the claim follows. 

    By the previous paragraph we can assume that $\pint{\lambda }{\alpha_{b-h}}=0$. 
    Let $T_{1} = \bigcup_{m=0}^{d} \{ b-(m+1)h \} $, where $d$ is the unique integer defined by $\varrho_{1} = \varrho_{b-h-1,h}(\lambda -\mu) = b-h- (d+1)h$.

  Given that $T_{1} \subset J = [1,b-h] \setminus [i-h+1,a-h+1]$, \Cref{claim: alpha a v} leads to $(Y\cup B)\sim_{T_{1}}\Phi^{\succ \alpha_{b-h,b-1}}$.
  Thus we can apply \Cref{lem: free P} to $K=Y\cup B$ and $p=p'=b-h$ to conclude
\begin{equation}
    \bfM_{\lambda -\mu}^{Y\cup B} = q^{D(P_{b-h-1,h})(\lambda - \mu)} \bfM_{P_{b-h-1,h}(\lambda - \mu)}^{Y\cup B}  .
\end{equation}

Therefore, the claim follows by  
  $D(P_{b-h-1,h})(\lambda - \mu)= D(P_{b-1,h})(\lambda ) -1$ and $P_{b-h-1,h}(\lambda - \mu ) = P_{b-1,h}(\lambda) -\mu_1$. 

  \medskip
  We now suppose that  \Cref{claim: pass tov stop in P} holds for some $1\leq u < b-a$. If $\pint{P_{b-1,h}^u(\lambda)}{\alpha_{b-h-u}}>0$ then we have 
  \begin{equation}
     P^{u+1}_{b-1,h}(\lambda) =P_{b-u,h}(P_{b-1,h}^u(\lambda))= P_{b-1,h}^u(\lambda) -\alpha_{{b-h-u,b-u-1}} . 
  \end{equation}

It follows that $ P^{u+1}_{b-1,h}(\lambda)-\mu_{u+1} =P_{b-1,h}^u(\lambda)  - \mu_{u}  $ and $D(P^{u+1}_{b-1,h})(\lambda) = D(P^{u}_{b-1,h})(\lambda) +1$. 
Thus, by replacing these expressions in \cref{eq:  pass to v stop in P} we obtain the claim for $u+1$ as we wanted.

 By the previous paragraph we can assume that $\pint{P_{b-1,h}^u(\lambda)}{\alpha_{b-h-u}}=0$. Let  
 \begin{equation}
     T_{u} = \bigcup_{m=0}^{d} [b-u-(m+1)h,b-(m+1)h],
 \end{equation}
 where $d$ is the unique integer defined by $\varrho_{u} = \varrho_{b-u-1,h}(  P_{b-1,h}^u(\lambda)- \mu_u) = b-u-h - (d+1)h$.

A direct computation and \Cref{lem: ceros por Ps} yield $\pint{P_{b-1,h}^u(\lambda)-\mu_{u}}{\alpha_t}=0$ for all $t\in  T_r\setminus \{b-h-u\}$ and $\pint{P_{b-1,h}^u(\lambda)-\mu_{u}}{\alpha_{b-h-u}}=-1$. 
Given that $T_{u} \subset J = [1,b-h] \setminus [i-h+1,a-h+1]$, \Cref{claim: alpha a v} leads to $(Y\cup B)\sim_{T_{1}}\Phi^{\succ \alpha_{b-h,b-1}}$.
  Thus we can apply \Cref{lem: free P} to $K=Y\cup B$ and $p=p'=b-h-u$ to conclude
\begin{equation}\label{eq: LALALA1}
    \bfM_{ P_{b-1,h}^u(\lambda)-\mu_{u}   }^{Y\cup B} = 
    q^{D(P_{b-h-u-1,h})(  P_{b-1,h}^u(\lambda)-\mu_{u}     )} 
    \bfM_{ P_{b-h-u-1,h}(P_{b-1,h}^u(\lambda)-\mu_{u}  ) }^{Y\cup B}  .
\end{equation}
We have 
\begin{equation}\label{eq: LALALA2}
    P_{b-h-u-1,h}(P_{b-1,h}^u(\lambda)-\mu_{u}  ) = P_{b-u-1}(P_{b-1,h}^u(\lambda))-\mu_{u+1} =P^{u+1}_{b-1,h}(\lambda) -\mu_{u+1}
\end{equation}
 and 
 \begin{equation}\label{eq: LALALA3}
    D(P_{b-h-u-1,h})(  P_{b-1,h}^u(\lambda)-\mu_{u}     ) =  
      D(P_{b-h-u-1,h})(  P_{b-1,h}^u(\lambda)    ) -1.
 \end{equation}
By replacing \cref{eq: LALALA2} and \cref{eq: LALALA3} into \cref{eq: LALALA1}, and using our inductive hypothesis we obtain
\cref{eq:  pass to v stop in P} for $u+1$. 
This finishes the proof of our claim. 
\end{proof}

By applying \Cref{claim: pass tov stop in P} to $u=b-a$ we get
\begin{equation} \label{eq: LALALA4}
    \bfM_{\lambda - \mu}^{Y\cup B} = q^{D(P_{b-1,h}^{b-a})(\lambda)- (b-a)} \bfM_{P_{b-1,h}^{b-a} (\lambda) }^{Y\cup B}.
\end{equation}

Using a similar argument (replacing the role of \Cref{lem: ceros por Ps} and  \Cref{lem: free P}  by \Cref{lem: ceros por Qs} and \Cref{lem: free Q1}, respectively) we obtain 
\begin{equation} \label{eq: LALALA5}
     \bfM_{P_{b-1,h}^{b-a} (\lambda) }^{Y\cup B} = q^{D(Q^{a-b}_{b-1,h})(P_{b-1,h}^{b-a} (\lambda) ) }   \bfM_{ Q^{a-b}_{b-1,h}(P_{b-1,h}^{b-a}(\lambda))  }^{Y\cup B} = q^{D(Q^{a-b}_{b-1,h})(P_{b-1,h}^{b-a} (\lambda) ) } \bfM_{v^{b-a}_{b-1,h}}(\lambda). 
\end{equation}
 Finally, by combining \cref{eq: LALALA4} and \cref{eq: LALALA5} we obtain \cref{eq: alpha a v} when $v^{b-a}_{b-1,h}(\lambda)$ is defined. 

 If $v^{b-a}_{b-1,h}(\lambda)$ is not defined we repeat the same argument. 
 In this case at least one of the $P$ or $Q$ operators involved in $v^{b-a}_{b-1,h}(\lambda)$ must be undefined. 
 At this point, the corresponding application of either 
 \Cref{lem: free P} or \Cref{lem: free Q1} will give $\bfM_{\lambda -\mu }^{Y\cup B}=0$, as required. 
\end{proof}

\begin{example}\rm  \label{ex: como funciona}
    We illustrate how \Cref{lem: lamda - alpha  a v} can be used to obtain a positive expansion of the $X_m$-terms in \eqref{eq: defi X in SID} as linear combinations of $\bfM^{\succeq \alpha_{i-m,j-m}}_{\nu}$-elements.

    Let $n=12$ and $\lambda =[0,0,1,2,1,2,0,0,0,5,2,5]_{\varpi}$. 
    We work with the positive root $\alpha_{4,8}$. 
    Notice that in this case $h=5$. 
    Applying \Cref{prop: second version refinado}, we obtain
    \begin{equation} \label{eq: expansion ejemplo}
        \bfM_{\lambda}^{\succ \alpha_{4,8}} = \bfM_{\lambda}^{\succeq \alpha_{4,8}}+  q^3\bfM_{\mathbf{P}_{8,3}(\lambda)}^{\succ \alpha_{4,8}} + q^2X_{1} + q^4X_{2} + q^6X_{3},
    \end{equation}
  
    where 
    \begin{equation}
    X_{m} = \bfM_{v^{m}_{8,5}(\lambda)}^{\succ \alpha_{4,8}} - \bfM_{v^{m}_{8,5}(\lambda) - \alpha_{4-m,8-m}}^{\succ \alpha_{4,8}}.
    \end{equation}
    We observe that, since the integer $k = k_{8,5}(\lambda) = 3$ exists, the term $$\bfM_{\mathbf{v}_{8,5}(\lambda)}^{\succ \alpha_{4,8}} - q^{D(\mathbf{P},5,4)(\mathbf{v}_{8,5}(\lambda))}\bfM_{\mathbf{P}_{5,5}\mathbf{v}_{8,5}(\lambda)}^{\succ \alpha_{4,8}}$$ does not occur in the sum \eqref{eq: expansion ejemplo}, as stated in \Cref{prop: second version refinado}.

    Let $\mu = v_{8,5}^3(\lambda) = [1, 0, 1, 1, 1, 2, 0, 0, 0, 4, 2, 5]_{\varpi}$.  
    We will focus on showing that the difference
\begin{equation}
  X_3=  \bfM_{\mu}^{\succ \alpha_{4,8}} -q\bfM_{\mu - \alpha_{1,5}}^{\succ \alpha_{4,8}}
\end{equation}
    can be written as a positive linear combination of $\bfM_{\nu}^{\succeq \alpha_{1,5}}$-elements.
    The verification of the positive expansion of the terms $X_1$ and $X_2$ is dealt with similarity, so that for the sake of brevity it is  left to the reader.

    Let $Y=\Phi^{\succ \alpha_{4,8}} \cup \{\alpha_{1,5}\}$. 
    For $J\subset \{6,7,8\}$, we set
    \[
        A_{J}=Y \cup \{\alpha_{c-h+1,c} \mid c \in J\}.
    \]
    We stress that $A_{\{6,7,8\}} = \Phi^{\succeq \alpha_{1,5}}$. 
    By the definition of $\bfM$-elements, we obtain
    \begin{equation}
       X_3= \bfM_{\mu}^{Y} = \bfM_\mu^{A_{\{6\}}} +q\bfM_{\mu-\alpha_{2,6}}^{A_{\{7\}}}
       +q^{2}\bfM_{\mu-\alpha_{2,6}-\alpha_{3,7}}^{A_{\{8\}}}
        +q^{3}\bfM_{\mu-\alpha_{2,6}-\alpha_{3,7}-\alpha_{4,8}}^{Y}.
    \end{equation}
    
    Applying \Cref{lem: lamda - alpha  a v}, we obtain
    \begin{equation}\label{eq: Exa A1}
        X_3= \bfM_{\mu}^{Y} = \bfM_\mu^{A_{\{6\}}} 
       +q^{4}\bfM_{v_{7,5}^2(\mu)}^{A_{\{8\}}}
        +q^{6}\bfM_{v_{8,5}^3 (\mu)}^{Y},
    \end{equation}
    since $v_{6,5}^1(\mu)$ is undefined, $D(v_{7,5}^2)(\mu)=4$, and $D(v_{8,5}^3 )(\mu)=6$. 

    Repeating the same computation with the weight $v_{8,5}^3(\mu)$ playing the role of $\mu$, we obtain
    \begin{equation}\label{eq: Exa A2}
        \bfM_{v_{8,5}^3(\mu)}^{Y} = \bfM_{v_{8,5}^3(\mu)}^{A_{\{6\}}} +q^4
        \bfM_{v_{7,5}^2v_{8,5}^3(\mu)}^{A_{\{8\}}}, 
    \end{equation}
    since $v_{8,5}^3v_{8,5}^3(\mu)$ and $v_{6,5}^1v_{8,5}^3(\mu)$ are undefined, and $D(v_{7,5}^2)(v_{8,5}^3(\mu))=4$.

    Therefore, combining \eqref{eq: Exa A1} and \eqref{eq: Exa A2}, we obtain
    \begin{equation}\label{eq: Exa A3}
        X_3=\bfM_{\mu}^{Y} = \bfM_\mu^{A_{\{6\}}} 
       +q^{4}\bfM_{v_{7,5}^2(\mu)}^{A_{\{8\}}}
        +q^{6}\bfM_{v_{8,5}^3(\mu)}^{A_{\{6\}}} +q^{10}
        \bfM_{v_{7,5}^2v_{8,5}^3(\mu)}^{A_{\{8\}}}.
    \end{equation}

The relevance of this expansion is that we have written a term indexed by the set $Y$ as a positive linear combination of terms indexed by strictly larger sets.

We now expand the terms on the right-hand side of \eqref{eq: Exa A3}.
Using the definition of $\bfM$-elements and \Cref{lem: lamda - alpha  a v}, we obtain
\begin{equation} \label{eq: Exa A4}
    \begin{array}{rl}
      \bfM_{\mu}^{A_{\{6\}}}  = &     \bfM_{\mu}^{A_{\{6,7\}}} +q     \bfM_{\mu-\alpha_{2,7}}^{A_{\{6,8\}}} +q^2 \bfM_{\mu-\alpha_{2,7}-\alpha_{3,8}}^{A_{\{6\}}}   \\[5pt]
       =  &   
       \bfM_{\mu}^{A_{\{6,7\}}} +q^2     \bfM_{v_{7,5}^1(\mu)}^{A_{\{6,8\}}} +q^4 \bfM_{v_{8,5}^2(\mu )}^{A_{\{6\}}}, 
    \end{array}
\end{equation}
since $D(v_{7,5}^1)(\mu)=2$ and $D(v_{8,5}^2)(\mu)=4$. 

Similarly, we obtain
\begin{equation}\label{eq: Exa A5}
    \bfM_{v_{8,5}^2(\mu )}^{A_{\{6\}}} =  \bfM_{v_{8,5}^2(\mu )}^{A_{\{6,7\}}}  + q^2  \bfM_{v_{7,5}^1v_{8,5}^2(\mu )}^{A_{\{6,8\}}},
\end{equation}
since $v_{8,5}^2v_{8,5}^2(\mu )$ is undefined and $D(v_{7,5}^1 )(v_{8,5}^2(\mu ))=2$. 

Combining \eqref{eq: Exa A4} and \eqref{eq: Exa A5}, we obtain
\begin{equation} \label{eq: Exa 6}
   \bfM_{\mu}^{A_{\{6\}}}=   \bfM_{\mu}^{A_{\{6,7\}}} +q^2     \bfM_{v_{7,5}^1(\mu)}^{A_{\{6,8\}}} +q^4 \bfM_{v_{8,5}^2(\mu )}^{A_{\{6,7\}}}  + q^6  \bfM_{v_{7,5}^1v_{8,5}^2(\mu )}^{A_{\{6,8\}}}. 
\end{equation}

Thus, we have obtained an expansion of $\bfM_{\mu}^{A_{\{6\}}}$, which is the first term on the right-hand side of \eqref{eq: Exa A3}.
As before, the relevance of this decomposition is that we have written a term indexed by the set $A_{\{6\}}$ as a positive linear combination of terms indexed by strictly larger sets.

Using similar computations, we can rewrite the other terms in \eqref{eq: Exa A3} as follows:

\begin{equation} \label{eq: Exa 7}
    \begin{array}{rl}
      \bfM_{v_{7,5}^2(\mu)}^{A_{\{8\}}} =       &  \bfM_{v_{7,5}^2(\mu)}^{A_{\{6,8\}}} \\[10pt]
     \bfM_{v_{8,5}^3(\mu)}^{A_{\{6\}}}  =     &  \bfM_{v_{8,5}^3(\mu)}^{A_{\{6,7\}}} +q^2 \bfM_{v_{7,5}^1v_{8,5}^3(\mu)}^{A_{\{6,8\}}}    \\[10pt]
     \bfM_{v_{7,5}^2v_{8,5}^3(\mu)}^{A_{\{8\}}} =   & 
      \bfM_{v_{7,5}^2v_{8,5}^3(\mu)}^{A_{\{6,8\}}}.
    \end{array}
\end{equation}

Thus, by substituting \eqref{eq: Exa 6} and \eqref{eq: Exa 7} into \eqref{eq: Exa A3}, we obtain an expansion of $X_3$ as a linear combination of $\bfM$-elements indexed by sets $A_{J}$ with $|J|=2$. 

Using \Cref{lem: lamda - alpha  a v} once again, we see that all the elements involved in this expansion satisfy
\[
\bfM_{\nu}^{A_{J}}=\bfM_{\nu}^{A_{\{6,7,8\}}} =\bfM_{\nu}^{\succeq \alpha_{1,5}},
\]
except for the term indexed by $\mu$, in which case we have
\begin{equation}
    \bfM_{\mu}^{A_{\{6,7\}}} = \bfM_{\mu}^{\succeq \alpha_{1,5}} + q^2\bfM_{v_{8,5}^1(\mu )}^{\succeq \alpha_{1,5}} .
\end{equation}

Putting all of the above equalities together, we obtain
\begin{equation}
   X_3= \bfM_{\mu}^Y = \sum_{\nu\in S} q^{d(\nu)} \bfM_{\nu}^{\succeq \alpha_{1,5}}, 
\end{equation}
where the weights in the set $S$ and the exponents $d(\nu)$ are given in \Cref{tab: good pair}. 
\end{example}

\begin{table}[h] 
    \centering
    \renewcommand{\arraystretch}{1.5}
    \begin{tabular}{|c|c|c|}\hline
       $\nu$  & $\nu_\varpi$   &  $d(\nu)$  \\ \hline 
        $\mu$ &  $[1, 0, 1, 1, 1, 2, 0, 0, 0, 4, 2, 5]_{\varpi}$  &   $0 $  \\ \hline
        $v^{3}_{8,5}(\mu)$ &  $[2, 0, 1, 0, 1, 2, 0, 0, 0, 3, 2, 5]_{\varpi}$  &   $6$  \\ \hline
        $v^1_{7,5}v^{3}_{8,5}(\mu)$ &  $[2, 1, 0, 0, 1, 2, 0, 0, 0, 3, 1, 6]_{\varpi}$  &   $8 $  \\ \hline
        $v^2_{7,5}v^{3}_{8,5}(\mu)$ &  $[3, 0, 0, 0, 1, 2, 0, 0, 0, 2, 2, 6]_{\varpi}$  &   $10 $  \\ \hline
        $v^2_{8,5}(\mu)$ &  $[1, 1, 1, 0, 1, 2, 0, 0, 0, 4, 1, 5]_{\varpi}$  &   $4 $  \\ \hline
        $v^{1}_{7,5}v^2_{8,5}(\mu)$ &  $[1, 2, 0, 0, 1, 2, 0, 0, 0, 4, 0, 6]_{\varpi}$  &   $6 $  \\ \hline
        $v^{1}_{8,5}(\mu)$ &  $[1, 0, 2, 0, 1, 2, 0, 0, 0, 4, 2, 4]_{\varpi}$  &   $2$  \\ \hline
        $v^{1}_{7,5}(\mu)$ &  $[1, 1, 0, 1, 1, 2, 0, 0, 0, 4, 1, 6]_{\varpi}$  &   $2 $  \\ \hline
        $v^{2}_{7,5}(\mu)$ &  $[2, 0, 0, 1, 1, 2, 0, 0, 0, 3, 2, 6]_{\varpi}$  &   $4 $  \\ \hline
    \end{tabular}
    \renewcommand{\arraystretch}{1}
    \caption{Weights involved in the expansion of $X_3$}
    \label{tab: good pair}
\end{table}

In order to generalize the above example we need to introduce a bit more of notation. 

\begin{definition}\rm
Let $\alpha_{i,j}\in \Phi^{\geq 3}$ and $h=j-i+1$. 
Let $1\leq m \leq h-2$.  
Consider a rectangle divided into $m$ unit boxes arranged in a single row, where each box can be colored either gray or white.
We enumerate the boxes from left to right from $j-m+1$ to $j$.

A \emph{filling sequence of rectangles} is a finite sequence  
$ F = (f_0, f_1, \ldots, f_\ell) $
of colored rectangles satisfying that $f_\ell$  is the rectangle with all boxes gray and this is the only occurrence of this coloring in the whole sequence.

We say that a filling sequence $F$ is \textbf{admissible} if the following two conditions hold:
\begin{enumerate}[(I)]
    \item 
    If a box is gray in $f_i$ at position $s$, then it remains gray in $f_{i+1}$.  
    (Equivalently, once a box becomes gray, it never turns white again.)
    \item
    If  $f_i \neq f_{i+1}$ then there is exactly one gray box in $f_{i+1}$ that is white in $f_{i}$.  
    Furthermore, the new gray box must be placed at the  leftmost white block in $f_i$. 
\end{enumerate}

Given a coloring $B$ we denote by $\mathcal{F}_m(\alpha_{i,j},B)$ the set of  all admissible filling sequences starting at $B$. 
Note that $\mathcal{F}_m(\alpha_{i,j},B)$ is infinite for every $m \geq 1$, unless $B$ is the coloring with all boxes gray. 
\end{definition}

\begin{example}\rm \label{exa: filling sequence}
We represent admissible filling sequences $F=(f_0,f_1,\ldots , f_\ell)$ as a rectangle with $\ell+1$ rows and $m$ columns, with $f_0$ on top and $f_m$ at the bottom. 
Let us fix $\alpha_{10,19}$ and $m=8$. 
In \Cref{fig: 6-admissible filling sequences} we illustrate an admissible filling $F=(f_0,f_1, \ldots , f_{11})$ for the parameters above starting at the colored rectangle $f_0$ with only one gray box indexed by $17$. 
 
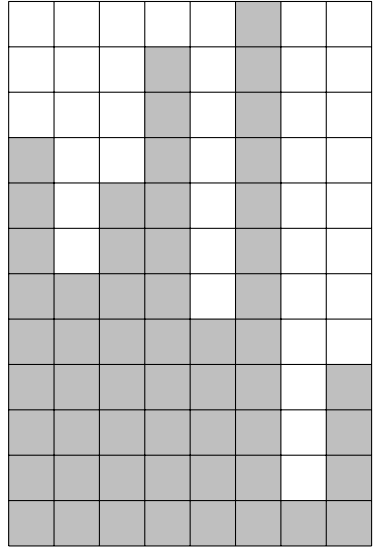
\begin{figure}[H]
\centering
\def\Cols{8}
\def\Rows{12}
% Cambia este valor: gray!70 = más oscuro, gray!40 = más claro
\def\CellFill{gray!50}
\begin{tikzpicture}[x=6mm,y=6mm]
  % --- Celdas negras ---
  \foreach \x/\y in {%
    5/11,
    5/10, 3/10,
    5/9, 3/9,
    5/8, 3/8, 0/8,
    5/7, 3/7, 0/7, 2/7,
    5/6, 3/6, 0/6, 2/6,
    5/5, 3/5, 0/5, 2/5,1/5,
    5/4, 3/4, 0/4, 2/4,1/4,4/4,
    5/3, 3/3, 0/3, 2/3,1/3,4/3,7/3,
    5/2, 3/2, 0/2, 2/2,1/2,4/2,7/2,
    5/1, 3/1, 0/1, 2/1,1/1,4/1,7/1,
    5/0, 3/0, 0/0, 2/0,1/0,4/0,7/0,6/0,
  }{
   \fill[\CellFill] (\x,\y) rectangle ++(1,1);
  }
\foreach \x/\n in {0/12,1/13,2/14,3/15,4/16,5/17,6/18,7/19}
{
    \node at (0.5+\x,13) {\n};
}

  % --- Rejilla ---
  \draw[step=1, thin] (0,0) grid (\Cols,\Rows);
\end{tikzpicture}
\label{fig: 6-admissible filling sequences}
\caption{An admissible filling sequence}
\end{figure}
\end{example}

We now associate to each admissible filling sequence a composition of operators. 

\begin{definition}\rm \label{def: good pair weights}
Let $\alpha_{i,j}\in \Phi^{\geq 3}$ and  $h=j-i+1$. 
Let $F=(f_0,f_1,\ldots , f_{\ell})$ be an admissible filling sequence.
We define the operator $X_F=X_{\ell-1}\cdots X_{1}$, where each operator $X_s$ is given as follows.
Let $L(f_s)$ be the leftmost white rectangle in $f_s$, $l_s$ be the number of squares in $L(f_s)$, $a_s$  be the label of the leftmost square in $L(f_s)$ and $b_s$  be the label of the rightmost square in $L(f_s)$. 
Then, we define
\begin{itemize}
    \item If $f_{s-1}=f_{s}$ then $X_s = v_{b_s,h}^{l_s}$.
     \item If $f_{s-1}\neq f_{s}$, $ a_{s-1}\neq b_{s-1}$ and $a_s=a_{s-1}$ then $X_s= v_{b_s,h}^{l_s}$.
    \item Otherwise, we set $X_{s}$ equal to the identity operator.  
\end{itemize}
\end{definition}

\begin{example}\rm
    With the same notation as in \Cref{exa: filling sequence} we have
    \begin{equation}
        X_F =    v_{18,10}^1 \circ  v_{18,10}^1 \circ v_{18,10}^1 \circ I\circ I\circ v_{13,10}^1\circ  v_{13,10}^1 \circ I\circ v_{14,10}^3 \circ v_{14,10}^3 . 
    \end{equation}
\end{example}

\begin{lemma}\rm\label{lem: pre - positivo caso bueno}
    Let  $\alpha_{i,j}\in \Phi^{\geq 3}$ and $h=j-i+1$. 
    Let $1\leq k \leq h-2$ and $Y = \Phi^{\succ \alpha_{i,j}}\cup \{ \alpha_{i-k,j-k}  \}$.
    Let $\lambda \in X^+$ be such that $\pint{\lambda }{\alpha_{t}}=0$ for all $t\in [j-k+1,j]$. 
    Then, for all  $B\subset [\alpha_{i-k+1,j-k+1},\alpha_{i,j}]\coloneqq I$, we have
    \begin{equation} \label{eq: pareja buena es positiva}
        \bfM_{\lambda}^{Y\cup B}= \sum_{F\in \mathcal{F}_k(\alpha_{i,j},B)} q^{D(X_F)(\lambda)} \bfM_{X_F(\lambda)}^{\succeq \alpha_{i-k,j-k}}. 
    \end{equation}
    Here, we identify a set $B\subset I$ with the coloring given by the rule that a square in position $j-k+1 \leq s \leq j$ is gray if and only if $\alpha_{s-h+1,s} \in B$.  
    Furthermore, the $\bfM$-elements in \eqref{eq: pareja buena es positiva} indexed by weights for which they are not defined are omitted. 
\end{lemma}

\begin{proof}  
   We proceed by downward induction on the size of the set $B\subset I$. 
   If $B=I$, then $Y\cup B = \Phi^{\succeq \alpha_{i-k,j-k}}$ and $\mathcal{F}_{k}(\alpha_{i,j},B)$ consists of a single filling sequence $\{F\}$ such that $X_F$ is the identity operator. 
   Therefore, there is nothing to prove. 
   
   We now fix $B\subset I$ with $|B|<|I|$ and assume that \cref{eq: pareja buena es positiva} holds for all $B_1\subset I$ with $|B|<|B_1|$ and all $\mu \in X^+$ satisfying $\pint{\mu}{\alpha_t}=0$ for all $t\in [j-k+1,j]$. 
   
   For this $B$, we define $a$ and $b$ as in \eqref{eq: defin a ,b , mu}.  
   We also define $J=[\alpha_{a-h+1,a}, \alpha_{b-h,b-1}]$.
   We note that, under the identification described in the statement, the set $J\subset B$ corresponds to the leftmost white region of $B$. 

   For $a\leq c \leq b-1$, we define
   \[
      A_c= Y\cup B \cup \{\alpha_{c-h+1,c}\}
      \qquad\text{and}\qquad
      J_c= [\alpha_{a-h+1,a}, \alpha_{c-h,c-1}].
   \]
   
   By the definition of the $\bfM$-elements, we have
 \begin{equation}  \label{eq: desc subintervalos}
    \bfM_{\lambda}^{Y\cup B} =   \bfM_{\lambda}^{A_a} +     \sum_{c=a+1}^{b-1} 
    q^{c-a}\bfM_{\lambda - \Sigma_{J_c} }^{A_{c} } +q^{b-a}\bfM_{\lambda-\Sigma_{J}}^{Y\cup B}  .
 \end{equation}
 
 Applying \Cref{lem: lamda - alpha  a v}, we obtain
 \begin{equation}  \label{eq: desc subintervalos dominantes}
    \bfM_{\lambda}^{Y\cup B} = 
    \bfM_{\lambda}^{A_a} +  \sum_{c=a+1}^{b-1} q^{D(v_{c-1,h}^{c-a})(\lambda)}\bfM_{v_{c-1,h}^{c-a}(\lambda)}^{A_{c}}  + 
    q^{D(v_{b-1,h}^{b-a})(\lambda )}\bfM_{v_{b-1,h}^{b-a}(\lambda )}^{Y\cup B}. 
 \end{equation}

Since $|Y\cup B|<|A_c|$ for all $a\leq c \leq b-1$, we may apply our inductive hypothesis to rewrite all terms indexed by the sets $A_c$ using \cref{eq: pareja buena es positiva}. 

Although the last term on the right-hand side of \cref{eq: desc subintervalos dominantes} is not covered by the downward induction on $|B|$, we can still apply \cref{eq: pareja buena es positiva} to it because $v_{b-1,h}^{b-a}(\lambda )< \lambda$, and hence we may invoke induction on the dominance order. 

The result then follows from the definition of admissible filling sequences and the operators associated with them. Namely, if
\[
F =(f_0,f_1, ..,f_{\ell})\in \mathcal{F}_k(\alpha_{i,j}, B),
\]
then, under the identification described in the statement, we must have either $f_1=\emptyset$ or $f_1=A_c$ for some $a\leq c \leq b-1$. 
Therefore, we obtain the following disjoint decomposition:
\begin{equation}
    \mathcal{F}(\alpha_{i,j},B) = \mathcal{F}(\alpha_{i,j},B)_{\emptyset} \cup \bigcup_{c=a}^{b-1} \mathcal{F}(\alpha_{i,j},B)_{A_c},  
\end{equation}
where
\[
\mathcal{F}(\alpha_{i,j},B)_{S}
=
\{ F\in \mathcal{F}(\alpha_{i,j},B) \mid f_1=S  \}.
\]
The claim now follows.
\end{proof}

\begin{proposition}\rm \label{prop: caso bueno}
 Let $\lambda \in X^{+}$, $\alpha_{i,j}\in \Phi^{\geq 3}$ and  $h=j-i+1$. 
    Let $1\leq k < h-1$ and suppose that $ \pint{\lambda}{\alpha_{t}} =0$ for $t\in [j-k+1,j]$. 
    Furthermore, we assume that $\mu_k\coloneqq v^k_{j,h}(\lambda) $ is defined.
    Then, we have
    \begin{equation} \label{sumsum}
        \bfM_{\mu_k}^{\succ \alpha_{i,j}} - q\bfM_{\mu_k - \alpha_{i-k,j-k}}^{\succ \alpha_{i,j}} =  \sum_{F\in \mathcal{F}_k(\alpha_{i,j},\emptyset)} q^{D(X_F)(\mu_k)}\bfM_{X_F(\mu_k)}^{ \succeq \alpha_{i-k,j-k}   }, 
    \end{equation}   
    where $\bfM$-elements indexed by undefined weights are neglected. 
\end{proposition}

\begin{proof}
    By the conditions imposed on $\lambda$ we have that $\mu_k$ satisfies the conditions in \Cref{lem: pre - positivo caso bueno}. 
    On the other hand, if $Y=\Phi^{\alpha_{i,j}}\cup \{\alpha_{i-k,j-k}\}$ then we have
    \begin{equation}
        \bfM_{\mu_{k}}^Y  = \bfM_{\mu_k}^{\succ \alpha_{i,j}} - q\bfM_{\mu_k - \alpha_{i-k,j-k}}^{\succ \alpha_{i,j}} . 
    \end{equation}
    Therefore, the result follows by applying \Cref{lem: pre - positivo caso bueno} for $B=\emptyset$. 
\end{proof}

\begin{example}\rm
Figure \ref{fig:rectangles} depicts the nine admissible filling sequences that contribute to the expansion of the element $X_3$ in \Cref{ex: como funciona}, as predicted by \Cref{prop: caso bueno}.
The weight associated with each filling sequence is displayed below the corresponding diagram.

Although there are infinitely many admissible filling sequences, only finitely many contribute to the expansion in \eqref{sumsum}.
\end{example}

\begin{figure}[ht]
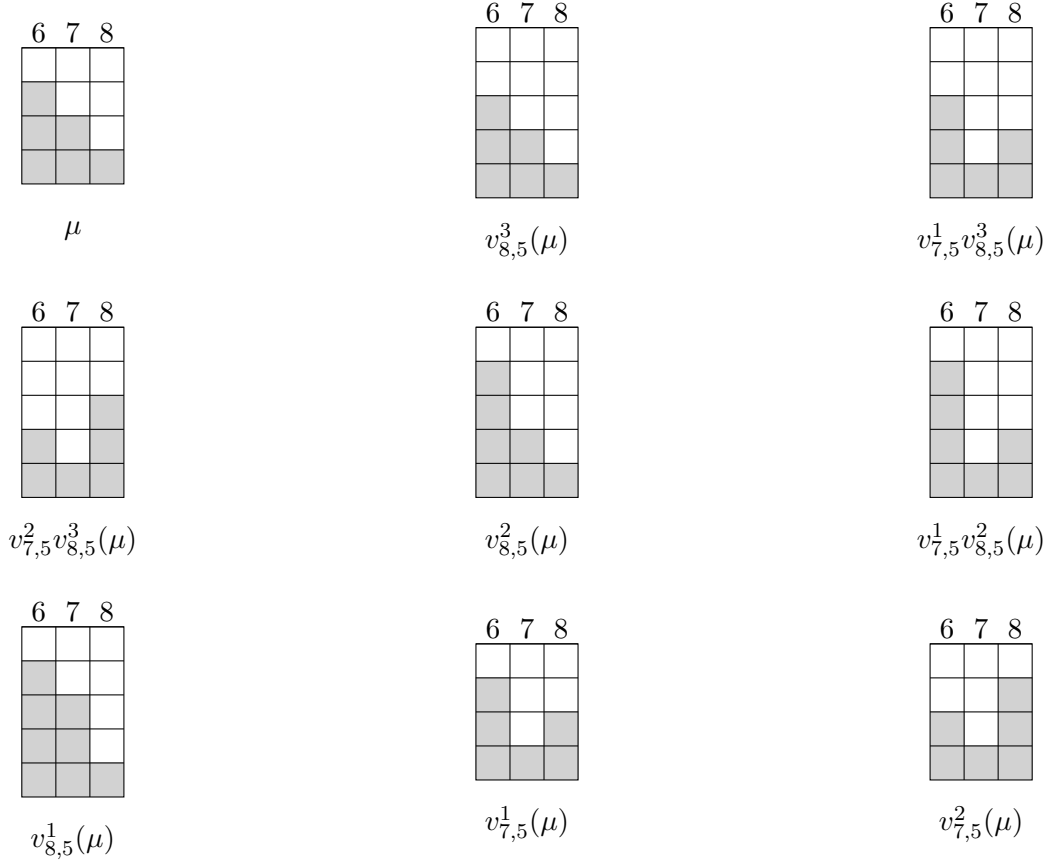

\centering

\newcommand{\rectcell}[2]{%
\begin{minipage}{0.28\textwidth}
\centering
#1

\vspace{2mm}
$#2$
\end{minipage}%
}

\rectcell{\grayrect[4]{1/6,1/7,1/8,2/6,2/7,3/6}}{\mu}
\hfill
\rectcell{\grayrect{1/6,1/7,1/8,2/6,2/7,3/6}}{v_{8,5}^{3}(\mu)}
\hfill
\rectcell{\grayrect{1/6,1/7,1/8,2/6,2/8,3/6}}{v_{7,5}^{1}v_{8,5}^{3}(\mu)}

\vspace{4mm}

\rectcell{\grayrect[5]{1/6,1/7,1/8,2/6,2/8,3/8}}
{v_{7,5}^{2}v_{8,5}^{3}(\mu)}
\hfill
\rectcell{\grayrect{1/6,1/7,1/8,2/6,2/7,3/6,4/6}}{v_{8,5}^{2}(\mu)}
\hfill
\rectcell{\grayrect{1/6,1/7,1/8,2/6,2/8,3/6,4/6}}
{v_{7,5}^{1}v_{8,5}^{2}(\mu)}

\vspace{4mm}

\rectcell{\grayrect{1/6,1/7,1/8,2/6,2/7,3/6,3/7,4/6}}{v_{8,5}^{1}(\mu)}
\hfill
\rectcell{\grayrect[4]{1/6,1/7,1/8,2/6,2/8,3/6}}{v_{7,5}^{1}(\mu)}
\hfill
\rectcell{\grayrect[4]{1/6,1/7,1/8,2/6,2/8,3/8}}{v_{7,5}^{2}(\mu)}

\caption{Filling sequences of rectangles and their corresponding weights}
\label{fig:rectangles}
\end{figure}

%%%%%%%%%%%%%%%%%%%%%%%%%%%%%%%%%%%%%%%%%%%%%%%%%%%%%%%%%%%%%%%%%

\subsection{Commutation Rules} \label{section COM RUl}

Having established in the previous section that good pairs expand positively, we now turn our attention to the bad pair appearing in the SID.

To handle this case, we require a collection of commutation rules between the relevant operators. The main purpose of this section is to state these rules.

For the sake of brevity, we omit their proofs. The arguments consist of a lengthy case-by-case analysis that, while technical, offers little additional insight into the underlying ideas. Instead, we present several examples that illustrate the rules and provide intuition for why they hold.

We emphasize that the commutation rules stated below do not apply to arbitrary dominant weights. Rather, they are valid only for a special class of dominant weights satisfying an additional condition.

\begin{definition}\rm
Let $\alpha_{i,j} \in \Phi^{\geq 2}$ and set $h = j-i+1$. We say that a dominant weight $\lambda$ satisfies Condition $\mathbf{Z}$ if $P_{i,h}(\lambda)$ is defined and, setting $\varrho \coloneq \varrho_{i,h}(\lambda)$, we have
$$ \pint{\lambda}{\alpha_t} = 0 \quad \text{for all integers } t \in [\varrho+h, j]. $$
\end{definition}

\begin{lemma}\label{lem: super-conmutation}
Let $\lambda \in X^+$, $\alpha_{i,j}\in \Phi^{\geq 2}$, and  $h=j-i+1$. 
Let $s$ be an integer such that $0\leq s\leq h-2$, and define $j_s\coloneqq j-s$. 
Assume that $\lambda$ satisfies Condition $\mathbf{Z}$ and that the weight
$\mathbf{P}_{j_s,h}\mathbf{P}_{i,h}(\lambda)$ is defined.
Then $\mathbf{P}_{j_s,h}(\lambda)$ is also defined, and
\begin{equation}\label{eq: super-commutation}
    \mathbf{P}_{j_s,h}(\lambda)
    =
    \mathbf{P}_{j_s,h}\mathbf{P}_{i,h}(\lambda).
\end{equation}
\end{lemma}

\begin{example}\label{exa:  conmutation}  \rm
In this example we illustrate  \Cref{lem: super-conmutation}.

Let us fix the parameters $(i, j) = (9, 11)$ and $h = 11 - 9 + 1 = 3$. According to the lemma, the integer $s$ can take values in the range $0 \le s \le h-2$, which gives $s \in \{0, 1\}$. 
Thus, the relevant indices are $j_0 = 11$ and $j_1 = 10$. 
Let
\[
    \lambda = [6, 2, 1, 1, 2, 0, 0, 0, 0, 0, 0, 1]_\varpi.
\]
We have
\begin{equation}
    P_{9,3}(\lambda) = \lambda - \alpha_{4,9} = [6, 3, 0, 1, 2, 0, 0, 0, -1, 1, 0, 1]_\varpi .
\end{equation}
In particular,  $\varrho_{9,3}(\lambda) = 4$. 
Therefore, $\lambda$ satisfies Condition $\mathbf{Z}$ since $\pint{\lambda}{\alpha_k} =0$ for all integers $k\in [7,11]$. 
We want to  verify \eqref{eq: super-commutation} for $j_0=11$ and $j_1=11$. 
For $j_0=11$ we have 
\begin{equation}
\begin{aligned}
    \mathbf{P}_{9,3}(\lambda) &= \lambda - \alpha_{4,9} - \alpha_{3,8} - \alpha_{5,7} - \alpha_{4,6} - \alpha_{3,5} \\
    &= [6, 4, 1, 0, 0, 0, 0, 0, 0, 1, 0, 1]_\varpi, \\[5pt]
    \mathbf{P}_{11,3}\mathbf{P}_{9,3}(\lambda) &= \mathbf{P}_{9,3}(\lambda) - \alpha_{3,11} - \alpha_{2,10} \\
    &= [7, 4, 0, 0, 0, 0, 0, 0, 0, 0, 0, 2]_\varpi.
\end{aligned}
\end{equation}
On the other hand, applying $\mathbf{P}_{11,3}$ directly to $\lambda$ yields:
\begin{equation}
\begin{aligned}
    \mathbf{P}_{11,3}(\lambda) &= \lambda - \alpha_{3,11} - \alpha_{5,10} - \alpha_{4,9} -\alpha_{3,8} - \alpha_{5,7} - \alpha_{4,6} - \alpha_{3,5} - \alpha_{2,4} \\
    &= [7, 4, 0, 0, 0, 0, 0, 0, 0, 0, 0, 2]_\varpi \\
    & = \mathbf{P}_{11,3}\mathbf{P}_{9,3}(\lambda).
\end{aligned}
\end{equation}
Similarly, for $j_1 = 10$ we have
\begin{equation}
\begin{aligned}
    \mathbf{P}_{10,3}\mathbf{P}_{9,3}(\lambda) &= \mathbf{P}_{9,3}(\lambda) - \alpha_{2,10} \\
    &= [7, 3, 1, 0, 0, 0, 0, 0, 0, 0, 1, 1]_\varpi, \\[5pt]
    \mathbf{P}_{10,3}(\lambda) &= \lambda - \alpha_{5,10} - \alpha_{4,9} - \alpha_{3,8} - \alpha_{5,7} - \alpha_{4,6}- \alpha_{3,5} - \alpha_{2,4} \\
    &= [7, 3, 1, 0, 0, 0, 0, 0, 0, 0, 1, 1]_\varpi \\ &= \mathbf{P}_{10,3}\mathbf{P}_{9,3}(\lambda).
\end{aligned}
\end{equation}
In both instances, the equality $\mathbf{P}_{j_s, 3}\mathbf{P}_{9,3}(\lambda) = \mathbf{P}_{j_s, 3}(\lambda)$ predicted by the lemma holds. 
\end{example}

\begin{example}\rm
Consider the same parameters as in \Cref{exa:  conmutation} but  the weight 
\[
    \lambda = [6, 2, 3, 0, 2, 1, 0, 0, 0, 0, 0, 1]_{\varpi}.
\]
We have $P_{9,3}(\lambda) = \lambda - \alpha_{1,9} = [5, 2, 3, 0, 2, 1, 0, 0, -1, 1, 0, 1]_{\varpi}$.
In particular, $\varrho_{9,3}(\lambda)=1$. 
Thus,  Condition $\mathbf{Z}$ requires $\langle \lambda, \alpha_t \rangle = 0$ for all $t \in  [4, 11]$. 
Since  $\pint{\lambda}{\alpha_5} = 2\neq 0$, we have that $\lambda$ does not satisfy Condition $\mathbf{Z}$. 

Let us show that \eqref{eq: super-commutation} does not hold in this case. 
We first notice that

\begin{equation}
\begin{aligned}
    \mathbf{P}_{9,3}(\lambda) &= \lambda - \alpha_{1,9}- \alpha_{6,8}- \alpha_{5,7}- \alpha_{4,6}- \alpha_{3,5} \\
    &= [5, 3, 3, 0, 1, 0, 0, 0, 0, 1, 0, 1]_\varpi.
\end{aligned}
\end{equation}

Therefore, we get
\begin{equation}
\begin{aligned}
    \mathbf{P}_{10,3}\mathbf{P}_{9,3}(\lambda) &= \mathbf{P}_{9,3}(\lambda)- \alpha_{5,10} \\
    &= [5, 3, 3, 1, 0, 0, 0, 0, 0, 0, 1, 1]_\varpi.
\end{aligned}
\end{equation}
On the other hand, we have
\begin{equation}
    \begin{aligned}
         \mathbf{P}_{10,3}(\lambda) &= \lambda- \alpha_{5,10} - \alpha_{4,9}- \alpha_{6,8}- \alpha_{5,7}- \alpha_{4,6}- \alpha_{3,5} \\
    &= [6, 3, 4, 0, 0, 0, 0, 0, 0, 0, 1, 1]_\varpi.
    \end{aligned}
\end{equation}
We conclude that 
$   \mathbf{P}_{10,3}\mathbf{P}_{9,3}(\lambda) \neq \mathbf{P}_{10,3}(\lambda)$. 
Thus \eqref{eq: super-commutation} fails for $j_1=10$. 
\end{example}

\begin{lemma}\rm\label{conm. 2}
Let $\alpha_{i,j} \in \Phi^{\geq 3}$ and set $h = j-i+1$. 
Let $s$ and $m$ be integers satisfying $0 \leq s < h-2$ and $1 \leq m \leq h-2-s$. 
Define $j_s = j-s$ and $i_s = i-s$. 
Let $F \in \mathcal{F}_m(\alpha_{i_s,j_s},B)$, where $B$ is the coloring in which all boxes are white, and let $X_F$ denote its corresponding operator. 
Suppose $\lambda \in X^{+}$ is such that $\pint{\lambda}{\alpha_{t}} = 0$ for all $t \in [i+1,j]$. 
If $X_F(v_{j_s,h}^m\mathbf{P}_{i,h}(\lambda))$ is defined, then $\mathbf{P}_{i,h}(X_F v_{j_s,h}^m(\lambda))$ is also defined, and the following equality holds:
\begin{equation}\label{eq: conm. 2}
    X_{F}(v_{j_s,h}^{m}\mathbf{P}_{i,h}(\lambda)) = \mathbf{P}_{i,h}(X_{F} v_{j_s,h}^m(\lambda)).
\end{equation}
\end{lemma}

%-------------------------------
%ENUNCIADO ANTERIOR
%-------------------------------

% \begin{lemma}\label{conm. 2}
%     Let $\lambda\in X^+$, $\alpha_{i,j} \in \Phi^{\geq 3}$, $h = j-i+1$, and $(s,m)$ be a pair integers satisfying  $0\leq s < h-2$, $1 \leq m \leq h-2-s$.
%     Set $j_s = j-s$.
%     Let $\mathcal{F}_m$ denote the set of $m$-fillings whose rightmost empty box is labeled by $j_s-m$. 
%     If $\lambda$ satisfies the condition Z, then
%     for any $F \in \mathcal{F}_m$ such that 
%     $X_{F}(v_{j_s,h}^{m}\mathbf{P}_{i,h}(\lambda))$
%     is defined, the following identity holds:
%     \begin{equation}\label{eq: conm. 2}
%          X_{F}(v_{j_s,h}^{m}\mathbf{P}_{i,h}(\lambda)) = \mathbf{P}_{i,h}(X_{F} \, v_{j_s,h}^m(\lambda)).
%     \end{equation}
% \end{lemma}

\begin{example}\rm
In this example, we illustrate the commutation relation described in \Cref{conm. 2}.

Let us fix the parameter $\alpha_{i,j} = \alpha_{11,14}$, $h = j-i+1 = 4$ and consider the operator $\mathbf{P}_{11,4}$. 
We will examine its commutation with the sequence of operators $X_F(v^{2}_{14,4}) = (v^1_{13,4})^3(v^{2}_{14,4})^2$.
This sequence arises from the $2$-filling illustrated in the following figure:

\begin{figure}[H]
\centering
\def\Cols{2}
\def\Rows{6}
% Cambia este valor: gray!70 = más oscuro, gray!40 = más claro
\def\CellFill{gray!50}
\begin{tikzpicture}[x=6mm,y=6mm]
  % --- Celdas negras ---
  \foreach \x/\y in {%
    1/3,
    1/2,
    1/1,
    0/0,1/0
  }{
   \fill[\CellFill] (\x,\y) rectangle ++(1,1);
  }
    \foreach \x/\n in {0/13,1/14}
    {
        \node at (0.5+\x,\Rows + 0.5) {\n};
    }
  % --- Rejilla ---
  \draw[step=1, thin] (0,0) grid (\Cols,\Rows);
\end{tikzpicture}
\caption{2-admissible filling sequences}
\label{fig: 2-admissible filling sequences}
\end{figure}

Consider the following weight:
\[
    \lambda = [3, 3, 2, 3, 3, 2, 0, 0, 0, 0, 0, 0, 0, 0, 2, 2, 0, 1, 5, 0]_\varpi.
\]

\medskip
\noindent First, we apply the composite operator $(v^1_{13,4})^3(v^{2}_{14,4})^2\mathbf{P}_{11,4}$ to $\lambda$:
\begin{equation}
\begin{aligned}
    \mathbf{P}_{11,4}(\lambda) &= \lambda - \alpha_{4,11} - \alpha_{3,10} - \alpha_{6,9} - \alpha_{5,8} - \alpha_{3,6} \\
    &= [3, 5, 2, 2, 3, 0, 0, 0, 0, 0, 0, 1, 0, 0, 2, 2, 0, 1, 5, 0]_\varpi, \\[5pt]
    v^{2}_{14,4}\mathbf{P}_{11,4}(\lambda) &= \mathbf{P}_{11,4}(\lambda) - \alpha_{3,17}- \alpha_{2,16} \\
    &= [4, 5, 1, 2, 3, 0, 0, 0, 0, 0, 0, 1, 0, 0, 2, 1, 0, 2, 5, 0]_\varpi, \\[5pt]
    (v^1_{13,4})^3(v^{2}_{14,4})^2\mathbf{P}_{11,4}(\lambda) &= v^{2}_{14,4}\mathbf{P}_{11,4}(\lambda) - \alpha_{3,17}- \alpha_{2,16}- 3\alpha_{2,19} \\
    &= [8, 2, 0, 2, 3, 0, 0, 0, 0, 0, 0, 1, 0, 0, 2, 0, 0, 3, 2, 3]_\varpi.
\end{aligned}
\end{equation}

\medskip
\noindent Next, we reverse the order: we apply the sequence of $v$-operators to the initial weight $\lambda$ first, followed by the $\mathbf{P}$-operator (i.e., we evaluate $\mathbf{P}_{11,4}(v^1_{13,4})^3(v^{2}_{14,4})^2(\lambda)$):
\begin{equation}
\begin{aligned}
    (v^1_{13,4})^3(v^{2}_{14,4})^2(\lambda) &= \lambda - 2\alpha_{3,17}- 2\alpha_{6,16}- 3\alpha_{2,19} \\
    &= [6, 2, 0, 3, 5, 0, 0, 0, 0, 0, 0, 0, 0, 0, 2, 0, 0, 3, 2, 3]_\varpi, \\[5pt]
    \mathbf{P}_{11,4}\Big((v^1_{13,4})^3(v^{2}_{14,4})^2(\lambda)\Big) &= (v^1_{13,4})^3(v^{2}_{14,4})^2(\lambda) - \alpha_{4,11} - \alpha_{3,10} - \alpha_{2,9}  - \alpha_{5,8} - \alpha_{4,7} - \alpha_{3,6} - \alpha_{2,5} \\
    &= [8, 2, 0, 2, 3, 0, 0, 0, 0, 0, 0, 1, 0, 0, 2, 0, 0, 3, 2, 3]_\varpi.
\end{aligned}
\end{equation}

Therefore, 
\[
    (v^1_{13,4})^3(v^{2}_{14,4})^2\mathbf{P}_{11,4}(\lambda) = \mathbf{P}_{11,4}(v^1_{13,4})^3(v^{2}_{14,4})^2(\lambda).
\] 
\end{example}
% Furthermore, observe that the intermediate weight 
% $$\mu = (v^1_{13,4})^3(v^{2}_{14,4})^2(\lambda) = [6, 2, 0, 3, 5, 0, 0, 0, 0, 0, 0, 0, 0, 0, 2, 0, 0, 3, 2, 3]_\varpi$$
% strictly satisfies Condition $\mathbf{Z}$, since $\varrho_{14,4}(\mu) = 4$ and $\mu_{8} = \mu_{9} = \mu_{10} = \mu_{11} = \mu_{12} = \mu_{13} = \mu_{14} = 0$.

\subsection{Proof of positivity}  \label{section proof of positivity}

The following lemma will allow us to handle the case when a bad pair occurs in the Second Inverse Decomosition. 

\begin{lemma} \label{lem: GranLemma}
    Let $\alpha_{i,j} \in \Phi^{\geq 2} $ and  $h=j-i+1$.
    For $\lambda \in X^{+}$ we define
    \begin{equation}
       \mathcal{A}_\lambda^{i,j} \coloneqq \sum_{\mu< \lambda}  \mathbb{N}[q] \bfM_{\mu}^{\succ\alpha_{i,j}}+
         \sum_{t=0}^{i-1} \left( \sum_{\substack{\mu \in X^+ \\\mu\leq  \lambda}}\mathbb{N}[q]\bfM_{\mu}^{\succeq \alpha_{i-t,j-t}} \right). 
    \end{equation}
If  $\lambda$ satisfies Condition $Z$ and  $\mathbf{P}_{i,h}(\lambda)$ is defined then
\begin{equation} \label{eq: DIFF}
    \bfM_{\lambda}^{\succ \alpha_{i-s,j-s}} - q^{D(\mathbf{P}_{i,h} )(\lambda)}  \bfM_{\mathbf{P}_{i,h}(\lambda)}^{\succ \alpha_{i-s,j-s}} \in  \mathcal{A}_\lambda^{i,j},
\end{equation}
for all $ 0\leq s \leq h-1  $. 
\end{lemma}

\begin{proof}
We proceed by induction on $i$. 
We notice that if $i<h$, then $i-h+1<1$.
Thus, the root $\alpha_{i-h+1,i}$ is not defined. 
It follows that $\mathbf{P}_{i,h}(\lambda)$ is not defined.
Therefore, we must have $i\geq h$, and the base case of our induction is the case $i=h$.

For $0\leq s \leq h-1$ we set $i_s=i-s$ and $j_{s}=j-s$. 
    
By combining \Cref{prop: second version refinado} and \Cref{prop: caso bueno}, we obtain
        \begin{equation}\label{eq:super-lemma-A}
        \bfM_{\lambda}^{\succ \alpha_{i_s,j_s}} =   \bfM_{\lambda}^{\succeq \alpha_{i_s,j_s}} +q^{D(\mathbf{P}_{j_s,h})(\lambda)} \bfM_{ \mathbf{P}_{j_s,h}(\lambda )}^{\succ \alpha_{i_s,j_s}} + \sum_{m=1}^{h-2}  q^{D(v_{j_s,h}^m)(\lambda)}  X_m + q^{D(\mathbf{v}_{j,h})(\lambda)}Y_s,
    \end{equation}
where 
\begin{equation}\label{eq: XXX}
    X_m = \displaystyle\sum_{F\in \mathcal{F}_m} q^{D(X_F)(v_{j_s,h}^m(\lambda))} \mathbf{M}_{X_F(v_{j_s,h}^m(\lambda))}^{\succeq \alpha_{i_s-m,j_s-m}}.
\end{equation}
and 
\begin{equation}\label{eq: YsubS}
    Y_s=  \bfM_{\mathbf{v}_{j_s,h}(\lambda )}^{\succ \alpha_{i_s,j_s}}-q^{D(\mathbf{P}_{i_s,h})(\mathbf{v}_{j_s,h}(\lambda ))}\bfM_{\mathbf{P}_{i_s,h}\mathbf{v}_{j_s,h}(\lambda )}^{\succ \alpha_{i_s,j_s}}.  
\end{equation}

On the other hand, we notice that $\pint{\mathbf{P}_{i,h}(\lambda)}{\alpha_{i+1}} = 1$ and $\pint{\mathbf{P}_{i,h}(\lambda)}{\alpha_{t}} = 0$ for all $t \in [i+2, j]$. 
Thus, \Cref{prop: second version refinado} and \Cref{prop: caso bueno} yield
\begin{equation} \label{eq: M de P_i}
     \mathbf{M}_{\mathbf{P}_{i,h}(\lambda)}^{\succ \alpha_{i_s ,j_s}} = \mathbf{M}_{\mathbf{P}_{i,h}(\lambda)}^{\succeq \alpha_{i_s ,j_s}} +  q^{D(\mathbf{P}_{j_s,h})(\mathbf{P}_{i,h}(\lambda))}\mathbf{M}_{\mathbf{P}_{j_s,h}\mathbf{P}_{i,h}(\lambda)}^{\succ \alpha_{i_s ,j_s}} + \sum_{m=1}^{h-2-s}q^{D(v_{j_s,h}^m)(\mathbf{P}_{i,h}(\lambda))}X'_m,
\end{equation}
where 
\begin{equation}\label{eq: XXX prime}
    X'_m =  \sum_{F\in \mathcal{F}_m} q^{D(X_F)(v_{j_s,h}^m(\mathbf{P}_{i,h}(\lambda)))} \mathbf{M}_{X_F(v_{j_s,h}^m(\mathbf{P}_{i,h}(\lambda)))}^{\succeq \alpha_{i_s-m,j_s-m}}.
\end{equation}

We insist on our convention that if a weight $\mu$ is not defined, then the corresponding $\bfM$-term is set equal to zero.
In order to save space, we set
\begin{equation}
    \Delta_\lambda^{i}(s) \coloneqq   \bfM_{\lambda}^{\succ \alpha_{i-s,j-s}} - q^{D(\mathbf{P}_{i,h} )(\lambda)}  \bfM_{\mathbf{P}_{i,h}(\lambda)}^{\succ \alpha_{i-s,j-s}}. 
\end{equation}

We first treat the base case $i=h$. 
We split the proof into three cases.

\begin{itemize}
    \item $s=h-1$.  We have $j_s=j-(h-1)=i$. Therefore, by \eqref{eq:super-lemma-A} we have
    \begin{equation}
\Delta_{\lambda}^i(s)  =   \bfM_{\lambda}^{\succeq \alpha_{i_s,j_s}}  + \sum_{m=1}^{h-2}  q^{D(v_{j_s,h}^m)(\lambda)}  X_m + q^{D(\mathbf{v}_{j,h})(\lambda)}Y_s.  
    \end{equation}
By \eqref{eq: XXX} we know that $X_m\in \mathcal{A}_{\lambda}^{i,j}$ for all $1\leq m \leq h-2$.  
Thus, we only need to analyze the element $Y_s$. 
We notice that in this case $i_s<i=h$.
By the discussion in the first paragraph of this proof, we conclude that the weight $\mathbf{P}_{i_s,h}\mathbf{v}_{j_s,h}(\lambda )$ is not defined. 
In accordance with our conventions, \eqref{eq: YsubS} reduces to 
\begin{equation}\label{eq: YYY solo}
       Y_s=  \bfM_{\mathbf{v}_{j_s,h}(\lambda )}^{\succ \alpha_{i_s,j_s}}.
\end{equation}
Since $\mathbf{v}_{j_s,h}(\lambda)<\lambda$, it follows that $Y_s\in \mathcal{A}_{\lambda}^{i,j}$. 
Therefore, $\Delta_{\lambda}^i(s) \in  \mathcal{A}_{\lambda}^{i,j}$, as we wanted to show. 

\item $s=h-2$. We combine \eqref{eq:super-lemma-A}
 and \eqref{eq: M de P_i} to compute $\Delta_{\lambda}^i(s)$.
 By using \Cref{lem: super-conmutation}, we get
 \begin{equation}
     \mathbf{P}_{j_s,h}(\lambda) = \mathbf{P}_{j_s,h}\mathbf{P}_{i,h}(\lambda) \qquad \mbox{and} \qquad D(\mathbf{P}_{j_s,h})(\lambda) = D(\mathbf{P}_{i,h})(\lambda) + D(\mathbf{P}_{j_s,h})(\mathbf{P}_{i,h})(\lambda). 
 \end{equation}
 Therefore, the corresponding terms cancel each other in $\Delta_{\lambda}^i(s)$.
 Thus, in this case the lemma reduces to proving that both 
 \begin{equation}
  \bfM_{\lambda}^{\succeq \alpha_{i_s,j_s}} - q^{D(\mathbf{P}_{i,h})(\lambda)} \bfM_{\mathbf{P}_{i,h}(\lambda)}^{\succeq \alpha_{i_s,j_s}}   \qquad \mbox{and} \qquad    Y_s
 \end{equation}
 belong to $\mathcal{A}_{\lambda}^{i,j}$. 

On the one hand, we have
\begin{equation} \label{eq: case h-2}
   \bfM_{\lambda}^{\succeq \alpha_{i_s,j_s}} - q^{D(\mathbf{P}_{i,h})(\lambda)} \bfM_{\mathbf{P}_{i,h}(\lambda)}^{\succeq \alpha_{i_s,j_s}} = \Delta_{\lambda}^i(s+1).  
\end{equation}
The right-hand side of \eqref{eq: case h-2} belongs to $\mathcal{A}_{\lambda}^{i,j}$ by the case $s=h-1$ treated above. 

On the other hand, if $h>2$, then $i_s<i=h$, so we can conclude that $Y_s\in \mathcal{A}_{\lambda}^{i,j}$ by arguing as in the previous case. 
If $h=2$, then $s=0$, so that $Y_s= \Delta_{\mathbf{v}_{j,h}(\lambda)}^i(s)$. 
Since $\mathbf{v}_{j,h}(\lambda) <\lambda$, we can argue by induction on the dominance order to conclude that $Y_s\in \mathcal{A}_{\lambda}^{i,j}$.

\smallskip
 \item Let $0\leq s < h-2$.
 We argue in the same fashion as in the previous case. 
That is, we combine \eqref{eq:super-lemma-A}
 and \eqref{eq: M de P_i} to compute $\Delta_{\lambda}^i(s)$. 
 After this, we use \Cref{lem: super-conmutation} to cancel the terms indexed by $ \mathbf{P}_{j_s,h}(\lambda) $ and $\mathbf{P}_{j_s,h}\mathbf{P}_{i,h}(\lambda)$. 
 Next, we use downward induction on $s$ to conclude that the difference on the left-hand side of \eqref{eq: case h-2} belongs to $\mathcal{A}_{\lambda}^{i,j}$. 
 Similarly, the term $Y_s$ can be shown to belong to $\mathcal{A}_{\lambda}^{i,j}$ by the same argument as in the previous case, namely: either $Y_s \in \mathcal{A}_{\lambda}^{i,j}$ if $s>0$ by \eqref{eq: YYY solo}, or $Y_s \in \mathcal{A}_{\lambda}^{i,j}$ if $s=0$ by induction on the dominance order. 
 
 All in all, we realize that negative terms in $\Delta_{\lambda}^i(s)$ can only arise from the differences between the $X_m$ and $X_m'$ terms. 
 Concretely, this case reduces to proving that 
\begin{equation}
   q^{D(v_{j_s,h}^m)(\lambda)}  X_m  - q^{D(\mathbf{P}_{i,h})(\lambda)} q^{D(v_{j_s,h}^m)(\mathbf{P}_{i,h}(\lambda))}X'_m \in \mathcal{A}_{\lambda}^{i,j}
\end{equation}
for all $1\leq m \leq h-2-s$. 

By combining \eqref{eq: XXX} and \eqref{eq: XXX prime}, together with the additivity of the degrees, it is enough to prove that  
\begin{equation} \label{eq: it is enough A}
   q^{D(X_F\, v_{j_s,h}^m)(\lambda)}  \mathbf{M}_{X_F(v_{j_s,h}^m(\lambda))}^{\succeq \alpha_{i_s-m,j_s-m}}  - q^{D(X_F\, v_{j_s,h}^m \mathbf{P}_{i,h})(\lambda)} \mathbf{M}_{X_F(v_{j_s,h}^m(\mathbf{P}_{i,h}(\lambda)))}^{\succeq \alpha_{i_s-m,j_s-m}}\in \mathcal{A}_{\lambda}^{i,j},
\end{equation}
for all $1\leq m \leq h-2-s$. 

If $X_F(v_{j_s,h}^m(\mathbf{P}_{i,h}(\lambda)))$ is not defined, then we are done by our conventions. 
Otherwise, by \Cref{conm. 2} we have that
\begin{equation}\label{eq: XXX conm}
    X_F(v_{j_s,h}^m(\mathbf{P}_{i,h}(\lambda))) = \mathbf{P}_{i,h}(X_F(v_{j_s,h}^m(\lambda))).
\end{equation}
and 
\begin{equation} \label{eq: XXX  conm degrees}
    D(  X_F \, v_{j_s,h}^m \mathbf{P}_{i,h}  )(\lambda) = D(\mathbf{P}_{i,h})(X_F \, v_{j,h}^m (\lambda)) + D(X_F \, v_{j,h}^m )(\lambda).
\end{equation}

By replacing \eqref{eq: XXX conm} and \eqref{eq: XXX  conm degrees} into \eqref{eq: it is enough A}, we see that it suffices to show that 
\begin{equation}
   \mathbf{M}_{X_F\, v_{j_s,h}^m(\lambda)}^{\succeq \alpha_{i_s-m,j_s-m}}  - q^{ D(\mathbf{P}_{i,h})( X_F \, v_{j_s,h}^m  (  \lambda))} \mathbf{M}_{  \mathbf{P}_{i,h} (X_F \, v_{j_s,h}^m (\lambda))}^{\succeq \alpha_{i_s-m,j_s-m}}\in \mathcal{A}_{\lambda}^{i,j}, 
\end{equation}
for all $1\leq m \leq h-2-s$. 
Finally, this last claim follows since $X_F \, v_{j_s,h}^m (\lambda)$ satisfies Condition $Z$ and by downward induction on $s$.
\end{itemize}
This completes the proof of the base case $i=h$. 
If $i>h$, we repeat the same argument word by word. 
Indeed, the only significant difference between the general case and the base case is the argument needed to ensure that the terms $Y_s$ that may arise throughout the analysis belong to $\mathcal{A}_{\lambda}^{i,j}$. In the case $s>0$, we must argue by induction on $i$ (noticing that if $s>0$, then $i_s<i$), in contrast with the base case, where we showed that such terms were not differences by using \eqref{eq: YYY solo}. 
\end{proof}

The following is \Cref{Teo D} in the introduction.

\begin{theorem}\label{teo: theo D in the main}
    Let $\alpha_{i,j} \in \Phi^{\geq 2}$ and $\lambda \in X^+$. 
    Then 
 \begin{equation}  \label{eq: casi casi final}
      \bfM_{\lambda}^{\succ \alpha_{i,j} } \in  \sum_{k=0}^{i-1}  \sum_{\mu \leq \lambda} \mathbb{N}[q] \bfM_{ \mu }^{\succeq \alpha_{i-k,j-k}}.      
 \end{equation}
\end{theorem}

\begin{proof}
    We proceed by induction on the dominance order. 
    The minimal elements in $X^+$ with respect to the dominance order are the zero weight and the fundamental weights. 
    For all these weights, it follows from \Cref{prop: second version refinado} and our conventions that 
    \begin{equation}
        \bfM_{\lambda}^{\succ \alpha_{i,j}} = \bfM_{\lambda}^{\succeq \alpha_{i,j}}. 
    \end{equation}
We now fix $\lambda \in X^{+}$ and assume that the result holds for all $\mu <\lambda$. 
We use \Cref{prop: second version refinado} 
to expand $\bfM_{\lambda}^{\succ \alpha_{i,j}}$. 
If the bad pair does not occur in that expansion we apply  induction on the dominance order (since $\mathbf{P}_{j,h}(\lambda)< \lambda$) and 
\Cref{prop: caso bueno}  to obtain \eqref{eq: casi casi final}. 
Otherwise, we also need to apply to  \Cref{lem: GranLemma} to the weight $\mathbf{v}_{j,h}(\lambda)$, which is easily seen to satisfy Condition $\mathbb{Z}$, to get \eqref{eq: casi casi final} in this case. 
\end{proof}

Finally, we are in position to prove our main theorem, this is, \Cref{Teo E} in the introduction, which we recall here for the reader's convenience.  

\begin{theorem}
     Let $\alpha_{i,j} \in \Phi^{\geq 2} $, $h=j-i+1$ and $\lambda \in X^+$. 
    Then, $ \bfM_{\lambda}^{\succ \alpha_{i,j} }$ can be written as a positive linear combination of the $h$-th pre-canonical basis. 
    In particular, $ \bfM_{\lambda}^{\succ \alpha_{n-h+1,n} }=\bfN^{h+1}_\lambda $ can be written as a positive linear combination of the $h$-th pre-canonical basis, which is the content of \Cref{Conj A}.  
\end{theorem}

\begin{proof}
  This   follows by a repeated application of \Cref{teo: theo D in the main}. 
\end{proof}

\bibliographystyle{alpha} 
\bibliography{references}

\end{document}